\let\oldvec\vec% Store \vec in \oldvec
\documentclass[invmat,numbook]{svjour} 
\let\newvec\vec% Store svjour vec in \newvec
\let\vec\oldvec% Restore \vec from \oldvec
\usepackage{amsmath}
\usepackage{calc}
\usepackage[T1]{fontenc}
\usepackage[full]{textcomp}
\usepackage[charter]{mathdesign}

\usepackage{mathtools}
\DeclareSymbolFont{cmcal}{OMS}{cmsy}{m}{n}
\SetSymbolFont{cmcal}{bold}{OMS}{cmsy}{b}{n}
\DeclareSymbolFontAlphabet{\mathcal}{cmcal}

\let\vec\newvec% Restore \vec from \newvec

\usepackage[all,cmtip]{xy}
\usepackage{xcolor}
\usepackage{enumitem} %to change item labels

\usepackage{mathtools}
\spnewtheorem{notation}[definition]{Notation}{\bfseries}{\itshape}

\usepackage[super]{nth}

\usepackage{etoolbox}
\tracingpatches
\makeatletter
\patchcmd{\@citeo}{\hskip0.1em}{\kern0.1em}{}{}
\patchcmd{\@citex}{\hskip0.1em}{\kern0.1em}{}{}
\makeatother

\newcommand{\eq}[1]{eq.\ (\ref{#1})}

\newcommand{\eqn}[2]{
\begin{equation}\label{#1}
#2
\end{equation}
}

\newcommand{\eqnalign}[2]{
	\begin{equation}
	\begin{aligned}\label{#1}
	#2
	\end{aligned}
	\end{equation}
}
\newcommand{\clapsubstack}[1]{
  \mathclap{\substack{#1}}
}

\makeatletter
\newcommand*{\Xbar}[1]{}%
\DeclareRobustCommand*{\Xbar}[1]{%
  \mathpalette\@Xbar{#1}%
}
\newcommand*{\@Xbar}[2]{%
  \sbox0{$#1\mathrm{#2}\m@th$}%
  \sbox2{$#1#2\m@th$}%
  \rlap{%
    \hbox to\wd2{%
      \hfill
      $\overline{%
        \vrule width 0pt height\ht0 %
        \kern\wd0 %
      }$%
    }%
  }%
  \copy2 %
}
\newcommand*{\Frozenbar}[1]{}%
\DeclareRobustCommand*{\Frozenbar}[1]{%
  \mathpalette\@Frozenbar{#1}%
}
\newcommand*{\@Frozenbar}[2]{%
  \sbox0{$#1\mathrm{W}\m@th$}%
  \sbox2{$#1#2\m@th$}%
  \rlap{%
    \hbox to\wd2{%
      \hfill
      $\overline{%
        \vrule width 0pt height\ht0 %
        \kern\wd0 %
      }$%
    }%
  }%
  \copy2 %
}
\makeatother

\catcode`\@=11 % This allows us to modify PLAIN macros.

\def\eqalign#1{\null\,\vcenter{\openup\jot\m@th
\ialign{\strut\hfil$\displaystyle{##}$&$\displaystyle{{}##}$\hfil
\crcr#1\crcr}}\,}

\catcode`\@=12 % at signs are no longer letters

\usepackage{isomath}  
\SetMathAlphabet{\mathbf}{normal}{OML}{mdput}{b}{n}

\def\mb#1{{\mathbf{#1}}}
\def\bm#1{{\boldsymbol{#1}}}

\newcommand{\cdotop}{\!\!\cdot\!}
\newcommand{\BCFTA}{\mathrm{BCFTA}}

\usepackage{tocbibind}

\newcommand{\naturalqed}{
\leavevmode\unskip\penalty9999 \hbox{}\nobreak\hfill
    \quad\hbox{$\natural$}
}

\def\a{\alpha}
\def\b{\beta}

\def\d{\delta}  
\def\e{\varepsilon} \def\ep{\epsilon}
  
\def\g{\gamma}  \def\G{\Gamma}
\def\k{\varkappa}

\def\l{\lambda}  \def\La{\Lambda}
\def\m{\mu}
\def\n{\nu}
\def\r{\rho}
\def\vr{\varrho}
  \def\O{\Omega}
\def\p{\psi}  
\def\s{\sigma}  \def\vs{\varsigma} \def\S{\Sigma}
\def\th{\theta}  \def\vt{\vartheta}
\def\t{\tau}
\def\w{\varphi}

\def\CF{{\cal F}}

\def\CZ{{\cal Z}}
\def\R{\mathbb{R}}
\def\C{\mathbb{C}}
\def\Z{\mathbb{Z}}

\def\I{\mathbb{I}}
\def\mB{\mathfrak{B}}

\def\mq{\mathfrak{q}}

\def\mP{\mathfrak{P}}
\def\mp{\mathfrak{p}}

\def\mC{\mathfrak{C}}
\def\mc{\mathfrak{c}}

\def\mm{\mathfrak{m}}

\def\mL{\mathfrak{L}}

\def\mg{\mathfrak{g}}
\def\mG{\mathfrak{G}}
\def\mf{\mathfrak{f}}
\def\mF{\mathfrak{F}}
\def\md{\mathfrak{d}}
\def\mD{\mathfrak{D}}

\def\mK{\mathfrak{K}}

\def\sB{\mathscr{B}}
\def\sC{\mathscr{C}}
\def\sD{\mathscr{D}}

\def\sF{\mathscr{F}}

\def\sK{\mathscr{K}}
\def\sG{\mathscr{G}}
\def\sL{\mathscr{L}}
\def\sM{\mathscr{M}}

\def\sP{\mathscr{P}}

\def\sU{\mathscr{U}}
\def\sV{\mathscr{V}}

\def\rd{\partial}

\newcommand{\Fr}[2]{\dfrac{#1}{#2}}
\newcommand{\fr}[2]{{\tfrac{#1}{#2}}}

\def\pr{\prime}
\def\ppr{{\prime\!\prime}}

\DeclareMathOperator{\Ker}{Ker}
\DeclareMathOperator{\proj}{pr}
\DeclareMathOperator{\eb}{eb}
\let\Im\relax
\DeclareMathOperator{\Im}{Im}
\def\gh{\mathit{gh}}

\def\cl{\mathit{cl}}
\def\Hom{\mathit{Hom}}
\def\Der{\mathit{Der}}
\def\Diff{\mathit{Diff}}
\def\End{\mathit{End}}
\def\Aut{\mathit{Aut}}

\def\kbar{{\color{blue}\hbar}}
\def\kkbar{[\![\kbar]\!]}

\newcommand{\category}[1]{\textbf{#1}}

\def\BQ{\mathit{\color{blue}BQFTA}}

\def\BC{\mathit{\color{blue}BCFTA}}

\newcommand{\Perm}{\mathrm{Perm}}

\usepackage{graphicx}
\makeatletter
\DeclareRobustCommand*\uell{\mathpalette\@uell\relax}
\newcommand*\@uell[2]{
  \setbox0=\hbox{$#1\ell$}
  \setbox1=\hbox{\rotatebox{10}{$#1\ell$}}
  \dimen0=\wd0 \advance\dimen0 by -\wd1 \divide\dimen0 by 2
  \mathord{\lower 0.1ex \hbox{\kern\dimen0\unhbox1\kern\dimen0}}
}

\def\bell{\pmb{\uell}}

\usepackage{shuffle}
\begin{document}

%\draftmode

\title{Homotopical Computations in Quantum Fields Theory}
\author{
Jae-Suk Park\thanks{jaesuk@postech.ac.kr}
%\thanks{This work was partially supported by IBS-R003-G1}
}

\institute{
Department of Mathematics, POSTECH, Pohang 37673, Republic of Korea
}

%
%\dedication{We're imperfect people pursuing perfect ideas, and 
%\\
%there's tremendous frustration in the gap.
%\\
%--- Marc Andreessen
%}
%

\maketitle

\begin{abstract}
This paper is a mathematical study of quantum correlation functions
in quantum field theory within a homotopy algebraic framework motivated from the BV 
quantization scheme. We
characterize quantum correlation functions
by algebraic homotopy theoretical methods which circumvent gauge fixing and perturbative Feynman diagrams.
We show that there is a universal algebraic structure, 
closely related with that of the WDVV equation,
governing quantum correlation functions 
of every quantum field theory in our framework up to a certain ambiguity.
The algebraic structure  is independent of
the details of the quantum expectation, other than
its existence with the prescribed symmetry, and comes with a concrete
algorithm for explicit computations.
We will also make  proposals for the precise  natures 
of quantum expectation and physical equivalence of quantum field theories.
%, which is based on a certain infinitesimal symmetry of the quantum expectation.
%Exploiting this symmetry systematically,  
%we attempt to  characterize quantum correlation functions
%by algebraic homotopy theoretical methods which circumvent gauge fixing and perturbative Feynman diagrams.
%We show that there is a universal algebraic structure, 
%closely related with that of the WDVV equation,
%governing quantum correlation functions 
%of every quantum field theory in our framework up to a certain ambiguity.
%The algebraic structure  is independent of
%the details of the quantum expectation, other than
%its existence with the prescribed symmetry, and comes with a concrete
%algorithm for explicit computations.
%We will also make  proposals for the precise  natures 
%of quantum expectation and physical equivalence of quantum field theories.
%We will treat the Planck constant  $\kbar$ as a formal parameter
%but our main results on the universal algebra governing quantum correlations will be exact.
\end{abstract}

\newpage

\setcounter{tocdepth}{3}
\tableofcontents
\parskip=1.8mm
\parindent=0mm

%\section{Introduction}
%
%
%This paper is a mathematical study of quantum correlation functions
%in quantum field theory within an algebraic framework derived from the Batalin--Vilkovisky 
%quantization scheme, which is based on a certain infinitesimal symmetry of the quantum expectation.
%Exploiting this symmetry systematically,  
%we attempt to  characterize quantum correlation functions
%by algebraic homotopy theoretical methods which circumvent gauge fixing and perturbative Feynman diagrams.
%We show that there is a universal algebraic structure, 
%closely related with that of the WDVV equation,
%governing quantum correlation functions 
%of every quantum field theory in our framework up to a certain ambiguity.
%The algebraic structure  is independent of
%the details of the quantum expectation, other than
%its existence with the prescribed symmetry, and comes with a concrete
%algorithm for explicit computations.
%We will also make  proposals for the precise  natures 
%of quantum expectation and physical equivalence of quantum field theories.
%We will treat the Planck constant  $\kbar$ as a formal parameter
%but our main results on the universal algebra governing quantum correlations will be exact.
%
%\setcounter{tocdepth}{3}
%\tableofcontents
%\parskip=1.8mm
%\parindent=0mm
%
%\subsection{Motivations}

\newpage

\section{Introduction}

\subsection{Background}

The  Batalin-Vilkovisky (BV) quantization scheme of classical field theory
is  a versatile framework
in dealing with a complete tower of infinitesimal
classical symmetries
--- the (gauge) symmetry of a classical action $S_{\cl}$, symmetries of the symmetry, etc.---
which provides physicists with not only  a general method of gauge fixing
but also a criterion for independence of  path integrals 
from gauge choices, so that  one may compute quantum correlation functions
by enumerating  Feynman diagrams after a suitable gauge choice \cite{BV}.
But ultimately, the scheme
is a homological algebraic
implementation of 
an infinitesimal symmetry of the quantum expectation 
that satisfies a Schwinger--Dyson type equation \cite{Dys,Sch,Henn,Park15}:
that is, the (non-existing) path integral measure is translation invariant before being twisted by $e^{-\fr{1}{\kbar}S_{\cl}}$.

Exploiting such a symmetry systematically,  
we attempt to  characterize quantum correlation functions
by algebraic homotopy theoretical methods which circumvent gauge fixing and perturbative Feynman diagrams.
We show that there is a universal algebraic structure, 
closely related with that of the WDVV equation,
governing quantum correlation functions 
of every quantum field theory in our framework up to a certain ambiguity.
The algebraic structure  is independent of
the details of the quantum expectation, other than
its existence with the prescribed symmetry, and comes with a concrete
algorithm for explicit computations.
We will also make  proposals for the precise  natures 
of quantum expectation and physical equivalence of quantum field theories.
We will treat the Planck constant  $\kbar$ as a formal parameter
but our main results on the universal algebra governing quantum correlations will be exact.

Algebraically, working over a ground field $\Bbbk$ of characteristic zero,  
we regard a BV quantization of a classical field theory as a blackbox producing:
\begin{itemize}
\item
a unital $\Z$-graded commutative and associative
algebra $\big(\sC, 1_\sC, \,\cdot\,\big)$ over $\Bbbk$,
\item an odd second order algebraic  differential operator
$\Delta$, satisfying $\Delta^2=\Delta 1_\sC=0$ 
whose failure to be a derivation
of the product $\cdot$ defines an odd Poisson bracket $(\ ,\,)_{\mathit{BV}}$,
and 
\item a quantum master action
$\mb{S}=S + \kbar S^{(1)}+\kbar^2 S^{(2)}+\cdots  \in \sC[\![\kbar]\!]$,
which is a solution to the quantum master equation: 
\[
\Delta e^{-\fr{1}{\kbar}\mb{S}}=0\Longleftrightarrow -\kbar\Delta \mb{S} +\Fr{1}{2}(\mb{S}, \mb{S})_{\mathit{BV}}=0
,
\]
and whose classical limit $S$ incorporates both the classical action $S_{\cl}$ and a complete tower of infinitesimal classical symmetries.\footnote{Given a classical field theory such as Yang--Mills theory, setting up and working out  its BV quantization are important mathematical challenges, for which we refer to Costello's book \cite{Cos}.}
\end{itemize}

We can organize those outcomes as a tuple $\big(\sC[\![\kbar]\!], 1_\sC, \,\cdot\,, \mb{K}\big)$, where $\big(\sC[\![\kbar]\!], 1_\sC, \;\cdot\;\big)$ is the topologically-free
$\Bbbk{\kkbar}$-algebra generated by $\big(\sC, 1_\sC, \,\cdot\,\big)$  and
$\mb{K}:\sC{\kkbar}\rightarrow \sC{\kkbar}$ is a $\Bbbk{\kkbar}$-linear differential defined by
$\mb{K}\coloneqq{} -\kbar \Delta +(\mb{S},\,)_{\mathit{BV}}$ 
%= \big(S, \,\big)_{\mathit{BV}} 
%+\kbar\left( -\Delta +\big(S^{(1)}, \,\big)_{\mathit{BV}}\right)
%+\kbar^2 \big(S^{(2)}, \,\big)_{\mathit{BV}} +\ldots
which satisfies $\mb{K}^2=\mb{K}1_\sC=0$ and,
for all homogeneous $\bm{x}_1,\bm{x}_2 \in \sC{\kkbar}$, 
\[
\mb{K}\big(\bm{x}_1\cdot \bm{x}_2\big)
- \mb{K}\bm{x}_1\cdot \bm{x}_2 -(-1)^{|\bm{x}_1|}\bm{x}_1\cdot \mb{K}\bm{x}_2 
=(-\kbar)\big(\bm{x}_1,\bm{x}_2\big)_{\mathit{BV}},
\]
while $(\ ,\,)_{\mathit{BV}}$ is a derivation of the product.
To summarize 
all those properties, we 
say that the tuple $\big(\sC[\![\kbar]\!], 1_\sC, \,\cdot\,, \mb{K}\big)$ is a BV-QFT algebra with quantum descendant
$\big(\sC[\![\kbar]\!], 1_\sC, \mb{K}, (\ ,\,)_{\mathit{BV}}\big)$,
which is a topologically-free unital sDGLA over $\Bbbk{\kkbar}$.\footnote{The acronym sDGLA stands for 
shifted differential graded Lie algebra; in such an algebra, the Lie bracket has degree $1$. This is a special case of an $sL_\infty$-algebra --- a homotopy Lie algebra with degrees shifted by one.}
Then a {\em quantum expectation} can be
interpreted as a cochain map $\bm{\mc}: \sC[\![\kbar]\!]\rightarrow \Bbbk[\![\kbar]\!]$ from the pointed cochain complex
$(\sC[\![\kbar]\!], 1_\sC, \mb{K})$ to  the pointed cochain complex $(\Bbbk[\![\kbar]\!], 1, 0)$ with zero differential,
both over $\Bbbk{\kkbar}$. The condition $\bm{\mc}(1_\sC)=1$ is a normalization and
the condition $\bm{\mc}\circ \mb{K}=0$ is the homological algebraic condition defining an {\em infinitesimal symmetry of the quantum expectation}. 
Another quantum expectation 
$\tilde{\bm{\mc}}$ is physically equivalent to $\bm{\mc}$
if they are homotopic $\bm{\mc}\sim \tilde{\bm{\mc}}$ as pointed cochain maps
or, equivalently, if they have the same cochain homotopy type $[\bm{\mc}]=[\tilde{\bm{\mc}}]$.
Such a variation of quantum expectation  
within its homotopy type corresponds to a change of gauge fixing in  the BV quantization
scheme and every quantum correlation function should be invariant of the homotopy type of the quantum expectation. 
%We like to exploits such the symmetry and redundancy systematically to characterize every quantum correlation functions up to certain ambiguity,  bypassing gauge fixing and perturbative Feynman path integrals.
%In our pursuit of such a direction, we will encounter both technical and conceptual issues in defining quantum correlators whose resolution may also shed new light on quantization scheme itself.

For example, we can consider a quantum observable  as a homogeneous element $\mb{O} \in \sC[\![\kbar]\!]$
satisfying $\mb{K}\mb{O}=0$ so that its expectation value $\langle\mb{O} \rangle_{\bm{\mc}}\coloneqq{} \bm{\mc}(\mb{O})$
depends only on the cochain homotopy type of  the quantum expectation  $\bm{\mc}$, i.e, $\langle\mb{O}\rangle_{\bm{\mc}}=
\langle\mb{O}\rangle_{\tilde{\bm{\mc}}}$ whenever  $\bm{\mc}\sim \tilde{\bm{\mc}}$.
Let $\mb{O}$ and $\tilde{\mb{O}}$ be quantum observables that belong to the same $\mb{K}$-cohomology class. 
Then, they have the same quantum expectation value $\langle\mb{O}\rangle_{\bm{\mc}}=\langle\tilde{\mb{O}}\rangle_{\bm{\mc}}$.
Therefore, we may consider the space of equivalence classes of quantum observables as the cohomology
$\mb{H}$ of the cochain complex $\big(\sC{\kkbar}, \mb{K}\big)$ on which the homotopy class of  quantum expectation
$\bm{\mc}$ induces
uniquely a $\Bbbk{\kkbar}$-linear map $\mb{\iota}:\mb{H}\rightarrow \Bbbk{\kkbar}$ such that $\mb{\iota}(1_{\mb{H}})=1$,
where $1_{\mb{H}}$ is the cohomology class of $1_\sC$.  Then, we might try to introduce an algebraic structure
on $\mb{H}$ governing quantum correlation functions  between and among quantum observables thought of as pairs, triples, and so on. However, we immediately run into an interesting conundrum. 

To wit, the products $\mb{O}\!\cdot\!\mb{O}$ and $\tilde{\mb{O}}\!\cdot\!\tilde{\mb{O}}$ of quantum
observables $\mb{O}$ and $\tilde{\mb{O}}$, which are assumed to have the same $\mb{K}$-cohomology class,  are not
quantum observables in general  and, even if they happen to be quantum observables by some accident, they may  belong to different  $\mb{K}$-cohomology classes, i.e.,
in general $\langle\mb{O}\!\cdot\!\mb{O}\rangle_{\bm{\mc}} \neq  \langle \tilde{\mb{O}}\!\cdot\!\tilde{\mb{O}}\rangle_{\bm{\mc}}$.
Therefore the naive
definition of $n$-fold quantum correlation functions among quantum observables as the quantum expectation value of
$n$-fold products of quantum observables is inadequate.
These problems in general originate from  the nature of the differential $\mb{K}$
which is not a derivation of the product. Rather the failure of $\mb{K}$ to be a derivation of the product $\cdot$ is divisible by $\kbar$. This property, on the other hand, is crucial in capturing quantum correlations:
Assume that $\mb{K}$
is a derivation of the product and
$\mb{O} - \langle\mb{O}\rangle_{\bm{\mc}}\cdot 1_\sC=\mb{K}\mb{\l}$. 
Then, the variance $\left\langle\big(\mb{O} - \langle\mb{O}\rangle_{\bm{\mc}}\cdot 1_\sC\big)^2\right\rangle_{\bm{\mc}}
\equiv\langle\mb{O}\!\cdot\!\mb{O}\rangle_{\bm{\mc}} - \langle\mb{O}\rangle_{\bm{\mc}}^2$ 
of any quantum observable $\mb{O}$ should always be zero. But we expect that the variance vanishes identically only in the classical limit
in the presence of quantum correlation.

A partial resolution of the above problems can be accomplished by adopting  the notion 
of a homotopical family of quantum observables.
We can regard a non-empty set $\{\mb{O}_a\}_{a \in I}$ of quantum observables as the image of 
a basis $\{e_a\}_{a \in I}$ of a $\Z$-graded vector space $V$ under a cochain map 
$\mb{\w}_1:\big(V{\kkbar}, 0\big) \rightarrow \big(\sC{\kkbar}, \mb{K}\big)$, 
i.e. $\mb{\w}_1:V{\kkbar}\rightarrow \sC{\kkbar}$ is a degree zero map satisfying 
$\mb{K}\circ\mb{\w}_1=0$ and $\mb{O}_a = \mb{\w}_1(e_a)$, for all $ a \in I$.  
On the other hand,
a homotopical  family of quantum observables is defined to be an $sL_\infty$-morphism 
$\xymatrix{\underline{\mb{\w}}: \big(V{\kkbar}, \underline{0}\big) \ar@{..>}[r] 
& \big(\sC[\![\kbar]\!], \mb{K}, (\ ,\,)_{\mathit{BV}}\big)}$,
where $\underline{\mb{\w}}= \mb{\w}_1, \mb{\w}_2, \ldots$
and $\mb{\w}_n=\w^{(0)}_n +\kbar \w^{(1)}+\ldots$ is a family of degree preserving $\Bbbk$-linear maps, 
parametrized by $\kbar$,
from the $n$th (super)symmetric power $S^n\! V$ of $V$ to $\sC$, i.e., $\mb{\w} \in \Hom\big(S^n\! V, \sC)^{0}{\kkbar}$.
Then,  there is an associated  family
$\underline{\mb{\Pi}}^{{\mb{\w}}}={\mb{\Pi}}^{{\mb{\w}}}_1, {\mb{\Pi}}^{{\mb{\w}}}_2,\ldots$ of  {\em quantum correlators},
where ${\mb{\Pi}}^{{\mb{\w}}}_n\in \Hom\big(S^n V, \sC)^{0}{\kkbar}$
satisfies
$\mb{K}\circ {\mb{\Pi}}^{{\mb{\w}}}_n=0$,
and the family of (joint) quantum moments can be defined as the following:
\[
\Big\{\left\langle {\mb{\Pi}}^{{\mb{\w}}}_n(e_{a_1}, \ldots, e_{a_n})\right\rangle_{\bm{\mc}}\Big|n\geq 1 \,;\, a_1,\ldots,a_n \in I\Big\}.
\]
This family is an invariant of the homotopy type of the quantum expectation $\bm{\mc}$. 
We also have 
$\left\langle {\mb{\Pi}}^{{\mb{\w}}}_n(e_{a_1}, \ldots, e_{a_n})\right\rangle_{\bm{\mc}} 
= \left\langle {\mb{\Pi}}^{\tilde{\mb{\w}}}_n(e_{a_1}, \ldots, e_{a_n})\right\rangle_{\bm{\mc}}$
for all $n\geq 1$ and $a_1,\ldots,a_n \in I$ whenever the two corresponding $sL_\infty$-morphisms $\underline{\mb{\w}}$ and 
$\underline{\tilde{\mb{\w}}}$ are homotopic. 
We can also view the homotopical family of 
quantum observables as the simultaneous quantum correction 
${\mb{\Pi}}^{{\mb{\w}}}_n(e_{a_1}, \ldots, e_{a_n})$ of every multiple product $\mb{O}_{a_1}\cdot\ldots \cdot \mb{O}_{a_n}$ 
of quantum observables in the set  $\{\mb{O}_a\}_{a \in I}$. For example, we have  
%${\mb{\Pi}}^{{\mb{\w}}}_1(e_{a_1})= \mb{O}_{a_1}$ and 
\begin{align*}
{\mb{\Pi}}^{{\mb{\w}}}_2(e_{a_1},e_{a_2})=& \mb{O}_{a_1}\cdotop\mb{O}_{a_2} -\kbar  \mb{O}_{a_1 a_2}
,\\
{\mb{\Pi}}^{{\mb{\w}}}_3(e_{a_1},e_{a_2}, e_{a_3})=& 
\mb{O}_{a_1}\cdotop  \mb{O}_{a_2}\cdotop  \mb{O}_{a_3}  
-\kbar  \mb{O}_{a_1 a_2}\cdotop  \mb{O}_{a_3}
-\kbar  \mb{O}_{a_1}\cdotop  \mb{O}_{a_2a_3}
-\kbar (-1)^{|\mb{O}_{a_1}||\mb{O}_{a_2}|} \mb{O}_{a_2}\cdotop  \mb{O}_{a_1a_3}
\\
&
+\kbar^2 \mb{O}_{a_1 a_2 a_3}
,
\end{align*}
etc., where $\mb{O}_{a_1\ldots \a_n}\coloneqq{}\mb{\w}_n(e_{a_1}, \ldots, e_{a_n})$. 
A homotopical family of quantum observables also corresponds to 
a formal deformation of  the quantum master action $\mb{S}$ to a new
quantum master action $\mb{S} + \mb{\Theta}^{\mb{\w}}$,  
where $\mb{\Theta}^{\mb{\w}} =t^a \mb{\w}_1(e_a) +\Fr{1}{2!}t^{a_1} t^{a_2} \mb{\w}_2(e_{a_1}, e_{a_2})+\ldots$
and  $\{t^a\}_{a\in I}$ are variables dual to $\{e_a\}_{a\in I}$.\footnote{We use Einstein summation convention throughout this paper.} 
In the deformation we have $\Delta e^{-\fr{1}{\kbar}\mb{S} + \mb{\Theta}^{\mb{\w}}}=0$ and 
homotopic $sL_\infty$-morphisms  produce equivalent deformations.  Then, we have the following generating function for the quantum moments --- depending only on the homotopy types
of $\bm{\mc}$ and $\underline{\mb{\w}}$ ---
\[
\qquad
\bm{\mc}\left( e^{-\fr{1}{\kbar} \mb{\Theta}^{\mb{\w}}}\right)
=1 +\sum_{\mathclap{n=1}}^\infty \Fr{1}{n!}t^{a_n}\cdots t^{a_1}
\left\langle {\mb{\Pi}}^{{\mb{\w}}}_n(e_{a_1}, \ldots, e_{a_n})\right\rangle_{\bm{\mc}}
.
\]

However, the above resolution is not entirely satisfactory since there are numerous inequivalent ways of extending
a cochain map $\mb{\w}_1:\big(V{\kkbar}, 0\big)\rightarrow \big(\sC{\kkbar}, \mb{K}\big)$ into non-homotopic
$sL_\infty$-morphisms. This means that quantum corrections to the products ${\mb{O}_{a_1}\cdot\ldots \cdot \mb{O}_{a_n}}$
can be arbitrary so that the interpretation of $\left\{\mb{\Pi}^{{\mb{\w}}}_n(e_{a_1}, \ldots, e_{a_n})\right\}$ as 
the family of joint quantum correlation functions among the set $\{\mb{O}_a\}_{a \in I}$ of quantum observables 
is ambiguous.  
The second approach, which is eventually our resolution to this problem adopts a kind of complimentary principle, swinging back and forth between the classical and quantum vantage points and will reveal to us, arguably, 
the true nature of quantum correlations.  

The classical limit $\big(\sC, 1_\sC,  \;\cdot\;, K\big)$ of the BV-QFT algebra $\big(\sC[\![\kbar]\!], 1_\sC, \;\cdot\;, \mb{K}\big)$ 
is a unital differential graded commutative algebra (CDGA) over $\Bbbk$, where $K \coloneqq{}\big(S,\ \big)_{BV}$ and
$S=S_{\cl}+\ldots$ is the classical limit of the quantum master action $\mb{S}$ 
so that $S$ satisfies the classical BV master equation 
$\big(S, S\big)_{BV}=0$. 
In contrast to the differential $\mb{K} =K -\kbar\big( \Delta - (S^{(1)},\ )_{\mathit{BV}}\big)+\ldots$, 
the classical differential $K$ is a derivation of the product. 
We  say an element $O\in \sC$ is an off-shell classical observable if  $KO=0$ 
and two off-shell classical observables are equivalent if they have the same $K$-cohomology class. 
Then any product $O_1\cdot\ldots \cdot O_n$ of off-shell classical observables $O_1,\ldots, O_n$ 
is an off-shell classical observable, whose equivalence class  depends only on the equivalence classes of  $O_1,\ldots, O_n$.
In fact, it is the cohomology $(H, 1_H, 0)$ of the classical pointed cochain complex $\big(\sC, 1_\sC, K\big)$ 
that is the arena of classical physics, where the differential $K$ can be viewed as an odd vector field on an appropriate
classical fields space whose vanishing loci is the classical equation motion space 
--- the moduli space defined by  classical equations of motion
modulo classical symmetry,
and $H$ corresponds to the space of $\Bbbk$-valued function(als) on the classical equation motion space.
Therefore, it is appropriate to define  a classical observable as a homogeneous element in $H$ --- an equivalence class of
off-shell classical observables.  

Then we will propose precise notions for a quantization of off-shell classical observables 
and corresponding quantum correlators among classical observables, which will lead us to a universal
and exactly computable
algebraic structure  governing every quantum correlation.

\subsection{The Results}

We will work with  {\em binary QFT algebras}, which are a natural generalization of BV-QFT algebras which share the same relevant universal properties.
A binary QFT algebra is a tuple $\sC{\kkbar}_\BQ=\big(\sC[\![\kbar]\!], 1_\sC, \,\cdot\,,\mb{K}\big)$,
where $\big(\sC[\![\kbar]\!], 1_\sC, \,\cdot\,\big)$ is a topologically-free unital $\Z$-graded commutative
and associative algebra over $\Bbbk{\kkbar}$ 
and
$\big(\sC[\![\kbar]\!], 1_\sC,\mb{K}\big)$
is a pointed and  topologically-free cochain complex over $\Bbbk{\kkbar}$, 
which satisfies a sequence of
$\kbar$-compatibility axioms between the product $\cdot$ and the differential $\mb{K}$---namely
the failure of $\mb{K}$ to be a derivation of $\cdot$ is divisible by $\kbar$ 
and the $n$th iterated failure is divisible by $\kbar^n$.  
Then  the role of the topologically-free sDGLA  associated to a BV-QFT algebra 
is played by a topologically-free unital $sL_\infty$-algebra  
$\big(\sC[\![\kbar]\!], 1_\sC, \underline{\bell}=\mb{K}, \bell_2,\bell_3,\cdots\big)$ over $\Bbbk{\kkbar}$,
called the quantum descendant  of $\sC{\kkbar}_\BQ$.

The classical limit $\big(\sC, 1_\sC, \,\cdot\,,K\big)$ of the binary QFT algebra $\sC{\kkbar}_\BQ$ 
is still  a unital CDGA  over $\Bbbk$.
Therefore, the underlying pointed cochain complex $\big(\sC, 1_\sC, K\big)$ over $\Bbbk$ 
is homotopy equivalent to its cohomology 
$(H, 1_H, 0)$, which is regarded as the arena of classical physics.
We fix  a homotopy equivalence $(f, h, s)$, where
$s: \sC \rightarrow \sC$ is a splitting and
both $f:H\rightarrow \sC$ and $h:\sC\rightarrow H$ are
pointed cochain quasi-isomorphisms.
In particular, $f$ is a $\Bbbk$-linear choice of a set of off-shell representatives of all classical observables.
A quantization of every off-shell classical observable is defined
as a deformation $(\mb{f},\mb{h}, \mb{s})$ of $(f, h, s)$ such that  $(\mb{f},\mb{h}, \mb{s})$  
is a homotopy equivalence between 
$\big(H{\kkbar}, 1_H, 0\big)$ and $\big(\sC[\![\kbar]\!], 1_\sC, \mb{K}\big)$
 as pointed cochain complexes over $\Bbbk{\kkbar}$.
In general, there are obstructions order by order in $\kbar$ to such a deformation. 
These obstructions can be organized into a differential  $\mb{\kappa} =\kbar \k^{(1)}+ \kbar^2 \k^{(2)}+\ldots$ in such a way that 
$(\mb{f},\mb{h}, \mb{s})$   is a homotopy equivalence between 
$\big(H{\kkbar}, 1_H, \mb{\kappa}\big)$ and $\big(\sC[\![\kbar]\!], 1_\sC, \mb{K}\big)$ as
pointed cochain complexes over $\Bbbk{\kkbar}$.
This construction, which is  an application of standard homological perturbation
theory,  is not unique but is as canonical as possible in the sense that 
everything depends at most on the choice of the splitting $s$
and, in particular, the condition $\mb{\kappa}=0$ is independent of this choice.
We say  a binary QFT algebra is anomaly-free if $\mb{\kappa}=0$, so that
any off-shell classical observable admits a quantization.

We will develop a general theory including the anomalous case but
here
we state the main theorems of this paper restricted to the anomaly-free case, 
which has a more straightforward physical interpretation.

\begin{theorem}\label{intra}
Let $\sC{\kkbar}_\BQ$ be an anomaly-free binary QFT algebra
with associated quantum descendant unital $sL_\infty$-algebra  $\big(\sC[\![\kbar]\!], 1_\sC, \underline{\bell}\big)$.
Then,  there is a distinguished unital $sL_\infty$-quasi-isomorphism
\[\xymatrix{\underline{\mb{\phi}}: \big(H{\kkbar}, 1_H, \underline{0}\big) \ar@{..>}[r] 
& \big(\sC[\![\kbar]\!], 1_\sC, \underline{\bell}\big)}\] with the following factorization and $\kbar$-finiteness property.

To state the property, define the associated family $\underline{\mb{\Pi}}=\mb{\Pi}_1,\mb{\Pi}_2,\ldots$ of
quantum correlators,
%which is regarded as  universal homotopical family of quantum observables of the theory.
%Let $\underline{\mb{\Pi}}^0=\mb{\Pi}^0_1,\mb{\Pi}^0_2,\ldots$ be  the associated quantum correlators,
for all $n\geq 1$ and  homogeneous $v_1,\ldots, v_n \in H$, as
\[
\mb{\Pi}_n({v}_1, \cdots, {v}_n))
\coloneqq{}
\sum_{\mathclap{\mp\in P(n)}}
(-\kbar)^{n-|\mp|}\e(\mp)\,
\mb{\phi}\big({v}_{B_1}\big)\cdot\dotsc\cdot\mb{\phi}\big({v}_{B_{|\mp|}}\big)
.
\]

The condition is then that for any quantum expectation $\bm{\mc}$ we have a $\Bbbk{\kkbar}$-linear functional 
$\mb{\iota}:H{\kkbar} \rightarrow \Bbbk{\kkbar}$ such that
the $n$-fold quantum correlation function $\bm{\mc}\circ \mb{\Pi}_n$ 
admits a factorization 
\[
\bm{\mc}\circ \mb{\Pi}_n =\mb{\iota}\circ \grave{\mb{\pi}}_n : S^n H{\kkbar}\longrightarrow \Bbbk{\kkbar}
,
\] 
where $\grave{\mb{\pi}}_1$ is the identity $\I_H$ map on $H$   and, for all $n\geq 2$,
\begin{itemize}

\item $\grave{\mb{\pi}}_n$ has polynomial dependence on $\kbar$ of degree not higher than $n-2$:
\[
\grave{\mb{\pi}}_n = \grave{\pi}^{(0)}_n+( -\kbar) \grave{\pi}^{(1)}_n+\ldots +(-\kbar)^{\color{blue}n-2}  \grave{\pi}^{(n-2)}_n,
\]
where $\grave{\pi}^{(j)}_n: S^n H \rightarrow H$ for $j=0,1,\ldots, n-2$ and

\medskip

\item 
$\grave{\mb{\pi}}_{n}(v_1,\ldots, v_{n-1}, 1_H)= \grave{\mb{\pi}}_{n-1}(v_1,\ldots, v_{n-1})$.
\end{itemize}
\end{theorem}

The first part of our theorem says that $H{\kkbar}$, viewed as an 
unital $sL_\infty$-algebra over $\Bbbk{\kkbar}$ with the trivial $sL_\infty$-structure
 $\big(H{\kkbar}, 1_H, \underline{0}\big)$, is quasi-isomorphic to
the unital $sL_\infty$-algebra  $\big(\sC[\![\kbar]\!], 1_\sC, \underline{\bell}\big)$,
whenever $\mb{\kappa}=0$.
It also says that there is a {\em distinguished}  unital $sL_\infty$-quasi-isomorphism $\underline{\mb{\phi}}$ between them.
The rest of the theorem explicates the consequences of this. 
%which  defines a canonical and universal quantum correlators   $\underline{\mb{\Pi}}$.

For a quantum expectation $\bm{\mc}:\sC{\kkbar}\rightarrow \Bbbk{\kkbar}$, 
we call $\mb{\iota}\coloneqq{}\bm{\mc}\circ \mb{f}:H{\kkbar}\rightarrow \Bbbk{\kkbar}$ the on-shell quantum expectation. We regard the on-shell quantum expectation as a family of $\Bbbk$-linear maps parametrized by $\kbar$ from $H$ to $\Bbbk$:
the quantum expectation value of 
a classical observable $v \in H$
is $\left\langle \mb{f}(v)\right\rangle_{\bm{\mc}}=\iota^{(0)}(v) +\kbar \iota^{(1)}(v)+\kbar^2\iota^{(2)}(v) +\ldots \in \Bbbk{\kkbar}$.
%We assume an agnostic viewpoint about such the on-shell quantum expectation and move on to 
%define the canonical quantum correlators to obtain some exact results on quantum correlations beyond the formal computations.
We call $\bm{\mc}\circ \mb{\Pi}_n$ the $n$-fold quantum correlation function,
and regard it as a family of $\Bbbk$-linear maps parametrized by $\kbar$ from $S^n H$ to $\Bbbk$.
The family of joint quantum moments of  a non-empty set of classical observables $\{v_1,\ldots, v_k\} \subset H$ is
defined to be the family
\[
\left\{
\left\langle \mb{\Pi}_n(v_{j_1},\ldots, \ldots,v_{j_n})\right\rangle_{\bm{\mc}}
=\bm{\mc}\circ \mb{\Pi}_n(v_{j_1},\ldots, \ldots,v_{j_n})\in \Bbbk{\kkbar}
\Big| n\geq 1;\;  1\leq j_1,\ldots, j_n \leq n
\right\}
\]
of invariants of the homotopy type of quantum expectation $\bm{\mc}$. 

The factorization property $\bm{\mc}\circ \mb{\Pi}_n =\mb{\iota}\circ \grave{\mb{\pi}}_n$ says that $n$-fold quantum correlation functions
are determined by the on-shell quantum expectation  $\mb{\iota}$ 
and $\grave{\mb{\pi}}_n$, which has the distinguished property that it only has polynomial dependence in $\kbar$ of degree at most $n-2$ for all $n\geq 2$.
This pushes convergence issues,  when $\kbar$ is no-longer treated as a formal variable, for quantum correlation functions 
to the on-shell quantum expectation, about which this paper takes an agnostic viewpoint.

There can be infinitely many different unital $sL_\infty$-quasi-isomorphisms
from $H{\kkbar}$ to $\sC{\kkbar}$. For each such quasi-isomorphism $\underline{\mb{\w}}$
we can define an associated family $\underline{\mb{\Pi}}^{\mb{\w}}$ of quantum correlators
and corresponding quantum correlation functions $\bm{\mc}\circ \underline{\mb{\Pi}}^{\mb{\w}}$
so that $\bm{\mc}\circ \mb{\Pi}_n^{\mb{\w}}= \mb{\iota}\circ \mb{\pi}_n^{\mb{\w}}$,
where ${\mb{\pi}}^{\mb{\w}}_n\coloneqq{} \mb{h}\circ  \mb{\Pi}_n^{\mb{\w}}$.  
However, if $\underline{\mb{\w}}$ is not homotopic to the distinguished morphism 
$\underline{\mb{\phi}}$ as a unital $sL_\infty$-morphism, then
${\mb{\pi}}^{\mb{\w}}_n$ is only a formal power series in $\kbar$ for all $n\geq 2$
in general, rather than a polynomial.
We call the $sL_\infty$-homotopy type
of $\underline{\mb{\phi}}$ a {\em quantum structure} on  $\sC{\kkbar}_\BQ$.

The proof of Theorem~\ref{intra} will be constructive: there is a concrete algorithm
to determine the family   $\underline{\grave{\mb{\pi}}}=\grave{\mb{\pi}}_1, \grave{\mb{\pi}}_2,\ldots$.
The next main theorem is about the internal structure of the family  $\underline{\grave{\mb{\pi}}}$:
that it is determined by the components $\grave{\pi}^{(n-2)}_n$ for $n\geq 2$.

\begin{theorem}\label{intrb}
There is a  family $\underline{\grave{m}}= \grave{m}_2, \grave{m}_3, \ldots$,
which determines $\grave{\mb{\pi}}_2, \grave{\mb{\pi}}_3, \ldots$
by the following recursive formula:  for all $n\geq 2$ and homogeneous $v_1,\ldots, v_n \in H$,
\[
\grave{\mb{\pi}}_{n} ({v}_1,\dotsc,{v}_n)
=
\sum_{\clapsubstack{\mp \in P(n)\\|B_{|\mp|}|=n-|\mp|+1\\ n-1\sim_\mp n  }}
(-\kbar)^{n-|\mp|-1}
\e(\mp)\; 
\grave{\mb{\pi}}_{|\mp|}\left({v}_{B_1}, \cdots, {v}_{B_{\mp-1}},
\grave{{m}}\big({v}_{B_{|\mp|}}\big)\right).
\]

Moreover, the family $\underline{\grave{m}}$ has the following properties:

\begin{itemize}
\item symmetry: 
$\grave{m}_n$ is a $\Bbbk$-linear map of degree $0$ from $S^n H$  to $H$ for all $n\geq 2$;

\medskip
\item unity: 
$\grave{m}_2(1_H, v_1)=v_1$, while $\grave{m}_{n}(1_H, v_1,\ldots, v_{n-1})=0$ 
for all $n\geq 3$;

\medskip
\item generalized associativity:  for all $n\geq 0$,
\begin{align*}
\sum_{\mathclap{\vs\subset [n]}} &
\e(\vs\sqcup\vs^c)
\grave{m}\big(v_{\vs}\otimes \grave{m}(v_{\vs^c}\otimes w_1\otimes w_2)\otimes w_{3} \big)
\\
&=\sum_{\mathclap{\vs\subset [n]}}
\e(\vs\sqcup\vs^c)
(-1)^{|v_{\vs^c}||w_1|}
\grave{m}\big(v_{\vs}\otimes w_{1}\otimes \grave{m}(v_{\vs^c}\otimes w_2\otimes w_3) \big)
.
\end{align*}
\end{itemize}
See Theorems \ref{lakib} and \ref{xdalem} for the full details.
%where
%$v_1,\ldots, v_n, w_1,w_2,w_3$ are homogeneous elements  in $H$
%and  the sums for generalize associativity
%are over all  subsets 
%$\vs\subset [n]$ of  the ordered set $[n]=\{1,2,\ldots, n\}$,
%$\vs^c$ is the complimentary to $\vs$,
%$v_\vs = v_{j_1}\otimes \ldots \otimes  v_{j_{\sharp(\vs)}}$
%if $\vs=\{j_1,\ldots, j_{\sharp(\vs)}\}$, $j_1 <\ldots< j_{\sharp(\vs)}$, 
%$\e(\vs\sqcup \vs^c)$ is the 
%the Koszul sign for the permutation 
%$\check{\s}\left(v_1\otimes \ldots \otimes v_n \right)=v_{\vs}\otimes v_{\vs^c}$,
%and $\grave{m}(v_1\otimes\ldots\otimes v_n)= \grave{m}_n(v_1,\ldots, v_n)$, for all $n\geq 2$.
\end{theorem}

We call the tuple $\big(H, 1_H, \underline{\grave{m}}\big)$ the {\em on-shell quantum correlation algebra}
of the anomaly-free binary QFT algebra $\sC{\kkbar}_\BQ$
and call $\underline{\grave{\mb{\pi}}}$  the family of iterated quantum correlation products
generated by $\underline{\grave{m}}$.

Here are some explicit expressions for $\underline{\grave{\mb{\pi}}}$ in terms of $\underline{\grave{m}}$
as the notation for these interesting relations are yet to be introduced: 
\begin{align*}
\grave{\mb{\pi}}_2(v_1, v_2)=& \grave{m}_2(v_1, v_2)
,\\
\grave{\mb{\pi}}_3(v_1, v_2, v_3)=& \grave{m}_2\big(v_1, \grave{m}_2(v_2, v_3)\big)
-\kbar \grave{m}_3(v_1, v_2, v_3) 
,\\
\grave{\mb{\pi}}_4(v_1, v_2, v_3, v_4)=&
\grave{m}_2\big(v_1, \grave{m}_2(v_2, \grave{m}_2(v_3,v_4)\big)
-\kbar \grave{m}_2\big(v_1, \grave{m}_3(v_2, v_3, v_4)\big)
\\
&
-\kbar (-1)^{|v_1|| v_2|} \grave{m}_2\big(v_2, \grave{m}_3(v_1, v_3, v_4)\big)
-\kbar \grave{m}_3\big(v_1, v_2, \grave{m}_2(v_3,v_4)\big) 
\\
&
+\kbar^2 \grave{m}_4(v_1,v_2,v_3,v_4)
.
\end{align*}

The following theorem will provide us an algorithm
to determine the family $\underline{\grave{m}}$ directly by a certain sequence of classical cohomology computations:
\begin{theorem}\label{intrc}
There is a family $\underline{\phi}^{-1} =\phi^{-1}_2,\phi^{-1}_3,\ldots$ 
of  $\Bbbk$-linear maps  $\phi^{-1}_n$ from $S^{n-2}H\otimes S^2 H$ to $\sC$ of degree $-1$
so that $\grave{m}_n=h\circ M_n$  for all $n\geq 2$,
where
\begin{align*}
M_n(v_1,&\ldots, v_n)
\\
\coloneqq{}
&
\sum_{\clapsubstack{\mp \in P(n)\\|\mp|=2\\ n-1\nsim_\mp n}}
\ep(\mp)
{\phi}\big({v}_{B_1}\big)\cdot {\phi}\big({v}_{B_{2}}\big)
-\sum_{\clapsubstack{\mp \in P(n)\\ \big|B_{|\mp|}\big|=n -|\mp|+1\\ n-1\sim_\mp n\\ \color{red} |\mp|\neq 1}}
\e(\mp)
{\phi}_{|\mp|}\Big({v}_{B_1},\ldots, {v}_{B_{|\mp|-1}}, \grave{{m}}({v}_{B_{|\mp|}})\Big)
\\
&
+\sum_{\clapsubstack{\mp \in P(n)\\ n-1 \sim_\mp n\\ \color{red} |\mp|\neq 1}}
\e(\mp)
\ell_{|\mp|}\left({\phi}\big(J{v}_{B_{1}}\big), \ldots,{\phi}\big(J{v}_{B_{|\mp|-1}}\big), \phi^{-1}\big({v}_{B_{|\mp|}}\big)\right)
,
\end{align*}
and $\underline{\phi}$ is the classical limit of the distinguished unital $sL_\infty$-quasi-morphism $\underline{\mb{\phi}}$.
\end{theorem}

Note the  classical limit $\big(\sC, 1_\sC, \underline{\ell}\big)$ of the quantum descendant algebra $\big(\sC[\![\kbar]\!], 1_\sC, \underline{\bell}\big)$ is a unital $sL_\infty$-algebra over $\Bbbk$. 
Then the classical limit $\underline{\phi}$ of the distinguished unital $sL_\infty$-quasi-isomorphism $\underline{\mb{\phi}}$ in Theorem \ref{intra}
is a distinguished  unital $sL_\infty$-quasi-isomorphism 
$\xymatrix{\underline{\phi}:\big(H,1_H, \underline{0}\big)\ar@{.>}[r]&
\big(\sC, 1_\sC, \underline{\ell}\big)}$.  
If we assume that $H$ is a finite dimensional $\Z$-graded vector space,
it follows that the  deformation functor defined by the Maurer--Cartan equation 
of the unital $sL_\infty$-algebra $\big(\sC, 1_\sC, \underline{\ell}\big)$
is pro-representable --- by the completed symmetric algebra $\widehat{S}(H^*)$ of the dual vector space $H^*$  of $H$ ---
so that we can associate to it a based smooth formal super-moduli space, where
the homotopy type of a unital $sL_\infty$-quasi-isomorphism from $H$ to $\sC$ can be viewed as affine
coordinates. (See \cite{Konst,BaKo}).
Therefore we can associate a based smooth formal super-moduli space
$\sM_o$ to an anomaly-free binary QFT algebra with finite dimensional classical cohomology $H$
and equip this moduli space with distinguished affine coordinates from the homotopy type of  $\underline{\mb{\phi}}$
which we call {\em quantum coordinates} as they came from a quantum structure.

The notion of quantum coordinates turns out to 
be closely related with that of special coordinates in the moduli space of 
type IIB topological string theory \cite{COGP,BCOV,BaKo} 
or the flat coordinates in the moduli space of universal unfoldings of an isolated
singularity \cite{Saito}, which is a crucial concept in mirror symmetry. 
Moreover, the on-shell quantum correlation algebra $\big(H, 1_H, \underline{\grave{m}}\big)$ is equivalent to a formal Frobenius manifold  without the flat metric. That is, $\underline{\grave{m}}$ satisfies the Witten--Dijkraaf--Verlinde--Verlinde (WDDV) equation \cite{DVV,WittenA}.

Consider Theorem \ref{intra} and assume that $H$ is a finite dimensional $\Z$-graded vector space over $\Bbbk$.
It is convenient to introduce homogeneous coordinates  $t_H=\{t^\a\}$ on $H$ so that
$\{\rd_\a =\rd/\rd\! t^\a\}$ form a homogeneous basis of $H$ with the distinguished element $\rd_0=1_H$.
Then we extend $\rd_\a$ as a derivation on $\Bbbk[\![t_H]\!]$. From $\underline{\mb{\phi}}$ and
$\underline{\grave{\mb{\pi}}}$, define
\begin{align*}
\mb{\Theta}\coloneqq{}&\sum_{\mathclap{n=1}}^\infty \Fr{1}{n!}t^{\a_n}\cdots t^{\a_1}
\mb{\phi}_n\big(\rd_{\a_1},\ldots, \rd_{\a_n}\big)
\in \big(\sC[\![t_H]\!]\big)^{0}{\kkbar}
,\\
\grave{\bm{T}}^\g \coloneqq{} &t^\g + \sum_{\mathclap{n=2}}^\infty \Fr{1}{n! (-\kbar)^{n-1}}t^{\a_n}\cdots t^{\a_1} 
\grave{\mb{\pi}}_{\a_1\cdots\a_n}{}^\g 
\in \Bbbk[\![t_H]\!][\![\kbar^{-1}]\!]
,
\end{align*}
where $\left\{\grave{\mb{\pi}}_{\a_1\cdots \a_n}{}^\g\right\}$ is the set of structure constants defined
for the operators $\grave{\mb{\pi}}_n$ as $\grave{\mb{\pi}}_n\big(\rd_{\a_1}, \ldots, \rd_{\a_n}\big) = \grave{\mb{\pi}}_{\a_1\cdots \a_n}{}^\g \rd_\g$.
Then,  we have $\rd_0\mb{\Theta}=1_\sC$ and
\[
\mb{K}e^{-\fr{1}{\kbar}\mb{\Theta}} =0
\Longleftrightarrow \mb{K}\mb{\Theta} 
+\Fr{1}{2!}\bell_2\big(\mb{\Theta}, \mb{\Theta}\big)+\Fr{1}{3!}\bell_3\big(\mb{\Theta}, \mb{\Theta}, \mb{\Theta}\big)
+\ldots=0,
\] 
so that $\mb{\Theta}$ is a universal solution to the Maurer--Cartan equation 
of the unital $sL_\infty$-algebra $\big(\sC[\![\kbar]\!], 1_\sC, \mb{K},\bell_2,\bell_3,\ldots\big)$,
corresponding to a distinguished choice of affine coordinates on the associated $\Z$-graded formal based moduli space $\sM_o$
with tangent space $H$.
We also have the generating function $\bm{\CZ}_{\bm{\mc}}$ 
of all quantum correlation functions with respect to a quantum expectation
$\bm{\mc}$ defined as follows:
\[
\bm{\CZ}_{\bm{\mc}}
\coloneqq{} \Big< e^{-\fr{1}{\kbar}\mb{\Theta}}\Big>_{\bm{\mc}}
= 1+\sum_{\mathclap{n=1}}^\infty \Fr{1}{n! (-\kbar)^n}t^{\a_n}\cdots t^{\a_1} 
\Big<\mb{\Pi}_n(\rd_{\a_1},\ldots, \rd_{\a_n})\Big>_{\bm{\mc}}.
\]
From the factorization property, we obtain 
the identity 
$
\bm{\CZ}_{\bm{\mc}} = 1 -\Fr{1}{\kbar} \grave{\bm{T}}^\g \big< \mb{f}(\rd_\g)\big>_{\bm{\mc}}
$
so that the quantum expectation values $\left\{ \big< \mb{f}(\rd_\g)\big>_{\bm{\mc}}\right\}$ and 
$\left\{\grave{\bm{T}}^\g\right\}$ determine every quantum correlation function.

Now  the unital $sL_\infty$-quasi-isomorphism $\underline{\mb{\phi}}^0$ being {\em distinguished}  means
that  $\{\grave{\bm{T}}^\g \}$ 
is a formal power series in $\kbar^{-1}$
since $\grave{\mb{\pi}}_n$ has at most degree $n-2$ polynomial dependence on $\kbar$.
If we consider the structure constants $\left\{ \grave{m}_{\a_1\cdots \a_n}{}^\g\right\}$
of $\underline{\grave{m}}$ in Theorem \ref{intrb}, i.e.,
$\grave{m}_n\big(\rd_{\a_1}, \ldots, \rd_{\a_n}\big) = \grave{m}_{\a_1\cdots \a_n}{}^\g \rd_\g$,
and define
\[
\grave{A}_{\a\b}{}^\g = \grave{m}_{\a\b}{}^\g  + \sum_{\mathclap{n=1}}^\infty \Fr{1}{n!}t^{\r_n}\cdots t^{\r_1}
\grave{m}_{\r_1\cdots\r_n \a\b}{}^\g
\in \Bbbk[\![t_H]\!],
\]
it can be checked that
$\big\{\grave{\bm{T}}^\g\big\}$ is the unique solution in formal power series in $\kbar^{-1}$ 
to the following system of formal differential equations:
\[
\kbar\rd_\a\rd_\b \grave{\bm{T}}^\g + \grave{A}_{\a\b}{}^\r\rd_\r \grave{\bm{T}}^\g=0
,\qquad
\rd_0 \grave{\bm{T}}^\g = \d_{0}{}^\g -\Fr{1}{\kbar} \grave{\bm{T}}^\g
,
\]
with the boundary conditions  $\grave{\bm{T}}^\g\big|_{t_H=0}=\rd_\b \grave{\bm{T}}^\g\big|_{t_H=0}-\d_\b{}^\g=0$,
where $\d_\b{}^\g$ denotes the Kronecker delta.
Finally, the properties of $\underline{\grave{m}}^0$ in Theorem \ref{intrb}
imply that $\big\{\grave{A}_{\a\b}{}^\g\big\}$ satisfy the following relations:

$\quad$-- unity: 
$\grave{A}_{0\b}{}^\g = \d_\b{}^\g$,

$\quad$--  super-commutativity:
$\grave{A}_{\a\b}{}^\g =(-1)^{|t^\a||t^\b|}\grave{A}_{\b\a}{}^\g$ and
$\rd_\a \grave{A}_{\b\g}{}^\s = (-1)^{|t^\a||t^\b|} \rd_\b \grave{A}_{\a\g}{}^\s$,

$\quad$--  associativity:
$\grave{A}_{\a\b}{}^\r\grave{A}_{\r \g}{}^\s =\grave{A}_{\b\g}{}^\r \grave{A}_{\a\r}{}^\s$,

so that the triple $\big(H\otimes \Bbbk[\![t_H]\!], \rd_0, \ast \big)$
is a unital super-commutative associative algebra over $\Bbbk[\![t_H]\!]$,
where $\rd_\a\ast \rd_\b  \coloneqq{} \grave{A}_{\a\b}{}^\g \rd_\g$.

The proof of our main theorems involves solutions to 
what we call the master equations for the families of quantum correlators
at levels $0$ and $1$,
where the family  $\underline{\mb{\Pi}}$  of quantum correlators 
in Theorem \ref{intra} is the level $0$ case. In the due course, it will become clear that
there should be a tower of quantum correlators for all levels, $n \geq 0$.
%,  and two quantum field theories
%can be considered as physically equivalent if and only if the corresponding towers of quantum correlation functions are isomorphic.
We shall need an another new and more versatile framework, beyond the scope of this paper, to characterize this total structure. 
Here we will use the master equation for level $1$ quantum correlators
as an auxiliary device in proving Theorems \ref{intrb} and \ref{intrc}.

We will also define morphisms and homotopy types of morphisms of binary QFT algebras
to form the category $\category{BQFTA}(\Bbbk)$ and the homotopy category $\mathit{ho}\category{BQFTA}(\Bbbk)$ of binary QFT algebras, and 
the (homotopy) category of BV-QFT algebras
will arise as a full subcategory of $(\mathit{ho})\category{BQFTA}(\Bbbk)$.  

A binary QFT algebra is both a topologically-free unital $\Z$-graded commutative
and associative algebra over $\Bbbk{\kkbar}$ and a pointed 
and  topologically-free cochain complex over $\Bbbk{\kkbar}$ together with 
a set of $\kbar$-compatibility conditions between the differential and the product as was mentioned before.
A morphism of binary QFT algebras  is defined similarly as a pointed cochain map,
satisfying a corresponding set of $\kbar$-compatibility conditions involving the products 
in the source and target binary QFT algebras, so that the failure of being a algebra homomorphism
is divisible by $\kbar$ and  the $n$th iterated failure is divisible by $\kbar^n$ 
(the classical limit of a morphism of binary QFT algebras is a morphism of unital CDGAs over $\Bbbk$).
These $\kbar$-compatibility conditions
are crucial in organizing  quantum correlations algebraically and can themselves be assembled to give 
a functor  $\bm{\mK}:  \category{BQFTA}(\Bbbk)\rightsquigarrow \category{UsL}_\infty(\Bbbk{\kkbar})$,
called the quantum descendant functor,
to the category $\category{UsL}_\infty(\Bbbk{\kkbar})$
of unital $sL_\infty$-algebras over $\Bbbk{\kkbar}$. 
We define homotopy classes of  morphisms of binary QFT algebras 
so that $\bm{\mK}$ is a homotopy functor, i.e., it
induces a well-defined functor from the homotopy category 
$\mathit{ho}\category{BQFTA}(\Bbbk)$ to the homotopy category 
$\mathit{ho}\category{UsL}_\infty(\Bbbk{\kkbar})$ of unital $sL_\infty$-algebras.

An impetus for these definitions 
is to have the following:
\begin{theorem}\label{intrd}
A homotopy equivalence of anomaly-free binary QFT algebras induces an isomorphism of on-shell quantum correlation algebras
and, in particular, sends a flat structure to a flat structure.
\end{theorem}

We then define a {\em  binary QFT} as a diagram $\xymatrix{ \sC{\kkbar}_{\BQ}\ar[r]^-{\bm{\mc}} & \Bbbk{\kkbar}}$
in  the category $\category{BQFTA}(\Bbbk)$, where 
$\Bbbk{\kkbar}$ is regarded as a binary QFT algebra concentrated in degree zero
and  is actually an initial object in the category. 
We call the binary QFT algebra morphism  $\bm{\mc}$ a  strong quantum expectation,
as it is not only a pointed cochain map but also satisfies the $\kbar$-compatibility condition between the products
in $\sC$ and $\Bbbk$. Namely, let $\bm{\mK}(\bm{\mc})=\underline{\mb{\chi}} =\mb{\chi}_1, \mb{\chi}_2,\ldots$,  then we have
$\mb{\chi}_1=\bm{\mc}$ and
$(-\kbar)
\mb{\chi}_2(\bm{x}_1, \bm{x}_2)=
\mb{\chi}_1(\bm{x}_1\cdot\bm{x}_2)- \mb{\chi}_1(\bm{x}_1)\mb{\chi}_1(\bm{x}_2)$
etc. In particular, 
$\mb{\chi}_1(\bm{x}_1) =\bm{\mc}(\bm{x}_1)\equiv \langle \bm{x}_1\rangle_{\bm{\mc}}$
is the quantum expectation value of $\bm{x}_1$
and
$
(-\kbar) \mb{\chi}_2(\bm{x}_1, \bm{x}_2)=
\langle \bm{x}_1\cdot \bm{x}_2\rangle_{\!\bm{\mc}}-\langle \bm{x}_1\rangle_{\!\bm{\mc}}\langle  \bm{x}_2\rangle_{\!\bm{\mc}}$ 
is the covariance between $\bm{x}_1$ and $\bm{x}_2$,
which should be non-zero in general due to quantum correlations but vanishes in the classical limit. We may call
 $\underline{\mb{\chi}}$ 
 the family of quantum cumulants, or the family of connected quantum correlation functions,
measuring the strength and depth of quantum correlations between and among events thought of as singletons, pairs, triples, and so on. 
The family of quantum cumulants has the structure of a unital $sL_\infty$-morphism 
$\underline{\mb{\chi}}:\big(\sC{\kkbar}, 1_\sC, \underline{\bell}\big)\dasharrow \big(\Bbbk{\kkbar}, 1, \underline{0}\big)$
between the quantum descendant unital $sL_\infty$-algebras. 
It is natural to declare that
two binary QFTs $\xymatrix{ \sC{\kkbar}_{\BQ} \ar[r]^-{\bm{\mc}} & \Bbbk{\kkbar}}$
and $\xymatrix{ \sC^\pr{\kkbar}_{\BQ}\ar[r]^-{\bm{\mc}^\pr} & \Bbbk{\kkbar}}$ are 
physically equivalent if the following diagram in  the category $\category{BQFTA}(\Bbbk)$ of binary QFT algebras
is commutative up to homotopy 
\[
\xymatrix{
\ar[dr]_-{\bm{\mc}}\sC{\kkbar}_{\BQ}\ar[rr]^-{\bm{\mf}} & &\sC^\pr{\kkbar}_{\BQ} 
\ar[dl]^-{\bm{\mc}^\pr}
\\
&\Bbbk{\kkbar}&
}
\]
and  $\bm{\mf}$ is a homotopy equivalence of binary QFT algebras
---
Physically equivalent binary QFTs have isomorphic quantum correlation functions.

%In the due course it will become clear that the notion binary QFT algebra 
%need to be further generalized, and the ultimate generalization will be simply
%called a QFT algebra.  Then, there should be category
%$\category{QFTA}(\Bbbk)$  and homotopy category $\mathit{ho}\category{QFTA}(\Bbbk)$ of QFT algebras
%such that there is a homotopy equivalent structure of QFT algebra on its classical cohomology.
%This, in particular, implies that there should be a structure of QFT algebra on the classical cohomology of a BV-QFT algebra
%as well homotopy equivalent to the  BV-QFT algebra as a QFT algebra, but this will be another story.

\subsection{Organization}

This paper is organized as follows:

In Sect.~\ref{section: binary QFT algebras}, we define  the (homotopy) category $(\mathit{ho})\category{BQFTA}(\Bbbk)$ of binary QFT algebras
together with the homotopy functor 
$\bm{\mK}:  \category{BQFTA}(\Bbbk)\rightsquigarrow \category{UsL}_\infty(\Bbbk{\kkbar})$.

In Sect.~\ref{section: binary qft}, we define the notion of a homotopical family of quantum observables and discuss several consequences of the definition.
%We will also make a proposal that a quantum expectation associated to a binary QFT algebra is 
%not just a pointed cochain map but is a morphism
%of binary QFT algebra to the initial object $\Bbbk{\kkbar}$ of $\mathit{ho}\category{BQFTA}(\Bbbk)$.
We show that homotopy equivalent binary QFT algebras have isomorphic sets of homotopical families of quantum observables.

In Sect.~\ref{section: quantization},  we study  quantization of the classical off-to-on-shell retract
and prove an important technical lemma called homotopy $\kbar$-divisibility. We also comment on
the dependence of quantization of classical observables on the data of the classical off-to-on-shell retract.

In Sect.~\ref{section: mastering qc}, we define the levels zero and one master equations for quantum correlators and find
canonical solutions without assuming the anomaly-free condition. This section is the technical core
of this paper

In Sect.~\ref{section: anomaly-free theory}, we specialize to the anomaly-free case, leading to Theorems \ref{intra}, \ref{intrb}, \ref{intrc} and \ref{intrd}
together with some physical interpretations.  We further specialize to the case that $H$ is finite dimensional
to discuss relationships between the universal algebraic structure governing quantum correlations
and the WDDV equation and some simple but historically important examples.

We have a pedagogical  Appendix on the theme of $sL_\infty$-algebras. 
The first part is a self-contained review on the category and homotopy category of $sL_\infty$-algebras. 
The second part is a  sketch of the recipe for constructing a classical BV master action out of a classical action
incorporating both the complete tower of classical symmetries and subsequent gauge fixing.
The key point is that every notion in an off-shell formalism for classical physics 
should be defined modulo the classical equation of motion and 
that the concomitant coherence issues are naturally resolved using the language of $sL_\infty$-morphisms. 
%The third part is 
%a sketch of commutative homotopy probability theory, which is about infinitesimal symmetry
%of expectation of commutative algebraic probability space combined with the notion of classical independence,
%related with the homotopy category of $sL_\infty$-algebras  
%--- the theory  was developed also as a preparation for this paper \cite{Park15}.
%
\subsection{Acknowledgement}

I would  like to thank Gabriel Drummond-Cole for a proofreading.
I am  grateful to 
Serguei Barannikov, Daniel Bennequin, Xiaojun Chen, Cheolhyun Cho, Gabriel Drummond-Cole, Fr\'ed\'eric H\'elein,
Minhyong Kim, Calin Larzaroiu, Changzhen Li, Si Li, Jeehoon Park, Koiji Saito, Jim Stasheff, Mathieu Stienon, Dennis Sullivan,  John Terilla,  
Thomas Tradler, Siye Wu, Philsang Yoo,  Mahmoud Zeinalian for  discussions on related subjects.

\section{The homotopy category of binary QFT algebras}
\label{section: binary QFT algebras}

The main purpose of this section is to define
the category $\category{BQFTA}(\Bbbk)$ and homotopy category $ho\category{BQFTA}(\Bbbk)$
of binary QFT algebras over $\Bbbk$. 
A BV-QFT algebra is a special kind of binary QFT algebra and the category
$\category{BV-QFTA}(\Bbbk)$ and homotopy category $\mathit{ho}\category{BV-QFTA}(\Bbbk)$
category of BV QFT algebras shall be defined as full subcategories of 
$\category{BQFTA}(\Bbbk)$ and $\mathit{ho}\category{BQFTA}(\Bbbk)$, respectively.
We also
consider briefly the (homotopy) category $(\mathit{ho})\BCFTA(\Bbbk)$ of binary CFT algebras,
where a binary CFT algebra appears as as the combined classical limit of a binary QFT algebra and its quantum descendant.\footnote{The acronyms
QFT and CFT stand wishfully for Q(uantum) F(ield) T(heory) and C(lassical) F(ield) T(heory), respectively.}

\subsection{Notation}

Fix a ground field $\Bbbk$ of characteristic zero, usually $\R$ or $\C$.

Let ${V}$ be  a $\Z$-graded $\Bbbk$-vector space. We shall call the $\Z$-grading
the \emph{ghost number} and denote it $\gh$, i.e., we have a decomposition 
${V} = \bigoplus_{j \in \Z} {V}^{j}$ 
and an element $v\in {V}^j$ is said to have ghost number $j=\gh(v)$. 
We often use the notation $(-1)^{|v|}$ instead of $(-1)^{\gh(v)}$ for the
parity $\pm 1$ of $v$ as well as the notation $J v= (-1)^{|v|} v$.

We use  the notation $\otimes$ for the tensor product over $\Bbbk$,
the notation $T^n {V}$  for the $n$th tensor power 
of ${V}$ and the notation $\Xbar{T}({V}) = \bigoplus_{n=1}^\infty T^n {V}$
for the reduced tensor module generated by $V$. 
We use  the notation $\odot$ for the symmetric product over $\Bbbk$,
the notation $S^n {V}$  for the $n$th symmetric power 
of ${V}$ and the notation $\Xbar{S}({V}) = \bigoplus_{n=1}^\infty S^n {V}$
for the reduced symmetric module generated by $V$. 
Note that $S^n {V}$ is generated by expressions $\{v_1\odot\ldots \odot v_n\}$
where  $v_1,\dotsc, v_n \in {V}$ are homogeneous elements in $V$
up to the equivalence relation that 
$v_1\odot\ldots \odot v_n=\ep(\s)v_{\s(1)}\odot\ldots\odot v_{\s(n)}$ for all $\s \in \Perm_n$,
where $\ep(\s)=\pm 1$ is the Koszul sign determined by decomposing a permutation $\s$ 
of  the ordered set $[n]=\{1,2,\ldots, n\}$ as a composite of transpositions and applying the (super)-commutativity  
of $\odot$ (namely that $v_1\odot v_2=(-1)^{|v_1||v_2|}v_2\odot v_1$).
All these modules have $\Z$-gradings by ghost number induced
from that of ${V}$. 

We treat the Planck constant $\kbar$ as a formal parameter with $\gh(\kbar)=0$.
A $\Z$-graded $\Bbbk{\kkbar}$-module $\mb{V}=\bigoplus_{j \in \Z} \mb{V}^{j}$ is topologically free if it is isomorphic to 
${V}{\kkbar}=\bigoplus_{j \in \Z} {V}^{j}{\kkbar}$,
where ${V}$ is a  $\Z$-graded $\Bbbk$-vector space and
\[
{V}^j{\kkbar}\coloneqq{}\left\{\sum_{\mathclap{n\geq 0}} \kbar^n v^{(n)} \big| v^{(n)}\in {V}^j\right\}.
\]
The completed tensor product 
$\bm{{V}}\bm{\otimes} \bm{{W}}$
of $\mb{{V}}\otimes_{\Bbbk{\kkbar}} \mb{{W}}$, where $\mb{{V}}$ and 
$\mb{{W}}$ are  $\Bbbk{\kkbar}$-modules,
is defined to be the inverse limit 
$\varprojlim \;\big(\mb{{V}}\otimes_{\Bbbk{\kkbar}} \mb{W}\big)/\kbar^n 
\cdot \big(\mb{{V}}\otimes_{\Bbbk{\kkbar}}\mb{W}\big)$.
The corresponding completed symmetric product is denoted by $\bm{{V}}\bm{\odot} \bm{{W}}$.
If $\mb{{V}}\cong {V}{\kkbar}$ and $\mb{W}\cong {W}{\kkbar}$ are topologically free, then 
both $\mb{{V}}\bm{\otimes} \mb{W}$
and  $\mb{{V}}\bm{\odot} \mb{W}$ are also topologically free: we have
$\mb{{V}}\bm{\otimes} \mb{W} \cong ({V}\otimes W){\kkbar}$ and  
$\mb{{V}}\bm{\odot} \mb{W} \cong ({V}\odot W){\kkbar}$.
Therefore we use the notation 
$T^n{V}{\kkbar}$ for $\overbrace{{V}{\kkbar}\bm{\otimes} \ldots \bm{\otimes} {V}{\kkbar}}^n$,
and the notation $S^n{V}{\kkbar}$ for 
$\overbrace{{V}{\kkbar}\bm{\odot} \ldots \bm{\odot} {V}{\kkbar}}^n$,
etc. 

The space of  $\Bbbk$-linear maps  between 
$\Z$-graded vector spaces ${V}$ and ${W}$
is denoted by $\Hom\big({V},{W}\big)=\bigoplus_{j\in\Z} \Hom\big({V},{W})^j$,
where $\Hom\big({V},{W}\big)^j$ is the space of $\Bbbk$-linear maps increasing the ghost number by $j$.
Let $\mb{\b}= \b^{(0)}+\kbar \b^{(1)}+\kbar^2 \b^{(2)}+\cdots$  be a family parametrized by $\kbar$, with $\b^{(n)} \in \Hom(V,W)^j$.
Then $\mb{\b}$ determines a $\Bbbk{\kkbar}$-linear map, denoted by the same symbol, 
between the topologically-free modules $V{\kkbar}$ and $W{\kkbar}$ increasing the ghost number by $j$
by $\kbar$-adic continuity: for all $\bm{v}=v^{(0)}+\kbar v^{(1)} +\ldots \in V{\kkbar}$ we have
\[
\mb{\b}(\bm{v})
= \sum_{\mathclap{n=0}}^\infty \kbar^n \sum_{\mathclap{i=0}}^n  \b^{(n-i)}\big(v^{(i)}\big),
\]
and the converse is also true. 
In other words, a $\Bbbk{\kkbar}$-linear map $\mb{\b}:{V}{\kkbar}\rightarrow W{\kkbar}$
is determined by its restriction to ${V}$.
Accordingly, 
the space of all  $\Bbbk{\kkbar}$-linear maps from $V{\kkbar}$ to $W{\kkbar}$
will be denoted by  
\[\Hom ({V},W){\kkbar}=\bigoplus_{j\in \Z}\Hom(V,W)^j{\kkbar}.\]
The projection $\b^{(0)}$
of $\mb{\b} \in \Hom ({V},W){\kkbar}$
to $\Hom({V},W)$ is said to be the {\em classical
limit}  of $\mb{\b}$, and we often us $\b$ as shorthand notation for $\b^{(0)}$.
Composition of  $\mb{\b} \in \Hom({U} ,{V})^j{\kkbar}$ and
$\mb{\g}\in \Hom({V}, {W} )^k{\kkbar}$
is defined in a similar fashion:
\[
\mb{\g}\circ \mb{\b}
=\sum_{\mathclap{n=0}}^\infty\kbar^n \sum_{\mathclap{i=0}}^n \g^{(n-i)}\circ \b^{(i)}
\in  \Hom\big({U}, {W} \big)^{j+k}{\kkbar}.
\]

We often use the notation $\mb{\b}_n(\bm{v}_1,\ldots, \bm{v}_n)$ in place of either 
$\mb{\b}_n(\bm{v}_1\bm{\otimes}\ldots\bm{\otimes} \bm{v}_n)$ for $\mb{\b}_n \in \Hom (T^n{V},W){\kkbar}$,
or $\mb{\b}_n(\bm{v}_1\bm{\odot}\ldots\bm{\odot} \bm{v}_n)$ for $\mb{\b}_n \in \Hom (S^n{V},W){\kkbar}$.
It is obvious that
a family $\underline{\mb{\b}} =\mb{\b}_1, \mb{\b}_2,\ldots$ of $\mb{\b}_n \in \Hom\big(S^n {V}, {W}\big)^j{\kkbar}$ for $n\geq 1$
determine  uniquely a ${\mb{\b}} \in  \Hom\big(\Xbar{S}({V}), {W}\big)^j{\kkbar}$ and vice versa such that
${\mb{\b}}_n ={\mb{\b}}\circ \eb_{{S}^{^n}\!\!{V}{\kkbar}}$, for all $n\geq 1$, where 
$\eb_{S^n {V}{\kkbar}}:  S^n{V}{\kkbar} \rightarrow \Xbar{S}({V}){\kkbar}$
is the natural embedding. We use the notation $\underline{\mb{\b}}$ and $\mb{\b}$ interchangeably.

We shall denote an element of ${V}{\kkbar}$ by a bold letter, 
i.e., $\mb{v}\in {V}{\kkbar}$, and an element of ${V}$ by an italic letter, i.e.,
$x \in {V}$, and will write the formal power series expansion of $\mb{v}\in{V}{\kkbar}$ 
as
$\mb{v} = v^{(0)} + \kbar v^{(1)} +\kbar^2 v^{(2)}+\cdots$.
We shall often denote $v^{(0)}$ by $v$ and say that $v$ is the classical limit of $\mb{v}$.

A \emph{partition} of the ordered  set $[n]=\{1,2,\ldots, n\}$ is a decomposition 
$\mp = B_1\sqcup B_2\sqcup \cdots\sqcup B_{|\mp|}$
into pairwise disjoint and non-empty subsets $B_i$ called {\em blocks}. 
We denote the number of blocks 
in the partition $\mp$ by $|\mp|$ and the size of a block $B$ by $|B|$. 
We shall use the {\em strictly ordered} representation for a partition. 
That is, blocks are ordered by the maximum element of each block and 
each block is ordered via the ordering induced from the natural numbers. 
We denote the set of all partitions of $[n]$ by $P(n)$.
For example, we have
$P(1)=\{1\}$, $P(2)=\{1,2\}, \;\{1\}\sqcup \{2\}$ and
\[
P(3)=\{1,2,3\}, 
\;\;\{1,2\}\sqcup \{3\},
\;\;\{2\}\sqcup \{1,3\},
\;\;\{1\}\sqcup \{2,3\},
\;\;\{1\}\sqcup\{2\}\sqcup\{3\}.
\]
%A partition $\mp=B_1\sqcup B_2\sqcup \cdots \sqcup B_{|\mp|}$ of the set $[n]$ 
%is called {\em interval} if $B_1=\{1,2,\cdots, |B_1|\}$ and, for $2\leq k \leq |\mp|$,
%\[
%B_k=\left\{\sum_{\mathclap{j=1}}^{k-1}|B_j|+1,\sum_{\mathclap{j=1}}^{k-1}|B_j|+2,\cdots, \sum_{\mathclap{j=1}}^{k-1}|B_j|+|B_k|\right\},
%\]
%by the obvious reason.
%We denote the set of all interval partitions of $[n]$ by $I(n)$.
%For example, we have
%$I(1)=P(1)$, $I(2)=P(2)$ and
%\[
%I(3)=\{1,2,3\}, 
%\;\;\{1,2\}\sqcup \{3\},
%%\;\;\{2\}\sqcup \{1,3\},
%\;\;\{1\}\sqcup \{2,3\},
%\;\;\{1\}\sqcup\{2\}\sqcup\{3\}
%\]
For $k,k^\pr$ in $[n]$, we use the notation 
$k \sim_\mp k^\pr$ if both $k$ and $k^\pr$ belong to the same block in the partition $\mp$ 
and the notation $k\nsim_\mp k^\pr$ otherwise.
For a given set of $n$ homogeneous elements $\bm{v}_1,\dotsc, \bm{v}_n \in {V}{\kkbar}$,
the Koszul sign $\e(\mp)$ for a partition $\mp =B_1\sqcup \cdots \sqcup B_{|\mp|} \in P(n)$
means the Koszul sign $\e(\s)$ associated with the permutation $\s \in \mathit{Perm}_n$, i.e. the unique permuation $\s$ that satisfies the equation
\[
\bm{v}_{B_1}\bm{\otimes} \bm{v}_{B_2}\bm{\otimes}\ldots\bm{\otimes} \bm{v}_{B_{|\mp|}}
=\bm{v}_{\s(1)}\bm{\otimes}\bm{v}_{\s(2)}\bm{\otimes} \ldots\bm{\otimes} \bm{v}_{\s(n)},
\]
where $\bm{v}_B = \bm{v}_{j_1}\bm{\otimes} \bm{v}_{j_2}\bm{\otimes} \ldots\bm{\otimes}\bm{v}_{j_{|B|}}$ 
if $B=\left\{{j_1}, {j_2}, \cdots, {j_{|B|}}\right\}$.
For a family $\underline{\mb{\b}}=\mb{\b}_1,\mb{\b}_2,\dots $ of $\mb{\b}_n\in \Hom\big(S^n {V}, W\big)^{j}$, for all $n\geq 1$, 
and  a block $B=\{j_1,\ldots, j_r\}$ of a partition $\mp \in P(n)$, 
the notation $\mb{\b}\big(\bm{v}_{B}\big)$  is taken to mean $\mb{\b}_r\big(\bm{v}_{j_1},\dotsc, \bm{v}_{j_r}\big)$.

\subsection{Binary QFT algebras}

Let $\sC$ be a $\Z$-graded vector space over $\Bbbk$.
A structure of a binary QFT algebra on $\sC$ 
will consist of the structure of a QFT complex and the structure of a unital super-commutative and associative algebra
satisfying an $\kbar$-compatibility condition.

\begin{definition}
The structure of a QFT complex  on $\sC$ is a tuple $\big(\sC{\kkbar}, 1_\sC, \mb{K}\big)$,
where $1_\sC \in \sC^0$ and  $\mb{K}= K^{(0)}+\kbar K^{(1)}+\kbar^2 K^{(2)}+\ldots \in\Hom (\sC,\sC)^1{\kkbar}$,  with the properties that
$\mb{K}1_\sC=\mb{K}\circ \mb{K}=0$ and $1_\sC\neq K^{(0)} x$ for all $x\in \sC$.
\end{definition}
We remark that the conditions $\mb{K}1_\sC=\mb{K}\circ \mb{K}=0$ are equivalent to 
the set of  conditions that  $K^{(n)}1_\sC =\sum_{\mathclap{j=0}}^n K^{(j)}\circ K^{(n-j)}=0$ for all $n\geq 0$,
and  the condition that $1_\sC\neq K^{(0)} x$ for all $x\in \sC$ implies
that $1_\sC\neq \mb{K} \bm{x}$ for all $\bm{x}\in \sC{{\kkbar}}$ --- hence the
$\mb{K}$-cohomology class of $1_\sC$ is non-trivial.
We shall usually denote $K^{(0)}$ by $K$, so that $K \in \Hom\big(\sC, \sC\big)^1$ 
and $K\circ K=K1_\sC=0$.
Therefore,  the tuple $(\sC, 1_\sC, K)$ is a pointed cochain complex over $\Bbbk$ 
such that the $K$-cohomology class of the cocycle $1_\sC$ is non-trivial.
We call $\big(\sC, 1_\sC, {K}\big)$ the classical limit of the
QFT complex  $\big(\sC{\kkbar}, 1_\sC, \mb{K}\big)$.

In the geometrical picture of the BV quantization scheme, the classical limit  
$K$ of $\mb{K}$ corresponds to  
an odd nilpotent vector field on the space of all fields and anti-fields 
whose vanishing locus in the space of all classical fields is the solution space
of the classical (on-shell) equations of motion.  
Then the  $K$-cohomology is the space of functions on this on-shell motion space
modulo classical symmetry (the space of classical observables modulo classical equivalence).
This motivates the following  definition. 

\begin{definition}
A QFT complex  is on-shell if the differential vanishes in the classical limit.
\end{definition}

A unital $\Z$-graded commutative associative algebra on $\sC$
is a tuple $\big(\sC, 1_\sC, \;\cdot\;\big)$, where $1_\sC \in \sC^0$ and $\cdot$ is a $\Bbbk$-bilinear product on $\sC$, 
so that
we have $\gh(x\cdot y)=\gh(x)+\gh(y)$, 
$x\cdot 1_\sC = x$, $x\cdot y =(-1)^{|x||y|}y\cdot x$ and $x\cdot (y\cdot z) =(x\cdot y)\cdot z$
for all homogeneous elements $x,y,z \in\sC$.
Then,  we have the canonical structure $\big(\sC{\kkbar}, 1_\sC, \;\cdot\;\big)$ of 
a unital $\Z$-graded commutative associative algebra over $\Bbbk{\kkbar}$ on the topologically-free
module $\sC{\kkbar}$
with the  product $\bm{x}\cdot \bm{y} \coloneqq{}\sum_{{n=0}}^\infty\kbar^n \sum_{{i=0}}^n x^{(i)}\cdot y^{(n-i)}$
for all $\bm{x}=x^{(0)}+\kbar x^{(1)}+\ldots  \in \sC{\kkbar}$ and $\bm{y}=y^{(0)}+\kbar y^{(1)}+\ldots  \in \sC{\kkbar}$.

\begin{definition}\label{bBKftalg}
A QFT complex $\big(\sC{\kkbar}, 1_\sC, \;\cdot\;, \mb{K}\big)$ structure on $\sC$ is 
$\kbar$-compatible with a unital $\Z$-graded commutative associative algebra structure
$\big(\sC, 1_\sC, \,\cdot\,\big)$ if 
the family $\underline{\bell} = {\bell}_1, {\bell}_2,\ldots$ defined
recursively  for all $n\geq 1$ and homogeneous $\bm{x}_1,\ldots, \bm{x}_n \in \sC{\kkbar}$ via the equation
\eqn{defqdesalg}{
\mb{K}\big(\bm{x}_1\cdots \bm{x}_n\big)
=
\sum_{\clapsubstack{\mp \in P(n)\\ |B_i| = n-|\mp|+1}}
(-\kbar)^{n-|\mp|}\ep(\mp)
J\!\bm{x}_{B_1}\cdot\ldots \cdot J\!\bm{x}_{B_{i-1}}\cdot \bell(\bm{x}_{B_i})
\cdot \bm{x}_{B_{i+1}}\cdot\ldots\cdot \bm{x}_{B_{|\mp|}}
,
}
has the property that $\bell_n \in \Hom(S^n \sC, \sC)^1{\kkbar}$ for all $n\geq 1$ (a priori there should be negative powers of $\kbar$).
Then, 
\begin{itemize}
\item
we call the tuple $\sC{\kkbar}_\BQ\coloneqq{}\big(\sC{\kkbar}, 1_\sC, \;\cdot\;, \mb{K}\big)$ a \emph{binary QFT algebra structure} on $\sC$, and

\medskip
\item we call the tuple $\bm{\mK}(\sC{\kkbar}_\BQ )\coloneqq{}\big(\sC{\kkbar}, 1_\sC,\underline{\bell}\big)$ the \emph{quantum descendant} of 
the binary QFT algebra $\sC{\kkbar}_\BQ$.
\end{itemize}
\end{definition}

\begin{remark}
Note that the recursive formula \eq{defqdesalg} for the family $\underline{\bell}$   can be rewritten as follows:
\begin{align*}
(-\kbar)^{n-1}\bell_n(\bm{x}_1,\ldots, \bm{x}_n) =
&\mb{K}\big(\bm{x}_1\cdot\ldots\cdot \bm{x}_n\big)
\\
&
-\sum_{\clapsubstack{\mp \in P(n)\\ |B_i| = n-|\mp|+1\\ \color{red}|\mp|\neq 1}}
(-\kbar)^{n-|\mp|}\e(\mp)\,
J\!\bm{x}_{B_1}\cdot\ldots\cdot J\!\bm{x}_{B_{i-1}}\cdot \bell(\bm{x}_{B_i})
\cdot \bm{x}_{B_{i+1}}\cdot\ldots\cdot \bm{x}_{B_{|\mp|}},
\end{align*}
where the right-hand-side of the above depends only on $\bell_1,\ldots, \bell_{n-1}$.
Therefore, the formula determine the family $\underline{\bell}$ uniquely.
Note also that the condition $ |B_i| = n-|\mp|+1$ implies
that the blocks $B_1,\ldots, B_{i-1}, B_{i+1},\ldots, B_{|\mp|}$ are singletons.
For example, we have $\bell_1 =\mb{K}$ and
\begin{align*}
(-\kbar)\bell_2(\bm{x}_1,\bm{x}_2)& = \mb{K}(\bm{x}_1\cdot \bm{x}_2) 
- \mb{K}\bm{x}_1\cdot \bm{x}_2 - J\!\bm{x}_1\cdot \mb{K}\bm{x}_2
,\\
(-\kbar)^2\bell_3(\bm{x}_1,\bm{x}_2,\bm{x}_3)& = \mb{K}(\bm{x}_1\cdot \bm{x}_2\cdot \bm{x}_3) 
- \mb{K}\bm{x}_1\cdot \bm{x}_2 \cdot \bm{x}_3- J\!\bm{x}_1\cdot \mb{K}\bm{x}_2\cdot \bm{x}_3
-J\!\bm{x}_1\cdot J\!\bm{x}_2\cdot \mb{K}\bm{x}_3 
\\
&\ \ 
-(-\kbar)\left(
\bell_2(\bm{x}_1, \bm{x}_2)\cdot \bm{x}_3 
+ J\!\bm{x}_1\cdot \bell_2(\bm{x}_2, \bm{x}_3) 
+ (-1)^{|\bm{x}_1||\bm{x}_2|}J\!\bm{x}_2\cdot \bell_2(\bm{x}_1,\bm{x}_3)\right)
.
% -(-\kbar)\bell_2(\bm{x}_1, \bm{x}_2)\cdot \bm{x}_3 
% -(-\kbar) J\!\bm{x}_1\cdot \bell_2(\bm{x}_2, \bm{x}_3) 
% -(-\kbar) (-1)^{|\bm{x}_1||\bm{x}_2|}J\!\bm{x}_2\cdot \bell_2(\bm{x}_1,\bm{x}_3)
% .
\end{align*}
Again, without imposing any compatibility, the definition of $\underline{\bell}$ would only imply that $(-\kbar)^{n-1}\bell_n$ was in $\Hom(S^n \sC, \sC)^{1}{\kkbar}$. The
$\kbar$-compatibility condition requires that in fact $\bell_n$ is in $\Hom(S^n \sC, \sC)^{1}{\kkbar}$ for all $n\geq 1$,
imposing a set of non-trivial relations between the product 
$\cdot$ and the differential $\mb{K}$.
\naturalqed
\end{remark}

\begin{lemma}\label{qdslinftyalgebra}
The quantum descendant $\bm{\mK}(\sC{\kkbar}_\BQ )=\big(\sC{\kkbar}, 1_\sC, \underline{\bell}\big)$ 
of a binary QFT algebra  $\sC{\kkbar}_\BQ$
is a topologically-free  unital $sL_\infty$-algebra over $\Bbbk{\kkbar}$, i.e.,
for all $n\geq 1$ and homogeneous $\bm{x}_1,\ldots, \bm{x}_{n}\in \sC{\kkbar}$,
%we have $\bell_n \in  \Hom(S^n \sC, \sC)^{1}{\kkbar}$ and
we have $\bell_n \in  \Hom(S^n \sC, \sC)^{1}{\kkbar}$ and
\eqnalign{uslinftyalgebra}{
\bell_{n}\big(\bm{x}_1,\ldots, \bm{x}_{n-1}, 1_\sC\big)=0
,\\
\\
\sum_{\clapsubstack{|\mp| \in P(n)\\ |B_i|=n-|\mp|+1}}\e(\mp)\,\bell_{|\mp|}\!\left(
J\bm{x}_{B_1},\ldots, J\bm{x}_{B_{i-1}}, \bell(\bm{x}_{B_{i}}), \bm{x}_{B_{i+1}}, \ldots, 
\bm{x}_{B_{|\mp|}}\right)=0
.
}
\end{lemma}

\begin{proof}
By definition, we have $\bell_n \in \Hom\big(S^n \sC, \sC\big)^{1}{\kkbar}$ for all $n\geq 1$.
Since $\bell_1=\mb{K}$ and $\mb{K}1_\sC=0$, we have $\bell_1(1_\sC)=0$. 
It can be checked  that
$\bell_{n}\big(\bm{x}_1,\ldots, \bm{x}_{n-1}, 1_\sC\big)=0$
for all $n\geq 1$ and $\bm{x}_1,\ldots, \bm{x}_{n}\in \sC{\kkbar}$
by an easy induction. It remains to show the second set of relations in \eq{uslinftyalgebra}.

From the family $\underline{\bell}=\bell_1,\bell_2,\ldots$ 
of $\bell_n \in \Hom(S^n \sC, \sC)^{1}{\kkbar}$,
we define a $\Bbbk{\kkbar}$-linear operator $\mb{\d}_{{\bell}}:\Xbar{S}(\sC){\kkbar}\rightarrow \Xbar{S}(\sC){\kkbar}$
with $\gh\big(\mb{\d}_{{\bell}}\big)=1$
by defining, for all $n\geq 1$ and homogeneous 
$\bm{x}_1,\ldots, \bm{x}_n \in \sC{\kkbar}$,
\eqnalign{qcoder}{
\mb{\d}_{{\bell}}({\bm{x}}_1 \bm{\odot} &\dotsc \bm{\odot} {\bm{x}}_n)\coloneqq{}
\\
&
\sum_{\clapsubstack{\mp \in P(n)\\ |B_i| = n-|\mp|+1}}
(-\kbar)^{n-|\mp|}\e(\mp)\,
J\!{\bm{x}}_{B_1}\bm{\odot} \dotsc\bm{\odot} J\!{\bm{x}}_{B_{i-1}}\bm{\odot} 
{\bell}({\bm{x}}_{B_i})\bm{\odot} {\bm{x}}_{B_{i+1}}
\bm{\odot} \dotsc\bm{\odot} {\bm{x}}_{B_{|\mp|}}
.
}
It is straightforward to check that 
for all $n\geq 1$ and  homogeneous $\bm{x}_1,\ldots, \bm{x}_n \in \sC{\kkbar}$,
\eqnalign{qxcoder}{
\proj_{\sC{\kkbar}}
& \circ\mb{\d}_{{\bell}}\circ  \mb{\d}_{{\bell}}({\bm{x}}_1
\bm{\odot}\dotsc\bm{\odot} {\bm{x}}_n)
\\
=&(-\kbar)^{n-1}\sum_{\clapsubstack{|\mp| \in P(n)\\ |B_i|=n-|\mp|+1}}
\e(\mp)\bell_{|\mp|}\left(
J\!\bm{x}_{B_1},\ldots, J\!\bm{x}_{B_{i-1}}, \bell(\bm{x}_{B_{i}}), \bm{x}_{B_{i+1}}, \ldots, 
\bm{x}_{B_{|\mp|}}\right).
}
Therefore, all that remains is to establish that
$\proj_{\sC{\kkbar}}  \circ\mb{\d}_{{\bell}}\circ  \mb{\d}_{{\bell}}=0$.

From the $\Z$-graded commutative and associative product $\cdot$ of $\sC{\kkbar}_\BQ$,
we define an operator $\mb{\pi} \in  \Hom\left(\Xbar{S}(\sC),\sC\right)^{0}{\kkbar}$ 
for all $n\geq 1$ and  ${\bm{x}}_1,\ldots,{\bm{x}}_n \in \sC{\kkbar}$ as
\eqn{qcoderz}{
\mb{\pi}\big({\bm{x}}_1\bm{\odot}\ldots\bm{\odot} {\bm{x}}_n\big)\coloneqq{}{\bm{x}}_1\cdot\ldots\cdot {\bm{x}}_n
.
}
Let $\mb{\pi}_n: =\mb{\pi}\circ \eb_{S^n\sC{\kkbar}}$, for all $n\geq 1$.
We note that $\mb{\pi}_1 =\I_{\sC{\kkbar}}$, the identity map on $\sC{\kkbar}$.

Now the set of relations in \eq{defqdesalg} defining the family $\underline{\bell}$
can be rewritten as follows:
\eqn{qdera}{
\mb{K}\circ \mb{\pi} =\mb{\pi}\circ \mb{\d}_{\bell},
}
which implies that $\mb{\pi}\circ \mb{\d}_{{\bell}}\circ  \mb{\d}_{{\bell}}=0$,
since $(\mb{\pi}\circ \mb{\d}_{{\bell}})\circ  \mb{\d}_{{\bell}}
=\mb{K}\circ (\mb{\pi}\circ  \mb{\d}_{{\bell}})=(\mb{K}\circ\mb{K})\circ \mb{\pi}$ and $\mb{K}\circ\mb{K}=0$.
From $\mb{\pi}_1 =\I_{\sC{\kkbar}}$, we conclude that 
$\proj_{\sC{\kkbar}}\circ \mb{\d}_{{\bell}}\circ  \mb{\d}_{{\bell}}=0$.
\naturalqed
\end{proof}

\begin{remark}\label{qderd}
Consider the reduced symmetric coalgebra 
$\Xbar{S}^{\mathit{co}}\!(\sC)=\left(\Xbar{S}(\sC){\kkbar},
\Xbar{\blacktriangle}\right)$ cogenerated by $\sC$ (see Appendix~\ref{appendix: bar}.)
The $\Bbbk$-linear coproduct 
\[\Xbar{\blacktriangle}:\Xbar{S}(\sC)\rightarrow \Xbar{S}(\sC)\otimes \Xbar{S}(\sC)\]
uniquely induces
a $\Bbbk{\kkbar}$-linear coproduct, denoted by the
same symbol, 
\[\Xbar{\blacktriangle}:\Xbar{S}(\sC){\kkbar}\rightarrow \Xbar{S}(\sC){\kkbar}\bm{\otimes} \Xbar{S}(\sC){\kkbar}\]
by $\kbar$-adic continuity so that 
$\Xbar{S}^{\mathit{co}}\!(\sC){\kkbar}=\left(\Xbar{S}(\sC){\kkbar}, \Xbar{\blacktriangle}\right)$
is a (topologically-free)  $\Z$-graded cocommutative coalgebra over $\Bbbk{\kkbar}$.

Define $\breve{\bell} \in  \Hom\left(\Xbar{S}(\sC),\sC\right)^{1}{\kkbar}$ as
$\breve{\bell}_n\coloneqq{}(-\kbar)^{n-1}\bell_n$, where $\breve{\bell}\circ \eb_{S^n\sC{\kkbar}}$.
Consider  the unique extension $\mD(\breve{\bell})$ of $\breve{\bell}$ 
as a coderivation  on  $\Xbar{S}^{\mathit{co}}\!(\sC){\kkbar}$:
we have,  for all $n\geq 1$ and  homogeneous $\bm{x}_1,\ldots, \bm{x}_n \in \sC{\kkbar}$,
\[
\mD(\breve{\bell})({\bm{x}}_1 \bm{\odot} \dotsc \bm{\odot} {\bm{x}}_n)
=
\sum_{\clapsubstack{\mp \in P(n)\\ |B_i| = n-|\mp|+1}}
\e(\mp)\,
J\!{\bm{x}}_{B_1}\bm{\odot} \dotsc\bm{\odot} J\!{\bm{x}}_{B_{i-1}}\bm{\odot} 
\breve{\bell}({\bm{x}}_{B_i})\bm{\odot} {\bm{x}}_{B_{i+1}}
\bm{\odot} \dotsc\bm{\odot} {\bm{x}}_{B_{|\mp|}}
.
\]
(See Lemma~\ref{fcoder} in Appendix~\ref{appendix: bar}.)
We note that $\mb{\d}_{{\bell}}=\mD(\breve{\bell})$. Also note
that the condition $\proj_{\sC{\kkbar}}\circ \mD(\breve{\bell})\circ \mD(\breve{\bell})=0$
actually implies that $\mD(\breve{\bell})\circ \mD(\breve{\bell})=0$. Therefore, we have $\mb{\d}_{\bell}\circ\mb{\d}_{\bell}=0$.
\naturalqed
\end{remark}

\begin{example}
The ground field $\Bbbk$ has the obvious structure $\big(\Bbbk{\kkbar}, 1, \cdot, 0\big)$
of a binary QFT algebra, denoted by $\Bbbk{\kkbar}$, with the zero differential $0$. The quantum descendant of this binary QFT algebra
is the unital $sL_\infty$-algebra  $\big(\Bbbk{\kkbar}, 1, \underline{0}\big)$, i.e., the
zero $sL_\infty$-structure $\underline{0}$ on $\Bbbk{\kkbar}$.
\end{example}

\begin{definition}
 A binary QFT algebra $\big(\sC{\kkbar}, 1_\sC, \;\cdot\;, \mb{K}\big)$ is a BV-QFT algebra 
if the quantum descendant $\big(\sC{\kkbar}, 1_\sC, \underline{\bell}\big)$ 
is a unital sDGLA,  i.e., $\bell_n=0$ for all $n\geq 3$, and $\bell_2$ does not depend
on $\kbar$, i.e., $\bell_2=\ell^{(0)}_2=(\ ,\,)_{\mathit{BV}} \in\Hom\big(S^2\sC,\sC\big)^1$.
\end{definition}

Fix a binary QFT algebra $\big(\sC{\kkbar}, 1_\sC, \;\cdot\;, \mb{K}\big)$ 
and let $\Big(\sC{\kkbar}, 1_\sC, \underline{\bell}\Big)$ be its
quantum descendant unital $sL_\infty$-algebra. 
It is straightforward to check the following lemmas:

\begin{lemma}\label{hiact}
For any $\mb{\g} \in \sC{\kkbar}$  and nilpotent $\mb{\vt} \in \sC^{0}{\kkbar}$, we have
\begin{align*}
\mb{K} \left( \mb{\g}\cdot e^{-\fr{1}{\kbar}\mb{\vt}}\right)   
=
&\left(\mb{K}\mb{\g} +\sum_{\mathclap{n=2}}^\infty\Fr{1}{(n-1)!} \bell_n\big(\mb{\vt},\ldots,\mb{\vt}, \mb{\g}\big)\right)
\cdot e^{-\fr{1}{\kbar}\mb{\th}}
\\
&
-\Fr{1}{\kbar} \left( \mb{K}\mb{\vt} 
+\sum_{\mathclap{n=2}}^\infty\Fr{1}{n!} \bell_n\big(\mb{\vt},\ldots,\mb{\vt}\big)\right)\cdot\mb{\g} \cdot e^{-\fr{1}{\kbar}\mb{\th}}
,
\end{align*}
where $e^{-\fr{1}{\kbar}\mb{\vt}}\coloneqq{}1_\sC -\Fr{1}{\kbar}\mb{\vt} +\Fr{1}{2!\kbar^2}\mb{\vt}\cdot \mb{\vt}
 -\Fr{1}{3!\kbar^3}\mb{\vt}\cdot \mb{\vt}\cdot \mb{\vt} +\ldots 
\in \sC(\!(\kbar)\!)^0$.
\end{lemma}

\begin{lemma}\label{remBKftalg}
We have $\bell_1=\mb{K}$
and, for all $n\geq 2$ and homogeneous $\bm{x}_1,\ldots, \bm{x}_n \in \sC{\kkbar}$,
\begin{align*}
-\kbar\bell_n\big(\bm{x}_1,\ldots, \bm{x}_n\big)  = &
\bell_{n-1}\big(\bm{x}_1,\ldots, \bm{x}_{n-2}, \bm{x}_{n-1}\cdot \bm{x}_n\big)
-\bell_{n-1}\big(\bm{x}_1,\ldots, \bm{x}_{n-2}, \bm{x}_{n-1}\big)\cdot \bm{x}_n
\\
&
-(-1)^{|\bm{x}_{n-1}|(|\bm{x}_1|+\ldots+|\bm{x}_{n-2}|)} 
J\bm{x}_{n-1}\cdot\bell_{n-1}\big(\bm{x}_1,\ldots, \bm{x}_{n-2}, \bm{x}_{n}\big),
\end{align*}
so that the failure of $\bell_n$ being a derivation of the product is divisible by $\kbar$.
\end{lemma}

%
%\begin{lemma}[Definition]
%\label{cldesalg}
%%Let $\mb{K}= K +\kbar K^{(1)}+\ldots$
%%and $\bell_n =\ell_n +\kbar \ell^{(1)}_n +\ldots$, for all $n\geq 1$, where $\ell_1= K$.
%%Then,
%The classical limit $\big(\sC, 1_\sC, \,\cdot\,, K\big)$ of the binary QFT algebra 
%is a unital  CDGA over $\Bbbk$ and
%the classical limit $\big(\sC, 1_\sC, \underline{\ell}\big)$ of 
%the quantum descendant algebra, where $\ell_1=K$, is a
%unital $sL_\infty$-algebra over $\Bbbk$
%and, for all $n\geq 2$ and homogeneous $x_1,\ldots, x_n \in \sC$, we have
%\begin{align*}
%\ell_{n-1}\big(x_1,\ldots, x_{n-2}, x_{n-1}\cdot x_n\big)=&
%\ell_{n-1}\big(x_1,\ldots, x_{n-2}, x_{n-1}\big)\cdot x_n
%\\
%&
%+(-1)^{|x_{n-1}|(|x_1|+\ldots+|x_{n-2}|)} Jx_{n-1}\cdot
%\ell_{n-1}\big(x_1,\ldots, x_{n-2}, x_{n}\big)
%.
%\end{align*}
%We say the tuple $\big(\sC, 1_\sC, \,\cdot\,, \underline{\ell}\big)$
%a binary CFT algebra.
%\end{lemma}
%
%

\subsection{Morphisms of binary QFT algebras}

A morphism of binary QFT algebras 
will be defined as a morphism of the underlying QFT complexes satisfying  
a $\kbar$-compatibility condition.

\begin{definition} 
Let 
$\big(\sC{\kkbar}, 1_\sC, \mb{K}\big)$ 
and $\big(\sC^\pr{\kkbar}, 1_{\sC^\pr},\mb{K}^\pr\big)$ be QFT complexes. 
A \emph{morphism} between them is a map 
$\bm{\mf}=\mf^{(0)} +\kbar \mf^{(1)}+\kbar^2 \mf^{(2)}+\ldots\in \Hom\big(\sC,\sC^\pr\big)^{0}{\kkbar}$
satisfying the conditions $\bm{\mf}(1_\sC) = 1_{\sC^\pr}$ and $\mb{K}^\pr\circ \bm{\mf}= \bm{\mf}\circ \mb{K}$.
A morphism of QFT complexes is a \emph{quasi-isomorphism} if it induces an isomorphism of $\Bbbk{\kkbar}$-modules
on cohomology.
\end{definition}

\begin{remark}
A QFT complex is a pointed cochain complex on a topologically-free $\Bbbk{\kkbar}$-module. The cohomology of a QFT complex is a $\Bbbk{\kkbar}$-module
but not necessarily a topologically-free $\Bbbk{\kkbar}$-module.
The classical limit $\mf^{(0)} \in  \Hom\big(\sC,\sC^\pr\big)^0$ of a morphism 
$\bm{\mf}$ of QFT complexes
is a pointed cochain map between the pointed cochain complexes $\big(\sC, 1_\sC, K\Big)$ and $\big(\sC^\pr, 1_{\sC^\pr}, K^\pr\Big)$, i.e., $f(1_\sC)= 1_{\sC^\pr}$ and $K^\pr\circ \mf^{(0)} =\mf^{(0)}\circ K$,
and $\mf^{(0)}$ is a quasi-isomorphism whenever $\bm{\mf}$ is a quasi-isomorphism.
\naturalqed
\end{remark}

Consider binary QFT algebras
$\sC{\kkbar}_\BQ=\left(\sC{\kkbar}, 1_\sC, \;\cdot\;, \mb{K}\right)$ 
and $\sC^\pr{\kkbar}_\BQ
=\left(\sC^\pr{\kkbar}, 1_{\sC^\pr}, \;\cdot^\pr\;, \mb{K}^\pr\right)$.

%\begin{definition}\label{hcompahomotopy}
%A $\kbar$-compatibility  homotopy from $\sC{\kkbar}_\BQ$ to $\sC^\pr{\kkbar}_\BQ$ is a family 
%$\underline{\mb{\eta}}=\mb{\eta}_1,\mb{\eta}_2,\ldots$
%of  $\mb{\eta}_n\in \Hom\big(S^n\sC, \sC^\pr\big)^{-1}{\kkbar}$ 
%such that $\mb{\eta}_1=0$ and,
%for all $n\geq 1$ and $\bm{x}_1,\ldots, \bm{x}_n \in \sC{\kkbar}$,
%\[
%\mb{\eta}_{n+1}(\bm{x}_1,\ldots, \bm{x}_{n},1_\sC)=\mb{\eta}_{n}(\bm{x}_1,\ldots, \bm{x}_{n})
%.
%\]
%We use the notations $\underline{\mb{\eta}}$ and $\mb{\eta} \in \in \Hom\big(\Xbar{S}(\sC), \sC^\pr\big)^{-1}{\kkbar}$
%interchangeably,  where $\mb{\eta}(\bm{x}_1\odot\ldots\odot \bm{x}_n)\coloneqq{} \mb{\eta}_n(\bm{x}_1,\ldots,\bm{x}_n)$
%for all $n\geq 1$ and $\bm{x}_1,\ldots, \bm{x}_n \in \sC{\kkbar}$.
%\end{definition}

\begin{definition}\label{BKftalgm}
A morphism of  QFT complexes
$\bm{\mf}:\big(\sC{\kkbar}, 1_\sC, \mb{K}\big)
\rightarrow\big(\sC^\pr{\kkbar}, 1_{\sC^\pr},\mb{K}^\pr\big)$
is \emph{$\kbar$-compatible} (with the products) up to homotopy if 
the family $\underline{\mb{\psi}}^{\!\bm{\mf}} = {\mb{\psi}}^{\!\bm{\mf}} _1, {\mb{\psi}}^{\!\bm{\mf}}_2,\ldots$ defined recursively for all $n\geq 1$ and homogeneous elements by the equation
$\bm{x}_1,\ldots, \bm{x}_n \in \sC{\kkbar}$,
\begin{align*}
\bm{\mf}\big(\bm{x}_1\cdot\ldots\cdot \bm{x}_n\big)=
&\sum_{\mathclap{\mp\in P(n)}}^{{}}
(-\kbar)^{n-|\mp|}
\ep(\mp)\mb{\psi}^{\!\bm{\mf}} \big(\bm{x}_{B_1}\big)\cdot^\pr\dotsc\cdot^\pr
\mb{\psi}^{\!\bm{\mf}} \big(\bm{x}_{B_{|\mp|}}\big),
\end{align*}
has the property that  $\mb{\psi}^{\!\bm{\mf}}_n \in \Hom\big(S^n \sC, \sC^\pr\big)^{0}{\kkbar}$ 
for all $n\geq 1$ (a priori it should have a negative power of $\kkbar$).
Then,

\begin{itemize}

\item
we call $\bm{\mf}$ a  \emph{morphism of binary QFT algebras} from $\sC{\kkbar}_\BQ$ to 
$\sC^\pr{\kkbar}_\BQ$, and

\medskip

\item
we call the family $\underline{\mb{\psi}}^{\!\bm{\mf}}$ the quantum descendant of $(\bm{\mf})$,  
and denote it by $\bm{\mK}(\bm{\mf})=\underline{\mb{\psi}}^{\!\bm{\mf}}$.
\end{itemize}

\end{definition}

\begin{remark}
We shall usually denote $\underline{\mb{\psi}}^{\!\bm{\mf}}$ by $\underline{\mb{\psi}}$
when the context is clear.
The recursive formula for $\underline{\mb{\psi}}=\bm{\mK}(\bm{\mf})$ in the above definition  can be rewritten as follows:
\begin{align*}
(-\kbar)^{n-1}\mb{\psi}_n(\bm{x}_1,&\ldots, \bm{x}_n)
=
\bm{\mf}\big(\bm{x}_1\cdot\ldots \cdot \bm{x}_n\big)
-\sum_{\clapsubstack{\mp \in P(n)\\ \color{red}|\mp|\neq 1}}(-\kbar)^{n-|\mp|}
\ep(\mp)\mb{\psi}\big(\bm{x}_{B_1}\big)\cdot^\pr\dotsc\cdot^\pr
\mb{\psi}\big(\bm{x}_{B_{|\mp|}}\big)
,
\end{align*}
where the right hand side of the equation above depends only on $\mb{\psi}_1,\ldots, 
\mb{\psi}_{n-1}$.
Therefore, the formula determines the family $\underline{\mb{\psi}}$ uniquely.
For example, we have $\mb{\psi}_1=\bm{\mf}$ and
\begin{align*}
(-\kbar)\mb{\psi}_2(\bm{x}_1,\bm{x}_2) =
&
\mb{\psi}_1(\bm{x}_1\cdot \bm{x}_2) 
- \mb{\psi}_1(\bm{x}_1)\cdot^\pr\!\mb{\psi}_1(\bm{x}_2)
,\\
(-\kbar)^2\mb{\psi}_3(\bm{x}_1,\bm{x}_2,\bm{x}_3) =
&
\mb{\psi}_1(\bm{x}_1\cdot \bm{x}_2\cdot \bm{x}_3) 
- \mb{\psi}_1(\bm{x}_1)\cdot^\pr\!\mb{\psi}_1(\bm{x}_2)\cdot^\pr\!\mb{\psi}_1(\bm{x}_3)
\\
&-(-\kbar)(\mb{\psi}_1(\bm{x}_1)\cdot^\pr\!\mb{\psi}_2(\bm{x}_2,\bm{x}_3)
+(-1)^{|\bm{x}_1||\bm{x}_2|}\mb{\psi}_1(\bm{x}_2)\cdot^\pr\!\mb{\psi}_2(\bm{x}_1,\bm{x}_3)
\\
&\hphantom{-(-\kbar)(}+\mb{\psi}_2(\bm{x}_1,\bm{x}_2)\cdot^\pr\!\mb{\psi}_1(\bm{x}_3))
\end{align*}
Again, it is only a priori true that $(-\kbar)^{n-1}\mb{\psi}_n \in  \Hom\big(S^n \sC, \sC^\pr\big)^{0}{\kkbar}$, while the
$\kbar$-compatibility condition imposes the further demand that $\mb{\psi}_n \in \Hom\big(S^n \sC, \sC^\pr\big)^{0}{\kkbar}$ for all $n\geq 1$ 
imposing a set of non-trivial restrictions on $\bm{\mf}$. 
\naturalqed
\end{remark}

Let  $\big(\sC{\kkbar}, 1_\sC, \underline{\bell}\big)=\bm{\mK}\left(\sC{\kkbar}_\BQ\right)$ 
and
$\big(\sC^\pr{\kkbar}, 1_{\sC^\pr}, \underline{\bell}^\pr\big)
=\bm{\mK}\left(\sC^\pr{\kkbar}_\BQ\right)$
be the quantum descendants, which are unital $sL_\infty$-algebras 
by {\em Lemma \ref{qdslinftyalgebra}}.

\begin{lemma}\label{BKftdesm}
The quantum descendant  
$\underline{\mb{\psi}}=\bm{\mK}(\bm{\mf})$ 
of a binary QFT algebra morphism   
$\xymatrix{\sC{\kkbar}_\BQ\ar[r]^-{\bm{\mf}}& \sC^\pr{\kkbar}_\BQ}$
is  a morphism of the descendant unital $sL_\infty$-algebras
$
\xymatrix{\big(\sC{\kkbar}, 1_\sC,  \underline{\bell}\big)
\ar@{..>}[r]^-{\underline{\mb{\psi}}}& \big(\sC^\pr{\kkbar}, 1_{\sC^\pr}, \underline{\bell}^\pr\big)
}$ .
That is, 
for all $n\geq 1$ and homogeneous elements $\bm{x}_1,\ldots, \bm{x}_{n}\in \sC{\kkbar}$,
%we have $\bell_n \in  \Hom(S^n \sC, \sC)^{1}{\kkbar}$ and
we have $\mb{\psi}_n \in \Hom\big(S^n \sC, \sC^\pr\big)^{0}{\kkbar}$ and
\eqnalign{uslinftymorphism}{
\mb{\psi}_n(\bm{x}_1,\ldots, \bm{x}_{n-1}, 1_\sC)-1_{\sC^\pr}\cdot \d_{n,1} &=0
,\\
&\\
\sum_{\clapsubstack{|\mp| \in P(n)}}
\e(\mp)
\bell^\pr_{|\mp|}\Big(
{\mb{\psi}}\big(\bm{x}_{B_1}\big),\ldots,{\mb{\psi}}\big( \bm{x}_{B_{|\mp|}}\big)
\Big) 
&
\\
-\sum_{\clapsubstack{|\mp| \in P(n)\\ |B_i|=n-|\mp|+1}}
\e(\mp)
{\mb{\psi}}_{|\mp|}\Big(J\bm{x}_{B_1},\ldots, J\bm{x}_{B_{i-1}}, \bell\big(\bm{x}_{B_{i}}\big), \bm{x}_{B_{i+1}},\ldots,\bm{x}_{B_{|\mp|}}\Big)&=0
.
}
\end{lemma}

\begin{proof}
%We need to show that,
%for all $n\geq 1$ and homogeneous 
%$\bm{x}_1,\ldots, \bm{x}_n \in \sC{\kkbar}$, we have
%$\mb{\psi}_n \in \Hom\big(S^n \sC, \sC^\pr\big)^{0}{\kkbar}$,
%$\mb{\psi}_n(\bm{x}_1,\ldots, \bm{x}_{n-1}, 1_\sC)=1_{\sC^\pr}\cdot \d_{n,1}$
%and
%\eqnalign{slinftyalgebra}{
%\sum_{\clapsubstack{|\mp| \in P(n)}}
%\e(\mp)
%\bell^\pr_{|\mp|}\Big(
%{\mb{\psi}}\big(\bm{x}_{B_1}\big),\ldots,{\mb{\psi}}\big( \bm{x}_{B_{|\mp|}}\big)
%\Big) 
%&
%\\
%-\sum_{\clapsubstack{|\mp| \in P(n)\\ |B_i|=n-|\mp|+1}}
%\e(\mp)
%{\mb{\psi}}_{|\mp|}\Big(J\bm{x}_{B_1},\ldots, J\bm{x}_{B_{i-1}}, \bell\big(\bm{x}_{B_{i}}\big), \bm{x}_{B_{i+1}},\ldots,\bm{x}_{B_{|\mp|}}\Big)&=0
%.
%}
By definition,
$\mb{\psi}_n \in \Hom\big(S^n \sC, \sC^\pr\big)^{0}{\kkbar}$ for all $n\geq 1$. 
We have $\mb{\psi}_1(1_\sC) =1_{\sC^\pr}$ since $\mb{\psi}_1=\bm{\mf}$ 
and $\bm{\mf}(1_\sC)=1_{\sC^\pr}$. 
It can be checked that
$\mb{\psi}_{n}\big(\bm{x}_1,\ldots, \bm{x}_{n-1}, 1_\sC) =0$
for all $n\geq 2$ by  a straightforward induction.
It remains to show the second set of relations in \eq{uslinftymorphism}.

From the family $\underline{\mb{\psi}}$, 
we define a $\Bbbk{\kkbar}$-linear map 
$\mb{\Psi}_{{\mb{\psi}}}\in \Hom\big(\Xbar{S}(\sC), \Xbar{S}(\sC^\pr)\big)^{0}{\kkbar}$
such that, for all $n\geq 1$ and  homogeneous $\bm{x}_1,\ldots, \bm{x}_n \in \sC{\kkbar}$,
\eqnalign{qcalm}{
\mb{\Psi}_{{\mb{\psi}}}({\bm{x}}_1\bm{\odot}\dotsc &\bm{\odot} {\bm{x}}_n)
\coloneqq{}\sum_{\clapsubstack{\mp \in P(n)}}
(-\kbar)^{n-|\mp|}\e(\mp)\,
\mb{\psi}\big({\bm{x}}_{B_1}\big)\bm{\odot} \dotsc\bm{\odot} \mb{\psi}\big({\bm{x}}_{B_{|\mp|}}\big)
.
}
Consider the pair of $\Bbbk{\kkbar}$-linear operators $\mb{\d}_{{\bell}}:\Xbar{S}(\sC){\kkbar}\rightarrow \Xbar{S}(\sC){\kkbar}$
and $\mb{\d}_{{\bell}^\pr}:\Xbar{S}(\sC^\pr){\kkbar}\rightarrow \Xbar{S}(\sC^\pr){\kkbar}$
associated with $\underline{\bell}$ and $\underline{\bell}^\pr$, respectively (See \eq{qcoder}).
It is straightforward see that, for all $n\geq 1$ and homogeneous 
$\bm{x}_1,\ldots, \bm{x}_{n}\in \sC{\kkbar}$,
\eqnalign{qcoderchain}{
\proj_{\sC^\pr{\kkbar}}
\circ & \big(\mb{\d}_{{\bell}^\pr}\circ\mb{\Psi}_{{\mb{\psi}}}  
-\mb{\Psi}_{\mb{\psi}}\circ\mb{\d}_{{\bell}}\big)
\big(\bm{x}_1\bm{\odot}\ldots\bm{\odot}\bm{x}_n\big)
\\
=
&
(-\kbar)^{n-1}\sum_{\clapsubstack{|\mp| \in P(n)}}
\e(\mp)\,
\bell^\pr_{|\mp|}\left(
{\mb{\psi}}\big(\bm{x}_{B_1}\big),\ldots,{\mb{\psi}}\big( \bm{x}_{B_{|\mp|}}\big)
\right) 
\\
&
-(-\kbar)^{n-1}\sum_{\clapsubstack{|\mp| \in P(n)\\ |B_i|=n-|\mp|+1}}
\e(\mp)\,
{\mb{\psi}}_{|\mp|}\left(J\bm{x}_{B_1},\ldots, J\bm{x}_{B_{i-1}}, 
\bell\big(\bm{x}_{B_{i}}\big), \bm{x}_{B_{i+1}},\ldots,\bm{x}_{B_{|\mp|}}\right)
.
}
Comparing the above with the desired relations in \eq{uslinftyalgebra},
all that remains is to establish that
$\proj_{\sC^\pr{\kkbar}}
\circ  \big(\mb{\d}_{{\bell}^\pr}\circ\mb{\Psi}_{{\mb{\psi}}}  
-\mb{\Psi}_{\mb{\psi}}\circ\mb{\d}_{{\bell}}\big)=0$.

Define $\mb{\pi} \in \Hom\left(\Xbar{S}(\sC), \sC\right)^{0}{\kkbar}$
and  $\mb{\pi}^\pr \in \Hom\left(\Xbar{S}(\sC^\pr), \sC^\pr\right)^{0}{\kkbar}$
as follows:
\begin{align*}
\mb{\pi}\big(\bm{x}_1\odot\ldots\odot\bm{x}_n\big) &\coloneqq \bm{x}_1\cdot\ldots\cdot \bm{x}_n
&&\text{for all } n\geq 1 \text{ and }\bm{x}_1,\ldots, \bm{x}_n \in \sC{\kkbar},
\\
\mb{\pi}^\pr\big(\bm{x}^\pr_1\odot\ldots\odot\bm{x}^\pr_n\big)&\coloneqq\bm{x}^\pr_1\cdot^\pr\ldots\cdot^\pr \bm{x}^\pr_n
&&\text{for all } n\geq 1 \text{ and }
\bm{x}^\pr_1,\ldots, \bm{x}^\pr_n \in \sC^\pr{\kkbar}.
\end{align*}

%\item $\mb{\eta} \in  \Hom\left(\Xbar{S}(\sC),\sC^\pr\right)^{0}{\kkbar}$ such that
%$\mb{\eta}\big(\bm{x}_1\odot\ldots\odot\bm{x}_n\big)\coloneqq{}\mb{\eta}_n \big(\bm{x}_1,\ldots,\bm{x}_n\big)$
%for all $n\geq 1$ and  $\bm{x}_1,\ldots, \bm{x}_n \in \sC{\kkbar}$.
Then, from the definition of the quantum descendant algebra, we have the following relations:
(See \eq{qdera} and Remark \ref{qderd}.)
\eqn{qdesfg}{
\mb{K}\circ \mb{\pi} =\mb{\pi}\circ \mb{\d}_{{\bell}}
,\qquad
\mb{K}^\pr\circ \mb{\pi}^\pr =\mb{\pi}^\pr\circ \mb{\d}_{{\bell}^\pr}
,\qquad
\mb{\d}_{{\bell}}\circ \mb{\d}_{{\bell}}=\mb{\d}_{{\bell}^\pr}\circ \mb{\d}_{{\bell}^\pr}=0.
}

Now we note that  the system of equations for the quantum descendant  
$\underline{\mb{\psi}}=\bm{\mK}(\bm{\mf})$ 
in Definition \ref{BKftalgm} can be rewritten as follows:
\eqn{simqdm}{
\bm{\mf}\circ \mb{\pi} = \mb{\pi}^\pr\circ\mb{\Psi}_{{\mb{\psi}}} 
.
}
Applying $\mb{K}^\pr$ to the above, we obtain that
\eqn{simqdmss}{
\mb{K}^\pr\circ \bm{\mf}\circ \mb{\pi} =\mb{K}^\pr\circ \mb{\pi}^\pr\circ\mb{\Psi}_{{\mb{\psi}}}
}
whose left-hand side is
\begin{align*}
\big(\mb{K}^\pr\circ \bm{\mf}\big)\circ \mb{\pi} 
= \bm{\mf}\circ\big(\mb{K}\circ \mb{\pi} \big)
= \big(\bm{\mf}\circ \mb{\pi}\big)\circ \mb{\d}_{{\bell}}
=\mb{\pi}^\pr\circ\mb{\Psi}_{{\mb{\psi}}}\circ \mb{\d}_{{\bell}}
,
\end{align*}
and whose right-hand side is
\begin{align*}
\big(\mb{K}^\pr\circ \mb{\pi}^\pr\big)\circ\mb{\Psi}_{{\mb{\psi}}}-\mb{K}^\pr\circ\mb{\eta}\circ \mb{\d}_{{\bell}}
=
\mb{\pi}^\pr\circ \mb{\d}_{{\bell}^\pr}\circ\mb{\Psi}_{{\mb{\psi}}}
.
\end{align*}
Therefore,  the equality in \eq{simqdmss} reduces to
$\mb{\pi}^\pr\circ \big(\mb{\d}_{{\bell}^\pr}\circ\mb{\Psi}_{{\mb{\psi}}}  
-\mb{\Psi}_{{\mb{\psi}}}\circ\mb{\d}_{{\bell}}\big)=0$. 
From $\mb{\pi}^\pr\circ \eb_{\sC^\pr{\kkbar}} =\I_{\sC^\pr{\kkbar}}$,
we conclude that
$\proj_{\sC^\pr{\kkbar}}\circ \big(\mb{\d}_{{\bell}^\pr}\circ\mb{\Psi}_{{\mb{\psi}}} 
 -\mb{\Psi}_{{\mb{\psi}}}\circ\mb{\d}_{{\bell}}\big)=0$.
\naturalqed
\end{proof}

\begin{remark}\label{qcoalgd}
Consider the topologically-free reduced symmetric coalgebras 
$\Xbar{S}^{\mathit{co}}\!(\sC){\kkbar}$
and $\Xbar{S}^{\mathit{co}}\!(\sC^\pr){\kkbar}$.
Define $\breve{\mb{\psi}} \in  \Hom\left(\Xbar{S}(\sC),\sC^\pr\right)^{0}{\kkbar}$ as
$\breve{\mb{\psi}}_n\coloneqq{}(-\kbar)^{n-1}\mb{\psi}_n$, where $\breve{\mb{\psi}}\circ \eb_{S^n\sC{\kkbar}}$.
Consider  the unique extension $\mF(\breve{\mb{\psi}})$ of $\breve{\mb{\psi}}$ 
as a coalgebra map from  $\Xbar{S}^{\mathit{co}}\!(\sC){\kkbar}$ to $\Xbar{S}^{\mathit{co}}\!(\sC^\pr){\kkbar}$:
we have,  for all $n\geq 1$ and  homogeneous $\bm{x}_1,\ldots, \bm{x}_n \in \sC{\kkbar}$,
\[
\mF(\breve{\mb{\psi}})({\bm{x}}_1\bm{\odot}\dotsc \bm{\odot} {\bm{x}}_n)
=\sum_{\clapsubstack{\mp \in P(n)}}
\e(\mp)\,
\breve{\mb{\psi}}\big({\bm{x}}_{B_1}\big)\bm{\odot} \dotsc\bm{\odot} \breve{\mb{\psi}}\big({\bm{x}}_{B_{|\mp|}}\big)
,
\]
and $\proj_{\sC^\pr{\kkbar}}\circ \mF(\breve{\mb{\psi}})= \breve{\mb{\psi}}$.
(See Lemma \ref{fcoalg} in Appendix~\ref{appendix: bar}.)
We note that $\mb{\Psi}_{\mb{\psi}}=\mF(\breve{\mb{\psi}})$. 
Also note
that the condition 
$
\proj_{\sC^\pr{\kkbar}}\circ \left(
\mD(\breve{\bell}^\pr)\circ \mF(\breve{\mb{\psi}})  -\mF(\breve{\mb{\psi}})\circ \mD(\breve{\bell})\right)=0
$
actually implies that 
$\mD(\breve{\bell}^\pr)\circ \mF(\breve{\mb{\psi}})  -\mF(\breve{\mb{\psi}})\circ \mD(\breve{\bell})=0$.
 Therefore, we have $\mb{\d}_{{\bell}^\pr}\circ\mb{\Psi}_{{\mb{\psi}}} =\mb{\Psi}_{{\mb{\psi}}}\circ\mb{\d}_{{\bell}}$.
\naturalqed
\end{remark}

\begin{definition}
A morphism $\bm{\mf}$ of binary QFT algebras is a quasi-isomorphism if 
$\bm{\mf}$ is a quasi-isomorphism as a morphism of underlying
QFT complexes.
\end{definition}

Recall that a unital $sL_\infty$-morphism  $\underline{\mb{\psi}}=\mb{\psi}_1,\mb{\psi}_2,\ldots$ is a quasi-isomorphism 
if $\mb{\psi}_1$ is a   quasi-isomorphism of the underlying pointed cochain complexes and that $\mb{\psi}_1=\bm{\mf}$
if  $\underline{\mb{\psi}}=\bm{\mK}(\bm{\mf})$ is the quantum descendant of a binary QFT algebra morphism 
$\bm{\mf}$. Therefore, 
the quantum descendant of a binary QFT algebra quasi-isomorphism is a unital $sL_\infty$-quasi-isomorphism.

% \begin{lemma}
% Every binary QFT algebra $\sC{\kkbar}_\BQ$ has at least one morphism into itself given by
% $\I_{\sC{\kkbar}}:\sC{\kkbar}_\BQ \rightarrow \sC{\kkbar}_\BQ$, whose quantum descendant
% is $\bm{\mK}(\I_{\sC{\kkbar}})= \I_{\sC{\kkbar}}, 0,0,0,\ldots$.
% \end{lemma}
% \begin{proof}Trivial.
% \naturalqed
% \end{proof}

\begin{theorem}
The composition  $\bm{\mf}^\pr\circ \bm{\mf}$ 
of consecutive morphisms of binary QFT algebras
$\xymatrix{\sC{\kkbar}_\BQ\ar[r]^-{\bm{\mf}}&\sC^\pr{\kkbar}_\BQ\ar[r]^-{\bm{\mf}^\ppr} &\sC^\pr{\kkbar}_\BQ}$ 
as pointed cochain maps
is a morphism of binary QFT algebras
$\xymatrix{\sC{\kkbar}_\BQ\ar[r]^-{\bm{\mf}^\ppr\circ\bm{\mf}} &\sC^\pr{\kkbar}_\BQ}$.
Moreover,
\begin{itemize}
\item 
equipped with morphisms of binary QFT algebras and this composition, binary QFT algebras over $\Bbbk$ form a category $\category{BQFTA}(\Bbbk)$, 
and

\item
the assignment to each binary QFT algebra its quantum descendant algebra and to each morphism of binary QFT algebras
its quantum descendant morphism is a functor  
\[
\bm{\mK}: \category{BQFTA}(\Bbbk)\rightsquigarrow \category{UsL}_\infty(\Bbbk{\kkbar})
\] 
from the category  $\category{BQFTA}(\Bbbk)$ of  binary QFT algebras 
to the category $\category{UsL}_\infty(\Bbbk{\kkbar})$ of unital $sL_\infty$-algebras.
\end{itemize}

\end{theorem}

\begin{proof}

%Consider a sequence of morphisms of  binary QFT algebra morphisms:
%\eqn{zuxdia}{
%\xymatrix{
%\sC{\kkbar}_\BQ
%\ar[r]^{\bm{\mf}} 
%&\sC^\pr{\kkbar}_\BQ
%\ar[r]^{\bm{\mf}^\pr} 
%&
%\sC^\ppr{\kkbar}_\BQ
%\ar[r]^{\bm{\mf}^\ppr} 
%&
%\sC^{\ppr\pr}{\kkbar}_\BQ
%.
%}
%}
It is obvious that $\bm{\mf}^\pr\circ \bm{\mf}$  is a pointed cochain map
from $\sC{\kkbar}$ to  $\sC^\ppr{\kkbar}$.
It remains to show that $\bm{\mf}^\pr\circ \bm{\mf}$ also satisfies the $\Bbbk$-compatibility condition and that the quantum descendant $\bm{\mK}(\bm{\mf}^\pr\circ \bm{\mf})$ of $\bm{\mf}^\pr\circ \bm{\mf}$
is the composition $\bm{\mK}(\bm{\mf}^\pr)\bullet \bm{\mK}( \bm{\mf})$
of the quantum descendants $\bm{\mK}( \bm{\mf})$ of $\bm{\mf}$ and $\bm{\mK}( \bm{\mf})$ 
of $\bm{\mf}^\pr$ as morphisms of unital $sL_\infty$-algebras.

Let $\big(\sC{\kkbar}, 1_{\sC}, \underline{\bell}\big)$, 
$\big(\sC^\pr{\kkbar}, 1_{\sC^\pr}, \underline{\bell}^\pr\big)$ and 
$\big(\sC{\kkbar}, 1_{\sC^\ppr}, \underline{\bell}^\ppr\big)$ be the quantum descendants of
$\sC{\kkbar}_\BQ$, $\sC{\kkbar}_\BQ$ and $\sC{\kkbar}_\BQ$, respectively.
By hypothesis, we have the following
unital $sL_\infty$-morphisms between quantum descendant algebras:
\begin{align*}
\underline{\mb{\psi}}^{\bm{\mf}}\coloneqq{}&
\xymatrix{
\bm{\mK}\big(\bm{\mf}\big): \big(\sC{\kkbar}, 1_\sC, \underline{\bell}\big)
\ar@{..>}[r]&
\big(\sC^\pr{\kkbar}, 1_{\sC^\pr}, \underline{\bell}^\pr\big)}
,\\
\underline{\mb{\psi}}^{\bm{\mf}^\pr}\coloneqq{}&
\xymatrix{
\bm{\mK}\big(\bm{\mf}^\pr\big):
\big(\sC^\pr{\kkbar}, 1_{\sC^\pr}, \underline{\bell}^\pr\big)
\ar@{..>}[r]&
\big(\sC^\pr{\kkbar}, 1_{\sC^\ppr}, \underline{\bell}^\ppr\big),}
\end{align*}
where
\eqnalign{bydoa}{
\bm{\mf}\circ \mb{\pi} = \mb{\pi}^\pr\circ \mb{\Psi}_{\mb{\psi}^{\bm{\mf}}} 
,\qquad
\bm{\mf}^\pr\circ\mb{\pi}^\pr =
\mb{\pi}^\ppr\circ \mb{\Psi}_{{\mb{\psi}}^{\bm{\mf}^\pr}}.
}
Applying $\bm{\mf}^\pr$ to the first identity above and using the second identity, we have
$\bm{\mf}^\pr\circ\bm{\mf}\circ \mb{\pi} =
\bm{\mf}^\pr\circ\mb{\pi}^\pr\circ \mb{\Psi}_{\mb{\psi}^{\bm{\mf}}}
=\mb{\pi}^\ppr\circ \mb{\Psi}_{\mb{\psi}^{\bm{\mf}^\pr}}\circ \mb{\Psi}_{\mb{\psi}^{\bm{\mf}}}$.
Therefore we obtain
\eqn{bydob}{
(\bm{\mf}^\pr\circ\bm{\mf})\circ \mb{\pi} = 
\mb{\pi}^\ppr\circ \mb{\Psi}_{{\mb{\psi}}^{\bm{\mf}^\pr}}\circ \mb{\Psi}_{{\mb{\psi}^{\bm{\mf}}}} 
.
}

Recall that the composition $\underline{\mb{\psi}}^{\bm{\mf}^\pr}\bullet \underline{\mb{\psi}^{\bm{\mf}}}$
of the unital $sL_\infty$-morphisms $\underline{\mb{\psi}}^{\bm{\mf}^\pr}$ and $\underline{\mb{\psi}^{\bm{\mf}}}$ is
defined for all $n\geq 1$ and homogeneous $\bm{x}_1,\ldots,\bm{x}_n \in \sC{\kkbar}$, via
\[
(\mb{\psi}^{\bm{\mf}^\pr}\bullet \mb{\psi}^{\bm{\mf}}\big)(\bm{x}_1\odot\ldots\odot \bm{x}_n)
\equiv \big(\underline{\mb{\psi}}^{\bm{\mf}^\pr}\bullet \underline{\mb{\psi}^{\bm{\mf}}}\big)_n
(\bm{x}_1,\ldots,\bm{x}_n) 
=\sum_{\mathclap{\mp\in P(n)}}\e(\mp) 
\mb{\psi}^{\bm{\mf}^\pr}_{|\mp|}\big(\mb{\psi}^{\bm{\mf}}(\bm{x}_{B_1}), \ldots, \mb{\psi}(\bm{x}_{B_{|\mp|}})\big)
,
\]
and $\xymatrix{\underline{\mb{\psi}}^{\bm{\mf}^\pr}\bullet \underline{\mb{\psi}}^{\bm{\mf}}:
\big(\sC{\kkbar}, 1_\sC, \underline{\bell}\big)\ar@{..}[r]&
\big(\sC^\ppr{\kkbar}, 1_{\sC^\ppr}, \underline{\bell}^\ppr\big)}$ is a unital $sL_\infty$-morphism.
In particular, we have 
$\big(\underline{\mb{\psi}}^{\bm{\mf}^\pr}\bullet \underline{\mb{\psi}}^{\bm{\mf}}\big)_n 
\in \Hom\left(S^n{\sC}, \sC^\ppr\right)^{0}{\kkbar}$ for all $n\geq 1$.
\begin{claim}
We have the identity 
$\mb{\Psi}_{{\mb{\psi}}^{\bm{\mf}^\pr}}\circ \mb{\Psi}_{{\mb{\psi}^{\bm{\mf}}}} = 
\mb{\Psi}_{{\mb{\psi}}^{\bm{\mf}^\pr}\bullet {\mb{\psi}^{\bm{\mf}}}}$.
\end{claim}

\begin{proof}
Define $\breve{\mb{\psi}}^{\bm{\mf}} \in \Hom\left(\Xbar{S}(\sC),\sC^\pr\right)^{0}{\kkbar}$ 
and $\breve{\mb{\psi}}^{\bm{\mf}^\pr} \in \Hom\left(\Xbar{S}(\sC^\pr),\sC^\ppr\right)^{0}{\kkbar}$ 
so that  $\breve{\mb{\psi}}^{\bm{\mf}}_n=(-\kbar)^{n-1}\mb{\psi}^{\bm{\mf}}_n$
and $\breve{\mb{\psi}}^{\bm{\mf}^\pr}_n=(-\kbar)^{n-1}\mb{\psi}^{\bm{\mf}^\pr}_n$ for all $n\geq 1$.
Consider the unique coalgebra maps 
$\mF\big(\breve{\mb{\psi}}\big):\Xbar{S}^{\mathit{co}}\!(\sC){\kkbar}\rightarrow \Xbar{S}^{\mathit{co}}\!(\sC^\pr){\kkbar}$
and $\mF\big(\breve{\mb{\psi}^\pr}\big):\Xbar{S}^{\mathit{co}}\!(\sC^\pr){\kkbar}
\rightarrow \Xbar{S}^{\mathit{co}}\!(\sC^\ppr){\kkbar}$ determined by 
$\breve{\mb{\psi}}^{\bm{\mf}}$ and $\breve{\mb{\psi}}^{\bm{\mf}^\pr}$, respectively.
Then, we have $\mF\big(\breve{\mb{\psi}}^{\bm{\mf}^\pr}\big)\circ \mF\big(\breve{\mb{\psi}}^{\bm{\mf}^\pr}\big) = 
\mF\big(\breve{\mb{\psi}}^{\bm{\mf}^\pr}\bullet\breve{\mb{\psi}}^{\bm{\mf}}\big)$ --- see Lemma \ref{barfunctor} in Appendix~\ref{appendix: bar}.
From Remark \ref{qcoalgd}, we have
$\mb{\Psi}_{\mb{\psi}^{\bm{\mf}}}=\mF\big(\breve{\mb{\psi}}^{\bm{\mf}^\pr}\big)$ and 
$\mb{\Psi}_{\mb{\psi}^{\bm{\mf}^\pr}}=\mF\big(\breve{\mb{\psi}}^{\bm{\mf}^\pr}\big)$. 
We note that
\begin{align*}
(\breve{\mb{\psi}}^{\bm{\mf}^\pr}\bullet\breve{\mb{\psi}}^{\bm{\mf}}\big)(\bm{x}_1\odot\ldots\odot \bm{x}_n)
&=\sum_{\mathclap{\mp\in P(n)}}\e(\mp) \,
\breve{\mb{\psi}}^{\bm{\mf}^\pr}_{|\mp|}
\left(\breve{\mb{\psi}}^{\bm{\mf}}(\bm{x}_{B_1}), \ldots, \breve{\mb{\psi}}^{\bm{\mf}}(\bm{x}_{B_{|\mp|}})\right)
\\
&=
(-\kbar)^{n-1}\sum_{\mathclap{\mp\in P(n)}}\e(\mp) \,
{\mb{\psi}}^{\bm{\mf}^\pr}_{|\mp|}
\left(\mb{\psi}^{\bm{\mf}}(\bm{x}_{B_1}), \ldots, \mb{\psi}^{\bm{\mf}}(\bm{x}_{B_{|\mp|}})\right).
\end{align*}
It follows that $\breve{\mb{\psi}}^{\bm{\mf}^\pr}\bullet\breve{\mb{\psi}}^{\bm{\mf}} =
\breve{(\mb{\psi}^{\bm{\mf}^\pr}\bullet\mb{\psi}^{\bm{\mf}})}$, 
which implies that 
$\mb{\Psi}_{{\mb{\psi}}^{\bm{\mf}^\pr}}\circ \mb{\Psi}_{{\mb{\psi}^{\bm{\mf}}}} 
= \mb{\Psi}_{{\mb{\psi}^{\bm{\mf}^\pr}}\bullet {\mb{\psi}^{\bm{\mf}}}}$.
\naturalqed
\end{proof}

Therefore, \eq{bydob} is equivalent to the the identity
$(\bm{\mf}^\pr\circ\bm{\mf})\circ \mb{\pi} = 
\mb{\pi}^\ppr\circ  \mb{\Psi}_{\mb{\psi}^{\bm{\mf}^\pr}\bullet \mb{\psi}^{\bm{\mf}^\pr}}$,
which implies that $\bm{\mf}^\pr\circ\bm{\mf}:\sC{\kkbar}_\BQ \rightarrow\sC^\pr{\kkbar}_\BQ$ is a morphism of 
binary QFT algebras whose quantum descendant $\underline{\mb{\psi}}^{\bm{\mf}^\pr\circ\bm{\mf}}$ is 
$\underline{\mb{\psi}}^{\bm{\mf}^\pr}\bullet \underline{\mb{\psi}}^{\bm{\mf}}$, i.e.,
\eqn{bydod}{
\bm{\mK}\big(\bm{\mf}^\pr\circ \bm{\mf}\big)=\mb{\mK}\big(\bm{\mf}^\pr\big)\bullet \mb{\mK}\big(\bm{\mf}\big).
}
It is trivial that $\I_{\sC{\kkbar}}:\sC{\kkbar}_\BQ \rightarrow \sC{\kkbar}_\BQ$ 
is the identity morphism for every 
binary QFT algebra
$\sC{\kkbar}_\BQ$, and that its quantum descendant 
is the identity morphism on the quantum descendant unital $sL_\infty$-algebra.
% It follows that binary QFT algebras forms a
% category $\category{BQFTA}(\Bbbk)$, whose objects are binary QFT algebras and morphisms are morphisms
% of binary QFT algebras,
% and
% $\bm{\mK}: \category{BQFTA}(\Bbbk)\rightsquigarrow \category{UsL}_\infty(\Bbbk{\kkbar})$ 
% is a functor
% from the category  $\category{BQFTA}(\Bbbk)$ of  binary QFT algebras 
% to the category $\category{UsL}_\infty(\Bbbk{\kkbar})$ of unital $sL_\infty$-algebras.
\naturalqed
\end{proof}

%
%\begin{remark}\label{qdesma}
%The first non-trivial  $\kbar$-compatibility condition is that the failure of
%$\bm{\mf}$ being an algebra homomorphism 
%is divisible by $\kbar$ or, equivalently, the classical limit $f$ of $\bm{\mf}=f +\kbar f^{(1)}+\ldots$ is an
%algebra homomorphism, i.e., $f(x_1\cdot x_2)= f(x_1)\cdot^\pr f(x_2)$. In general,
%it can be checked that,
%for all $n\geq 2$ and homogeneous $\bm{x}_1,\ldots, \bm{x}_n \in \sC{\kkbar}$, we have
%\begin{align*}
%(-\kbar)\mb{\psi}_n\big(\bm{x}_1,\ldots, \bm{x}_n\big)  = &
%\mb{\psi}_{n-1}\big(\bm{x}_1,\ldots, \bm{x}_{n-2}, \bm{x}_{n-1}\cdot \bm{x}_n\big)
%-\sum_{\clapsubstack{ \mp \in P(n)\\ |\mp|=2\\ n-1\nsim n}}\mb{\psi}(\bm{x}_{B_1})\cdot^\pr \mb{\psi}(\bm{x}_{B_2})
%.
%\end{align*}
%\naturalqed
%\end{remark}

A morphism  of binary QFT algebras has
some easily checkable but illuminating properties summarized by the following two lemmas.

Let  $\xymatrix{\sC{\kkbar}_\BQ\ar[r]^-{\bm{\mf}} &\sC^\pr{\kkbar}_\BQ}$ be a morphism of binary QFT algebras
whose quantum descendant is 
$\xymatrix{\big(\sC{\kkbar}, 1_{\sC}, \underline{\bell}\big)\ar@{..>}[r]^{\underline{\mb{\psi}}}&\big(\sC^\pr{\kkbar}, 1_{\sC^\pr}, \underline{\bell}^\pr\big)}$. Then,

\begin{lemma}\label{physwav}
For any $\mb{\g} \in \sC{\kkbar}$  and nilpotent $\mb{\th} \in \sC^{0}{\kkbar}$, 
there are $\mb{\g}^\pr \in  \sC^\pr{\kkbar}$ and $\mb{\th}^\pr \in \sC^{\pr 0}{\kkbar}$
given by
\[
\mb{\g}^\pr =\mb{\psi}_1(\mb{\g})+ \sum_{\mathclap{n=1}}^\infty \Fr{1}{n!}\mb{\psi}_{n+1}\big(\mb{\g},\mb{\th},\ldots, \mb{\th}\big)
,\qquad
\mb{\th}^\pr=\sum_{\mathclap{n=1}}^\infty \Fr{1}{n!}\mb{\psi}_{n}\big(\mb{\th},\ldots, \mb{\th}\big)
\]
such that
$\bm{\mf}\left( \mb{\g} \cdot e^{-\fr{1}{\kbar}\mb{\th}}\right) = 
 \mb{\g}^\pr \cdot e^{-\fr{1}{\kbar}\mb{\th}^\pr}$.
\end{lemma}

\begin{lemma}\label{remBKftmor}
We have $\mb{\psi}_1 =\bm{\mf}$ and
for all $n\geq 1$ and  homogeneous $\bm{x}_1,\ldots, \bm{x}_{n+1} \in \sC{\kkbar}$
\begin{align*}
(-\kbar)\mb{\psi}_{n+1}\big(\bm{x}_1,\ldots, \bm{x}_{n+1}\big)  = &
\mb{\psi}_{n}\big(\bm{x}_1,\ldots, \bm{x}_{n-1}, \bm{x}_{n}\cdot \bm{x}_{n+1}\big)
-\sum_{\clapsubstack{ \mp \in P(n+1)\\ |\mp|=2\\ n\nsim n+1}}\mb{\psi}(\bm{x}_{B_1})\cdot^\pr \mb{\psi}(\bm{x}_{B_2})
.
\end{align*}
\end{lemma}

\begin{remark} Consider Lemma \ref{physwav}.
If we regard $\mb{\g} \cdot e^{-\fr{1}{\kbar}\mb{\th}}$  as a wave function
it can be called a
physical wave function if
$\mb{K} \left( \mb{\g} \cdot e^{-\fr{1}{\kbar}\mb{\th}}\right) =0$. This condition is equivalent to the
following by  Lemma \ref{hiact}:
\begin{align*}
\mb{K}\mb{\vt} 
+\sum_{\mathclap{n=2}}^\infty\Fr{1}{n!} \bell_n\big(\mb{\vt},\ldots,\mb{\vt}\big) &=0
,\\
\mb{K}\mb{\g} +\sum_{\mathclap{n=2}}^\infty\Fr{1}{(n-1)!} \bell_n\big(\mb{\vt},\ldots,\mb{\vt}, \mb{\g}\big)&=0
.
\end{align*}
Note that
$\mb{K}^\pr \left( \mb{\g}^\pr \cdot e^{-\fr{1}{\kbar}\mb{\th}^\pr}\right) = 
\bm{\mf}\left( \mb{K}\left( \mb{\g} \cdot e^{-\fr{1}{\kbar}\mb{\th}}\right)\right)$ since $\mb{K}^\pr\circ\bm{\mf}=\bm{\mf}\circ
\mb{K}$.
Therefore, a morphism of binary QFT algebras sends  physical wave functions to
physical wave functions. 
\naturalqed
\end{remark}

\begin{remark} Consider the first three examples of the relations in Lemma \ref{remBKftmor}.
We have $\mb{\psi}_1 =\bm{\mf}$ and
\begin{align*}
(-\kbar)\mb{\psi}_2\big(\bm{x}_1, \bm{x}_2)=
&
\mb{\psi}_1\big(\bm{x}_1\cdot \bm{x}_2) - \mb{\psi}_1\big(\bm{x}_1) \cdot^\pr \mb{\psi}_1\big( \bm{x}_2) 
,\\
(-\kbar)\mb{\psi}_3\big(\bm{x}_1, \bm{x}_2,\bm{x}_3)=
&
\mb{\psi}_2\big(\bm{x}_1, \bm{x}_2\cdot \bm{x}_3) 
- \mb{\psi}_2\big(\bm{x}_1,\bm{x}_2) \cdot^\pr \mb{\psi}_1\big( \bm{x}_3) 
\\
&
- (-1)^{|\bm{x}_1||\bm{x}_2|}\mb{\psi}_1\big(\bm{x}_2) \cdot^\pr \mb{\psi}_1\big(\bm{x}_1, \bm{x}_3),
\end{align*}
The first non-trivial  $\kbar$-compatibility condition for a pointed cochain map $\bm{\mf}$ to be a morphism of
binary QFT algebras is that the failure of $\bm{\mf}$ being an algebra homomorphism 
is $(-\kbar)\mb{\psi}_2$, which divisible by $\kbar$. Then the second non-trivial  $\kbar$-compatibility condition
is that the failure of $\mb{\psi}_2 \in \Hom\big(S^2\sC,\sC^\pr\big)^0{\kkbar}$ being 
an ``algebra homomorphism'' is $(-\kbar)\mb{\psi}_3$, which divisible by $\kbar$, and so on.
\naturalqed
\end{remark}

\subsection{The homotopy category of binary QFT algebras}
\label{subsec: homotopy cat of binary qft alg}

In this subsection, we introduce the notion of homotopy of  morphisms of binary QFT algebras
to define the homotopy category $\mathit{ho}\category{BQFTA}(\Bbbk)$ of binary QFT algebras
and show that the quantum descendant functor
$\bm{\mK}: \category{BQFTA}(\Bbbk) \rightsquigarrow \category{UsL}_\infty(\Bbbk{\kkbar})$
is a homotopy functor---that it induces a functor 
$\mathit{ho}\bm{\mK}: \mathit{ho}\category{BQFTA}(\Bbbk) \rightsquigarrow \mathit{ho}\category{UsL}_\infty(\Bbbk{\kkbar})$ 
from the homotopy category $\mathit{ho}\category{BQFTA}(\Bbbk)$ of binary QFT algebras to the homotopy
category $\mathit{ho}\category{UsL}_\infty(\Bbbk{\kkbar})$ of unital $sL_\infty$-algebras over $\Bbbk{\kkbar}$.

We begin with introducing the notion of a homotopy pair of binary QFT algebras
after fixing some notation that will be useful later in the section.

Let $\sU{\kkbar}$ and $\sU^\pr{\kkbar}$ be $\Z$-graded topologically-free $\Bbbk{\kkbar}$-modules. 
Introduce a formal time parameter $\t$ and
denote by $\Hom\left(\sU, \sU^\pr\right)^{k}{\kkbar}[\t]$ 
the space of $\Bbbk{\kkbar}$-linear maps from $\sU{\kkbar}$ to $\sU^\pr{\kkbar}$
of ghost number $k$ parametrized by $\t$ with polynomial dependence:
$\mb{\a}(\t) \in \Hom\left(\sU, \sU^\pr\right)^{k}{\kkbar}[\t]$ if $\mb{\a}(\t)=
\mb{\a}_0 + \mb{\a}_1 \t +\ldots +\mb{\a}_n \t^n$, $\mb{\a}_j \in \Hom\left(\sU, \sU^\pr\right)^{k}{\kkbar}$
for $j=0,1,\ldots, n$.  Let $\mb{\a}^\pr(\t)=\mb{\a}^\pr_0 + \mb{\a}^\pr_1 \t +\ldots +\mb{\a}^\pr_{n^\pr} \t^{n^\pr} 
\in \Hom\left(\sU^\pr, \sU^\ppr\right)^{l}{\kkbar}[\t]$, where $\sU^\ppr{\kkbar}$ 
is another $\Z$-graded topologically-free $\Bbbk{\kkbar}$-module,
then we define
%$\mb{\a}^\pr(\t)\circ \mb{\a}(\t) \in \Hom\left(\sU, \sU^\ppr\right)^{k}{\kkbar}[\t]$ as 
\[\mb{\a}^\pr(\t)\circ \mb{\a}(\t) := \sum_{\mathclap{i=0}}^{n^\pr + n} \left(\mb{\a}^\pr_j\circ \mb{\a}_{i-j}\right) \t^i
 \in \Hom\left(\sU, \sU^\ppr\right)^{k+l}{\kkbar}[\t]
\]

Fix two binary QFT algebras as follows: 
\begin{align*}
\sC{\kkbar}_\BQ&=\big(\sC{\kkbar}, 1_{\sC}, \,\cdot\,,\mb{K}\big),&
\bm{\mK}\big(\sC{\kkbar}_\BQ\big)&=\big(\sC{\kkbar}, 1_{\sC}, \underline{\bell}\big)
,\\
\sC^\pr{\kkbar}_\BQ&=\big(\sC^\pr{\kkbar}, 1_{\sC^\pr}, \,\cdot^\pr\,,\mb{K}^\pr\big), &
\bm{\mK}\big(\sC^\pr{\kkbar}_\BQ\big)&=\big(\sC^\pr{\kkbar}, 1_{\sC^\pr}, \underline{\bell}^\pr\big)
.
\end{align*}
Define $\mb{\pi} \in \Hom\left(\Xbar{S}(\sC), \sC\right)^{0}{\kkbar}$
and  $\mb{\pi}^\pr \in \Hom\left(\Xbar{S}(\sC^\pr), \sC^\pr\right)^{0}{\kkbar}$
as in the proof of Lemma~\ref{BKftdesm}:
\begin{align*}
\mb{\pi}\big(\bm{x}_1\odot\ldots\odot\bm{x}_n\big) &\coloneqq \bm{x}_1\cdot\ldots\cdot \bm{x}_n
&&\text{for all } n\geq 1 \text{ and }\bm{x}_1,\ldots, \bm{x}_n \in \sC{\kkbar},
\\
\mb{\pi}^\pr\big(\bm{x}^\pr_1\odot\ldots\odot\bm{x}^\pr_n\big)&\coloneqq\bm{x}^\pr_1\cdot^\pr\ldots\cdot^\pr \bm{x}^\pr_n
&&\text{for all } n\geq 1 \text{ and }
\bm{x}^\pr_1,\ldots, \bm{x}^\pr_n \in \sC^\pr{\kkbar}.
\end{align*}

% \begin{itemize}
% \item
% $\mb{\pi}\big(\bm{x}_1\odot\ldots\odot\bm{x}_n\big) \coloneqq{}\bm{x}_1\cdot\ldots\cdot \bm{x}_n$
% for all $n\geq 1$ and  $\bm{x}_1,\ldots, \bm{x}_n \in \sC{\kkbar}$,
% and
% \item
% $\mb{\pi}^\pr\big(\bm{x}^\pr_1\odot\ldots\odot\bm{x}^\pr_n\big) \coloneqq{}\bm{x}^\pr_1\cdot^\pr\ldots\cdot^\pr \bm{x}^\pr_n$
% for all $n\geq 1$ and  $\bm{x}^\pr_1,\ldots, \bm{x}^\pr_n \in \sC^\pr{\kkbar}$.
% \end{itemize}
Define $\breve{\bell} \in \Hom\left(\Xbar{S}(\sC),\sC\right)^{1}{\kkbar}$ 
and $\breve{\bell}^\pr \in \Hom\left(\Xbar{S}(\sC^\pr),\sC^\pr\right)^{1}{\kkbar}$ 
via $\breve{\bell}_n=(-\kbar)^{n-1}\bell_n$
and $\breve{\bell}^\pr_n=(-\kbar)^{n-1}\bell^\pr_n$ for all $n\geq 1$.
Let $\mD(\breve{\bell}):\Xbar{S}(\sC){\kkbar}\rightarrow \Xbar{S}(\sC){\kkbar}$ and 
$\mD(\breve{\bell}^\pr):\Xbar{S}(\sC^\pr){\kkbar}\rightarrow \Xbar{S}(\sC^\pr){\kkbar}$
be the unique coderivations determined by $\breve{\bell}$ and $\breve{\bell}^\pr$, respectively. 
Then, we have the following relations:
\[
\mb{K}\circ \mb{\pi} = \mb{\pi}\circ \mD(\breve{\bell})
,\qquad
\mb{K}^\pr\circ \mb{\pi}^\pr = \mb{\pi}^\pr\circ \mD(\breve{\bell}^\pr)
,\qquad
\mD(\breve{\bell})\circ \mD(\breve{\bell})=\mD(\breve{\bell}^\pr)\circ \mD(\breve{\bell}^\pr).
\]
Remember or notice that $\mD(\breve{\bell})=\mb{\d}_{\bell}$ and $\mD(\breve{\bell}^\pr)=\mb{\d}_{\bell^\pr}$.

Now we are ready to introduce the notion of a homotopy pair. 
For each pair $\big(\bm{\mf}(\t)\big|\mb{\xi}(\t)\big)$ in $\Hom\big(\sC, \sC^\pr\big)^{0}{\kkbar}[\t]\bigoplus  \Hom\big(\sC, \sC^\pr\big)^{-1}{\kkbar}[\t]$,  we can associate another pair
\[
\left(\breve{\mb{\psi}}(\t)\big|\breve{\mb{\l}}(\t)\right)\in
\Hom\left(\Xbar{S}(\sC), \sC^\pr\right)^{0}{\kkbar}[\t]\bigoplus\Hom\left(\Xbar{S}(\sC), \sC^\pr\right)^{-1}{\kkbar}[\t]
\]
implicitly defined by the equations
\eqn{hpqdes}{
\bm{\mf}(\t)\circ\mb{\pi} =
\mb{\pi}^\pr\circ\mF\left(\breve{\mb{\psi}}(\t)\right)
,\qquad
\mb{\xi}(\t)\circ\mb{\pi}=
\mb{\pi}^\pr\circ\La\left(\breve{\mb{\psi}}(\t), \breve{\mb{\l}}(\t)\right)
,}
where for all $n\geq 1$ and  homogeneous $\bm{x}_1,\ldots,\bm{x}_n \in \sC{\kkbar}$, we write
\begin{align*}
\mF\big(\breve{\mb{\psi}}(\t)\big)(\bm{x}_1\bm{\odot} \ldots\bm{\odot} \bm{x}_n)\coloneqq{}
&
\sum_{\clapsubstack{|\mp| \in P(n)}}
\e(\mp)
\breve{\mb{\psi}}(\t)\big(\bm{x}_{B_1}\big)\bm{\odot}\ldots\bm{\odot}\breve{\mb{\psi}}(\t)\big( \bm{x}_{B_{|\mp|}}\big)
,\\
\La\big(\breve{\mb{\psi}}(\t), \breve{\mb{\l}}(\t)\big)(\bm{x}_1\bm{\odot} \ldots\bm{\odot} \bm{x}_n)\coloneqq{}
&
\sum_{\clapsubstack{|\mp| \in P(n)}}\ \ 
\sum_{\mathclap{i=1}}^{|\mp|}\,\e(\mp)\,
\breve{\mb{\psi}}(\t)\big(J\!\bm{x}_{B_1}\big)\bm{\odot}\ldots\bm{\odot}\breve{\mb{\psi}}(\t)\big(J\!\bm{x}_{B_{i-1}}\big)
\bm{\odot} 
\\
&
\qquad\times
\breve{\mb{\l}}(\t)\big(\bm{x}_{B_i}\big)
\bm{\odot}\breve{\mb{\psi}}(\t)\big(\bm{x}_{B_{i+1}}\big)\bm{\odot}\ldots\bm{\odot}
\breve{\mb{\psi}}(\t)\big( \bm{x}_{B_{|\mp|}}\big)
.
\end{align*}
From the condition $\mb{\pi}^\pr_1=\I_{\sC^\pr{\kkbar}}$,
we can check that \eq{hpqdes} actually determines  the pair $\left(\breve{\mb{\psi}}(\t)\big|\breve{\mb{\l}}(\t)\right)$ 
uniquely from the pair $\big(\bm{\mf}(\t)\big|\mb{\xi}(\t)\big)$.
%such that 
%\[\breve{\mb{\psi}}(\t)\in\Hom\left(\Xbar{S}(\sC), \sC^\pr\right)^{0}{\kkbar}[\t]
%,\qquad
%\breve{\mb{\l}}(\t)\in \Hom\left(\Xbar{S}(\sC), \sC^\pr\right)^{-1}{\kkbar}[\t]
%.
%\]
For example, we have  $\breve{\mb{\psi}}_1(\t)=\bm{\mf}(\t)$ and $\breve{\mb{\l}}_1(\t)=\mb{\xi}(\t)$ at the first level and
\begin{align*}
\breve{\mb{\psi}}_2(\t)(\bm{x}_1,\bm{x}_2) \coloneqq{} 
&\breve{\mb{\psi}}_1(\t)(\bm{x}_1\cdot\bm{x}_2)
-\breve{\mb{\psi}}_1(\t)(\bm{x}_1)\cdot^\pr \breve{\mb{\psi}}_1(\t)(\bm{x}_2)
,\\
\breve{\mb{\l}}_2(\t)(\bm{x}_1,\bm{x}_2)\coloneqq{}&\mb{\xi}(\t)(\bm{x}_1\cdot\bm{x}_2)
-\breve{\mb{\l}}_1(\t)(\bm{x}_1)\cdot^\pr \breve{\mb{\psi}}_1(\t)(\bm{x}_2)
-\breve{\mb{\psi}}_1(\t)(J\!\bm{x}_1)\cdot^\pr \breve{\mb{\l}}_1(\t)(\bm{x}_2)
\end{align*}
at the second level.

Then, we say that $\left(\breve{\mb{\psi}}(\t)\big|\breve{\mb{\l}}(\t)\right)$
is the descendant of $\big(\bm{\mf}(\t)|\mb{\xi}(\t)\big)$,
and denote this relationship
$\breve{\bm{\mK}}\big(\bm{\mf}(\t)\big|\mb{\xi}(\t)\big)=\left(\breve{\mb{\psi}}(\t)\big|\breve{\mb{\l}}(\t)\right)$.

\begin{definition}\label{hpairBKft}
A pair 
$\big(\bm{\mf}(\t)\big|\mb{\xi}(\t)\big)\in\Hom\big(\sC, \sC^\pr\big)^{0}{\kkbar}[\t]\bigoplus  \Hom\big(\sC, \sC^\pr\big)^{-1}{\kkbar}[\t]$
with the descendant $\left(\breve{\mb{\psi}}(\t)\big|\breve{\mb{\l}}(\t)\right)=\breve{\bm{\mK}}\big(\bm{\mf}(\t)\big|\mb{\xi}(\t)\big)$
is called a homotopy pair of binary QFT algebras from $\sC{\kkbar}_\BQ$ to $\sC^\pr{\kkbar}_\BQ$, and denote this
\[
\xymatrix{\sC{\kkbar}_\BQ\ar@{=>}[rr]^{(\bm{\mf}(\t)|\mb{\xi}(\t))} && \sC^\pr{\kkbar}_\BQ},
\]
if it satisfies the following axioms:

\begin{itemize} 
\item the unital homotopy flow equation that
\eqn{hpqflow}{
\mb{\xi}(\t)(1_\sC)=0
,\qquad
\Fr{d}{d\!\t} \bm{\mf}(\t) =\mb{K}^\pr\circ \mb{\xi}(\t) + \mb{\xi}(\t) \circ\mb{K}
}
and
\item
the $\kbar$-condition that
$\Fr{1}{(-\kbar)^{n-1}}\breve{\mb{\l}}_n(\t)\in  \Hom\big(S^n\sC, \sC^\pr\big)^{-1}{\kkbar}[\t]$
for all  $n\geq 1$.
\end{itemize}
\end{definition}

\begin{proposition}\label{BKfhomp}
Consider a homotopy pair $\xymatrix{\sC{\kkbar}_\BQ\ar@{=>}[rr]^{(\bm{\mf}(\t)|\mb{\xi}(\t))} && \sC^\pr{\kkbar}_\BQ}$ of  binary QFT algebras
whose descendant is $\left(\breve{\mb{\psi}}(\t)\big| \breve{\mb{\l}}(\t)\right)$.
Assume that $\bm{\mf}(0)$ is a morphism of binary QFT algebras. Then,
we have
\begin{enumerate}
\item $
\xymatrix{\bm{\mf}(\t):\sC{\kkbar}_\BQ \ar[r] & \sC^\pr{\kkbar}_\BQ}$ 
is a (uniquely defined) family of
binary QFT algebra morphisms,
and 

\medskip
\item $\xymatrix{\big(\underline{{\mb{\psi}}}(\t), \underline{{\mb{\l}}}(\t)\big):
\big(\sC{\kkbar}, 1_\sC, \underline{\bell}\big)\ar@{:>}[r] &\big(\sC^\pr{\kkbar}, 1_{\sC^\pr}, \underline{\bell}^\pr\big)}
$
is a homotopy pair of unital $sL_\infty$-algebras, where
${\mb{\psi}}_n(\t) \coloneqq{}\Fr{1}{(-\kbar)^{n-1}}\breve{\mb{\psi}}_n(\t)$,
${\mb{\l}}_n(\t) \coloneqq{}\Fr{1}{(-\kbar)^{n-1}}\breve{\mb{\l}}_n(\t)$ for $n\geq 1$,
and $\underline{{\mb{\psi}}}(0)=\bm{\mK}\big(\bm{\mf}(0)\big)$.
\end{enumerate}
\end{proposition}

%Note that the explicit formulas for $\underline{\breve{\mb{\psi}}}(\t)=\breve{\mb{\psi}}_1(\t), \breve{\mb{\psi}}_2(\t),\ldots$ 
%and $\underline{\breve{\mb{\l}}}(\t)=\breve{\mb{\l}}_1(\t), \breve{\mb{\l}}_2(\t),\ldots$ are
%such that, for all $n\geq 1$ and $x_1,\ldots, x_n \in \sC$,
%\begin{align*}
%\bm{\mf}(\t)\big(x_1\cdot\ldots\cdot x_n\big)
%=
%&
%\sum_{\clapsubstack{\mp \in P(n)}}
%\e(\mp)
%\breve{\mb{\psi}}(\t)\big(x_{B_1}\big)\cdot^\pr\ldots\cdot^\pr\breve{\mb{\psi}}(\t)\big(v_{B_{|\mp|}}\big)
%,\\
%%
%\mb{\xi}(\t)\big(x_1\cdot\ldots\cdot x_n\big)
%=
%&
%\!\!\sum_{\clapsubstack{\mp \in P(n)}}
%\sum_{\mathclap{i=1}}^{|\mp|}
%\e(\mp)
%\mb{\breve{\psi}}(\t)\big(Jx_{B_1}\big)\cdot^\pr\ldots\cdot^\pr\mb{\breve{\psi}}(\t)\big(Jx_{B_{i-1}}\big)
%\\
%&
%\qquad\qquad\qquad
%\cdot^\pr\breve{\mb{\l}}(\t)\big(v_{B_i}\big)
%\cdot^\pr \mb{\breve{\psi}}(\t)\big(v_{B_{i+1}}\big)\cdot^\pr\ldots\cdot^\pr\mb{\breve{\psi}}(\t)\big(x_{B_{|\mp|}}\big).
%\end{align*}
%For example, $\breve{\mb{\psi}}_1(\t) \coloneqq{}\bm{\mf}(\t)$, $\breve{\mb{\l}}_1(\t) \coloneqq{}\mb{\xi}_1(\t)$ and
%\begin{align*}
%\breve{\mb{\psi}}_2(\t)(x_1,x_2) \coloneqq{}& \breve{\mb{\psi}}_1(\t)(x_1\cdot x_2) 
%-  \breve{\mb{\psi}}_1(\t)(x_1)\cdot^\pr \breve{\mb{\psi}}_1(\t)(x_2) 
%,\\
%\breve{\mb{\l}}_2(\t)(x_1,x_2) \coloneqq{}& \breve{\mb{\l}}_1(\t)(x_1\cdot x_2)
% - \breve{\mb{\l}}_1(\t)(x_1)\cdot^\pr \breve{\mb{\psi}}_1(\t)(x_2) 
% - \breve{\mb{\psi}}_1(\t)(Jx_1)\cdot^\pr\breve{\mb{\l}}_1(\t)(x_2).
%\end{align*}
%etc.

\begin{proof}
Note that the unital homotopy flow equation \eq{hpqflow} has the following unique solution  with initial condition $\bm{\mf}(0)$:
\[
\bm{\mf}(\t) = 
\bm{\mf}(0) + \mb{K}^\pr \circ \left(\int_{0}^{\t}\mb{\xi}(\s)d\s\right) +\left(\int_{0}^{\t}\mb{\xi}(\s)d\s\right)\circ \mb{K}
,
\]
satisfying $\bm{\mf}(\t)(1_\sC)=1_{\sC^\pr}$. It follows that 
$\bm{\mf}(\t)$ is  a family of pointed cochain maps from $\big(\sC{\kkbar},1_{\sC},\mb{K}\big)$
to $\big(\sC{\kkbar},1_{\sC^\pr},\mb{K}^\pr\big)$. We need to check that
$\bm{\mf}(\t)$ also satisfies the $\kbar$-compatibility condition.

We first show that  $\xymatrix{
\big(\underline{\breve{\mb{\psi}}}(\t), \underline{\breve{\mb{\l}}}(\t)\big):
\big(\sC{\kkbar}, 1_\sC, \underline{\breve{\bell}}\big)
\ar@{:>}[r] &
\big(\sC^\pr{\kkbar}, 1_{\sC^\pr}, \underline{\breve{\bell}}^\pr\big)
}$
is a homotopy pair of unital $sL_\infty$-algebras --- see Definition \ref{slinftyhflow} in { Appendix~\ref{appendix: l infinity algebras}}.
Then, we show that $\xymatrix{
\big(\underline{{\mb{\psi}}}(\t), {{\mb{\l}}}(\t)\big):
\big(\sC{\kkbar}, 1_\sC, \underline{{\bell}}\big)
\ar@{:>}[r] &
\big(\sC^\pr{\kkbar}, 1_{\sC^\pr}, \underline{{\bell}}^\pr\big)
}$
is a homotopy pair between the quantum descendant unital $sL_\infty$-algebras --- this   implies
that $\bm{\mf}(\t)$ is a one-parameter family of binary QFT morphisms
such that $\underline{\mb{\psi}}(\t)=\bm{\mK}\big(\bm{\mf}(\t)\big)$.

From
the homotopy flow equations in \eq{hpqflow},
we obtain that
\begin{align*}
\Fr{d}{d\!\t} \bm{\mf}(\t) \circ\mb{\pi}
=
\mb{K}^\pr\circ \mb{\xi}(\t) \circ\mb{\pi}
+ \mb{\xi}(\t) \circ\mb{K}\circ\mb{\pi}
=\mb{K}^\pr\circ \mb{\xi}(\t) \circ\mb{\pi}
+ \mb{\xi}(\t) \circ \mb{\pi}\circ \mD(\breve{\bell})
,
\end{align*}
where we have used $\mb{K}\circ  \mb{\pi} =\mb{\pi}\circ \mD(\breve{\bell})$.
Therefore, we have the following identity:
\eqnalign{hpqflowy}{
\Fr{d}{d\!\t} \big(\bm{\mf}(\t) \circ\mb{\pi}\big)
-\mb{K}^\pr\circ\big( \mb{\xi}(\t) \circ\mb{\pi}\big)
-\big(\mb{\xi}(\t) \circ \mb{\pi}\big)\circ \mD(\breve{\bell})=0
.
}
Substituting for $\bm{\mf}(\t) \circ\mb{\pi}$ and $\mb{\xi}(\t) \circ \mb{\pi}$ with   \eq{hpqdes} and using 
$\mb{K}^\pr\circ\mb{\pi}^\pr=\mb{\pi}^\pr\circ \mD(\breve{\bell}^\pr)$,
%
%we have
%\[
%\eqalign{
%\Fr{d}{d \!\t}\bm{\mf}(\t)\circ\mb{\pi} =
%&\mb{\pi}^\pr\circ\Fr{d}{d \!\t} \mF\left(\breve{\mb{\psi}}(\t)\right)
%-\mb{K}^\pr\circ \Fr{d}{d \!\t}\mb{\eta}(\t)
%-\Fr{d}{d \!\t}\mb{\eta}(\t)\circ \mD(\breve{\bell})
%,\\
%%
%\mb{K}^\pr\circ\mb{\xi}(\t)\circ\mb{\pi}=
%&\mb{K}^\pr\circ\mb{\pi}^\pr\circ\La\left(\breve{\mb{\psi}}(\t), \breve{\mb{\l}}(\t)\right)
%+\mb{K}^\pr\circ\mb{\ups}(\t)\circ \mD(\breve{\bell})
%\\
%=
%&\mb{\pi}^\pr\circ \mD(\breve{\bell}^\pr)\circ\La\left(\breve{\mb{\psi}}(\t), \breve{\mb{\l}}(\t)\right)
%+\mb{K}^\pr\circ\mb{\ups}(\t)\circ \mD(\breve{\bell})
%,\\
%%
%\mb{\xi}(\t)\circ\mb{\pi}\circ \mD(\breve{\bell})=
%&\mb{\pi}^\pr\circ\La\left(\breve{\mb{\psi}}(\t), \breve{\mb{\l}}(\t)\right)\circ \mD(\breve{\bell})
%-\mb{K}^\pr\circ \mb{\ups}(\t) \circ \mD(\breve{\bell})
%,}
%\]
%we obtain that
%\begin{align*}
%\Big(\Fr{d}{d\t}\bm{\mf}(\t)
%&
%- \mb{K}^\pr\circ \mb{\xi}(\t)
%- \mb{\xi}(\t) \circ\mb{K}\Big)\circ\mb{\pi}
%\\
%&
% = \mb{\pi}^\pr\circ \Fr{d}{d\t}\mF\left(\breve{\mb{\psi}}(\t)\right)
%-\mb{K}^\pr\circ\mb{\pi}^\pr\circ\La\left(\breve{\mb{\psi}}(\t), \breve{\mb{\l}}(\t)\right)
%- \mb{\xi}(\t) \circ\mb{\pi}\circ \sD\big(\breve{\bell}\big)
%\\
%&
%= \mb{\pi}^\pr\circ \left(\Fr{d}{d\t}\mF\left(\breve{\mb{\psi}}(\t)\right)
%-\sD\big(\breve{\bell}^\pr\big)\circ\La\left(\breve{\mb{\psi}}(\t), \breve{\mb{\l}}(\t)\right)
%-\La\left(\breve{\mb{\psi}}(\t), \breve{\mb{\l}}(\t)\right)\circ \sD\big(\breve{\bell}\big)\right)
%.
%\end{align*}
%From the homotopy flow equation for the pair $\big(\bm{\mf}(\t),\mb{\xi}(\t)\big)$, we have
we obtain that
\eqn{hpqflowz}{
\mb{\pi}^\pr\circ \left(\Fr{d}{d\!\t}\mF\left(\breve{\mb{\psi}}(\t)\right)
-\sD\big(\breve{\bell}^\pr\big)\circ\La\left(\breve{\mb{\psi}}(\t), \breve{\mb{\l}}(\t)\right)
-\La\left(\breve{\mb{\psi}}(\t), \breve{\mb{\l}}(\t)\right)\circ \sD\big(\breve{\bell}\big)\right)
=0.
}
From $\mb{\pi}^\pr_1=\I_{\sC^\pr{\kkbar}}$, we have
\[
\proj_{\sC^\pr{\kkbar}}\circ \left(\Fr{d}{d\t}\mF\left(\breve{\mb{\psi}}(\t)\right)
-\sD\big(\breve{\bell}^\pr\big)\circ\La\left(\breve{\mb{\psi}}(\t), \breve{\mb{\l}}(\t)\right)
-\La\left(\breve{\mb{\psi}}(\t), \breve{\mb{\l}}(\t)\right)\circ \sD\big(\breve{\bell}\big)\right)
=0.
\]
In components, this condition means that,
for all $n\geq 1$ and homogeneous $\bm{x}_1,\ldots,\bm{x}_n \in \sC{\kkbar}$,
we have
\begin{align}
\label{brehflow}
    & 
\Fr{d}{d\!t}\breve{\mb{\psi}}_n(\t)\big(\bm{x}_1,\ldots, \bm{x}_n\big)= 
\\
    & \mathmakebox[\displaywidth][r]{
    \sum_{\mathclap{\mp\in P(n)}}\ \ \sum_{\mathclap{i=1}}^{|\mp|}\e(\mp)
\breve{\bell}^\pr_{|\mp|}\Big( \breve{\mb{\psi}}(\t)\big(J\!{x}_{B_1}\big),\ldots,
\breve{\mb{\psi}}(\t)\big(J\!\bm{x}_{B_i}\big),\breve{\mb{\l}}(\t)\big(x_{B_i}\big),\breve{\mb{\psi}}(\t)\big(\bm{x}_{B_{i+1}}\big),
\ldots, \breve{\mb{\psi}}(\t)\big(\bm{x}_{B_{|\mp|}}\big)\Big)
    }\notag
  \\
  &
  \mathmakebox[\displaywidth][r]{
+\sum_{\clapsubstack{\mp\in P(n)\\ |B_i|=n-|\mp|+1}}\e(\mp)
\breve{\mb{\l}}_{|\mp|}(\t)\Big(J\!\bm{x}_{B_1},\ldots, J\!\bm{x}_{B_{i-1}}, \breve{\bell}\big(\bm{x}_{B_i}\big),  \bm{x}_{B_{i+1}},\ldots,
 \bm{x}_{B_{|\mp|}}\Big).
   }\notag
\end{align}
It is straightforward to check that $\breve{\mb{\l}}_n(x_1,\ldots, x_{n-1}, 1_\sC)=0$ 
for all $n\geq 1$ and homogeneous $\bm{x}_1,\ldots,\bm{x}_{n-1}\in\sC{\kkbar}$. 
Therefore,
\[\xymatrix{
\big(\underline{\breve{\mb{\psi}}}(\t), \underline{\breve{\mb{\l}}}(\t)\big):
\big(\sC{\kkbar}, 1_\sC, \underline{\breve{\bell}}\big)
\ar@{:>}[r] &
\big(\sC^\pr{\kkbar}, 1_{\sC^\pr}, \underline{\breve{\bell}}^\pr\big)
}\]
is a homotopy pair of unital $sL_\infty$-algebras.

Now set ${\mb{\psi}}_n(\t)=\Fr{1}{(-\kbar)^{n-1}}\breve{\mb{\psi}}_n(\t)$, for all $n\geq 1$.
If ${\mb{\psi}}_n(\t)$ is in $\Hom\big(S^n\sC, \sC^\pr\big)^{0}{\kkbar}[\t]$ for all $n\geq 1$, 
then the first relation in \eq{hpqdes} becomes
\eqn{fakedes}{
\bm{\mf}(\t)\circ\mb{\pi} =
\mb{\pi}^\pr\circ\mb{\Psi}_{{\mb{\psi}(\t)}},
}
which implies that $\bm{\mf}(\t)$ is a family of morphisms of binary QFT algebras
and $\underline{{\mb{\psi}}}(\t)=\bm{\mK}\big(\bm{\mf}(\t)\big)$.
We are going to show that 
${\mb{\psi}}_n(\t)$ is indeed in $\Hom\big(S^n\sC, \sC^\pr\big)^{0}{\kkbar}[\t]$ for all $n\geq 1$.

From the assumption that $\bm{\mf}(0)$ is a morphism of binary QFT algebras,
it follows that $\underline{{\mb{\psi}}}(0)=\bm{\mK}\big(\bm{\mf}(0)\big)$ is a unital $sL_\infty$-morphism
from $\big(\sC{\kkbar}, 1_\sC, \underline{{\bell}}\big)$ to $\big(\sC^\pr{\kkbar}, 1_{\sC^\pr}, \underline{{\bell}}^\pr\big)$.
In particular we have ${\mb{\psi}}_n(0) \in \Hom\big(S^n\sC, \sC^\pr\big)^{0}{\kkbar}$ for all $n\geq 1$.
From $\breve{\mb{\psi}}_1(\t) = \bm{\mf}(\t)   \in  \Hom\big(\sC, \sC^\pr\big)^{0}{\kkbar}[\t]$,
we have ${\mb{\psi}}_1(\t)  \in  \Hom\big(\sC, \sC^\pr\big)^{0}{\kkbar}[\t]$.
Fix $n > 1$ and assume that $\mb{\psi}_k(\t) \in  \Hom\big(S^k\sC, \sC^\pr\big)^{0}{\kkbar}$ for all $k < n$
and all $\t$. Note that ${\mb{\l}}_n(\t) \coloneqq{}\Fr{1}{(-\kbar)^{n-1}}\breve{\mb{\l}}_n(\t)\in  \Hom\big(S^k\sC, \sC^\pr\big)^{-1}{\kkbar}$ for all $n\geq 1$ by $\kbar$-compatibility.
From \eq{brehflow}, we have, for homogeneous $\bm{x}_1,\ldots,\bm{x}_n \in \sC{\kkbar}$,
\begin{align*}
\Fr{d}{d\t}&\breve{\mb{\psi}}_n(\t)\big(\bm{x}_1,\ldots, \bm{x}_n\big)
\\
= 
&
{(-\kbar)^{n-1}}\sum_{\mathclap{\mp\in P(n)}}\ \ \sum_{\mathclap{i=1}}^{|\mp|}
\e(\mp)
\bell^\pr_{|\mp|}\Big( \mb{\psi}(\t)\big(Jx_{B_1}\big),\ldots,
\mb{\l}(\t)\big(x_{B_1}\big),\ldots, \mb{\psi}(\t)\big(x_{B_{|\mp|}}\big)\Big)
\\
&
+{(-\kbar)^{n-1}}\sum_{\clapsubstack{\mp\in P(n)\\ |B_i|=n-|\mp|+1}}\e(\mp)
\mb{\l}_{|\mp|}(\t)\Big(J\bm{x}_{B_1},\ldots, J\bm{x}_{B_{i-1}}, \bell\big(\bm{x}_{B_i}\big),  \bm{x}_{B_{i+1}},\ldots,
 \bm{x}_{B_{|\mp|}}\Big),
\end{align*}
which implies that  $\Fr{d}{d\!\t}{\mb{\psi}}_n(\t) \in \Hom\big(S^n\sC, \sC^\pr\big)^{0}{\kkbar}[\t]$.
Combined with the condition ${\mb{\psi}}_n(0) \in \Hom\big(S^n\sC, \sC^\pr\big)^{0}{\kkbar}$,
we can conclude that ${\mb{\psi}}_n(\t)$ is in fact contained in $\Hom\big(S^n\sC, \sC^\pr\big)^{0}{\kkbar}[\t]$, as desired.
It follows by induction that  ${\mb{\psi}}_n(\t)$ is contained in $\in \Hom\big(S^n\sC, \sC^\pr\big)^{0}{\kkbar}[\t]$ for all $n\geq 1$.  
Therefore, we have, for all $n\geq 1$ and homogeneous $\bm{x}_1,\ldots,\bm{x}_n \in \sC{\kkbar}$,
\begin{align*}
\Fr{d}{d\t}{\mb{\psi}}_n(\t)&\big(\bm{x}_1,\ldots, \bm{x}_n\big)
\\
= 
&
\sum_{\mathclap{\mp\in P(n)}}\ \ \sum_{\mathclap{i=1}}^{|\mp|}
\e(\mp)
\bell^\pr_{|\mp|}\Big( \mb{\psi}(\t)\big(Jx_{B_1}\big),\ldots,
\mb{\l}(\t)\big(x_{B_i}\big),\ldots, \mb{\psi}(\t)\big(x_{B_{|\mp|}}\big)\Big)
\\
&
+\sum_{\clapsubstack{\mp\in P(n)\\ |B_i|=n-|\mp|+1}}\e(\mp)
\mb{\l}_{|\mp|}(\t)\Big(J\bm{x}_{B_1},\ldots, J\bm{x}_{B_{i-1}}, 
\bell\big(\bm{x}_{B_i}\big),  \bm{x}_{B_{i+1}},\ldots,\bm{x}_{B_{|\mp|}}\Big)
,
\end{align*}
which means that
$\xymatrix{
\big(\underline{{\mb{\psi}}}(\t), \underline{{\mb{\l}}}(\t)\big):
\big(\sC{\kkbar}, 1_\sC, \underline{\bell}\big)
\ar@{:>}[r] &
\big(\sC^\pr{\kkbar}, 1_{\sC^\pr}, \underline{\bell}^\pr\big)
}$
is a homotopy pair between the quantum descendant unital $sL_\infty$-algebras.
Now \eq{fakedes} implies that $\underline{\mb{\psi}}(\t)=\bm{\mK}\big(\bm{\mf}(\t)\big)$, so that
$\bm{\mf}(\t):\sC{\kkbar}_\BQ\longrightarrow \sC^\pr{\kkbar}_\BQ$ is a
uniquely defined  $1$-parameter
family of morphisms of binary QFT algebras.
\naturalqed
\end{proof}

\begin{remark}
It can be checked that the identity \eq{hpqflowz} implies that
\eqn{hpqflowxz}{
\Fr{d}{d\t}\mF\left(\breve{\mb{\psi}}(\t)\right)=
\sD\big(\breve{\bell}^\pr\big)\circ\La\left(\breve{\mb{\psi}}(\t), \breve{\mb{\l}}(\t)\right)
+\La\left(\breve{\mb{\psi}}(\t), \breve{\mb{\l}}(\t)\right)\circ \sD\big(\breve{\bell}\big)
.
}
\end{remark}
%\begin{proposition}\label{hpflow}
%Let $\xymatrix{\big(\bm{\mf}(\t),\mb{\eta}(\t)\big|\mb{\xi}(\t),\mb{\ups}(\t)\big):\sC{\kkbar}_\BQ\ar@{=>}[r] & \sC^\pr{\kkbar}_\BQ}$ 
%be a homotopy pair of binary QFT algebras.
%Then, $(\bm{\mf}(\t),\mb{\eta}(\t))$ is a $1$-parameter family of morphisms of binary QFT algebra
%whenever $(\bm{\mf}(0),\mb{\eta}(0))$ is a morphism of  binary QFT algebra.
%\end{proposition}
%
%
%\begin{proof}
%Note that the homotopy flow equation is a unique solution for a fixed initial condition $F(0)$:
%\[
%\bm{\mf}(\t) = \bm{\mf}(0) + \mb{K}^\pr \circ \mb{\zeta}(\t) + \mb{\zeta}(\t)\circ \mb{K}
%\]
%where $\mb{\zeta}(\t)=\int_{0}^{\t}\mb{\xi}(\s)d\s$ and $\mb{\zeta}(\t)(1_\sC)=0$. 
%It follows that $\bm{\mf}(\t)$ is an $1$-parameter family of
%morphisms of QFT complex if $\bm{\mf}(0)$ was a morphism of QFT complex. 
%Assume that $\bm{\mf}(0)$ also has the $\kbar$-compatibility 
%that  $\Fr{1}{(-\kbar)^{n-1}}\breve{\mb{\psi}}_n(0) \in \Hom\big(S^n\sC, \sC^\pr\big)^{0}{\kkbar}$ for all $n\geq 1$, 
%i.e., $\bm{\mf}(0)$ is morphism of  binary QFT algebra. 
%From {\em Lemma \ref{BKfhomp}},
%it follows  that  $\Fr{1}{(-\kbar)^{n-1}}\breve{\mb{\psi}}_n(\t) \in \Hom\big(S^n\sC, \sC^\pr\big)^{0}{\kkbar}$ for all $\t$ 
%and for all $n\geq 1$. Therefore, $\bm{\mf}(\t)$ is a $1$-parameter family of morphisms of binary QFT algebra.
%\naturalqed
%\end{proof}

\begin{definition}
Two morphisms $\bm{\mf}, \bm{\tilde\mf}: \sC{\kkbar}_{\BQ}\longrightarrow\sC^\pr{\kkbar}_{\BQ}$ of 
binary QFT algebras are homotopic, which we denote by $\bm{\mf}\sim_\kbar \bm{\tilde\mf}$ 
if there is a homotopy pair $\big(\bm{\mf}(\t)\big|\mb{\xi}(\t)\big)$ 
of binary QFT algebras
such that $\bm{\mf}(0)=\bm{\mf}$ and 
$\bm{\mf}(1)=\bm{\tilde\mf}$.
\end{definition}

It is clear that $\sim_{\kbar}$ is an equivalence relation. 
We say that two morphisms of binary QFT algebras have the same homotopy type if they are homotopic to each other.

\begin{lemma}\label{hodesfuc}
Quantum descendants of homotopic morphisms of  binary QFT algebras  are homotopic as  morphisms of 
the quantum descendant unital $sL_\infty$-algebras.
\end{lemma}

\begin{proof}
Assume that $\bm{\mf}\sim_\kbar \bm{\tilde\mf}: \sC{\kkbar}_\BQ\longrightarrow\sC^\pr{\kkbar}_\BQ$ 
are homotopic morphisms of binary QFT algebras.
Then, by definition, there is a homotopy pair$\big(\bm{\mf}(\t)\big|\mb{\xi}(\t)\big)$ 
of binary QFT algebras
such that $\bm{\mf}(0)=\bm{\mf}$ and 
$\bm{\mf}(1)=\bm{\tilde\mf}$.
From Proposition \ref{BKfhomp}, it follows that
there is a corresponding homotopy pair $\big(\underline{{\mb{\psi}}}(\t), \underline{{\mb{\l}}}(\t)\big)$
between the quantum descendant unital $sL_\infty$-algebras 
such that $\underline{\mb{\psi}}(0)=\bm{\mK}\big(\bm{\mf}\big)$ and
$\underline{\tilde{\mb{\psi}}}(1)=\bm{\mK}\big(\tilde{\bm{\mf}}\big)$.
Therefore $\underline{\mb{\psi}}(0)$ and $\underline{\mb{\psi}}(1)$ are homotopic as morphisms of unital $sL_\infty$-algebras.
\naturalqed
\end{proof}

\begin{definition}
A morphism $\bm{\mf}:\sC{\kkbar}_\BQ\longrightarrow\sC^\pr{\kkbar}_\BQ$ of binary
QFT algebras is a homotopy equivalence if there is a
morphism  $\bm{\mf}^\pr:\sC^\pr{\kkbar}_\BQ\longrightarrow\sC{\kkbar}_\BQ$
of binary QFT algebras such that $\bm{\mf}\circ\bm{\mf} \sim_{\kbar} \I_{\sC{\kkbar}}$
and $\bm{\mf}\circ\bm{\mf}^\pr \sim_{\kbar}\I_{\sC{\kkbar}}$.
\end{definition}

A homotopy equivalence of binary
QFT algebras is automatically a quasi-isomorphism of binary
QFT algebras. It is also obvious that the quantum descendant of 
a homotopy equivalence of binary QFT algebras
is a homotopy equivalence of the quantum descendant unital $sL_\infty$-algebras.
We are going to define the homotopy category of binary QFT algebras
as the category whose objects are  binary QFT algebras and whose morphisms
are the homotopy types of morphisms of binary QFT algebras. Then the homotopy type of a homotopy equivalence
is an isomorphism in the homotopy category. 

\begin{theorem}
There is a homotopy category $\mathit{ho}\category{BQFTA}(\Bbbk)$ whose objects are binary QFT algebras over $\Bbbk$ and whose morphisms are
the homotopy types of morphisms of binary QFT algebras. 
\end{theorem}

The above theorem is a corollary of the following proposition:

\begin{proposition}\label{techniq}
Consider the following diagram in the category $\category{BQFTA}(\Bbbk)$ of binary QFT algebras:
$
\xymatrix{
\sC{\kkbar}_{\BQ}
\ar@/^/[r]^{\bm{\mf}}
\ar@/_/[r]_{\bm{\tilde\mf}}
&\sC^\pr{\kkbar}_{\BQ}
\ar@/^/[r]^{\bm{\mf}^\pr}
\ar@/_/[r]_{\bm{\tilde\mf}^\pr}
&\sC^\ppr {\kkbar}_{\BQ},
}
$ 
and
assume that $\bm{\mf}  \sim_{\kbar} \bm{\tilde\mf}$ and
$\bm{\mf}^\pr\sim_{\kbar} \bm{\tilde\mf}^\pr$
Then, 
we have
$\bm{\mf}^\pr\circ \bm{\mf}\sim_{\kbar}
\bm{\tilde\mf}^\pr\circ \bm{\tilde\mf}$
as morphisms of binary QFT algebra from $\sC{\kkbar}_{\BQ}$ to
$\sC^\pr{\kkbar}_{\BQ}$, and
the homotopy type of $\bm{\mf}^\pr\circ \bm{\mf}$ 
depends only on the homotopy types of $\bm{\mf}^\pr$
and $\bm{\mf}$.
\end{proposition}

\begin{proof}
From the assumption in the proposition, we have
a sequence of homotopy pairs of binary QFT algebras
\[
\xymatrix{
\sC{\kkbar}_\BQ\ar@{=>}[rr]^-{(\bm{\mf}(\t)|\mb{\xi}(\t))} && \sC^\pr{\kkbar}_\BQ
\ar@{=>}[rr]^-{(\bm{\mf}^\pr(\t)|\mb{\xi}^\pr(\t)} && \sC^\ppr{\kkbar}_\BQ
}
\]
such that
\begin{align*}
\begin{cases}
\bm{\mf}(0)=\bm{\mf}
,\\
\bm{\mf}(1)=\bm{\tilde\mf}
,
\end{cases}
&\quad{}
\begin{cases}
\bm{\mf}^\pr(0)=\bm{\mf}^\pr
,\\
\bm{\mf}^\pr(1)=\bm{\tilde\mf}^\pr.
\end{cases}
\end{align*}
% \[
% \left\{
% \eqalign{
% \bm{\mf}(0)=\bm{\mf}
% ,\cr
% \bm{\mf}(1)=\bm{\tilde\mf}
% ,}\right.
% \qquad
% \left\{
% \eqalign{
% \bm{\mf}^\pr(0)=\bm{\mf}^\pr
% ,\cr
% \bm{\mf}^\pr(1)=\bm{\tilde\mf}^\pr.
% }\right.
% \]
% \end{proof}
% \end{document}
Let
\[
\big(\breve{\mb{\psi}}(\t),\breve{\mb{\l}}(\t)\big)=
\breve{\bm{\mK}}\big(\bm{\mf}(\t)\big|\mb{\xi}(\t)\big)
,\quad
\big(\breve{\mb{\psi}}(\t),\breve{\mb{\l}}(\t)\big)=
\breve{\bm{\mK}}\big(\bm{\mf}(\t)\big|\mb{\xi}(\t)\big).
\]
We define the composition of the given homotopy pairs of binary QFT algebras
as follows:
\[
\big(\bm{\mf}^\pr(\t)\big|\mb{\xi}^\pr(\t)\big)
\diamond
\big(\bm{\mf}(\t)\big|\mb{\xi}(\t)\big)
\coloneqq{}\big(\bm{\mf}^\ppr(\t)\big|\mb{\xi}^\ppr(\t)\big),
\]
where
\[
\begin{cases}
\bm{\mf}^\ppr(\t)\coloneqq{}& \bm{\mf}^\pr(\t)\circ \bm{\mf}(\t)
,\\
\mb{\xi}^\ppr(\t)\coloneqq{}&\bm{\mf}^\pr(\t)\circ\mb{\xi}(\t)+\mb{\xi}^\pr(\t)\circ\bm{\mf}(\t)
,\\
\end{cases}
\]
Note that
\[
\bm{\mf}^\ppr(0)=\bm{\mf}\circ \bm{\mf}^\pr
,\qquad
\bm{\mf}^\ppr(1)=\bm{\tilde\mf}\circ \bm{\tilde\mf}^\pr
\]
Then what we need  to show is  that $\big(\bm{\mf}^\ppr(\t)\big|\mb{\xi}^\ppr(\t)\big)$
is a homotopy pair of binary QFT algebras from $\sC{\kkbar}_\BQ$ to $\sC^\ppr{\kkbar}_\BQ$.
It can be checked easily that $\big(\bm{\mf}^\ppr(\t)\big|\mb{\xi}^\ppr(\t)\big)$ satisfies the unital homotopy flow equation:
it is obvious that $\mb{\xi}^\ppr(\t)(1_\sC)=0$, 
and
\[
\Fr{d}{d\!\t}\bm{\mf}^\ppr(\t)=
\Fr{d}{d\!\t}\bm{\mf}^\pr(\t)\circ \bm{\mf}(\t)
+\bm{\mf}^\pr(\t)\circ \Fr{d}{d\!\t}\bm{\mf}(\t)
,\qquad
\begin{cases}
\Fr{d}{d\!\t} \bm{\mf}(\t) 
=
\mb{K}^\pr\circ \mb{\xi}(\t) 
+ \mb{\xi}(\t) \circ\mb{K}
,\\\\
\Fr{d}{d\!\t} \bm{\mf}^\pr(\t) 
=
\mb{K}^\ppr\circ \mb{\xi}^\pr(\t) 
+ \mb{\xi}^\pr(\t) \circ\mb{K}^\pr
,
\end{cases}
\]
so it is trivial that $\Fr{d}{d\!\t}\bm{\mf}^\ppr(\t)=\mb{K}^\ppr\circ \mb{\xi}^\ppr(\t) +\mb{\xi}^\ppr(\t) \circ\mb{K}$.
What remains is to check is the $\kbar$-condition. 

Consider 
the descendant $\left(\breve{\mb{\psi}}^\ppr(\t), \breve{\mb{\l}}^\ppr(\t)\right)$  of $\big(\bm{\mf}^\ppr(\t)\big|\mb{\xi}^\ppr(\t)\big)$:
\eqn{hpqdesppr}{
\begin{cases}
\bm{\mf}^\ppr(\t)\circ\mb{\pi} =
\mb{\pi}^\ppr\circ\mF\left(\breve{\mb{\psi}}^\ppr(\t)\right)
,\\
\mb{\xi}^\ppr(\t)\circ\mb{\pi}=
\mb{\pi}^\ppr\circ\La\left(\breve{\mb{\psi}}^\ppr(\t), \breve{\mb{\l}}^\ppr(\t)\right)
,
\end{cases}
}
From
\[
\begin{cases}
\bm{\mf}^\pr(\t)\circ\mb{\pi}^\pr =
\mb{\pi}^\ppr\circ\mF\left(\breve{\mb{\psi}}^\pr(\t)\right)
,\\
\mb{\xi}(\t)^\pr\circ\mb{\pi}^\pr=
\mb{\pi}^\ppr\circ\La\left(\breve{\mb{\psi}}^\pr(\t), \breve{\mb{\l}}^\pr(\t)\right)
\end{cases}
\qquad
\begin{cases}
\bm{\mf}(\t)\circ\mb{\pi} =
\mb{\pi}^\pr\circ\mF\left(\breve{\mb{\psi}}(\t)\right)
,\\
\mb{\xi}(\t)\circ\mb{\pi}=
\mb{\pi}^\pr\circ\La\left(\breve{\mb{\psi}}(\t), \breve{\mb{\l}}(\t)\right)
\end{cases}
\]
we  obtain that
\eqn{hpqdespr}{
\begin{cases}
\bm{\mf}^\ppr(\t)\circ\mb{\pi} =
\mb{\pi}^\ppr\circ\mF\left(\breve{\mb{\psi}}^\pr(\t)\right)\circ\mF\left(\breve{\mb{\psi}}(\t)\right)
,\\
\mb{\xi}^\ppr(\t)\circ\mb{\pi}=
\mb{\pi}^\ppr\circ\left(\mF\left(\breve{\mb{\psi}}^\pr(\t)\right)\circ\La\left(\breve{\mb{\psi}}(\t), \breve{\mb{\l}}(\t)\right)
+\La\left(\breve{\mb{\psi}}^\pr(\t), \breve{\mb{\l}}^\pr(\t)\right)\circ\mF\left(\breve{\mb{\psi}}^\pr(\t)\right)\right)
\end{cases}
}
Comparing this with \eq{hpqdesppr}, we  obtain  that
\begin{align*}
\proj_{\sC^\ppr{\kkbar}}\circ\mF\left(\breve{\mb{\psi}}^\ppr(\t)\right) =
&
\proj_{\sC^\ppr{\kkbar}}\circ\mF\left(\breve{\mb{\psi}}^\pr(\t)\right)\circ\mF\left(\breve{\mb{\psi}}(\t)\right)
,\\
\proj_{\sC^\ppr{\kkbar}}\circ\La\left(\breve{\mb{\psi}}^\ppr(\t), \breve{\mb{\l}}^\ppr(\t)\right)=
&\proj_{\sC^\ppr{\kkbar}}\circ\mF\left(\breve{\mb{\psi}}^\pr(\t)\right)\circ\La\left(\breve{\mb{\psi}}(\t), \breve{\mb{\l}}(\t)\right)
\\
&
+\proj_{\sC^\ppr{\kkbar}}\circ\La\left(\breve{\mb{\psi}}^\pr(\t), \breve{\mb{\l}}^\pr(\t)\right)\circ\mF\left(\breve{\mb{\psi}}^\pr(\t)\right)
,
\end{align*}
which implies the following:
\begin{itemize}
\item from
$\breve{\mb{\psi}}^\ppr(\t)=\proj_{\sC^\ppr{\kkbar}}\circ\mF\left(\breve{\mb{\psi}}^\ppr(\t)\right)$,
we have
\eqn{fsfuy}{
\breve{\mb{\psi}}^\ppr(\t) = \breve{\mb{\psi}}^\pr(\t) \circ\mF\left(\breve{\mb{\psi}}(\t)\right)
\equiv \breve{\mb{\psi}}^\pr(\t)\bullet \breve{\mb{\psi}}(\t);
}
\item from
$\breve{\mb{\l}}^\ppr(\t)=\proj_{\sC^\ppr{\kkbar}}\circ \La\left(\breve{\mb{\psi}}^\ppr(\t), \breve{\mb{\l}}^\ppr(\t)\right)$, we have
\eqn{fsfuisay}{
\breve{\mb{\l}}^\ppr(\t) = \breve{\mb{\psi}}^\pr(\t)\circ \La\left(\breve{\mb{\psi}}(\t), \breve{\mb{\l}}(\t)\right)
+\breve{\mb{\l}}^\pr(\t) \circ\mF\left(\breve{\mb{\psi}}(\t)\right)
.
}
\end{itemize}

Recall that for all $n\geq 1$, we have
\begin{align*}
{\mb{\l}}^\pr(\t)_n\coloneqq{}&\Fr{1}{(-\kbar)^{n-1}}\breve{\mb{\l}}^\pr(\t)_n \in \Hom\big(S^n\sC^\pr, \sC^\ppr)^{-1}{\kkbar}
,\\
{\mb{\l}}(\t)_n\coloneqq{}&\Fr{1}{(-\kbar)^{n-1}}\breve{\mb{\l}}(\t)_n\in \Hom\big(S^n\sC, \sC^\pr)^{-1}{\kkbar}
,\\
{\mb{\psi}}^\pr(\t)_n\coloneqq{}&\Fr{1}{(-\kbar)^{n-1}}\breve{\mb{\psi}}^\pr(\t)_n\in \Hom\big(S^n\sC^\pr, \sC^\ppr)^{0}{\kkbar}
,\\
{\mb{\psi}}(\t)_n\coloneqq{}&\Fr{1}{(-\kbar)^{n-1}}\breve{\mb{\l}}(\t)_n\in \Hom\big(S^n\sC, \sC^\pr)^{0}{\kkbar}
,
\end{align*}
From \eq{fsfuisay}, we can check that 
${\mb{\l}}^\ppr(\t)_n\coloneqq{}\Fr{1}{(-\kbar)^{n-1}}\breve{\mb{\l}}^\ppr(\t)  \in \Hom\big(S^n\sC, \sC^\ppr)^{-1}{\kkbar}$.
Therefore, we conclude that $\big(\bm{\mf}^\ppr(\t)\big|\mb{\xi}^\ppr(\t)\big)$
is a homotopy pair of binary QFT algebras from $\sC{\kkbar}_\BQ$ to $\sC^\ppr{\kkbar}_\BQ$.
\naturalqed
\end{proof}

Note that $\big(\breve{\mb{\psi}}^\ppr(\t)\big|\breve{\mb{\l}}^\ppr(\t)\big)$ is the descendant of 
$\big(\bm{\mf}^\ppr(\t)\big|\mb{\xi}^\ppr(\t)\big)$ and $\big(\underline{\mb{\psi}}^\ppr(\t)\big|\underline{\mb{\l}}^\ppr(\t)\big)$
is the homotopy pair of the quantum descendant unital $sL_\infty$-algebras 
from $\big(\sC{\kkbar}, 1_{\sC}, \underline{\bell}\big)$ to $\big(\sC^\ppr{\kkbar}, 1_{\sC^\ppr}, \underline{\bell}^\ppr\big)$
with $\underline{\mb{\psi}}^\ppr(0) = \underline{\mb{\psi}}^\pr\circ \underline{\mb{\psi}}$ and
$\underline{\mb{\psi}}^\ppr(1) = \underline{\mb{\tilde\psi}}^\pr\circ \underline{\mb{\tilde\psi}}$,
where  $\underline{\mb{\psi}}^\pr=\bm{\mK}\big(\bm{\mf}^\pr\big)$,
$\underline{\mb{\psi}}=\bm{\mK}\big(\bm{\mf}\big)$, 
$\underline{\mb{\tilde\psi}}^\pr=\bm{\mK}\big(\bm{\tilde\mf}^\pr\big)$,
and
$\underline{\mb{\tilde\psi}}=\bm{\mK}\big(\bm{\tilde\mf}\big)$.
Combined with Lemma \ref{hodesfuc},  the above proposition also implies the following.

\begin{theorem}
The quantum descendant functor $\bm{\mK}: \category{BQFTA}\big(\Bbbk\big)
\rightsquigarrow \category{UsL}_\infty(\Bbbk{\kkbar})$ is a homotopy functor,
i.e., it induces a functor 
$
\mathit{ho}\bm{\mK}: \mathit{ho}\category{BQFTA}\big(\Bbbk\big)
\rightsquigarrow \mathit{ho}\category{UsL}_\infty(\Bbbk{\kkbar})$
between the respective homotopy categories.
\end{theorem}

\subsection{The category and homotopy category of binary CFT algebras}

We obtain the
notion of binary CFT (classical field theory) algebra over $\Bbbk$
by taking the classical limit $\big(\sC, 1_{\sC}, \,\cdot\,, \underline{\ell}\big)$ of 
the tuple $\big(\sC{\kkbar}, 1_{\sC}, \,\cdot\,, \underline{\bell}\big)$, which combines the
structure of a binary QFT algebra $\big(\sC{\kkbar}, 1_{\sC}, \,\cdot\,, \mb{K}\big)$ with
its quantum descendant $\big(\sC{\kkbar}, 1_{\sC}, \underline{\bell}\big)$.
The classical limit of the quantum descendant of a morphism of binary QFT algebras gives
rise naturally to a morphism of the associated binary CFT algebras, etc., and yields
the following definition of the category $\BCFTA(\Bbbk)$ of binary CFT algebras over $\Bbbk$.

\begin{definition}[Category of binary CFT algebras]\label{cldesalg}
A binary CFT algebra over $\Bbbk$ is a tuple
$\sC_\BC=\big(\sC, 1_\sC, \,\cdot\,, \underline{\ell}\big)$, where 
$\big(\sC, 1_\sC, \,\cdot\,\big)$ is a unital $\Z$-graded commutative and associative algebra
and $\big(\sC, 1_\sC, \underline{\ell}\big)$ is a unital $sL_\infty$-algebra,
satisfying the compatibility condition that,
for all $n\geq 1$ and homogeneous $x_1,\ldots, x_{n+1} \in \sC$, we have
\begin{align*}
\ell_{n}\big(x_1,\ldots, x_{n-1}, x_{n}\cdot x_{n+1}\big)=&
\ell_{n}\big(x_1,\ldots, x_{n-1}, x_{n}\big)\cdot x_{n+1}
\\
&
+(-1)^{|x_{n}|(|x_1|+\ldots+|x_{n-1}|)} Jx_{n}\cdot
\ell_{n}\big(x_1,\ldots, x_{n-2}, x_{n+1}\big)
.
\end{align*}
A morphism  $\xymatrix{\underline{\psi}: \sC_\BC\ar@{..>}[r]&\sC^\pr_\BC}$ of binary CFT algebras 
is a morphism of the underlying unital $sL_\infty$-algebras
such that
for all $n\geq 1$ and homogeneous ${x}_1,\ldots, {x}_{n+1} \in \sC$, we have
\begin{align*}
{\psi}_{n}\big({x}_1,\ldots, {x}_{n-1}, {x}_{n}\cdot {x}_{n+1}\big)=
\sum_{\clapsubstack{ \mp \in P(n+1)\\ |\mp|=2\\ n\nsim n+1}}{\psi}({x}_{B_1})\cdot^\pr {\psi}({x}_{B_2})
.
\end{align*}
The composition of two consecutive morphisms of binary CFT algebras is
the composition as morphisms of the underlying unital $sL_\infty$-algebras.
\end{definition}
It can be checked that the composition of  two consecutive morphisms of binary CFT algebras is
a morphism of binary CFT algebras so that binary CFT algebras over $\Bbbk$
form a category $\BCFTA(\Bbbk)$.

Now we turn to the homotopy category $\mathit{ho}\BCFTA(\Bbbk)$ of binary CFT algebras over $\Bbbk$.  

\begin{definition}\label{hpofbcfta}
A homotopy pair $\xymatrix{\big({\psi}(\t)\big|{\l}(\t)\big):\sC_\BC\ar@{:>}[r]&\sC^\pr_\BC}$ of binary CFT algebras 
is a homotopy pair  of the
underlying unital $sL_\infty$-algebras 
with the following set of additional conditions: for all $n\geq 1$ and homogeneous $x_1,\ldots,x_n \in \sC$,
\begin{align*}
{\l}_{n}(\t)\big( &{x}_1,\ldots, {x}_{n-1},  {x}_{n}\cdot {x}_{n+1}\big)
\\
&=
\sum_{\clapsubstack{ \mp \in P(n+1)\\ |\mp|=2\\ n\nsim n+1}}{\l}(\t)({x}_{B_1})\cdot^\pr {\psi}(\t)({x}_{B_2})
+{\psi}(\t)(J\!{x}_{B_1})\cdot^\pr {\l}(\t)({x}_{B_2})
.
\end{align*}
\end{definition}

Recall that $\big({\psi}(\t)\big|{\l}(\t)\big)$ is a homotopy pair 
of the
underlying unital $sL_\infty$-algebras from 
$\big(\sC, 1_\sC, \underline{\ell}\big)$ to $\big(\sC^\pr, 1_{\sC^\pr}, \underline{\ell}^\pr\big)$
if 
\[\big({\psi}(\t)\big|{\l}(\t)\big)\in \Hom\left(\Xbar{S}(\sC),\sC^\pr\right)^0[\t]\oplus  
\Hom\left(\Xbar{S}(\sC),\sC^\pr\right)^{-1}[\t]\]
and the data satisfies both
the unit condition that $\l(\t)\big(x_1\odot\ldots x_{n-1}\odot 1_\sC\big)=0$ for all $n\geq 1$ and $x_1,\ldots,x_n \in \sC$,
and  the homotopy flow equation generated by $\l(\t)$, namely
\[
\Fr{d}{d\!t}\mF\big(\psi(\t)\big) = \mD(\ell^\pr)\circ \La\big(\psi(\t),\l(\t)\big) +\La\big(\psi(\t),\l(\t)\big)\circ \mD(\ell).
\]
It follows that $\underline{\psi}(\t)$ is a uniquely defined family of unital $sL_\infty$-morphisms
whenever $\underline{\psi}(0)$ is a unital $sL_\infty$-morphism.  The additional set of conditions in Definition \ref{hpofbcfta}
for a homotopy pair $\big({\psi}(\t)\big|{\l}(\t)\big)$ of binary CFT algebras
implies that  $\underline{\psi}(\t)$ is a uniquely defined family of binary CFT algebra morphisms
whenever $\underline{\psi}(0)$ is a binary CFT algebra morphism. 
Therefore, we say that two morphisms $\underline{\psi}$ and $\underline{\tilde\psi}$ of binary CFT algebras 
are homotopic and have the same homotopy type 
if there is a homotopy pair $\big({\psi}(\t)\big|{\l}(\t)\big)$ of binary CFT algebras such that 
$\underline{\psi}(0)=\underline{\psi}$ and $ \underline{\psi}(1)=\underline{\tilde\psi}$.

It can be checked that the homotopy type of the composition of two consecutive morphisms
of binary CFT algebras depends only on the homotopy types of the constituents.
Therefore, we can form the homotopy category $\mathit{ho}\BCFTA(\Bbbk)$ of binary CFT algebras,
whose objects are binary CFT algebras and whose morphisms are the homotopy types of morphisms of binary CFT algebras.

\begin{remark}
From Lemma \ref{remBKftalg}, it follows that the combined classical limit of 
a binary QFT algebra and its quantum descendant is a binary CFT algebra.
From Lemma \ref{remBKftmor}, it follows that
the classical limit of  the quantum descendant $\underline{\mb{\psi}}$ of a morphism
$\mb{f}$ of binary QFT algebras (recalling that ${\mb{\psi}}_1=\mb{f}$)
is a morphism of the associated binary CFT algebras. Considering those relationships
and the properties of the quantum descendant functor, it should be obvious that
the composition of  two consecutive morphisms of binary CFT algebras is
a morphism of binary CFT algebra so that binary CFT algebras form a category $\BCFTA(\Bbbk)$.
From the definitions of the homotopy category  $\mathit{ho}\category{BQFTA}(\Bbbk)$ of binary QFT algebras
and the homotopy functoriality of the quantum descendant functor in Sect.~\ref{subsec: homotopy cat of binary qft alg}, one can check
that the classical limits of the quantum descendants of homotopic morphisms of binary QFT algebras are
homotopic as morphisms of the associated binary CFT algebras.
This poses the interesting problem of whether one can setup a mathematical framework for a quantization
of a binary CFT algebra $\sC_\BC$ to a binary QFT algebra $\sC{\kkbar}_\BQ$, whose classical
limit in the above sense gives us back $\sC_\BC$ with appropriate functoriality. 
Such a framework may shed new light
on the passage from classical field theory to quantum field theory, but is beyond the scope of this
paper.
\naturalqed
\end{remark}
%\begin{lemma}[Definition]
%\label{cldesalg}
%%Let $\mb{K}= K +\kbar K^{(1)}+\ldots$
%%and $\bell_n =\ell_n +\kbar \ell^{(1)}_n +\ldots$, for all $n\geq 1$, where $\ell_1= K$.
%%Then,
%The classical limit $\big(\sC, 1_\sC, \,\cdot\,, K\big)$ of the binary QFT algebra 
%is a unital  CDGA over $\Bbbk$ and
%the classical limit $\big(\sC, 1_\sC, \underline{\ell}\big)$ of 
%the quantum descendant algebra, where $\ell_1=K$, is a
%unital $sL_\infty$-algebra over $\Bbbk$
%and, for all $n\geq 2$ and homogeneous $x_1,\ldots, x_n \in \sC$, we have
%\begin{align*}
%\ell_{n-1}\big(x_1,\ldots, x_{n-2}, x_{n-1}\cdot x_n\big)=&
%\ell_{n-1}\big(x_1,\ldots, x_{n-2}, x_{n-1}\big)\cdot x_n
%\\
%&
%+(-1)^{|x_{n-1}|(|x_1|+\ldots+|x_{n-2}|)} Jx_{n-1}\cdot
%\ell_{n-1}\big(x_1,\ldots, x_{n-2}, x_{n}\big)
%.
%\end{align*}
%We say the tuple $\big(\sC, 1_\sC, \,\cdot\,, \underline{\ell}\big)$
%a binary CFT algebra.
%\end{lemma}

\section{Binary QFTs}
\label{section: binary qft}

The binary QFT algebra $\Bbbk{\kkbar}$ is initial in the category $\category{BQFTA}(\Bbbk)$ and represents an initial object
in the homotopy category $ho\category{BQFTA}(\Bbbk)$.
A {\em binary QFT} is a slice over the initial object in the homotopy category $\mathit{ho}\category{BQFTA}(\Bbbk)$, i.e., a diagram of the form
\[
\xymatrix{
\sC{\kkbar}_\BQ\ar[r]^-{[\bm{\mc}]} & \Bbbk{\kkbar},
}
\] 
where  $\sC{\kkbar}_\BQ=\big(\sC{\kkbar}, 1_\sC, \,\cdot\,, \mb{K}\big)$ is a binary QFT algebra
and $[\bm{\mc}]$ is a homotopy type of binary QFT algebra morphisms to $\Bbbk{\kkbar}$.
In practice, we choose a representative, say $\bm{\mc}$, of $[\bm{\mc}]$ and
regard a binary QFT as a diagram 
$
\xymatrix{
\sC{\kkbar}_\BQ\ar[r]^-{\bm{\mc}} & \Bbbk{\kkbar}
}
$
in the category $\category{BQFTA}(\Bbbk)$ and consider only those notions and quantities that are
{\em invariants} of the homotopy type of $\bm{\mc}$.  We shall call  $\bm{\mc}$
a {\em strong} quantum expectation to contrast with a quantum expectation which is merely a pointed cochain map.

\subsection{Homotopical families of quantum observables}

Fix a binary QFT $\xymatrix{\sC{\kkbar}_\BQ\ar[r]^-{\bm{\mc}} & \Bbbk{\kkbar}}$
and let  $\big(\sC{\kkbar}, 1_\sC, \underline{\bell}\big)$ be the quantum descendant of 
$\sC{\kkbar}_\BQ$. Then, $\big(\sC{\kkbar}, 1_\sC, \underline{\bell}\big)$ 
is a unital $sL_\infty$-algebra.
Let $\underline{\mb{\chi}}=\bm{\mK}(\bm{\mc})$ be the quantum descendant of $\bm{\mc}$ so that for all $n\geq 1$ and homogeneous $\bm{x}_1,\ldots, \bm{x}_n \in \sC{\kkbar}$,
\begin{align*}
\bm{\mc}\big(\bm{x}_1\cdot\ldots\cdot\bm{x}_n\big)=
&
\sum_{\mathclap{\mp\in P(n)}}^{{}}
(-\kbar)^{n-|\mp|}
\ep(\mp)\mb{\chi}\!\big(\bm{x}_{B_1}\big)\ldots
\mb{\chi}\!\big(\bm{x}_{B_{|\mp|}}\big)
,
\end{align*}
has the property that  $\mb{\chi}_n \in \Hom\big(S^n \sC, \Bbbk\big)^{0}{\kkbar}$ for all $n\geq 1$.
Equivalently, we have 
\[
\bm{\mc}\circ \mb{\pi} = \mb{\pi}^{\Bbbk{\kkbar}}\circ \mb{\Psi}_{\mb{\chi}} 
\]
where $\mb{\pi}(\bm{x}_1\odot\ldots \odot \bm{x}_n)\coloneqq{} \bm{x}_1\cdot\ldots \cdot \bm{x}_n$
for all $n\geq 1$ and $\bm{x}_1, \ldots, \bm{x}_n \in \sC{\kkbar}$ and
$\mb{\pi}^{\Bbbk{\kkbar}}(\bm{a}_1\odot\ldots \odot \bm{a}_n)\coloneqq{} \bm{a}_1\cdot\ldots \cdot \bm{a}_n$
for all $n\geq 1$ and $\bm{a}_1, \ldots, \bm{a}_n \in \Bbbk{\kkbar}$.
Recall  that 
$\mb{K}\circ \mb{\pi} =\mb{\pi}\circ \mb{\d}_{\bell}$ and
$\xymatrix{\underline{\mb{\chi}}: \big(\sC{\kkbar},  1_\sC,\underline{\bell}\big)
 \ar@{..>}[r] & \big(\Bbbk{\kkbar}, 1, \underline{0}\big)}$ is a unital $sL_\infty$-morphism,
 i.e., $ \mb{\Psi}_{\mb{\chi}}\circ \mb{\d}_{\bell}=0$.

Now we introduce the notion of a homotopical family of quantum observables.

For any $\Z$-graded vector space $V$ over $\Bbbk$,
we can can regard $V{\kkbar}$ as a topologically-free
$sL_\infty$-algebra $\big(V{\kkbar}, \underline{0}\big)$ over $\Bbbk{\kkbar}$ with the zero $sL_\infty$-structure $\underline{0}$.
\begin{definition}
A homotopical family of quantum observables is a pair ${\sV}=\big(V; [\underline{\mb{\w}}]\big)$
where $V$ is a $\Z$-graded vector space over $\Bbbk$ and $[\underline{\mb{\w}}]$
is the homotopy type of an $sL_\infty$-morphism from $\big(V{\kkbar}, \underline{0}\big)$ to
$\big(\sC{\kkbar}, \underline{\bell}\big)$.
\end{definition}
In practice we work with a representative $\underline{\mb{\w}}$ of $[\underline{\mb{\w}}]$
and consider only those notions and quantities that are invariants of the homotopy type. 
The above definition is based on the following observation:

%%$
%and the associated quantum correlator is 
%$\mb{\Pi}^{\mb{\w}} \in \Hom\left(\Xbar{S}(V), \sC\right)^0{\kkbar}$
%such that, for all $n\geq 1$ and homogeneous $v_1,\ldots, v_n \in V$,
%\[
%\mb{\Pi}^{{\mb{\w}}}(v_1\odot\ldots\odot v_n)
%=\sum_{\mathclap{\mp\in P(n)}}^{{}}
%(-\kbar)^{n-|\mp|}
%\ep(\mp)\mb{\w}\!\big({v}_{B_1}\big)\cdot\dotsc\cdot
%\mb{\w}\!\big({v}_{B_{|\mp|}}\big).
%\]
%Then, we have
%\begin{enumerate}
%\item
%$\mb{\Pi}^{{\mb{\w}}}:\left(\Xbar{S}V{\kkbar}, 0\right)\rightarrow \big(\sC{\kkbar},\mb{K}\big)$
%is a cochain map, i.e., $\mb{K}\circ \mb{\Pi}^{{\mb{\w}}}=0$;
%\item
%$\mb{\Pi}^{{\tilde{\mb{\w}}}}-\mb{\Pi}^{{\mb{\w}}} =\mb{K}\mb{\S}$ 
%for some $\mb{\S} \in \Hom\big(\Xbar{S}(V), \sC\big)^{-1}{\kkbar}$ 
%whenever $\underline{\tilde{\mb{\w}}}\sim\underline{\mb{\w}}$ as $sL_\infty$-morphisms.
%\end{enumerate}
%
%
%\end{definition}

%
%in the homotopy category of topologically-free $sL_\infty$-algebras over $\Bbbk{\kkbar}$.
%\end{definition}
%In practice we work with a representative $\underline{\mb{\w}}$ 
%and
%consider only those notions and quantities that are invariants of the homotopy type $\left[\underline{\mb{\w}}\right]$

\begin{lemma}
Let  $\xymatrix{\underline{\mb{\w}}: \big(V{\kkbar},\underline{0}\big)\ar@{..>}[r]&
\big(\sC{\kkbar}, \underline{\bell}\big)}$ be a morphism of  $sL_\infty$-algebras and define
$\mb{\Pi}^{\mb{\w}} \in \Hom\left(\Xbar{S}(V), \sC\right)^0{\kkbar}$
so that, for all $n\geq 1$ and homogeneous $v_1,\ldots, v_n \in V$,
\[
\mb{\Pi}^{{\mb{\w}}}(v_1\odot\ldots\odot v_n)
=\sum_{\mathclap{\mp\in P(n)}}^{{}}
(-\kbar)^{n-|\mp|}
\ep(\mp)\mb{\w}\!\big({v}_{B_1}\big)\cdot\dotsc\cdot
\mb{\w}\!\big({v}_{B_{|\mp|}}\big).
\]
Then, we have
\begin{enumerate}
\item
$\mb{\Pi}^{{\mb{\w}}}:\left(\Xbar{S}V{\kkbar}, 0\right)\rightarrow \big(\sC{\kkbar},\mb{K}\big)$
is a cochain map, i.e., $\mb{K}\circ \mb{\Pi}^{{\mb{\w}}}=0$;
\item
$\mb{\Pi}^{{\tilde{\mb{\w}}}}-\mb{\Pi}^{{\mb{\w}}} =\mb{K}\mb{\S}$ 
for some $\mb{\S} \in \Hom\big(\Xbar{S}(V), \sC\big)^{-1}{\kkbar}$ 
whenever $\underline{\tilde{\mb{\w}}}\sim\underline{\mb{\w}}$ as $sL_\infty$-morphisms.
\end{enumerate}
\end{lemma}

\begin{proof}
For an $sL_\infty$-morphism $\xymatrix{\underline{\mb{\w}}: \big(V{\kkbar},\underline{0}\big)\ar@{..>}[r]&
\big(\sC{\kkbar}, \underline{\bell}\big)}$, 
define $\mb{\Psi}_{\mb{\w}}$ so that 
for all $n\geq 1$ and homogeneous $v_1,\ldots, v_n \in V$,
\begin{align*}
\mb{\Psi}_{\mb{\w}}(v_1\odot\ldots\odot v_n) 
\coloneqq{}
\sum_{\mathclap{\mp\in P(n)}}^{{}}
(-\kbar)^{n-|\mp|}
\ep(\mp)\mb{\w}\!\big({v}_{B_1}\big)\bm{\odot}\dotsc\bm{\odot}
\mb{\w}\!\big({v}_{B_{|\mp|}}\big)
.
\end{align*}
Then, we have the identities $\mb{\d}_{\bell}\circ \mb{\Psi}_{\mb{\w}}=0$
and
$\mb{\Pi}^{{\mb{\w}}} = \mb{\pi}\circ \mb{\Psi}_{\mb{\w}}$.
It follows that
$\mb{K}\circ\mb{\Pi}^{{\mb{\w}}} =\mb{K}\circ \mb{\pi}\circ \mb{\Psi}_{\mb{\w}}
= \mb{\pi}\circ \mb{\d}_{\bell}\circ \mb{\Psi}_{\mb{\w}}=0$.
Therefore, we have the first property that
$
\mb{\Pi}^{{\mb{\w}}}:\big(\Xbar{S}V{\kkbar}, 0\big)\rightarrow \big(\sC{\kkbar},\mb{K}\big)
$
is a cochain map whenever $\underline{\mb{\w}}$ is an $sL_\infty$-morphism.
For the second property, consider two homotopic $sL_\infty$-morphisms $\underline{\mb{\w}}$ and  $\underline{\tilde{\mb{\w}}}$.
Then there is a homotopy pair 
$\xymatrix{
\big(\underline{\mb{\w}}(\t), \underline{\mb{\l}}(\t)\big)
: \big(V{\kkbar},\underline{0}\big)\ar@{:>}[r]&
\big(\sC{\kkbar}, \underline{\bell}\big)}$ of $sL_\infty$-algebras such that 
$\underline{\mb{\w}}(0)=\underline{\mb{\w}}$ and $\underline{\mb{\w}}(1)=\underline{\tilde{\mb{\w}}}$.
Set $\breve{\mb{\w}}_n(\t) =(-\kbar)^{n-1}\mb{\w}_n(\t)$ and  $\breve{\mb{\l}}_n(\t) =(-\kbar)^{n-1}\mb{\l}_n(\t)$,
for all $n\geq 1$. Then $\xymatrix{
\big(\underline{\breve{\mb{\w}}}(\t), \underline{\breve{\mb{\l}}}(\t)\big)
: \big(V{\kkbar},\underline{0}\big)\ar@{:>}[r]&
\big(\sC{\kkbar}, \underline{\breve{\bell}}\big)}$ is a homotopy pair of $sL_\infty$-algebras such that 
$\underline{\breve{\mb{\w}}}(0)=\underline{\breve{\mb{\w}}}$ and $\underline{\breve{\mb{\w}}}(1)=\underline{\breve{\tilde{\mb{\w}}}}$. (Recall that 
$\mD\big(\breve{\bell}\big) = \mb{\d}_{\bell}$ and $\mF\big(\breve{\mb{\w}}\big) =\mb{\Psi}_{\mb{\w}}$.)
Therefore, we have
\[
\Fr{d}{d\t}\mF\big(\breve{\mb{\w}}(\t)\big) = \mD\big(\breve{\bell}\big)\circ \La\big(\breve{\mb{\w}}(\t),\breve{\mb{\l}}(\t)\big)
\]
which implies that
\[
\mF\big(\breve{\mb{\w}}(\t)\big) = \mF\big(\breve{\mb{\w}}\big) 
+ \mD\big(\breve{\bell}\big)\circ  \int^\t_0 \La\big(\breve{\mb{\w}}(\s),\breve{\mb{\l}}(\s)\big)d\!\s.
\]
It follows
that
\begin{align*}
\mb{\pi}\circ \mF\big(\breve{\mb{\w}}(\t)\big) 
=   
&
\mb{\pi}\circ\mF\big(\breve{\mb{\w}}\big)   
+ \mb{\pi}\circ\mD\big(\breve{\bell}\big)\circ  \int^\t_0 \La\big(\breve{\mb{\w}}(\s),\breve{\mb{\l}}(\s)\big)d\!\s
\\
= 
&
\mb{\Pi}^{{\mb{\w}}}  
+ \mb{K}\circ \mb{\pi}\circ \int^\t_0  \La\big(\breve{\mb{\w}}(\s),\breve{\mb{\l}}(\s)\big)d\!\s
,
\end{align*}
which implies that
$\mb{\Pi}^{\tilde{\mb{\w}}}= \mb{\Pi}^{{\mb{\w}}}  
+ \mb{K}\circ \mb{\pi}\circ\int^1_0  \La\big(\breve{\mb{\w}}(\t),\breve{\mb{\l}}(\t)\big)d\!\t$.
\naturalqed
\end{proof}

We sometimes refer to an $sL_\infty$-morphism 
$\xymatrix{\big(V{\kkbar},\underline{0} \big) \ar@{..>}[r]^-{\underline{\mb{\w}}} & \big(\sC{\kkbar}, \underline{\bell}\big)}$ 
as  a \emph{homotopical family of quantum observables} and call $\mb{\Pi}^{{\mb{\w}}} \in \Hom\left(\Xbar{S}(V), \sC\right)^{0}{\kkbar}$
the associated \emph{quantum correlator}.  The $n$th component
$\mb{\Pi}^{{\mb{\w}}}_n=\mb{\Pi}^{{\mb{\w}}} \circ\eb_{S^n V{\kkbar}}\in \Hom\left({S}^nV, \sC\right)^{0}{\kkbar}$ of $\mb{\Pi}^{{\mb{\w}}}$ is called the \emph{$n$-fold quantum correlator}. For example, we have
\begin{align*}
\mb{\Pi}_1^{{\mb{\w}}}(v_1) =& \mb{\w}_1(v_1)
,\\
\mb{\Pi}_2^{{\mb{\w}}}(v_1, v_2) =& \mb{\w}_1(v_1)\cdot \mb{\w}_1(v_2) +(-\kbar)
\mb{\w}_2(v_1,v_2)
,\\
\mb{\Pi}^{{\mb{\w}}}_3(v_1, v_2, v_3) =&
\mb{\w}_1(v_1)\cdot \mb{\w}_1(v_2)\cdot \mb{\w}_1(v_3) 
+(-\kbar)\mb{\w}_1(v_1)\cdot \mb{\w}_2(v_2,v_3)
\\
&
+(-\kbar)(-1)^{|v_1||v_2|}\mb{\w}_1(v_2)\cdot \mb{\w}_2(v_1,v_3)
+(-\kbar)\mb{\w}_2(v_1,v_2)\cdot \mb{\w}_1(v_3) 
\\
&
+(-\kbar)^2\mb{\w}_3(v_1,v_2,v_3)
.
\end{align*}

Consider a homotopical family of quantum observables ${\sV}=\left(V;[\underline{\mb{\w}}]\right)$
and let ${\underline{\mb{\w}}}$
be a representative of $\big[\underline{\mb{\w}}\big]$.
Then we have the following diagram:
\[
\xymatrix{
\left(\Xbar{S}(V){\kkbar}, 0\right)\ar@/^1.5pc/[rrrr]^-{ \bm{\mc}\circ\mb{\Pi}^{\mb{\w}}}
 \ar[rr]_-{ \mb{\Pi}^{\mb{\w}}} &&\big(\sC{\kkbar}, \mb{K}\big) \ar[rr]_-{ \bm{\mc}} && 
\big(\Bbbk{\kkbar},0\big),
\\
\left(V{\kkbar}, \underline{0}\right) \ar@{..>}@/_1.5pc/[rrrr]_-{\underline{\mb{\chi}}\bullet \underline{\mb{\w}}}
\ar@{..>}[rr]^-{\underline{\mb{\w}}} &&\big(\sC{\kkbar}, \underline{\bell}\big) \ar@{..>}[rr]^-{\underline{\mb{\chi}}} 
&& \big(\Bbbk{\kkbar},\underline{0}\big),
}
\]
where the first line is in the category of cochain complexes over $\Bbbk{\kkbar}$, while the second line
is in the category of  topologically-free $sL_\infty$-algebras over $\Bbbk{\kkbar}$.
\begin{itemize}
\item
Note that the composition $\bm{\mc}\circ\mb{\Pi}^{\mb{\w}}$ is an invariant of the homotopy type $[\underline{\mb{\w}}]$ of the
$sL_\infty$-morphism $\underline{\mb{\w}}$, since for any $\underline{\tilde{\mb{\w}}}\sim\underline{\mb{\w}}$
we have $\bm{\mc}\circ\mb{\Pi}^{\tilde{\mb{\w}}} - \bm{\mc}\circ\mb{\Pi}^{\mb{\w}}=\bm{\mc}\circ\mb{K}\circ \mb{\S}=0$,
so that this gives an intrinsic notion attached to $\mb{\sV}$.
Note also that $\bm{\mc}\circ\mb{\Pi}^{\mb{\w}}$ is an invariant of the homotopy type $[\bm{\mc}]$ of the
quantum expectation $\bm{\mc}$, since for any $\tilde{\bm{\mc}}\sim\bm{\mc}$
we have $\tilde{\bm{\mc}}\circ\mb{\Pi}^{{\mb{\w}}} - \bm{\mc}\circ\mb{\Pi}^{\mb{\w}}=\mb{\l}\circ\mb{K}\circ \mb{\Pi}^{\mb{\w}}=0$
so that it also gives an intrinsic notion attached to the binary QFT.
We call the invariant $\mb{\m}^\bm{\sV}\coloneqq{} \bm{\mc}\circ \mb{\Pi}^{\mb{\w}}\in \Hom\big(\Xbar{S}(V),\Bbbk\big)^{0}{\kkbar}$ 
{\em the quantum moment} of 
the homotopical family $\bm{\sV}$ of quantum observables.

\item
Note  that the composition $\underline{\mb{\chi}}\bullet \underline{\mb{\w}}$
is an $sL_\infty$-morphism from $\left(V{\kkbar}, \underline{0}\right)$ to $\big(\Bbbk{\kkbar},\underline{0}\big)$
and
is an invariant of the homotopy type $[\underline{\mb{\w}}]$ of  $\underline{\mb{\w}}$ so that it gives an intrinsic notion attached to $\mb{\sV}$.
Note also that $\underline{\mb{\chi}}\bullet \underline{\mb{\w}}$
is an invariant of the homotopy type $[\underline{\mb{\chi}}]$ of the quantum descendant 
$\underline{\mb{\chi}}=\bm{\mK}(\bm{\mc},\mb{\vr})$,
which depends only on the the homotopy type $[\bm{\mc}]$ of the
quantum expectation $\bm{\mc}$ so that it gives an intrinsic notion attached to the binary QFT.
We call the invariant $\underline{\mb{\chi}}^{\bm{\sV}}\coloneqq{} \underline{\mb{\chi}}\bullet \underline{\mb{\w}}$ {\em the quantum cumulant} of 
the homotopical family $\bm{\sV}$ of quantum observables.
\end{itemize}

The following lemma characterize the relationship between the quantum moment and quantum cumulant of 
a homotopical family  of quantum observables.

\begin{theorem}\label{morcuminhf}
Let $\mb{\m}^{{\sV}} =\bm{\mc}\circ \mb{\Pi}^{\mb{\w}}$ and $\underline{\mb{\chi}}^{{\sV}}=\underline{\mb{\chi}}\bullet \underline{\mb{\w}}$ 
be the quantum moment  and the quantum cumulant, respectively,
of a homotopical family $\bm{\sV}=\left(V; [\underline{\mb{\w}}]\right)$ of quantum observables. 
Then, for all $n\geq 1$ and homogeneous $v_1,\ldots, v_n \in V$,
we have
\[
\mb{\m}^{\bm{\sV}}(v_1\odot\ldots\odot v_n)
=\sum_{\mathclap{\mp\in P(n)}}^{{}}
(-\kbar)^{n-|\mp|}
\ep(\mp){\mb{\chi}}^{\bm{\sV}}\!\big({v}_{B_1}\big)\cdots
{\mb{\chi}}^{\bm{\sV}}\!\big({v}_{B_{|\mp|}}\big).
\]
\end{theorem}

\begin{proof}
From $\bm{\mc}\circ \mb{\pi} = \mb{\pi}^{\Bbbk{\kkbar}}\circ \mb{\Psi}_{\mb{\chi}}$
and $\mb{\Pi}^{{\mb{\w}}} = \mb{\pi}\circ \mb{\Psi}_{\mb{\w}}$
we obtain that
\begin{align*}
\mb{\m}^{\bm{\sV}}
=\bm{\mc}\circ \mb{\Pi}^{\mb{\w}}
=\bm{\mc}\circ \mb{\pi}\circ \mb{\Psi}_{\mb{\w}} 
=\mb{\pi}^{\Bbbk{\kkbar}}\circ \mb{\Psi}_{\mb{\chi}}\circ \mb{\Psi}_{\mb{\w}}
= \mb{\pi}^{\Bbbk{\kkbar}}\circ \mb{\Psi}_{\mb{\chi}^{\bm{\sV}}}
\end{align*}
where we have used 
$\mb{\Psi}_{\mb{\chi}}\circ \mb{\Psi}_{\mb{\w}}=\mb{\Psi}_{\mb{\chi}\bullet\mb{\w}}$ 
and $\underline{\mb{\chi}}^{\bm{\sV}}=\underline{\mb{\chi}}\bullet \underline{\mb{\w}}$  for the last equality.
Recall that 
\[
\mb{\Psi}_{\mb{\chi}^{\bm{\sV}}}(v_1\odot\ldots\odot v_n)
=\sum_{\mathclap{\mp\in P(n)}}^{{}}
(-\kbar)^{n-|\mp|}
\ep(\mp){\mb{\chi}}^{\bm{\sV}}\!\big({v}_{B_1}\big)\bm{\odot}\ldots \bm{\odot}
{\mb{\chi}}^{\bm{\sV}}\!\big({v}_{B_{|\mp|}}\big).
\]
Therefore  the identity $\mb{\m}^{\bm{\sV}} =\mb{\pi}^{\Bbbk{\kkbar}}\circ \mb{\Psi}_{\mb{\chi}^{\bm{\sV}}}$
is the desired formula in the theorem.
\naturalqed
\end{proof}

\begin{remark}
Fix a homotopical family $\bm{\sV}=\left(V, [\underline{\mb{\w}}]\right)$ of quantum observables
and assume that $V$ is a finite dimensional $\Z$-graded vector space. Choose
a homogeneous basis $e_V=\{e_a\}$ and let $t_V=\{t^a\}$ be the dual basis.
Consider the $\Z$-graded super-commutative
$\Bbbk$-algebra $\Bbbk[\![t_V]\!]$, where $t^{a_1}t^{a_2} =(-1)^{\gh(e_{a_1})\gh(e_{a_2})}t^{a_2}t^{a_1}$.
 Then, it is straightforward to check the following:

\begin{enumerate}
\item Let $\underline{\mb{\w}}$ be a representative of  $[\underline{\mb{\w}}]$ and let
\[
\mb{\Theta}^{\mb{\w}}\coloneqq{} \sum_{\mathclap{n=1}}^\infty \Fr{1}{n!} t^{a_n}\cdots t^{a_1} \mb{\w}_n\big(e_{a_1},\ldots, e_{a_n}\big) 
\in \big(\Bbbk[\![t_V]\!]\hat\otimes\sC\big)^{0}{\kkbar}.
\]
Then, we have $\mb{K} e^{-\fr{1}{\kbar}\mb{\Theta}^{\mb{\w}}}=0$
%\[
%\mb{K} e^{-\fr{1}{\kbar}\mb{\G}(t_1,\ldots, t_k)}=0  \quad\Longleftrightarrow \quad
%\mb{K}\mb{\G} + \sum_{\mathclap{n=2}}^\infty \Fr{1}{n!}\bell_n\big(\mb{\G},\ldots,\mb{\G}\big)=0,
%\]
and
\[
e^{-\fr{1}{\kbar}\mb{\Theta}^{\mb{\w}} } = 1_\sC 
+  \sum_{\mathclap{n=1}}^\infty \Fr{1}{(-\kbar)^n}\Fr{1}{n!} t^{a_n}\cdots t^{a_1}\mb{\Pi}^{\mb{\w}}_n\big(e_{a_1},\ldots, e_{a_n}\big).
\]

\item
Let $\mb{\m}^\bm{\sV}\coloneqq{} \bm{\mc}\circ \mb{\Pi}^{\mb{\w}}$ 
and $\underline{\mb{\chi}}^{\bm{\sV}}=\underline{\mb{\chi}}\bullet \underline{\mb{\w}}$, 
and define the following generating series 
\begin{align*}
\bm{\CZ}_{\!\bm{\sV}}
&
\coloneqq{} 1 + \sum_{\mathclap{n=1}}^\infty\Fr{1}{(-\kbar)^n} \Fr{1}{n!} t^{a_n}\cdots t^{a_1} \mb{\m}^{\bm{\sV}}_n\big(e_{a_1},\ldots, e_{a_n}\big)
\in \Bbbk[\![t_V]\!](\!(\kbar)\!)^0
,\\
\bm{\CF}_{\!\bm{\sV}}
&
\coloneqq{} \sum_{\mathclap{n=1}}^\infty \Fr{1}{n!} t^{a_n}\cdots t^{a_1} \mb{\chi}^{\bm{\sV}}_n\big(e_{a_1},\ldots, e_{a_n}\big)
\in \Bbbk[\![t_V]\!]^{0}{\kkbar}
.
\end{align*}
Then we have the identity $\bm{\CZ}_{\!\bm{\sV}} = e^{-\fr{1}{\kbar}\mb{\sF}_{\!\bm{\sV}}}$.
Note also that 
$\bm{\CZ}_{\!\bm{\sV}}\equiv \big<e^{-\fr{1}{\kbar}\mb{\Theta}^{\mb{\w}} }\big>_{\bm{\mc}}$.\naturalqed
\end{enumerate}
\end{remark}

\subsection{When are two binary QFTs physically equivalent?}

Let
$
\xymatrix{
\sC{\kkbar}_\BQ\ar[r]^-{\bm{\mc}} & \Bbbk{\kkbar}
}$
and
$
\xymatrix{
\sC^\pr{\kkbar}_\BQ\ar[r]^-{\bm{\mc}^\pr} & \Bbbk{\kkbar},
}
$ 
be binary QFTs 
and assume that 
there is a homotopy equivalence 
$\xymatrix{\sC{\kkbar}_\BQ\ar@/^/[r]^{\bm{\mf}} &\ar@/^/[l]^{\bm{\mf}^\pr}\sC^\pr{\kkbar}_\BQ}$
of binary QFT algebras, i.e., both $\bm{\mf}$ and  $\bm{\mf}^\pr$
are morphisms of binary QFT algebras such that  $\bm{\mf}^\pr\circ \bm{\mf}$
is homotopic to the identity map on $\sC{\kkbar}$ and 
$\bm{\mf}\circ \bm{\mf}^\pr$ is homotopic to the identity map on $\sC^\pr{\kkbar}$.
Note that $\bm{\mf}$ and  $\bm{\mf}^\pr$ are quasi-isomorphisms.
Assume also that
the following diagram in the category
of binary QFT algebras is commutative up to homotopy:
\[
\xymatrix{
\sC{\kkbar}_\BQ\ar[rr]^-{\bm{\mc}} \ar@/^/[d]^{\bm{\mf}} &&\Bbbk{\kkbar}
\\
\ar@/^/[u]^{\bm{\mf^\pr}} 
\sC^\pr{\kkbar}_\BQ\ar[rru]_{\bm{\mc}^\pr} &&
}.
\]
Then, we may say that the two binary QFTs are physically equivalent.
In the remaining part of this subsection, we check that such a physical equivalence
induces a isomorphism of the homotopical families of quantum observables that
preserves the quantum moment of every homotopical family of quantum observables.

Let $\sC{\kkbar}_\BQ=\big(\sC{\kkbar}, 1_\sC, \,\cdot\,, \mb{K}\big)$ and
$\sC^\pr{\kkbar}_\BQ=\big(\sC^\pr{\kkbar}, 1_{\sC^\pr}, \,\cdot^\pr\,, \mb{K}^\pr\big)$,
and let $\big(\sC{\kkbar}, 1_\sC, \underline{\bell}\big)$ and
$\big(\sC^\pr{\kkbar}, 1_{\sC^\pr}, \underline{\bell}^\pr\big)$ be the corresponding
quantum descendant algebras.
Let $\underline{\mb{\psi}}=\bm{\mK}(\bm{\mf})$ 
and $\underline{\mb{\psi}}^\pr=\bm{\mK}(\bm{\mf}^\pr)$
be the quantum descendants of the two homotopy inverses $\bm{\mf}$ and $\bm{\mf}^\pr$, respectively.
Then, we have the following homotopy equivalence in the category $\category{UsL}_\infty(\Bbbk{\kkbar})$:
\[
\xymatrix{
\big(\sC{\kkbar}, 1_\sC, \underline{\bell}\big)
\ar@{..>}@/^/[r]^-{\underline{\mb{\psi}}}
&
\ar@{..>}@/^/[l]^-{\underline{\mb{\psi}}^\pr}
\big(\sC^\pr{\kkbar}, 1_{\sC^\pr}, \underline{\bell}^\pr\big).
}
\]
Recall that
\eqn{hoea}{
\bm{\mf}\circ \mb{\pi} = \mb{\pi}^\pr \circ \mb{\Psi}_{\mb{\psi}} 
,\qquad
\bm{\mf}^\pr\circ \mb{\pi}^\pr =\mb{\pi} \circ \mb{\Psi}_{\mb{\psi}^\pr}
,
}
where
\begin{align*}
\mb{\pi}\big(\bm{x}_1\odot\ldots\odot\bm{x}_n\big) &\coloneqq \bm{x}_1\cdot\ldots\cdot \bm{x}_n
&&\text{for all } n\geq 1 \text{ and }\bm{x}_1,\ldots, \bm{x}_n \in \sC{\kkbar},
\\
\mb{\pi}^\pr\big(\bm{x}^\pr_1\odot\ldots\odot\bm{x}^\pr_n\big)&\coloneqq\bm{x}^\pr_1\cdot^\pr\ldots\cdot^\pr \bm{x}^\pr_n
&&\text{for all } n\geq 1 \text{ and }
\bm{x}^\pr_1,\ldots, \bm{x}^\pr_n \in \sC^\pr{\kkbar}.
\end{align*}

Let $\bm{\sV}=\left(V, \big[\underline{\mb{\w}}\big]\right)$
be a homotopical family of quantum observables in $\sC{\kkbar}_\BQ$
and let $\underline{\mb{\w}}$ be a representative of $\big[\underline{\mb{\w}}\big]$.
Then, 
$\xymatrix{\underline{\mb{\w}}^\pr\coloneqq{}
\underline{\mb{\psi}}\bullet\underline{\mb{\w}}:\big(V{\kkbar},\underline{0}\big)\ar@{..>}[r]
& \big(\sC^\pr{\kkbar},\underline{\bell}^\pr\big)}$ is an  $sL_\infty$-morphism 
whose homotopy type depends on $\underline{\mb{\w}}$ only via its homotopy type $\big[\underline{\mb{\w}}\big]$. 
Therefore $\bm{\sV}^\pr=\left(V, \big[\underline{\mb{\w}}^\pr\big]\right)$
is a  a homotopical family of quantum observables in $\sC^\pr{\kkbar}_\BQ$. 
Note also that $\underline{\mb{\psi}}^\pr\bullet \underline{\mb{\w}}^\pr
=\underline{\mb{\psi}}^\pr\bullet\underline{\mb{\psi}}\bullet\underline{\mb{\w}}
\sim \underline{\mb{\w}}$ if $\underline{\mb{\w}}^\pr=\underline{\mb{\psi}}\bullet\underline{\mb{\w}}$
so that we have
$\big[\underline{\mb{\psi}}^\pr\bullet \underline{\mb{\w}}^\pr\big]=\big[\underline{\mb{\w}}\big]$.
Obviously, the converse is also true. 
Therefore, the homotopy equivalence 
\[\xymatrix{\sC{\kkbar}_\BQ\ar@/^/[r]^{\bm{\mf}} &\ar@/^/[l]^{{\mf}^\pr}\sC^\pr{\kkbar}_\BQ}\]
induces an isomorphism of homotopical families of quantum observables.

Consider  homotopical families of quantum observables $\bm{\sV}$ in $\sC{\kkbar}_\BQ$ and
$\bm{\sV}^\pr$ in $\sC^\pr{\kkbar}_\BQ$ as defined above.
Recall that 
\[
\mb{\m}^{\mb{\sV}}\coloneqq{}\bm{\mc}\circ \mb{\Pi}^{\mb{\w}}
=\bm{\mc}\circ \mb{\pi}\circ \mF\big(\breve{\mb{\w}}\big)
,\qquad
\mb{\m}^{\pr\mb{\sV}^\pr}\coloneqq{}\bm{\mc}^\pr\circ \mb{\Pi}^{\mb{\w}^\pr}
=\bm{\mc^\pr}\circ \mb{\pi^\pr}\circ \mF\big(\breve{\mb{\w}}^\pr\big).
\]
From 
$\bm{\mc}\sim_\kbar \bm{\mc}^\pr\circ \bm{\mf}$
 and 
$\bm{\mf}\circ \mb{\pi} 
= \mb{\pi}^\pr \circ \mb{\Psi}_{\mb{\psi}} 
$
we have 
\begin{align*}
\mb{\m}^{\mb{\sV}}
= \bm{\mc}^\pr\circ \bm{\mf}\circ \mb{\pi}\circ \mb{\Psi}_{\mb{\w}}
= \bm{\mc}^\pr\circ  \mb{\pi}^\pr \circ \mb{\Psi}_{\mb{\psi}} \circ \mb{\Psi}_{\mb{\w}}
= \bm{\mc}^\pr\circ  \mb{\pi}^\pr \circ \mb{\Psi}_{\mb{\psi}\bullet\mb{\w}}
=
\bm{\mc}^\pr\circ \mb{\Pi}^{\mb{\w}^\pr}
=
\mb{\m}^{\pr\mb{\sV}^\pr},
\end{align*}
where we have used
$\mb{\Psi}_{\mb{\psi}} \circ \mb{\Psi}_{\mb{\w}}= \mb{\Psi}_{\mb{\psi}\bullet\mb{\w}}$
for the third equality, and
the fact that $ \mb{\Pi}^{\mb{\psi}\bullet\mb{\w}}=\mb{\pi}^\pr \circ \mb{\Psi}_{\mb{\psi}\bullet\mb{\w}}$
is cochain homotopic to  $\mb{\Pi}^{\mb{\w}^\pr}$ since $\mb{\w}^\pr \sim \mb{\psi}\bullet\mb{\w}$ as $sL_\infty$-morphisms
for the fourth equality.
Therefore, the induced  isomorphism of homotopical families of quantum observables
also preserves the quantum moment --- as well as the quantum cumulant by Theorem \ref{morcuminhf} --- of every homotopical family of quantum observables.

\section{Quantizations of classical observables}
\label{section: quantization}

Fix the structure $\sC{\kkbar}_\BQ=\big(\sC{\kkbar}, 1_\sC, \;\cdot\;,\mb{K}\big)$ 
of a binary QFT algebra on $\sC$ and consider the underlying  
QFT complex $\big(\sC{\kkbar}, 1_\sC, \mb{K}\big)$, which shall be called the {\em off-shell} QFT complex of 
$\sC{\kkbar}_\BQ$.  The classical limit $\big(\sC, 1_\sC, K\big)$ of the off-shell QFT complex is 
a pointed cochain complex over $\Bbbk$, whose cohomology $H$ is regarded as 
the space of equivalence classes of off-shell classical observables.  There is a distinguished element $1_H \in H$,
which is the $K$-cohomology class of $1_\sC$ and $\big(H, 1_H, 0\big)$ is a pointed cochain complex
over $\Bbbk$ with zero differential.
The main purpose of this subsection is to construct the structure of an on-shell QFT complex on 
$H$ that is homotopy equivalent to the off-shell QFT complex.
This will be the obstruction theory of  quantization of classical observables.

\begin{theorem}[Definition]
\label{anomaly}
There is an \emph{on-shell QFT complex} structure $\big(H{\kkbar}, 1_H, \mb{\kappa}\big)$ on the classical cohomology $H$ 
of a binary QFT algebra of $\sC{\kkbar}_\BQ$ which is homotopy equivalent to the off-shell QFT complex $\big(\sC{\kkbar}, 1_\sC, \mb{K}\big)$
of $\sC{\kkbar}_\BQ$.
We call $\sC{\kkbar}_\BQ$ \emph{anomaly-free} if $\mb{\kappa}=0$.
\end{theorem}

\begin{remark}
A QFT complex $\big(\sC{\kkbar},1_\sC, \mb{K}\big)$
is by definition a pointed cochain complex over $\Bbbk{\kkbar}$. 
However, we will rarely work with its cohomology
$\big(\mb{H}, 1_{\mb{H}}, 0\big)$, where $1_{\mb{H}}$ is the $\mb{K}$-cohomology class of $1_\sC$,
but almost exclusively with  the on-shell QFT complex $\big(H{\kkbar},1_H, \mb{\kappa}\big)$ 
on the classical cohomology $H$, which is quasi-isomorphic to $\big(\mb{H}, 1_{\mb{H}}, 0\big)$.
There are at least two reasons for this choice:
\begin{enumerate} 
\item We do not expect, in general, that the $\mb{K}$-cohomology group $\mb{H}$ is a topologically
free $\Bbbk{\kkbar}$-module, while both $H{\kkbar}$ and $\sC{\kkbar}$ will always be topologically free.
Here is a simple demonstration:
Let $v \in H\subset H{\kkbar}$ satisfy $\mb{\kappa}v=0$. Then $v$ can not be $\mb{\kappa}$-exact, since $\mb{\kappa}
=\kbar \k^{(1)}+\cdots$,
so that it gives a cohomology class in $\mb{H}$. 
Consider $\kbar v \in H{\kkbar}$, which always belongs to $\Ker \mb{\kappa}$. Now $\kbar v$ can be
$\mb{\kappa}$-exact. For example, there may be some $\eta\in H$ such that $\k^{(n)}\eta=0$ for all $n\geq 2$, while
$\k^{(1)}\eta = v$ so that $\kbar v = \mb{\kappa} \eta= \kbar \k^{(1)}\eta$.
We remark that $\mb{H}$ is a free $\Bbbk{\kkbar}$-module if $\mb{\kappa}=0$ 
since the cohomology of $(H{\kkbar}, \mb{\kappa}=0)$ is $H{\kkbar}$, which is isomorphic to $\mb{H}$
as topologically free $\Bbbk{\kkbar}$-modules. 
In general we may regard the on-shell QFT complex $\big(H{\kkbar}, 1_H,\mb{\kappa}\big)$ 
as a topologically-free resolution of $\big(\mb{H}, 1_{\mb{H}}, 0\big)$.

\medskip
\item
The arena of classical physics is the classical cohomology $H$ (the on-shell QFT complex) and we think of ourselves as mere mortals 
contemplating the quantum world
from this classical vantage point.
\naturalqed
\end{enumerate}

\end{remark}

The classical complex $(\sC, 1_\sC, K)$  is a pointed cochain complex over a field $\Bbbk$
where every quasi-isomorphism splits. 
Then the classical complex is homotopy equivalent to $\big(H, 1_H, 0\big)$.
We choose all the data of such a homotopy equivalence 
\eqn{cdfrd}{
\xymatrix{
\big(H,1_H,0\big)\ar@/^/[r]^f &\ar@/^/[l]^{h}\big(\sC, 1_\sC,K\big)\ar@(ul,ur)^{s}
},\qquad
\begin{cases}
h\circ f =\I_H
,\\
f\circ h = \I_\sC - K\circ{s} - {s} \circ K
,\\
\end{cases}
}
where both $f\in\Hom\big(H,\sC\big)^0$ and $h\in\Hom\big(\sC,H\big)^0$ are quasi-isomorphisms of pointed
cochain complexes satisfying $f(1_H) = 1_\sC$, $K\circ f=0$, $h(1_\sC)=1_H$ and $h\circ K =0$, and $s$ is an element in $\Hom\big(\sC, \sC\big)^{-1}$ satisfying  $K = K\circ s\circ K$ and
${s} (1_\sC)=0$.

We can and will choose a splitting $s$ such that it satisfies the following side condition:
\eqn{sidecon}{
s\circ s =s\circ f =h\circ s=0.
}
We call such a trio $(f, h,{s})$ a {\em classical off-to-on-shell retract} of the classical complex $\big(\sC, 1_\sC, K\big)$
--- the official name for such a trio is strong deformation retract with the side conditions in the homological perturbation theory.

\subsection{A quantization}

Associated to each classical off-to-on-shell retract $(f, h,{s} )$, 
we can construct an on-shell QFT complex structure  $\big(H{\kkbar}, 1_H, \mb{\kappa}\big)$
on $H$ and a {\em quantized off-to-on-shell retract} $(\mb{f},\mb{h}, \mb{s})$, whose classical limit is  
limit  $(f,h,s)$. That is, this data should satisfy
\eqn{qdfrd}{
\xymatrix{
\big(H{\kkbar}, 1_H,\mb{\kappa}\big)\ar@/^/[r]^{\mb{f}} 
&\ar@/^/[l]^{\mb{h}}\big(\sC{\kkbar},1_\sC, \mb{K}\big)\ar@(ul,ur)^{\mb{s}}},
\quad
\begin{cases}
\mb{h}\circ \mb{f} =\I_{H{\kkbar}}
,\\
\mb{f}\circ \mb{h} = \I_{\sC{\kkbar}} - \mb{K}\circ\mb{s} - \mb{s} \circ \mb{K}
,
\end{cases}
}
where 
$\mb{f} \in \Hom(H,\sC)^{0}{\kkbar}$ and $\mb{h}\in \Hom(\sC, H)^{0}{\kkbar}$ 
are (quasi-iso)morphisms of QFT complexes, i.e.,
\[
\mb{f}(1_H)= 1_\sC
,\quad
\mb{K}\circ \mb{f} =\mb{f}\circ \mb{\kappa}
,\qquad
\mb{h}(1_\sC)= 1_H
,\quad
\mb{h}\circ \mb{K} =\mb{\kappa}\circ \mb{h},
\]
and $\mb{s}\in \Hom(\sC,\sC)^{-1}{\kkbar}$ satisfies $\mb{s}(1_\sC)=0$ as well
as the following side conditions:
\eqn{qsidecon}{
\mb{s}\circ \mb{s}=\mb{s}\circ\mb{f}=\mb{h}\circ \mb{s}=0,\qquad {s}\circ \mb{s}={s}\circ\mb{f}=\mb{h}\circ {s}=0,
\qquad \mb{s}\circ {s}=\mb{s}\circ {f}={h}\circ \mb{s}=0.
}

\begin{proposition}\label{anomalyx}
Fix a classical off-to-on-shell retract $(f,h,{s})$
and define 
\[
\begin{cases}
\mb{\kappa} =&\qquad\kbar \k^{(1)} +\kbar^2 \k^{(2)} +\ldots
,\\
\mb{f} =&f+\kbar f^{(1)} +\kbar^2 f^{(2)} +\ldots
,\\
\end{cases}
\qquad
\begin{cases}
\mb{h} =&h+\kbar h^{(1)} +\kbar^2 h^{(2)} +\ldots
,\\
\mb{s} =&s+\kbar s^{(1)} +\kbar^2 s^{(2)} +\ldots
,
\end{cases}
\]
where, for all $n\geq 1$,
\begin{align*}
\k^{(n)} &= h \circ g^{(n)},&
f^{(n)} &= -{s}  \circ g^{(n)},
\end{align*}
where
\eqn{anxx}{
{g}^{(n)}\coloneqq{} \sum_{\mathclap{j=0}}^{n-1} K^{(n-j)}\circ  f^{(j)}  - \sum_{\mathclap{j=1}}^{n-1} f^{(n-j)}\circ \k^{(j)}
,}
\eqn{anyx}{
h^{(n)}= -u^{(n)}\circ s, 
\quad\text{ where }\quad
u^{(n)} \coloneqq{}\sum_{\mathclap{j=0}}^{n-1}h^{(j)}\circ K^{(n-j)} - \sum_{\mathclap{j=1}}^{n-1}\k^{(j)}\circ h^{(n-j)},
}
and

\eqn{anyxj}{
s^{(n)} =-\sum_{\mathclap{j=0}}^{n-1} s\circ K^{(n-j)}\circ s^{(j)}\equiv -\sum_{\mathclap{j=0}}^{n-1} s^{(j)}\circ K^{(n-j)}\circ s
.
}

Then $\big(H{\kkbar}, 1_H, \mb{\kappa}\big)$ is an on-shell QFT complex and $(\mb{f}, \mb{h},\mb{s})$
is a quantization of $(f,h,s)$.
\end{proposition}

\begin{proof}

\begin{enumerate}
\item From the side condition \eq{sidecon} and  by definition, the trio $(\mb{f},\mb{h},\mb{s})$ satisfies \eq{qsidecon}.
\item We check that $\mb{\kappa}^2 =\mb{\kappa}1_H=0$ and
$\mb{f}(1_H) = 1_\sC$ and $\mb{K}\circ\mb{f}=\mb{f}\circ \mb{\kappa}$
as follows:
It is obvious that $\mb{\kappa}^2 =0 \bmod \kbar$
and $\mb{K}\circ \mb{f}  = \mb{f}\circ\mb{\kappa} \bmod \kbar $,
since $\k^{(0)}=K\circ f =0$.
Fix $n\geq 2$ and assume that
$\mb{\kappa}^2 =0 \bmod \kbar^n$
and $\mb{K}\circ \mb{f}  = \mb{f}\circ\mb{\kappa} \bmod \kbar^n $.
Then, it can be checked that 
\eqn{claima}{
K\circ g^{(n)} =-\sum_{\mathclap{\ell=1}}^{n-1} f\circ \k^{(n-\ell)}\circ\k^{(\ell)}.
}
Applying $h\circ $ from the left to the above identity, we obtain that
\eqn{anxc}{
\sum_{\mathclap{\ell=1}}^{n-1} \k^{(n-\ell)}\circ\k^{(\ell)} =0,
}
and  $K \circ g^{(n)}=0$.
Let
$\k^{(n)} \coloneqq{} h \circ g^{(n)}$ and
$f^{(n)}\coloneqq{} -{s} \circ g^{(n)}$.
Then we obtain that
\begin{align*}
f\circ \k^{(n)} =& f\circ h \circ g^{(n)} 
\\
=& g^{(n)} - K\circ s\circ  g^{(n)} -s\circ K\circ  g^{(n)}
\\
=& g^{(n)} + K\circ f^{(n)}.
\end{align*}
This relation is equivalent to the following:
\eqn{anxf}{
\sum_{\mathclap{\ell=0}}^{n}  K^{(n-\ell)}\circ f^{(\ell)}  =\sum_{\mathclap{\ell=1}}^{n}  f^{(n-\ell)}\circ \k^{(\ell)}.
}
Combining our assumption with  \eq{anxc} and \eq{anxf}, 
we have $\mb{\kappa}^2 =0 \bmod \kbar^{n+1}$ and 
$\mb{K}\circ \mb{f}=\mb{f}\circ\mb{\kappa} \bmod  \kbar^{n+1}$.
By induction we conclude that $\mb{\kappa}^2 =0$ and $\mb{K}\circ \mb{f}=\mb{f}\circ\mb{\kappa}$.
It is straightforward to check that $\mb{\kappa}1_H=0$ and $\mb{f}(1_H) = 1_\sC$.
\item  
We check that $\mb{h}\circ \mb{K} =\mb{\kappa}\circ\mb{h}$ and $\mb{h}(1_H)= 1_\sC$ as follows:
Note that $\mb{h}\circ \mb{K}  = \mb{\kappa}\circ\mb{h} \bmod \kbar $,
since $h\circ K =\k^{(0)}=0$.
Fix $n\geq 2$ and assume that
$\mb{h}\circ \mb{K} =\mb{\kappa}\circ\mb{h} \bmod \kbar^n $.
Then, it can be checked that 
\eqn{claimb}{
u^{(n)}\circ K =0
,\qquad
\k^{(n)}\circ h=u^{(n)}\circ f\circ h
.
}
Let $h^{(n)}\coloneqq{} -u^{(n)}\circ s$. 
Then, we have
\begin{align*}
u^{(n)} 
&
= u^{(n)}\circ f\circ h + u^{(n)}\circ K\circ{s} +  u^{(n)}\circ {s} \circ K
\\&= u^{(n)}\circ f\circ h  +  u^{(n)}\circ {s} \circ K
\\
&
= \k^{(n)}\circ h - h^{(n)} \circ K
,
\end{align*}
which is equivalent to the following relation:
\eqn{anxg}{
\sum_{\mathclap{j=0}}^{n}h^{(j)}\circ K^{(n-j)} = \sum_{\mathclap{j=1}}^{n}\k^{(j)}\circ h^{(n-j)}.
}
Combining our assumption with  \eq{anxg}, we have
$\mb{h}\circ \mb{K} =\mb{\kappa}\circ\mb{h} \bmod \kbar^{n+1}$.
By induction we conclude that $\mb{h}\circ \mb{K} =\mb{\kappa}\circ\mb{h}$.
It is straightforward to check that $\mb{h}(1_\sC)= 1_H$.
\item 
From 
\[
\mb{f} = f -  \sum_{\mathclap{n=1}}^\infty  \kbar^n s\circ g^{(n)},
\qquad
\mb{h} = h - \sum_{\mathclap{n=1}}^\infty \kbar^n u^{(n)}\circ s
,
\]
and $s\circ s =h\circ s=s\circ f=0$, it is obvious that $\mb{h}\circ \mb{f} = h\circ f =\I_{H{\kkbar}}$.
\item
It remains to check that 
$\mb{f}\circ \mb{h} = \I_{\sC{\kkbar}} - \mb{K}\circ\mb{s} - \mb{s} \circ \mb{K}$.
Let $\mb{\xi}= \mb{f}\circ \mb{h} + \mb{K}\circ\mb{s} + \mb{s} \circ \mb{K} \in \Hom\big(\sC,\sC\big){\kkbar}$.
Consider the expansion 
$\mb{\xi} = {\xi}^{(0)} +\kbar {\xi}^{(1)} +\kbar^2 {\xi}^{(2)}+\ldots$,
where ${\xi}^{(n)}=(\mb{f}\circ \mb{h})^{(n)} + (\mb{K}\circ\mb{s})^{(n)}  + (\mb{s} \circ \mb{K})^{(n)}  
\in \Hom\big(\sC,\sC\big)$.
Note that $\mb{\xi}^{(0)}=\I$ since ${f}\circ {h} = \I_\sC - {K}\circ{s} - {s} \circ {K}$. We need to show that
${\xi}^{(n)}=0$ for all $n\geq 1$.   
A direct computation using the definitions of $(\mb{f},\mb{h},\mb{s})$
show that, for all $n\geq 1$,
\begin{align*}
(\mb{f}\circ \mb{h})^{(n)} = 
&
\sum_{\mathclap{k=0}}^{n-1}f^{(k)}\circ h^{(n-k)} + f^{(n)}\circ h
= -\sum_{\mathclap{k=0}}^{n-1} (\mb{f}\circ \mb{h})^{(k)}\circ K^{(n-k)}\circ s
+ f^{(n)}\circ h
\\
=
&
-\sum_{\mathclap{k=0}}^{n-1} (\mb{f}\circ \mb{h})^{(k)}\circ K^{(n-k)}\circ s
+\sum_{\mathclap{k=0}}^{n-1} s^{(k)}\circ K^{(n-k)}\circ K^{(0)}\circ s
- \sum_{\mathclap{k=0}}^n s^{(k)}\circ K^{(n-k)}.
\end{align*}
For $n=1$, we have $\mb{\xi}^{(1)}=0$:
\begin{align*}
(\mb{f}\circ \mb{h})^{(1)}=
&-f\circ h\circ K^{(1)}\circ s 
+s\circ K^{(1)}\circ K\circ s
- s\circ K^{(1)} - s^{(1)} \circ K
\\
=
&
- K^{(1)}\circ s
 +K\circ s\circ K^{(1)}\circ s 
 +s\circ\big( K\circ    K^{(1)}+ K^{(1)}\circ K\big)\circ s -(\mb{s}\circ \mb{K})^{(1)}
\\
=
& -(\mb{K}\circ \mb{s})^{(1)}-(\mb{s}\circ \mb{K})^{(1)}
.
\end{align*}
Fix $n\geq 2$ and assume that ${\xi}^{(1)}=\ldots = {\xi}^{(n-1)}=0$.
Then, 
we have
\begin{align*}
(\mb{f}\circ \mb{h})^{(n)} 
=
&
- K^{(n)}\circ s
+\sum_{\mathclap{k=0}}^{n-1} (\mb{K}\circ \mb{s})^{(k)}\circ K^{(n-k)}\circ s
\\
&
+\sum_{\mathclap{k=0}}^{n-1} (\mb{s}\circ \mb{K})^{(k)}\circ K^{(n-k)}\circ s
+ \sum_{\mathclap{k=0}}^{n-1} s^{(k)}\circ K^{(n-k)}\circ K^{(0)}\circ s
\\
&
- \sum_{\mathclap{k=0}}^n s^{(k)}\circ K^{(n-k)} 
\\
=
&-(\mb{K}\circ \mb{s})^{(n)}-(\mb{s}\circ \mb{K})^{(n)} 
+\sum_{\mathclap{k=1}}^n s^{(n-k)}\circ (\mb{K}\circ \mb{K})^{(k)}\circ s
\\
=
&-(\mb{K}\circ \mb{s})^{(n)}-(\mb{s}\circ \mb{K})^{(n)}
.
\end{align*}
Therefore ${\xi}^{(n)}=0$, so that ${\xi}^{(n)}=0$ for all $n\geq 1$ by induction.
\naturalqed
\end{enumerate}
\end{proof}

\subsection{Variations of  quantization}

Recall that the differential $\mb{\kappa}$ and the quantized
off-to-on-shell retract $(\mb{f},\mb{h},\mb{s})$
 in Proposition \ref{anomalyx} 
were obtained in terms of a fixed classical off-to-on-shell retract  $(f,h,{s} )$.  The following proposition show 
that these quantizations depend only up to an automorphism and homotopy on the choice of splitting data of the classical complex.
We will only consider the variations of $\mb{\kappa}$ and $\mb{f}$.
\begin{proposition}\label{anomalyz}
Let $(\mb{f},\mb{\kappa})$ and $(\mb{f}^\pr,\mb{\kappa}^\pr)$ be the duos 
associated with off-to-on-shell retracts $(f, h,{s} )$ and $(f^\pr, h^\pr,s^\pr)$,
respectively.  Then, there is a duo $(\mb{\xi}, \mb{\l})$ such that
\[
\mb{\kappa}^\pr =  \mb{\xi}^{-1}\circ\mb{\kappa}  \circ\mb{\xi}
,\qquad
\mb{f}^\pr  = \mb{f} \circ\mb{\xi}+ \mb{K}\circ\mb{\l} + \mb{\l}\circ\mb{\kappa}^\pr,
\]
where $\mb{\xi}= \I_H + \kbar \xi^{(1)} +\kbar^2 \xi^{(2)} +\ldots \in \Aut\big(H{\kkbar}\big)$ and
\[
\mb{\l}(1_H)=0,\qquad \mb{\l}=\l^{(0)}+\kbar \l^{(1)}+\kbar^2 \l^{(2)}+\ldots \in \Hom(H, \sC)^{-1}{\kkbar}.
\]
\end{proposition}

\begin{proof}
Recall that
$f^{(0)}\coloneqq{} f$, $f^{^\pr(0)}\coloneqq{}f^\pr$ and, for all $n\geq 1$, 
\eqn{anya}{
\begin{cases}
f^{(n)} = -{s} \circ g^{(n)},
\\
\k^{(n)} = h \circ g^{(n)},
\\
\end{cases}
\qquad
\begin{cases}
f^{\pr(n)} = -s^\pr \circ g^{\pr(n)},
\\
\k^{\pr (n)} = h^\pr \circ g^{\pr(n)},
\\
\end{cases}
}
where
\eqnalign{anyb}{
{g}^{(n)}\coloneqq{}& \sum_{\mathclap{\ell=0}}^{n-1} K^{(n-\ell)} \circ f^{(\ell)}  - \sum_{\mathclap{\ell=1}}^{n-1} f^{(n-\ell)}\circ\k^{(\ell)}
,\\
{g}^{\pr(n)}\coloneqq{}& \sum_{\mathclap{\ell=0}}^{n-1} K^{(n-\ell)} \circ f^{\pr(\ell)}  - \sum_{\mathclap{\ell=1}}^{n-1} f^{\pr(n-\ell)}\circ \k^{\pr(\ell)}.
}
Recall also that $K\circ g^{(n)}=K\circ g^{\pr(n)}=0$ and
\eqn{anyc}{
g^{ (n)}=f\circ \k^{(n)} -  K \circ f^{(n)},
\qquad
g^{\pr (n)}=f^\pr\circ\k^{\pr(n)} -  K \circ f^{\pr(n)}
.
}
We are going to construct $\mb{\xi}$ and $\mb{\l}$  inductively so
that $\xi^{(0)} = \I_H$, $\l^{(0)} = s\circ (f^{\pr}-f)$ and, for all $n\geq 1$,
\eqn{anyd}{
\begin{cases}
\xi^{(n)} = h\circ w^{(n)}
,\\
\l^{(n)} = -s\circ w^{(n)}
,
\end{cases}
}
where
\eqn{anye}{
w^{(n)}\coloneqq{} 
f^{\pr(n)}
-f^{(n)}
-\sum_{\mathclap{\ell=1}}^{n-1} f^{(n-\ell)}\circ\xi^{(\ell)}
-\sum_{\mathclap{\ell=0}}^{n-1}K^{(n-\ell)}\circ\l^{(\ell)} 
-\sum_{\mathclap{\ell=1}}^{n}\l^{(n-\ell)}\circ\k^{\pr(\ell)}
.
}

\begin{enumerate}
\item In this step, 
we construct $\mb{\l} \bmod \kbar^2$ and $\mb{\xi} \bmod \kbar^2$ and prove the theorem modulo $\kbar^2$.

%We begin with establishing  that 
%\eqn{anna}{
%\k^{\pr (1)}=\k^{(1)}.
%}
%Recall that
%\[
%\eqalign{
%\k^{(1)}\coloneqq{} h K^{(1)}f, \qquad
%\k^{\pr (1)}\coloneqq{}&  h^\pr K^{(1)}f^\pr,
%}
%\]
%where both $K^{(1)}f$ and  $K^{(1)}f^\pr$ are in $\sC\cap \Ker K$.
%Note that $h^\pr=h$ on $\sC\cap \Ker K$.  It follows that
%\[
%\k^{\pr (1)}= h K^{(1)}f^\pr =  h K^{(1)}f + h K^{(1)}Ks=h K^{(1)}f =\k^{(1)}
%\]
%where we have used $K^{(1)}Ks = -K K^{(1)}s$ and $h K=0$ for the second equality.
It is obvious, by definition, that 
$\k^{\pr(1)} =\k^{(1)}$.
Consider the identities $\mb{K} \circ \mb{f}= \mb{f}\circ \mb{\kappa}$ and $\mb{K}\circ  \mb{f}^\pr= \mb{f}^\pr\circ \mb{\kappa}^\pr$
modulo $\kbar ^2$:
\begin{align}
\label{modhb}
\begin{aligned}
K\circ  f&=0,\\
K^{(1)}\circ f+ K \circ f^{(1)}&=f\circ  \k^{(1)}.
\end{aligned}
&&
\begin{aligned}
K\circ  f^{\pr}&=0,\\
K^{(1)}\circ f^{\pr} + K\circ  f^{\pr(1)}&=f^{\pr}\circ \k^{(1)}.
\end{aligned}
\end{align}
% \eqn{modhb}{
% \eqalign{
% K\circ  f&=0,\cr 
% K^{(1)}\circ f+ K \circ f^{(1)}&=f\circ  \k^{(1)}.
% }
% \qquad
% \eqalign{
% K\circ  f^{\pr}&=0,\cr
% K^{(1)}\circ f^{\pr} + K\circ  f^{\pr(1)}&=f^{\pr}\circ \k^{(1)}.
% }
% }
% \end{enumerate}
% \end{proof}
% \end{document}
Note that both $f^\pr$ and $f$ are assumed to induce the identity map on the cohomology $H$, so that
$h\circ (f^\pr - f) = 0$. It follows that $f^\pr -f = - K\circ{s} \circ(f^\pr - f)$. 
Once we define
$\l^{(0)}\coloneqq{} -{s}\circ (f^\pr - f)$,
we conclude that $f^{\pr} = f + K\circ  \l^{(0)}$. From the equations in the second line of \eq{modhb},
we obtain that $(f^\pr - f)\circ\k^{(1)}= K^{(1)}\circ (f^\pr - f) + K\circ(f^{\pr(1)} -f^{(1)})$,
which implies that $Kw^{(1)}=0$, where
\[
w^{(1)}\coloneqq{}f^{\pr(1)} - f^{(1)} - K^{(1)}\circ \l^{(0)} - \l^{(0)}\circ \k^{(1)} \in\Hom(H,\sC)^{0}.
\]
Hence we obtain  a well-defined linear map $\xi^{(1)}\coloneqq{}h\circ w^{(1)} \in \Hom (H,H)^0$
satisfying the following relation:
\begin{align*}
f\circ \xi^{(1)} =&f \circ h\circ  w^{(1)}=w^{(1)} - K\circ  s \circ w^{(1)}
\\
=&f^{\pr(1)} - f^{(1)} - K^{(1)}\circ \l^{(0)} - \l^{(0)}\circ \k^{(1)} + K\circ \l^{(1)},
\end{align*}
where we have defined
$\l^{(1)}\coloneqq{} -s\circ w^{(1)} \in \Hom (H,\sC)^{-1}$.
Therefore, we conclude that
\eqn{modhaa}{
\k^{\pr (1)}=\k^{(1)},
\qquad
\begin{cases} 
f^{\pr}&=f+  K\circ  \l^{(0)}
,\\
f^{\pr(1)}&=f^{(1)}+ f\circ \xi^{(1)}
+K\circ  \l^{(1)}+ K^{(1)}\circ  \l^{(0)}+ \l^{(0)}\circ \k^{(1)}.
\end{cases}
}
Let $\mb{\xi}\coloneqq{} \I_H+\kbar \xi^{(1)} \bmod \kbar^{2}$ and
$\mb{\l}\coloneqq{} \l^{(0)}+\kbar \l^{(1)} \bmod \kbar^{2}$. Then
the relations in \eq{modhaa}
are equivalent to the following;
\begin{align*}
\begin{aligned}
\mb{\l}(1_H) &=0 \bmod \kbar^{2}
,
\\
\mb{\xi}(1_H)&= 1_H\bmod \kbar^{2}
,
\end{aligned}
\qquad
\begin{aligned}
\mb{\kappa}^{\pr}&=\mb{\xi}^{-1}\circ \mb{\kappa}\circ \mb{\xi}
\bmod \kbar^{2}
,\\
\mb{f}^{\pr}&= \mb{f}\circ \mb{\xi}
+ \mb{K}\circ \mb{\l} + \mb{\l}\circ \mb{\kappa}^{\pr}
\bmod \kbar^{2}.
\end{aligned}
\end{align*}
\item 
\label{item: inductive premise}Fix $n > 2$.
Assume that there are some $\mb{\xi}\coloneqq{} I_H+\kbar \xi^{(1)} +\cdots  +\kbar^{n-1} \xi^{(n-1)}\bmod \kbar^{n}$
and $\mb{\l}\coloneqq{} \l^{(0)}+\kbar \l^{(1)} +\cdots  +\kbar^{n-1} \l^{(n-1)}\bmod \kbar^{n}$ such that
\[
\begin{aligned}
\mb{\l}(1_H) &=0 \bmod \kbar^{n}
,
\\
&\mb{\xi}(1_H)&= 1_H\bmod \kbar^{n}
,\\
\end{aligned}
\qquad
\begin{aligned}
\mb{\kappa}^{\pr}&=\mb{\xi}^{-1}\circ \mb{\kappa}\circ \mb{\xi}\bmod \kbar^{n}
,\\
\mb{f}^{\pr}&= \mb{f}\circ \mb{\xi}
+ \mb{K}\circ \mb{\l} + \mb{\l}\circ \mb{\kappa}^{\pr}\bmod \kbar^{n}.
\end{aligned}
\]
\item

We claim that the following is a consequence of the assumptions in step~\eqref{item: inductive premise}:
\eqnalign{xclaimb}{
g^{\pr (n)} -g^{(n)}
=&
-K\circ \left(\sum_{\mathclap{\ell=1}}^{n-1} f^{(n-\ell)}\circ \xi^{(\ell)}
+\sum_{\mathclap{\ell=0}}^{n-1} K^{(n-\ell)}\circ \l^{(\ell)} 
+\sum_{\mathclap{\ell=1}}^{n-1}\l^{(n-\ell)}\circ \k^{\pr(\ell)}
\right)
\\
&
+ f(\mb{\xi}^{-1}\circ \mb{\kappa}\circ \mb{\xi})^{(n)}
-f\circ \k^{(n)}.
}
Applying $h$ to the above identity, we have
\begin{align*}
h \circ g^{\pr (n)}
-h\circ  g^{(n)}=
\big(\mb{\xi}^{-1}\circ \mb{\kappa}\circ \mb{\xi}\big)^{(n)}
-\k^{(n)}.
\end{align*}
Using $h \circ g^{\pr (n)}\equiv h^\pr \circ g^{\pr (n)}=\k^{\pr (n)}$, since $h=h^\pr$ on $\ker K$, and  $h\circ  g^{(n)}=\k^{ (n)}$, we obtain the relation
\eqn{annd}{
\k^{\pr (n)} =  \big(\mb{\xi}^{-1}\circ \mb{\kappa}\circ \mb{\xi}\big)^{(n)}.
}
It follows that
\begin{align*}
 g^{\pr (n)}
-g^{(n)}=
&
-K\circ \left(\sum_{\mathclap{\ell=1}}^{n-1} f^{(n-\ell)}\circ \xi^{(\ell)}
+\sum_{\mathclap{\ell=0}}^{n-1} K^{(n-\ell)}\circ \l^{(\ell)} 
+\sum_{\mathclap{\ell=1}}^{n-1}\l^{(n-\ell)}\circ \k^{\pr(\ell)}
\right)
\\
&
+ f\circ \k^{\pr (n)}
-f\circ \k^{(n)}
%\\
%=
%&-K\left(\sum_{\mathclap{\ell=1}}^{n-1} f^{(n-\ell)}\xi^{(\ell)}
%+\sum_{\mathclap{\ell=0}}^{n-1} K^{(n-\ell)}\l^{(\ell)} 
%+\sum_{\mathclap{\ell=1}}^{n}\l^{(n-\ell)}\k^{\pr(\ell)}
%\right)
%+ f^\pr \k^{\pr (n)}
%-f\k^{(n)}
.
\end{align*}
Using $f^\pr = f + K\l^{(0)}$, the above identity can be rewritten as follows; 
\begin{align*}
 g^{\pr (n)}
-g^{(n)}=
&-K\circ \left(\sum_{\mathclap{\ell=1}}^{n-1} f^{(n-\ell)}\circ \xi^{(\ell)}
+\sum_{\mathclap{\ell=0}}^{n-1} K^{(n-\ell)}\circ \l^{(\ell)} 
+\sum_{\mathclap{\ell=1}}^{n}\l^{(n-\ell)}\circ \k^{\pr(\ell)}
\right)
\\
&
+ f^\pr\circ  \k^{\pr (n)}
-f\circ \k^{(n)}
.
\end{align*}
On the other hand, the identities in \eq{anyc} imply that 
\[
 g^{\pr (n)}-g^{ (n)}
=
f^\pr\circ \k^{\pr(n)} -f\circ \k^{(n)} -  K\circ  f^{\pr(n)} +  K\circ  f^{(n)}.
\]
By combining the two identities above, we conclude that
\begin{align*}
K\circ  w^{(n)}=&0
,\\
w^{(n)}\coloneqq{} 
&
f^{\pr(n)}
-f^{(n)}
-\sum_{\mathclap{\ell=1}}^{n-1} f^{(n-\ell)}\circ \xi^{(\ell)}
-\sum_{\mathclap{\ell=0}}^{n-1}K^{(n-\ell)}\circ \l^{(\ell)} 
-\sum_{\mathclap{\ell=1}}^{n}\l^{(n-\ell)}\circ \k^{\pr(\ell)}
.
\end{align*}
Note that $w^{(n)}\in\Hom(H,\sC)^{0}$ and $w^{(n)}(1_H)=0$. 
Let $\xi^{(n)} \coloneqq{}h\circ w^{(1)} \in \Hom (H,H)^0$ and $\l^{(n)}\coloneqq{}- s\circ w^{(n)} \in \Hom (H,H)^{-1}$.
Then, we obtain that $f\circ \xi^{(n)} =f\circ  h\circ w^{(n)}=w^{(n)}  + K\circ \l^{(n)}$
and conclude that
\eqn{anne}{
\begin{aligned}
\xi^{(n)}(1_H)=\l^{(n)}(1_H)&=0
,\\
f^{\pr(n)}
&=f^{(n)}
+\sum_{\mathclap{\ell=1}}^{n} f^{(n-\ell)}\xi^{(\ell)}
+\sum_{\mathclap{\ell=0}}^{n}K^{(n-\ell)}\l^{(\ell)} 
+\sum_{\mathclap{\ell=1}}^{n}\l^{(n-\ell)}\k^{\pr(\ell)}
\end{aligned}
}
Let 
\begin{align*}
\mb{\xi}\coloneqq{}& 1+\kbar \xi^{(1)} +\cdots+ \kbar^n \xi^{(n)} \bmod \kbar^{n+1}
,\\
\mb{\l}\coloneqq{}& \l^{(0)}+\kbar \l^{(1)}+\cdots+ \kbar^n \l^{(n)}  \bmod \kbar^{n+1}.
\end{align*}
Then
the relations in \eq{anne} together with our assumptions in step~\eqref{item: inductive premise} are
equivalent to the following;
\[
\begin{aligned}
\mb{\l}(1_H) &=0 \bmod \kbar^{n+1}
,
\\
\mb{\xi}(1_H)&= 1_H\bmod \kbar^{n+1}
,\\
\end{aligned}
\qquad
\begin{aligned}
\mb{\kappa}^{\pr}&=\mb{\xi}^{-1}\circ\mb{\kappa}\circ\mb{\xi}\bmod \kbar^{n+1}
,\\
\mb{f}^{\pr}&= \mb{f}\circ \mb{\xi}
+ \mb{K}\circ\mb{\l} + \mb{\l}\circ \mb{\kappa}^{\pr}\bmod \kbar^{n+1}.
\end{aligned}
\]
so we are done. \naturalqed
 \end{enumerate}
\end{proof}

\begin{corollary}
Suppose
$\mb{\kappa}=\kbar^n \k^{(n)}+\kbar^{n+1}\k^{(n+1)}+\cdots$
has $\k^{(n)}\neq 0$ for some $n\geq 1$.
Then 
$\mb{\kappa}^\pr=\kbar^n \k^{\pr (n)}+\kbar^{n+1}\k^{\pr(n+1)}+\cdots$
and $\k^{\pr(n)}=\k^{(n)}$.
\end{corollary}

\subsection{Homotopy \texorpdfstring{$\kbar$}{h-bar}-divisibility}

The purpose of this subsection is to  define a $\Bbbk$-linear operator 
\[
\nabla_{(-\kbar)^{-1}}: \Hom\big(\Xbar T(H), \sC\big){\kkbar}\longrightarrow \Hom\big(\Xbar T(H), \sC\big){\kkbar},
\]
which will be used to prove a key technical Lemma \ref{hodge} for constructing a distinguished homotopical family of
quantum observables in the next section.

Fix a classical off-to-on-shell retract $(f,h,{s} )$
and  consider the associated deformation quantization
$\xymatrix{\left(H{\kkbar}, 1_H,\mb{\kappa}\right)\ar[r]^{\mb{f}}&\left(\sC{\kkbar},  1_\sC,\mb{K}\right)}$
of $\xymatrix{\left(H, 1_H,0\right)\ar[r]^{{f}}&\left(\sC,  1_\sC,K\right)}$.
Introduce the following $\Bbbk{\kkbar}$-linear operators
\begin{align*}
\mb{K}_\mathit{H\sC}:\;&\Hom\big(\Xbar T(H), \sC\big){\kkbar} \rightarrow \Hom\big(\Xbar T(H), \sC\big){\kkbar}
,\\
\mb{\kappa}_\mathit{HH}:\;&\Hom\big(\Xbar T(H), H\big){\kkbar}\rightarrow \Hom\big(\Xbar T(H), H\big){\kkbar}
,
\end{align*}
defined for all $n\geq 1$ and $v_1,\ldots,v_n \in H$ via
\begin{multline*}
\big(\mb{K}_\mathit{H\sC}\mb{\O}\big) ( {v}_1,\ldots,{v}_n)
\coloneqq{}
\\\mb{K}\mb{\O}( {v}_1,\ldots,{v}_n)
-(-1)^{|\mb{\O}|}\sum_{\mathclap{j=1}}^n 
\mb{\O}(J {v}_1,\ldots,Jv_{i-1},\mb{\kappa}{v}_j,v_{j+1},\ldots,{v}_n)
,\end{multline*}
\begin{multline*}
\big(\mb{\kappa}_\mathit{HH} \mb{\omega}\big) ( {v}_1,\ldots,{v}_n)
\coloneqq{}
\\\mb{\kappa}\mb{\omega}( {v}_1,\ldots,{v}_n)
-(-1)^{|\mb{\omega}|}\sum_{\mathclap{j=1}}^n 
\mb{\omega} (J {v}_1,\ldots,Jv_{i-1},\mb{\kappa}{v}_j,v_{j+1},\ldots,{v}_n),
\end{multline*}
where $\mb{\O}\in \Hom \big(\Xbar{T}(H), \sC\big){\kkbar}$ and $\mb{\omega}\in\Hom\big(\Xbar T(H), H\big){\kkbar}$.
It is trivial to check that $\mb{K}_\mathit{H\sC}^2=\mb{\kappa}_\mathit{HH}^2=0$. Note also 
that $\mb{f}\circ \mb{\omega}\in\Hom\big(\Xbar T(H), \sC\big){\kkbar}$
and $\mb{K}_\mathit{H\sC}\big(\mb{f}\circ \mb{\omega}\big) = \mb{f}\circ\big(\mb{\kappa}_\mathit{HH} \mb{\omega}\big)$ since
$
\mb{K}_\mathit{H\sC}(\mb{f}\circ \mb{\omega}) 
= \mb{K}\circ\mb{f}\circ \mb{\omega} -(-1)^{|\mb{\omega}|}\mb{f}\circ  \mb{\omega}\circ  \mb{\kappa}=\mb{f}\circ 
\left( \mb{\kappa}\circ \mb{\omega}-(-1)^{|\mb{\omega}|} \mb{\omega} \circ\mb{\kappa}\right)
$.

For a homogeneous element $\mb{\O} \in \Hom\big(\Xbar{T}(H), \sC\big){\kkbar}$, we 
define  $\nabla_{(-\kbar)^{-1}}\mb{\O}$ via
\eqn{hdbh}{
(-\kbar)\nabla_{(-\kbar)^{-1}}\mb{\O} \coloneqq{}\mb{\O} -\mb{f}\circ h\circ \O - \mb{K}_\mathit{H\sC} (s\circ \O) - s\circ K\circ \O.
}
Note that the right hand side vanishes is divisible by $\kbar$ because it vanishes in the classical limit:
\[
(\I_\sC - f\circ h - K\circ s -s\circ K)\circ \O=0,
\]
Therefore 
$\nabla_{\kbar^{-1}}\mb{\O}\in \Hom\big(\Xbar T(H), \sC\big){\kkbar}$.
If  the classical limit of $\mb{\O}$ vanishes, we have $\nabla_{(-\kbar)^{-1}}\mb{\O}=\Fr{1}{(-\kbar)}\mb{\O}$.

\begin{proposition}\label{dvdbyh}
Fix $k \in \Z$ with $k\geq 1$ and assume that there is a trio  $(\mb{\Omega}, \mb{\Xi}, \mb{\omega})$, where
$\mb{\Omega}$ and $\mb{\Xi}$ are in $\Hom\big(\Xbar{T}(H), \sC\big){\kkbar}$ and $\mb{\omega}$ is in $\Hom\big(\Xbar T(H), H\big){\kkbar}$,
which satisfies the following relation:
\[
\mb{K}_{\!H\!\sC} \mb{\Omega}= (-\kbar)^{k} \mb{\Xi} - \mb{f}\circ \mb{\omega}.
\]
Define, for $j=0,1,\ldots, k$, 
\begin{align*}
\mb{\Omega}^{[j]}&\coloneqq{}\left(\nabla_{(-\kbar)^{-1}}\right)^j \mb{\Omega}
,\\
(-\kbar)^j \mb{\omega}^{[j]}&\coloneqq{}\mb{\omega} +\sum_{\mathclap{i=0}}^{j-1} (-\kbar)^i\mb{\kappa}_{\!H\!H}\left( h\circ \Omega^{[i]}\right)
\end{align*}
where  $\O^{[j]}$ denotes the classical limit of $\mb{\Omega}^{[j]}$.  Then, we have
\begin{enumerate}
\item
 for $j=0,1,\ldots, k$,
\[
\mb{K}_{\!H\!\sC}\mb{\Omega}^{[j]}=(-\kbar)^{k-j}\mb{\Xi} - \mb{f}\circ \mb{\omega}^{[j]}
\quad\text{ and }\quad 
\begin{cases}
\mb{\Omega}^{[j]}\in \Hom\big(T(H), \sC\big){\kkbar}
,\\
\mb{\omega}^{[j]} \in \Hom\big(T(H), H\big){\kkbar}
,
\end{cases}
\]
\item  for $i=0,1,\ldots, k-1$,
\[
K\circ\Omega^{[i]}=0.
\]
\end{enumerate}
\end{proposition}

\begin{proof}
Recall that $\mb{\Omega}^{[j]}\in \Hom\big(T(H), \sC\big){\kkbar}$ for all
$j=0,1,\ldots, k$.
Note that $\mb{\Omega}^{[j]}$ and $\mb{\omega}^{[j]}$, $j=0,1,\ldots,k$, can be also defined recursively by
$\mb{\Omega}^{[0]}\coloneqq{}\mb{\Omega}$, $\mb{\omega}^{[0]}=\mb{\omega}$ and, for 
$1\leq j\leq k$,
\begin{align*}
\mb{\Omega}^{[j]}\coloneqq{}\nabla_{(-\kbar)^{-1}} \mb{\Omega}^{[j-1]}
,\qquad
(-\kbar)\mb{\omega}^{[j]}\coloneqq{}\mb{\omega}^{[j-1]} +\mb{\kappa}_{\!H\!H}\left( h\circ \Omega^{[j-1]}\right).
\end{align*}
From $\mb{\O}^{[0]} =\mb{\O}$ and $\mb{\omega}^{[0]}=\mb{\omega}$,  the
relation we have assumed is
\eqn{iuda}{
\mb{K}_{\!H\!\sC} \mb{\Omega}^{[0]}= (-\kbar)^{k} \mb{\Xi} - \mb{f}\circ \mb{\omega}^{[0]}.
}
From the condition that $k\geq 1$, the classical limit of this equation is
\[
K\circ\Omega^{[0]}=-f\circ \omega^{[0]},
\]
where $\omega^{[0]}\in \Hom\big(\Xbar{T}(H), H\big)$ is the classical limit of $\mb{\omega}^{[0]}$. 
Applying $h$ to this classical limit equation, we have
$0= h\circ f\circ \omega^{[0]}$ since $h\circ K=0$. From $ h\circ f=I_H$, we conclude that 
\eqn{iudb}{
\begin{aligned}
K\circ\Omega^{[0]}=0
,\\
\omega^{[0]}=0.
\end{aligned}
}
From, $\mb{\Omega}^{[1]} \coloneqq{}\nabla_{(-\kbar)^{-1}} \mb{\Omega}$ and $K\circ\Omega^{[0]}=0$, it follows that
\eqn{iudd}{
(-\kbar)\mb{\Omega}^{[1]} = \mb{\Omega}^{[0]} - \mb{f}\circ \left(h \circ \Omega^{[0]}\right) 
- \mb{K}_\mathit{H\sC}\left({s} \circ \Omega^{[0]}\right).
}
Applying $\mb{K}_\mathit{H\sC}$ to the above and using $\mb{K}_\mathit{H\sC}^2=0$ and
$\mb{K}_{\!H\!\sC}\mb{f}=0$, we obtain that
\[
(-\kbar)\mb{K}_{\!H\!\sC}\mb{\Omega}^{[1]}
=\mb{K}_{\!H\!\sC} \mb{\Omega}^{[0]} -  \mb{f}\circ\mb{\kappa}_\mathit{HH} \left(h \circ \Omega^{[0]}\right).
\]
From \eq{iuda}, the above relation gives
\eqnalign{iude}{
(-\kbar)\mb{K}_{\!H\!\sC}\mb{\Omega}^{[1]}
=&
(-\kbar)^k \mb{\Xi} 
-(-\kbar)\mb{f}\circ\mb{\omega}^{[1]},
}
where
\[
(-\kbar)\mb{\omega}^{[1]}\coloneqq{}\mb{\omega}^{[0]} +\mb{\kappa}_\mathit{HH}\left(h \circ \Omega^{[0]}\right).
\]
Note that the right-hand-side of the above equation vanishes in the classical limit since both $\mb{\omega}^{[0]}$ and 
$\mb{\kappa}_\mathit{HH}$ vanish in the classical limit. 
Therefore we have
\eqn{iudf}{
\mb{\omega}^{[1]} \in \Hom\big(T(H), H\big){\kkbar}.
}
Dividing \eq{iude} by $(-\kbar)$, we conclude that
\eqn{iudg}{
\mb{K}_{H\sC}\mb{\Omega}^{[1]}= (-\kbar)^{k-1} \mb{\Xi} 
-\mb{f}\circ\mb{\omega}^{[1]}.
}
Therefore, we are done if $k=1$. 

For $k\geq 2$, we work inductively as follows: Fix $n$ such that $1\leq n \leq k-1$ 
and assume that, 
for all $i=0,1,\ldots, n$,
\begin{enumerate}

\item
$\mb{\omega}^{[i]} \in \Hom\big(\Xbar T(H),H\big){\kkbar}$;

\item
$\mb{K}_{\!H\!\sC}\mb{\Omega}^{[i]}=(-\kbar)^{k-i}\mb{\Xi} - \mb{f}\circ \mb{\omega}^{[i]}$;

\item
$\omega^{[i]}=K\circ\O^{[i]}=0$, where $\omega^{[i]}$ is the classical limit of $\mb{\omega}^{[i]}$.

\end{enumerate}
 
From
$(-\kbar)\mb{\omega}^{[n+1]}\coloneqq{}\mb{\omega}^{[n]} +\mb{\kappa}_\mathit{HH}\left( h\circ \Omega^{[n]}\right)$
we conclude that 
\eqn{iudh}{
\mb{\omega}^{[n+1]} \in \Hom\big(\Xbar T(H),H\big){\kkbar}, 
}
since both the classical
limits of $\mb{\omega}^{[n]}$ and $\mb{\kappa}_\mathit{HH}$ vanish.
Note that
we have
\[
(-\kbar)^{n+1} \mb{\O}^{[n+1]} =
\mb{\Omega}
-\mb{f}\circ\left(\sum_{\mathclap{j=0}}^{n}(-\kbar)^j h\circ {\Omega}^{[i]}\right)
-\mb{K}_{H\sC}\left(\sum_{\mathclap{j=0}}^{n}(-\kbar)^j {s} \circ {\Omega}^{[i]} \right),
\]
where we have used $\mb{\O}^{[i]} =\left(\nabla_{(-\kbar)^{-1}}\right)^i \mb{\O}$ and 
the assumption that $K\circ\O^{[i]}=0$,
for all $i=0,1,\ldots, n$. 
Applying $\mb{K}_\mathit{H\sC}$ to the above, we have
\begin{align*}
(-\kbar)^{n+1} \mb{K}_\mathit{H\sC}\mb{\O}^{[n+1]} 
=
&
 \mb{K}_\mathit{H\sC}\mb{\Omega}
-\mb{f}\circ\left(\sum_{\mathclap{j=0}}^{n}(-\kbar)^j\mb{\kappa}_\mathit{HH}  \left(h\circ {\Omega}^{[i]}\right)\right)
\\
=
&
(-\kbar)^k\mb{\Xi}
-\mb{f}\circ\left(\mb{\omega}+\sum_{\mathclap{j=0}}^{n}(-\kbar)^j\mb{\kappa}_\mathit{HH}  \left(h\circ {\Omega}^{[i]}\right)\right)
\\
=
&
(-\kbar)^k\mb{\Xi}
-(-\kbar)^{n}\mb{f}\circ\left(\mb{\omega}^{[n]}+\mb{\kappa}_\mathit{HH}  \left(h\circ {\Omega}^{[n]}\right)\right)
.
\end{align*}
Using $(-\kbar)\mb{\omega}^{[n+1]}\coloneqq{}\mb{\omega}^{[n]} +\mb{\kappa}_\mathit{HH}\left( h\circ \Omega^{[n]}\right)$, 
the above identity becomes
\[
(-\kbar)^{n+1} \mb{K}_\mathit{H\sC}\mb{\O}^{[n+1]} =(-\kbar)^k\mb{\Xi}
-(-\kbar)^{n+1}\mb{f}\circ \mb{\omega}^{[n+1]}
\]
Since $n+1\leq k$, the above equation gives
\eqn{iudk}{
\mb{K}_\mathit{H\sC}\mb{\O}^{[n+1]} =(-\kbar)^{k-n-1}\mb{\Xi}-\mb{f}\circ \mb{\omega}^{[n+1]}.
}
For $n+1 <k$, the classical limit of this equation is $K\circ {\O}^{[n+1]} =-{f}\circ {\omega}^{[n+1]}$,
which in turn implies that 
\eqn{iudl}{
\begin{aligned}
K\circ\Omega^{[n+1]}=0
,\\
\omega^{[n+1]}=0.
\end{aligned}
} 
For $n+1=k$, we have
\eqn{iudm}{
\mb{K}_\mathit{H\sC}\mb{\O}^{[k]} =\mb{\Xi}-\mb{f}\circ \mb{\omega}^{[k]},
}
and we are done by induction.
\naturalqed
\end{proof}

As a corollary of  Proposition \ref{dvdbyh}, we conclude the following.

\begin{lemma}\label{hodge}  
For a trio  $(\mb{\Omega}, \mb{\Xi}, \mb{\omega})$ satisfying
, for $k\geq 1$;
\[
\mb{K}_{\!H\!\sC} \mb{\Omega}= (-\kbar)^{k} \mb{\Xi} - \mb{f}\circ \mb{\omega},
\qquad
\begin{cases}
\mb{\Omega},\mb{\Xi}\in \Hom\big(\Xbar T(H), \sC\big){\kkbar},\\ 
\mb{\omega} \in \Hom\big(\Xbar T(H), H\big){\kkbar},
\end{cases}
\]
we have
\eqn{solvafw}{
\mb{K}_{\!H\!\sC}\mb{\Omega}^{[k]}=\mb{\Xi} - \mb{f}\circ \mb{\omega}^{[k]}
\quad\text{ and }\quad 
\begin{cases}
\mb{\Omega}^{[k]}\in \Hom\big(\Xbar T(H), \sC\big){\kkbar}
,\\
\mb{\omega}^{[k]} \in \Hom\big(\Xbar T(H), H\big){\kkbar}
,
\end{cases}
}
where
\eqn{solvaww}{
\begin{aligned}
(-\kbar)^{k}\mb{\Omega}^{[k]}
\coloneqq{}&
\mb{\Omega}
-\mb{f}\circ\left(\sum_{\mathclap{i=0}}^{k-1}(-\kbar)^i h\circ {\Omega}^{[i]}\right)
-\mb{K}_{H\sC}\left(\sum_{\mathclap{i=0}}^{k-1}(-\kbar)^i {s} \circ {\Omega}^{[i]} \right)
,\\
(-\kbar)^k \mb{\omega}^{[k]}\coloneqq{}&\mb{\omega} +\sum_{\mathclap{i=0}}^{k-1} (-\kbar)^i\mb{\kappa}_{\!H\!H}\left( h\circ \Omega^{[i]}\right)
.
\end{aligned}
}
\end{lemma}

\section{Mastering quantum correlations}
\label{section: mastering qc}

The purpose of this section is to introduce the master equations for level $0$ and $1$ quantum correlators
and present canonical solutions. 
Throughout this section, we fix the following data.
\begin{enumerate} 
\item a binary QFT algebra $\sC{\kkbar}_\BQ=\big(\sC{\kkbar}, 1_\sC, \,\cdot\,,\mb{K}\big)$, where
\begin{enumerate}\item the tuple $\big(\sC{\kkbar}, 1_\sC, \underline{\bell}\big)$, is the quantum descendant unital $sL_\infty$-algebra
 and 
\item the tuple $\big(H{\kkbar}, 1_H, \mb{\kappa}\big)$ is the on-shell QFT complex,
\end{enumerate}
\item a classical off-to-on-shell retraction $(f, h, s)$
and a quantization $(\mb{f}, \mb{h},\mb{s})$ of it.
\end{enumerate}

\subsection{Master equation for the level zero quantum correlators}
\label{subsec: Master equation for the level zero quantum correlators}

In this subsection, we define the master equation governing the level zero quantum correlations
and find a canonical solution.

\begin{definition}\label{masterzero}
The level zero quantum master equation is 
a system of equations
for a tuple $\left\{\grave{\mb{\pi}}^0, \mb{\eta}^{-1},\grave{\bell},\mb{\phi}^0\right\}$,
where
\[
\begin{aligned}
\grave{\mb{\pi}}^0  &\in \Hom\left(\Xbar{S}(H), H\right)^{0}{\kkbar}
,\\
\grave{\bell} &\in \Hom\left(\Xbar{S}(H), H\right)^{1}{\kkbar}
,
\end{aligned}
\qquad
\begin{aligned}
\mb{\eta}^{-1}  &\in \Hom\left(\Xbar{S}(H),\sC\right)^{-1}{\kkbar}
,\\
\mb{\phi}^0 &\in \Hom\left(\Xbar{S}(H),\sC\right)^{0}{\kkbar}
,
\end{aligned}
\]
are defined recursively for all $n\geq 1$ and homogeneous $v_1,\ldots, v_n \in H$ by the equations
\begin{align*}
\mb{f}\left(\grave{\mb{\pi}}^0_n({v}_1, \cdots, {v}_n)\right)
=&
\sum_{\mathclap{\mp\in P(n)}}^{{}}
(-\kbar)^{n-|\mp|}\e(\mp)\,
\mb{\phi}^0\big({v}_{B_1}\big)\cdot\dotsc\cdot\mb{\phi}^0\big({v}_{B_{|\mp|}}\big)
\\
&
-\mb{K} \mb{\eta}^{-1}_n({v}_1,\dotsc, {v}_n)
\\
&
-\sum_{\clapsubstack{\mp \in P(n)\\ |B_i| = n-|\mp|+1}}
(-\kbar)^{n-|\mp|}
\e(\mp)\,
\mb{\eta}^{-1}_{|\mp|}\Big( J{v}_{B_1},\dotsc, J{v}_{B_{i-1}},\grave{\bell}\!
\big({v}_{B_i}\big), {v}_{B_{i+1}},\dotsc, {v}_{B_{|\mp|}}\Big)
,\\
 \mb{\kappa}\grave{\mb{\pi}}^0_n({v}_1,\dotsc,{v}_n)
 =
&
\sum_{\clapsubstack{\mp \in P(n)\\ |B_i| = n-|\mp|+1}}
(-\kbar)^{n-|\mp|}\e(\mp)\,
\grave{\mb{\pi}}^0_{|\mp|}\Big( J{v}_{B_1}, \dotsc, J{v}_{B_{i-1}}, 
\grave{\bell}\!\big({v}_{B_i}\big), {v}_{B_{i+1}},\dotsc, {v}_{B_{|\mp|}}\Big)
,
\end{align*}
with the following initial conditions:
\[
\grave{\mb{\pi}}^0_1=\I_{H{\kkbar}}
,\qquad
\mb{\eta}^{-1}_1=0
,\qquad
\mb{\phi}^0_1=\mb{f}
,\qquad
\grave{\bell}_1=\mb{\kappa}
.
\]
\end{definition}

\begin{remark}
For $n=1$, the (level zero) quantum master equation is  
\begin{align*}
\mb{f}\circ \mb{\pi}^0_1=&\mb{\phi}_1 -\mb{K}\circ\mb{\eta}^{-1}_1 -\mb{\eta}^{-1}_1\circ \grave{\bell}_1
,\\
\mb{\kappa}\circ \grave{\mb{\pi}}^0_1 = &\grave{\mb{\pi}}^0_1\circ \grave{\bell}_1.
\end{align*}
so that the initial conditions solve the equation for $n=1$.
\naturalqed
\end{remark}

The following proposition may be regarded as the appropriate integrability condition for the level zero quantum master equation.

\begin{proposition}\label{zerodes}
Let $\left\{\grave{\mb{\pi}}^0, \mb{\eta}^{-1},\grave{\bell},\mb{\phi}^0\right\}$ be a solution to the level zero 
quantum master equation. Then, we have
\begin{itemize}
\item The $L_\infty$ structure $\underline{\grave{\bell}}=\underline{0}$ whenever $\mb{\kappa}=0$.

\item The tuple
$\left(H{\kkbar}, 1_H,\underline{\grave{\bell}}\right)$ is a topologically-free unital $sL_\infty$-algebra over $\Bbbk{\kkbar}$, i.e.,
 for all $n\geq 1$ and homogeneous ${v}_1,\ldots, {v}_n \in H$, we have
\begin{align*}
\sum_{\clapsubstack{\mp \in P(n)\\ |B_i| = n-|\mp|+1}}\e(\mp)\,
\grave{\bell}_{|\mp|}\left( J{v}_{B_1}, \cdots, J{v}_{B_{i-1}}, \grave{\bell}({v}_{B_i}),  {v}_{B_{i+1}},\dotsc,  {v}_{B_{|\mp|}}\big)\right)
=0
,\\
\grave{\bell}_{n}\big({v}_1,\dotsc, {v}_{n-1},1_H\big)=0.
\end{align*}

\item The map $\underline{\mb{\phi}}^0$ satisfies the conditions to be a quasi-isomorphism of  topologically-free  unital $sL_\infty$-algebras
$\xymatrix{\underline{\mb{\phi}}^0: \left(H{\kkbar}, 1_H,\underline{\grave{\bell}}\right)\ar@{..>}[r] 
& \left(\sC{\kkbar}, 1_\sC,\underline{\bell}\right)}$, 
i.e., for all $n\geq 1$ and homogeneous ${v}_1,\ldots, {v}_n \in H$, we have
\begin{align*}
\sum_{\clapsubstack{\mp \in P(n)}}
\e(\mp)\,
\bell_{|\mp|}\!\left(\mb{\phi}^0\left( {v}_{B_1}\right), \cdots, \mb{\phi}\big(  {v}_{B_{|\mp|}}\big)\right)
=\sum_{\clapsubstack{\mp \in P(n)\\ |B_i| = n-|\mp|+1}}\e(\mp)\,
\mb{\phi}_{|\mp|}\!\left( J{v}_{B_1}, \cdots, \grave{\bell}({v}_{B_i}),  \dotsc,  {v}_{B_{|\mp|}}\big)\right)
,\\
\mb{\phi}^0_{n}\big({v}_1,\dotsc, {v}_{n-1}, 1_H\big)= 1_\sC\cdot \d_{n,1},
\end{align*}
where $\d_{1,1}=1$ and $\d_{n,1}=0$ if $n\neq 1$,
and $\mb{\phi}^0_1: \big(H{\kkbar}, 1_H,\grave{\bell}_1\big)\rightarrow \big(\sC{\kkbar}, 1_\sC,\bell_1\big)$ 
is a pointed cochain quasi-isomorphism.

\end{itemize}
\end{proposition}

\begin{proof}

It is obvious that $\grave{\bell}_n=0$ for all $n\geq 1$ whenever $\mb{\kappa}=0$.
It is also straightforward to check the following properties:

\begin{enumerate}[label=\alph*.,ref=\alph*]
\item (unitality of structure)
\label{item: unitality of structure}
$\grave{\bell}_{n}\big({v}_1,\dotsc, {v}_{n-1},1_H\big)=0$ for all $n\geq 1$ and $v_1,\ldots, v_{n-1} \in H$.
\item (unitality of morphism)
\label{item: unitality of morphism}
 $\mb{\phi}^0_1(1_H)=1_\sC$, 
 and  $\mb{\phi}^0_{n+1}\big({v}_1,\dotsc, {v}_{n}, 1_H\big)=0$  for all $n\geq 1$ and $v_1,\ldots, v_{n-1} \in H$.
\end{enumerate}

\begin{notation}
\label{notation: hbar extensions}
Consider the reduced symmetric coalgebra $\Xbar{S}^\mathit{co}(\sC){\kkbar}$ which is cogenerated by the topologically-free
$\Bbbk{\kkbar}$-module $\sC{\kkbar}$. 
We may extend the $\kbar$-shifted version of ${\bell}$ to a coderivation $\mb{\d}_{\bell}$
on $\Xbar{S}^\mathit{co}(\sC){\kkbar}$.
This coderivation is characterized  for all $n\geq 1$ and homogeneous $x_1,\ldots, x_n \in \sC$ by the formula
\[
\mb{\d}_{\bell}({x}_1\bm{\odot}\dotsc\bm{\odot} {x}_n)
\coloneqq{}
\sum_{\clapsubstack{\mp \in P(n)\\ |B_i| = n-|\mp|+1}}
(-\kbar)^{n-|\mp|}\e(\mp)\,
J\!{x}_{B_1}\bm{\odot} \dotsc\bm{\odot} J\!{x}_{B_{i-1}}\bm{\odot} {\bell}({x}_{B_i})\bm{\odot} {x}_{B_{i+1}}
\bm{\odot} \dotsc\bm{\odot} {x}_{B_{|\mp|}}.
\]
Define $\mb{\pi}^0 \in  \Hom\left(\Xbar{S}(\sC),\sC\right)^{0}{\kkbar}$  for all $n\geq 1$ and $\bm{x}_1,\ldots,\bm{x}_n \in \sC{\kkbar}$ to be
\[
\mb{\pi}^0\big(\bm{x}_1\bm{\odot}\ldots\bm{\odot} \bm{x}_k\big) =\bm{x}_1\cdot\ldots\cdot \bm{x}_k.
\]
Then, by the definition of the quantum descendant, we have
$\mb{K}\circ \mb{\pi}^0 =\mb{\pi}^0\circ \mb{\d}_{\bell}$, which implies that $ \mb{\d}_{\bell}\circ \mb{\d}_{\bell}=0$.

Similarly, consider the reduced symmetric coalgebra $\Xbar{S}^\mathit{co}(H){\kkbar}$ which is cogenerated by the topologically-free
$\Bbbk{\kkbar}$-module $H{\kkbar}$. 
Define a coderivation $\mb{\d}_{\!\grave{\bell}}$ on $\Xbar{S}^\mathit{co}(H){\kkbar}$ (as before)
and  a coalgebra map $\mb{\Psi}_{\mb{\phi}^0}: \Xbar{S}^\mathit{co}(H){\kkbar}\rightarrow \Xbar{S}^\mathit{co}(\sC){\kkbar}$
characterized by the following equations for all $n\geq 1$ and homogeneous $\bm{v}_1,\ldots, \bm{v}_n \in H{\kkbar}$:
\begin{align*}
\mb{\d}_{\!\grave{\bell}}(\bm{v}_1\bm{\odot}\dotsc\bm{\odot} \bm{v}_n)
&\coloneqq{}
\sum_{\clapsubstack{\mp \in P(n)\\ |B_i| = n-|\mp|+1}}
(-\kbar)^{n-|\mp|}\e(\mp)\,
J\!\bm{v}_{B_1}\bm{\odot} \dotsc\bm{\odot} J\!\bm{v}_{B_{i-1}}\bm{\odot} 
\grave{\bell}({v}_{B_i})\bm{\odot} \bm{v}_{B_{i+1}}\bm{\odot} \dotsc\bm{\odot} \bm{v}_{B_{|\mp|}}
,\\
\mb{\Psi}_{\mb{\phi}^0}\big(\bm{v}_1\bm{\odot}\dotsc\bm{\odot} \bm{v}_n\big)
&\coloneqq{}
\sum_{\clapsubstack{\mp\in P(n)}
}
(-\kbar)^{n-|\mp|}
\e(\mp)\mb{\phi}^0\!\left(\bm{v}_{B_1}\right)\bm{\odot}\ldots\bm{\odot}
\mb{\phi}^0\!\left(\bm{v}_{B_{|\mp|}}\right).
\end{align*}
\end{notation}

Using this notation the level zero quantum master equation can be written as follows:
\eqn{qmassp}{
\begin{aligned}
\mb{f}\circ\grave{\mb{\pi}}^0 &= \mb{\pi}^0\circ \mb{\Psi}_{\mb{\phi}^0} 
- \mb{K}\circ\mb{\eta}^{-1} 
-\mb{\eta}^{-1}\circ  \mb{\d}_{\!\grave{\bell}}
,\\
\mb{\kappa}\circ\grave{\mb{\pi}}^0 &=\grave{\mb{\pi}}^0 \circ  \mb{\d}_{\!\grave{\bell}}
.
\end{aligned}
}
Applying $\mb{\kappa}\circ $ to the second equation above
and using $\mb{\kappa}\circ \mb{\kappa}=0$, 
we obtain the equation
$
0=\mb{\kappa}\circ\grave{\mb{\pi}}^0 \circ \mb{\d}_{\!\grave{\bell}}
=\grave{\mb{\pi}}^0 \circ \mb{\d}_{\!\grave{\bell}}\circ \mb{\d}_{\!\grave{\bell}}
$.
From $\grave{\mb{\pi}}^0_1=\I_{H}$, we have
\eqn{qmassx}{
\proj_{H{\kkbar}}\circ \mb{\d}_{\!\grave{\bell}}\circ  \mb{\d}_{\!\grave{\bell}}=0
\Longleftrightarrow  \mb{\d}_{\!\grave{\bell}}\circ \mb{\d}_{\!\grave{\bell}}=0.
}
In components, we have, for all $n\geq 1$,
\begin{align*}
0=&\proj_{H{\kkbar}}\circ \mb{\d}_{\!\grave{\bell}}\circ \mb{\d}_{\!\grave{\bell}}
(\bm{v}_1\bm{\odot}\ldots\bm{\odot} \bm{v}_n)
\\
\equiv
&
(-\kbar)^{n-1}\sum_{\clapsubstack{\mp \in P(n)\\ |B_i| = n-|\mp|+1}}
\e(\mp)\,\grave{\bell}\Big(
J\!{v}_{B_1}\bm{\odot} \dotsc\bm{\odot} J\!\bm{v}_{B_{i-1}}\bm{\odot} 
\grave{\bell}(\bm{v}_{B_i})\bm{\odot} \bm{v}_{B_{i+1}}\bm{\odot} \dotsc\bm{\odot} \bm{v}_{B_{|\mp|}}\Big)
.
\end{align*}
Therefore  $\big(H{\kkbar}, \underline{\grave{\bell}}\big)$ is an $sL_\infty$-algebra over $\Bbbk{\kkbar}$.
Then, property~\eqref{item: unitality of structure} implies that $\Big(H{\kkbar}, 1_H,\underline{\grave{\bell}}\Big)$ 
is a unital $sL_\infty$-algebra over $\Bbbk{\kkbar}$.

Applying $\mb{K}\circ $ to the first part of \eq{qmassp} and using $\mb{K}\circ\mb{K}=0$, we obtain that
\eqn{qmassy}{
\mb{K}\circ\mb{f}\circ\grave{\mb{\pi}}^0 +\mb{K}\circ\mb{\eta}^{-1}\circ \mb{\d}_{\!\grave{\bell}}
= \mb{K}\circ\mb{\pi}^0\circ\mb{\Psi}_{\mb{\phi}^0}.
}
Consider the left hand side of this equation:
\begin{align*} 
\mb{K}\circ\mb{f}\circ\grave{\mb{\pi}}^0 
+\mb{K}\circ\mb{\eta}^{-1}\circ \mb{\d}_{\!\grave{\bell}}
&=
\mb{f}\circ \mb{\kappa}\circ\grave{\mb{\pi}}^0 
+\mb{K}\circ\mb{\eta}^{-1}\circ \mb{\d}_{\!\grave{\bell}}
\\&=
\mb{f}\circ\grave{\mb{\pi}}^0 \circ \mb{\d}_{\!\grave{\bell}}
+\mb{K}\circ\mb{\eta}^{-1}\circ \mb{\d}_{\!\grave{\bell}}
\\
&=
\mb{\pi}^0\circ \mb{\Psi}_{\mb{\phi}^0}\circ \mb{\d}_{\!\grave{\bell}}
- \mb{K}\circ\mb{\eta}^{-1}\circ \mb{\d}_{\!\grave{\bell}} 
-\mb{\d}_{\!\grave{\bell}}\circ \mb{\d}_{\!\grave{\bell}}
+\mb{K}\circ\mb{\eta}^{-1}\circ \mb{\d}_{\!\grave{\bell}}
\\
&=
\mb{\pi}^0\circ \mb{\Psi}_{\mb{\phi}^0}\circ \mb{\d}_{\!\grave{\bell}},
\end{align*}
where we have used $\mb{K}\circ\mb{f} =\mb{f}\circ \mb{\kappa}$ for the first equality, the second part of \eq{qmassp} 
for the second equality, the first part \eq{qmassp} for the third equality, and \eq{qmassx} for the last equality.
From  $\mb{K}\circ \mb{\pi}^0 =\mb{\pi}^0\circ \mb{\d}_{\!{\bell}}$, the right hand side of \eq{qmassy} is 
$\mb{K}\circ\mb{\pi}^0\circ \mb{\Psi}_{\mb{\phi}^0} = \mb{\pi}\circ \mb{\d}_{\!{\bell}}\circ \mb{\Psi}_{\mb{\phi}^0}$.
Therefore, \eq{qmassy} is equivalent to the following:
\[
\mb{\pi}^0 \circ\Big( \mb{\d}_{\!{\bell}}\circ\mb{\Psi}_{\mb{\phi}^0}
-\mb{\Psi}_{\mb{\phi}^0}\circ\mb{\d}_{\!\grave{\bell}}\Big)=0
.
\]
From $\mb{\pi}^0_1=\I_{\sC}$, we have
$\proj_{\sC{\kkbar}}\circ
\Big( \mb{\d}_{\!{\bell}}\circ\mb{\Psi}_{\mb{\phi}^0}
-\mb{\Psi}_{\mb{\phi}^0}\circ \mb{\d}_{\!\grave{\bell}}\Big)=0
$.
In components, we have, for all $n\geq 1$,
\begin{align*}
0= &
\proj_{\sC{\kkbar}}\circ\Big( \mb{\d}_{\!{\bell}}\circ\mb{\Psi}_{\mb{\phi}^0}
-\mb{\Psi}_{\mb{\phi}^0}\circ \mb{\d}_{\!\grave{\bell}}\Big)(\bm{v}_1\bm{\odot}\ldots\bm{\odot} \bm{v}_n)
\\
\equiv &
(-\kbar)^{n-1}\sum_{\clapsubstack{\mp \in P(n)}}
\e(\mp)\,
\bell\!\left(\mb{\phi}^0( \bm{v}_{B_1})\bm{\odot}\ldots\bm{\odot}\mb{\phi}^0(\bm{v}_{B_{|\mp|}})\right)
\\
&
-(-\kbar)^{n-1}\sum_{\clapsubstack{\mp \in P(n)\\ |B_i| = n-|\mp|+1}}\e(\mp)\,
\mb{\phi}^0\!\left( J{v}_{B_1}\bm{\odot}\ldots\bm{\odot} 
\grave{\bell}(\bm{v}_{B_i})\bm{\odot} \ldots\bm{\odot}  \bm{v}_{B_{|\mp|}}\right)
\end{align*}
Therefore, 
$\underline{{\mb{\phi}}^0}: \big(\sC{\kkbar}, \underline{{\bell}}\big)\rightarrow 
\big(H{\kkbar}, \underline{{\bell}}^H\big)$ is also an $sL_\infty$-morphism. Combined with
property~\eqref{item: unitality of structure}, we conclude that
$\xymatrix{\underline{\mb{\phi}^0}: \Big(H{\kkbar}, 1_H,\underline{\grave{\bell}}\Big)
\ar@{..>}[r] & \Big(\sC{\kkbar}, 1_\sC,\underline{{\bell}}\Big)}$
is a unital $sL_\infty$-morphism. 

Finally, we recall that $\bell_1=\mb{K}$, $\grave{\bell}_1=\mb{\kappa}$, and
$\mb{\phi}_1=\mb{f}$. Therefore $\underline{\mb{\phi}}^0$
is a unital $sL_\infty$-quasi-isomorphism since 
$\mb{f}:\big(H{\kkbar}, 1_H, \mb{\kappa}\big)\rightarrow \big(\sC{\kkbar}, 1_\sC, \mb{K}\big)$ 
is a pointed cochain quasi-isomorphism.
\naturalqed
\end{proof}

Now we explain a strategy to solve the level zero quantum master equation.

For $n=1$, the initial conditions solve the quantum master equation:
\[
\grave{\mb{\pi}}^0_1=\I_{H{\kkbar}}
,\qquad
\mb{\eta}^{-1}_1=0
,\qquad
\grave{\bell}_1=\mb{\kappa}
,\qquad
\mb{\phi}^{-1}_1=\mb{f}.
\]
It follows that $\mb{K}_\mathit{H\sC} \mb{\phi}^0_1=0$. 
Consider the quantum master equation for $n =2$:
\begin{align*}
\mb{f}\big(\grave{\mb{\pi}}^0_2(v_1,v_2)\big)
= 
&\mb{\phi}^0_1(v_1)\cdot \mb{\phi}^0_1(v_2) +(-\kbar)\mb{\phi}^0_2(v_1,v_2)
\\
&
-\mb{K}\mb{\eta}^{-1}_2(v_1,v_2) 
-\mb{\eta}^{-1}_2\big(\grave{\bell}_1(v_1),v_2\big)
-\mb{\eta}^{-1}_2\big(Jv_1,\grave{\bell}_1(v_2)\big)
,\\
\mb{\kappa}\grave{\mb{\pi}}^0_2(v_1,v_2) = & \grave{\mb{\pi}}^0_2\big(\grave{\bell}_1(v_1),v_2\big)
+\grave{\mb{\pi}}^0_2\big(Jv_1,\grave{\bell}_1(v_2)\big) +(-\kbar)\grave{\bell}_2(v_1,v_2).
\end{align*}
which can be written in the following form:
\begin{align*}
\mb{K}_\mathit{H\sC} \mb{\eta}^{-1}_{2}
+\mb{f}\circ \grave{\mb{\pi}}^0_{2}
-\mb{\O}^0_{2}
=
&
(-\kbar)\mb{\phi}^0_{2}
,\\
\mb{\kappa}_\mathit{HH}\grave{\mb{\pi}}^0_{2}
=&(-\kbar)\grave{\bell}_{2} 
,\\
\end{align*}
where $\mb{\O}^0_{2}\in \Hom\big(S^2H,\sC\big)^{0}{\kkbar}$ is defined by
\[
\mb{\O}^0_2(v_1,v_2) \coloneqq{}\mb{\phi}^0_1(v_1)\cdot \mb{\phi}^0_1(v_2)
.
\]
Then all we need  is to find
$\grave{\mb{\pi}}^0_2 \in \Hom\big(S^2H,H\big)^{0}{\kkbar}$
and $\mb{\eta}^{-1}_2 \in \Hom\big(S^2H,\sC\big)^{-1}{\kkbar}$ 
so that 
\[
\mb{\phi}^0_2\coloneqq{}\Fr{1}{(-\kbar)}\left(\mb{K}_\mathit{H\sC} \mb{\eta}^{-1}_{2}
+\mb{f}\circ \grave{\mb{\pi}}^0_{2}
-\mb{\O}^0_{2}\right) \in  \Hom\big(S^2H,\sC\big)^{0}{\kkbar}.
\]
Then we will also have 
$\grave{\bell}_{2}\coloneqq{}\Fr{1}{(-\kbar)}\mb{\kappa}_\mathit{HH}\grave{\mb{\pi}}^0_{2} \in \Hom\big(S^2H,\sC\big)^{1}{\kkbar}$
since $\mb{\kappa}_\mathit{HH}$ is divisible by $\kbar$. It is straightforward to check that
\[
\mb{K}_\mathit{H\sC}\mb{\O}^0_2(v_1,v_2) =(-\kbar)\bell_2\big(\mb{\phi}_1(v_1), \mb{\phi}_1(v_2)\big),
\]
so that the classical limit $\O^0_2 \in \Hom\big(S^2H,\sC\big)^0$ of $\mb{\O}^0_2$ satisfies $K\circ \O^0_2=0$. 
It follows that 
\[
\nabla_{(-\kbar)^{-1}}\mb{\O}^0_2 
=\Fr{1}{(-\kbar)}\Big(\mb{\O}^0_2 
-\mb{f}\circ h\circ \O^0_2 
- \mb{K}_\mathit{H\sC} (s\circ \O^0_2)\Big)
\in \Hom\big(S^2H,\sC\big)^{0}{\kkbar},
\]
where $\nabla_{(-\kbar)^{-1}}$ is the operator defined in \eq{hdbh}.
Therefore, the following is a solution to the level zero quantum master equation for $n=2$:
\[
\grave{\mb{\pi}}^0_2\coloneqq{}h\circ \O^0_2
,\qquad
\mb{\eta}^{-1}_2\coloneqq{}s\circ \O^0_2
,\qquad
\mb{\phi}^0_2\coloneqq{}-\mb{\nabla}_{(-\kbar)^{-1}}\mb{\O}^0_2
,\qquad
\grave{\bell}_{2}\coloneqq{}\Fr{1}{(-\kbar)}\mb{\kappa}_\mathit{HH}\grave{\mb{\pi}}^0_{2}.
\]

In general,  for $n\geq 2$,  the level zero quantum master equation can be written in the following form:
\begin{align*}
\mb{K}_\mathit{H\sC} \mb{\eta}^{-1}_{n}+\mb{f}\circ \grave{\mb{\pi}}^0_{n}
- \mb{\O}^0_{n}
=
&
(-\kbar)^{{n}-1}\mb{\phi}^0_{n}
,\\
\mb{\kappa}_\mathit{HH}\grave{\mb{\pi}}^0_{n} +\grave{\mb{\varpi}}^1_n
=&(-\kbar)^{{n}-1}\grave{\bell}_{n} 
,\\
\end{align*}
where $\mb{\O}^0_n \in \Hom\big(S^nH,\sC\big)^{0}{\kkbar}$ and $\grave{\mb{\varpi}}^1_n \in \Hom\big(S^nH,H\big)^{1}{\kkbar}$
are defined for homogeneous $v_1,\ldots, v_n \in H$ as
\begin{align*}
\mb{\O}^0_n({v}_1, \cdots, {v}_n)\coloneqq{}
&\mathbin{\hphantom{-}}
\sum_{\clapsubstack{\mp\in P(n)\\\color{red}|\mp|\neq 1}}(-\kbar)^{n-|\mp|}
\e(\mp)\mb{\phi}^0({v}_{B_1})\cdot\ldots\cdot
\mb{\phi}^0({v}_{B_{|\mp|}})
\\
&
-\sum_{\clapsubstack{\mp \in P(n)\\ |B_i| = n-|\mp|+1\\\color{red}|\mp|\neq n, 1}}
(-\kbar)^{n-|\mp|}
\e(\mp)
\mb{\eta}^{-1}_{|\mp|} \!\left( J\!{v}_{B_1},\dotsc, J\!{v}_{B_{i-1}},\grave{\bell}({v}_{B_i}), {v}_{B_{i+1}},\dotsc, {v}_{B_{|\mp|}}\right)
,\\
\grave{\mb{\varpi}}^1_n({v}_1, \cdots, {v}_n)\coloneqq{}&
-\sum_{\clapsubstack{\mp \in P(n)\\ |B_i| = n-|\mp|+1\\ {\color{red}|\mp|\neq 1, n}}}
(-\kbar)^{n-|\mp|}\e(\mp)
\grave{\mb{\pi}}^0_{|\mp|}\!\left( J\!{v}_{B_1}, \dotsc,J\!{v}_{B_{i-1}}, \grave{\bell}({v}_{B_i}), {v}_{B_{i+1}},\dotsc, {v}_{B_{|\pi|}}\right)
.
\end{align*}
Note that $\grave{\mb{\varpi}}^1_2=0$. Note also that $\mb{\O}^0_n$ 
and $\grave{\mb{\varpi}}^1_n$ depend only on 
$\left\{\mb{\pi}^0_k, \mb{\eta}^{-1}_k, \grave{\bell}_{\!k},\mb{\phi}^0_k\right\}$ for ${k < n}$.

Therefore, all we need is to find appropriate $\grave{\mb{\pi}}^0_n\in \Hom\big(S^nH,H\big)^{0}{\kkbar}$ 
and $\mb{\eta}^{-1}_n \in \Hom\big(S^nH,\sC\big)^{-1}{\kkbar}$ such that
\begin{align*}
\Fr{1}{(-\kbar)^{{n}-1}}\Big(\mb{K}_\mathit{H\sC} \mb{\eta}^{-1}_{n}
+\mb{f}\circ \grave{\mb{\pi}}^0_{n}
-\mb{\O}^0_{n}\Big) \in \Hom\big(S^nH,\sC\big)^{0}{\kkbar}
\text{ and}\\
\Fr{1}{(-\kbar)^{{n}-1}}\Big(\mb{\kappa}_\mathit{HH}\grave{\mb{\pi}}^0_{n} +\grave{\mb{\varpi}}^1_n\Big)
\in \Hom\big(S^nH,\sC\big)^{1}{\kkbar}.
\end{align*}

\begin{theorem}\label{solvemstzero}
There is a canonical solution 
$\left\{\grave{\mb{\pi}}^0, \mb{\eta}^{-1},\grave{\bell},\mb{\phi}^0\right\}$
to the level zero quantum master equation of Definition \ref{masterzero}
with
\begin{equation}
\label{eq: redundant initial conditions level zero}
\grave{\mb{\pi}}_1=\I_{H}
,\qquad
\mb{\eta}^{-1}_1=0
,\qquad
\mb{\phi}^0_1=\mb{f}
,\qquad
\grave{\bell}_1=\mb{\kappa}
,
\end{equation}
and for all $n\geq 2$
\begin{align*}
\grave{\mb{\pi}}^0_n 
&= 
\sum_{\mathclap{i=0}}^{n-2}(-\kbar)^i h\circ \big({\Omega}^{0}_n\big)^{[i]}
,
&
\mb{\phi}^0_n,
&
=-\left(\mb{\nabla}_{(-\kbar)^{-1}}\right)^{j}\mb{\O}^0_n,
\\
\mb{\eta}^{-1}_n &=\sum_{\mathclap{i=0}}^{n-2}(-\kbar)^i s\circ \big({\Omega}^{0}_n\big)^{[i]}
,
&
\grave{\bell}_n &=\Fr{1}{(-\kbar)^{n-1}}\left(
\grave{\mb{\varpi}}^1_n+\mb{\kappa}_\mathit{HH}\grave{\mb{\pi}}^0_n\right),
\end{align*}
where $ \big({\O}^{0}_n\big)^{[i]} \in \Hom\big(S^nH,\sC\big)$ is the classical limit of 
$\big(\mb{\O}^{0}_n\big)^{[i]} \coloneqq{}\left(\mb{\nabla}_{\kbar^{-1}}\right)^{i}\mb{\O}^0_n$, for $0\leq i\leq n-2$. 

For 
$v_1,\ldots, v_{n-1} \in H$, this solution satisfies
\begin{align*}
\grave{\mb{\pi}}^0_n\big(v_1,\ldots, v_{n-1}, 1_H\big)&=\grave{\mb{\pi}}^0_{n-1}\big(v_1,\ldots, v_{n-1}\big)
,\\
\mb{\eta}^{-1}_n\big(v_1,\ldots, v_{n-1}, 1_H\big)&=\mb{\eta}^{-1}_n\big(v_1,\ldots, v_{n-1}\big),
\end{align*}
\end{theorem}
\begin{remark}
The conditions of eq.~\eqref{eq: redundant initial conditions level zero} are mostly redundant; all except the first of them are explicitly required by Definition~\ref{masterzero}.
\end{remark}

\begin{proof}
Assume that 
\[
\big\{\grave{\mb{\pi}}^0_1,\ldots, \grave{\mb{\pi}}^0_{n-1}\big\}
,\quad\big\{0,\mb{\eta}^{-1}_2,\ldots, \mb{\eta}^{-1}_{n-1}\big\}
,\quad\big\{\grave{\bell}_1,\ldots, \grave{\bell}_{n-1}\big\}
,\quad\big\{\mb{\phi}^0_1,\ldots, \mb{\phi}^0_{n-1}\big\},
\]
is such a canonical solution to the master equation up to order $n-1$. 
It can be checked by an easy induction that,
for all $v_1,\ldots, v_{n-1} \in H$,
\eqn{anois}{
\mb{\O}^0_{n}(v_1,\ldots, v_{n-1}, 1_H) =
\mb{\O}^0_{n-1}(v_1,\ldots, v_{n-1}).
}
It can be also checked by a tedious but straightforward computation,
similar to that in the proof of Proposition \ref{zerodes},
that
\eqn{etaemmxa}{
\mb{K}_\mathit{H\sC} \mb{\O}^0_{n}  = (-\kbar)^{n-1}\mb{L}_n  -\mb{f}\circ\grave{\mb{\varpi}}^1_n,
}
where $\mb{L}_n \in \Hom\big(S^nH,\sC\big)^{1}{\kkbar}$ defined such that,
for homogeneous $v_1,\ldots, v_n$,
\begin{align*}
\mb{L}_n({v}_1, \cdots, {v}_n)
\coloneqq{}&
\sum_{\clapsubstack{\mp \in P(n)\\{\color{red}|\mp|\neq 1}}}\ep(\mp)
\bell_{|\mp|}\!\left(\mb{\phi}^0\left( {v}_{B_1}\right), \cdots, \mb{\phi}^0\big(  {v}_{B_{|\mp|}}\big)\right)
\\
&
-\sum_{\clapsubstack{\mp \in P(n)\\ |B_i| = n-|\mp|+1\\ \color{red}|\mp|\neq 1, n}}
\e(\mp)
\mb{\phi}^0_{|\mp|}\!\left( J{v}_{B_1}, \cdots, J{v}_{B_{i-1}}, \grave{\bell}({v}_{B_i}),  {v}_{B_{i+1}},\dotsc,  {v}_{B_{|\mp|}}\big)\right)
.
\end{align*}
Therefore, we can apply Lemma \ref{hodge} to obtain that
\eqn{zzxxh}{
\begin{aligned}
(-\kbar)^{n-1}\mb{\Omega}^{0[n-1]}_n
=&
\mb{\Omega}_n
-\mb{f}\circ\left(\sum_{\mathclap{i=0}}^{n-2}(-\kbar)^i h\circ \big({\Omega}^{0}_n\big)^{[i]}\right)
-\mb{K}_{H\sC}\left(\sum_{\mathclap{i=0}}^{n-2}(-\kbar)^i s\circ  \big({\Omega}^{0}_n\big)^{[i]}\right)
,\\
(-\kbar)^{n-1} \big(\grave{\mb{\varpi}}^{1}_n\big)^{[n-1]}
=&\grave{\mb{\varpi}}^1_n +\sum_{\mathclap{i=0}}^{n-2} (-\kbar)^i\mb{\kappa}_{\!H\!H}\left( h\circ  \big({\Omega}^{0}_n\big)^{[i]}\right)
,\\
\mb{K}_{\!H\!\sC}\big(\mb{\Omega}^{0}_n\big)^{[n-1]}
=&\mb{L}_n - \mb{f}\circ \big(\grave{\mb{\varpi}}^{1}_n\big)^{[n-1]}
,
\end{aligned}
}
where  
$\big(\mb{\O}^{0}_n\big)^{[j]}\coloneqq{}\left(\mb{\nabla}_{(-\kbar)^{-1}}\right)^{j}\mb{\O}^0_n \in \Hom\big(S^nH,\sC\big)^{0}{\kkbar}$ 
and $\big({\O}^{0}_n\big)^{[j]}\in \Hom\big(S^nH,\sC\big)^0$ denotes the
classical limit of $\big(\mb{\O}^{0}_n\big)^{[j]}$.
Setting 
\begin{equation}
\label{xsolvawwa}
\begin{aligned}
\grave{\mb{\pi}}^0_n &\coloneqq{} \sum_{\mathclap{j=0}}^{n-2}(-\kbar)^j h\circ \big({\Omega}^{0}_n \big)^{[j]}
&\text{in }&\Hom\big(S^nH, \sC\big)^{0}{\kkbar}
,\\
\mb{\eta}^{-1}_n &\coloneqq{}\sum_{\mathclap{j=0}}^{n-2}(-\kbar)^j s\circ \big({\Omega}^0_n\big)^{[j]} &\text{in }& \Hom\big(S^nH, \sC\big)^{-1}{\kkbar}
,\\
\mb{\phi}^0_n&\coloneqq{}-\big(\mb{\O}^{0}_n\big)^{[n-1]}  &\text{in }& \Hom\big(S^nH,\sC\big)^{0}{\kkbar}
\text{, and}\\
\grave{\bell}_n &\coloneqq{} \big(\grave{\mb{\varpi}}^{0}_n \big)^{[n-1]}&\text{in }& \Hom\big(S^nH,\sC\big)^{1}{\kkbar}
,
\end{aligned}
\end{equation}
we have
\eqn{xsolvawwb}{
\begin{aligned}
\mb{K}_\mathit{H\sC} \mb{\eta}^{-1}_{n}
+\mb{f}\circ \grave{\mb{\pi}}^0_{n}
-\mb{\O}^0_{n}
=
&
(-\kbar)^{{n}-1}\mb{\phi}^0_{n}
,\\
\mb{\kappa}_\mathit{HH}\grave{\mb{\pi}}^0_{n} +\grave{\mb{\omega}}^0_n
=&(-\kbar)^{{n}-1}\grave{\bell}_{n} 
,\\
\mb{K}_{\!H\!\sC}\mb{\phi}^0_n +\mb{L}_n =& \mb{f}\circ \grave{\bell}_n
.
\end{aligned}
}
and we are done.
From \eq{anois}, we also obtain that, for all
$v_1,\ldots, v_{n-1} \in H$,
\eqnalign{xsolvawwc}{
\grave{\mb{\pi}}^0_n\big(v_1,\ldots, v_{n-1}, 1_H\big)&=\grave{\mb{\pi}}^0_{n-1}\big(v_1,\ldots, v_{n-1}\big)
,\\
\mb{\eta}^{-1}_n\big(v_1,\ldots, v_{n-1}, 1_H\big)&=\mb{\eta}^{-1}_n\big(v_1,\ldots, v_{n-1}\big).
}
From  the relations in \eq{xsolvawwa}, \eq{xsolvawwb} and \eq{xsolvawwc} together with the assumption,
we have shown that 
\[
\big\{\grave{\mb{\pi}}^0_1,\ldots, \grave{\mb{\pi}}^0_{n}\big\}
,\quad\big\{0,\mb{\eta}^{-1}_2,\ldots, \mb{\eta}^{-1}_{n}\big\}
,\quad\big\{\grave{\bell}_1,\ldots, \grave{\bell}_{n}\big\}
,\quad\big\{\mb{\phi}^0_1,\ldots, \mb{\phi}^0_{n}\big\},
\]
constitute a canonical solution to the master equation up to order $n$.
Therefore, we are done by induction.
\naturalqed
\end{proof}

From now on 
$\left\{\grave{\mb{\pi}}^0, \mb{\eta}^{-1},\grave{\bell},\mb{\phi}^0\right\}$
will denote the canonical solution to the master equation for the level zero quantum correlators.

\subsection{Homotopy \texorpdfstring{$\kbar$}{h-bar}-divisibility II}
\label{subsec: homotopy h-divisibility 2}

Consider the unital $sL_\infty$-algebra $\big(H{\kkbar}, 1_H, \grave{\bell}\big)$. 
It is convenient to introduce
a $\Bbbk{\kkbar}$-linear operator
$\bm{\md}_{\grave{\bell}}: \big(S(H)\otimes S^j H\big){\kkbar}\rightarrow \big(S(H)\otimes S^j H\big){\kkbar}$
defined as follows:  
\begin{align*}
\bm{\md}_{\!\grave{\bell}} &(v_1\odot\ldots\odot v_n  \otimes w_{1}\odot\ldots\odot w_{j}\big)
=\bm{\d}_{\!\grave{\bell}}(v_1\odot\ldots\odot v_n)\otimes w_{1}\odot\ldots\odot w_{j}
\\
&
+\sum_{\mathclap{i=1}}^j\ \ \sum_{\mathclap{\vs \subset [n]}}
%(-1)^{|w_i|(|w_1|+\ldots+|w_{i-1}|)}
(-\kbar)^{n -|\vs|-1}
\e(\vs)\e(i,\underline{w}) v_{\vs}\otimes  \grave{\bell}\big(v_{\vs^c} \odot w_i\big)\odot  w_{1}\odot\ldots\odot \widehat{w_{i}}\odot\ldots\odot w_{j},
\end{align*}
where $\e(i,\underline{w})$ is the sign $(-1)^{|w_i|(|w_1|+\ldots+|w_{i-1}|)}$ and $\mb{\d}_{\!\grave{\bell}}: S(H){\kkbar}\rightarrow S(H){\kkbar}$ is as in Notation~\ref{notation: hbar extensions} and satisfies $\bm{\d}_{\grave{\bell}}\circ \bm{\d}_{\grave{\bell}}=0$.
Then it is also straightforward to show that $\bm{\md}_{\grave{\bell}}\circ \bm{\md}_{\grave{\bell}}=0$.

\begin{remark}
Let ${\mb{\r}}\in \Hom\big(S(H)\otimes S^j H, V\big){\kkbar}$
for a $\Z$-graded vector space $W$. Then, we have, for $n\geq j$,
\[
\big({\mb{\r}}\circ \bm{\md}_{\!\grave{\bell}}\big)(v_1,\ldots, v_n)
=\sum_{\clapsubstack{\mp \in P(n)\\|B_{|\mp|}|=n-|\mp|+1\\ n-j+1\nsim_\mp n-j+2\nsim_\mp \ldots\nsim_\mp n}}
(-\kbar)^{n-|\mp|}
\e(\mp)\; 
\grave{\mb{\r}}_{|\mp|}\left( J{v}_{B_1},\dotsc, J{v}_{B_{i-1}},\grave{\bell}\!
\left({v}_{B_i}\right), {v}_{B_{i+1}},\dotsc, {v}_{B_{|\mp|}}\right),
\]
where ${\mb{\r}}_n(v_1,\ldots, v_n)= {\mb{\r}}(v_1\odot \ldots\odot v_{n-j}\otimes v_{n-j+1}\odot\ldots\odot v_n)$.
\naturalqed
\end{remark}

Define  $\Bbbk{\kkbar}$-linear operators
\begin{align*}
\mb{K}^\infty_\mathit{H\sC}:\;&\Hom\big(S(H)\otimes S^j H, \sC\big){\kkbar} 
\rightarrow \Hom\big(S(H)\otimes S^j H, \sC\big){\kkbar}
,\\
\mb{\kappa}^\infty_\mathit{HH}:\;&\Hom\big(S(H)\otimes S^j H, H\big){\kkbar}
\rightarrow \Hom\big(S(H)\otimes S^j H, H\big){\kkbar}
,
\end{align*}
for all $n\geq 1$ and $v_1,\ldots,v_n \in H$ via the equations
\begin{align*}
\mb{K}^\infty_\mathit{H\sC}\mb{\Xi}\coloneqq{}
&\mb{K}\circ\mb{\Xi}
-(-1)^{|\mb{\Xi}|}\mb{\Xi}\circ \bm{\md}_{\grave{\bell}}
,\\
\mb{\kappa}^\infty_\mathit{HH} \mb{\xi}
\coloneqq{}
&\mb{\kappa}\circ \mb{\xi}
-(-1)^{|\mb{\xi}|}
\mb{\omega}\circ \bm{\md}_{\grave{\bell}}
,
\end{align*}
where $\mb{\Xi}\in \Hom\big(S(H)\otimes S^j H, \sC\big){\kkbar}$ and $\mb{\xi}\in\Hom\big(S(H)\otimes S^j H, H\big){\kkbar}$.
It is trivial that 
$\mb{K}^\infty_\mathit{H\sC}\circ \mb{K}^\infty_\mathit{H\sC}
=\mb{\kappa}^\infty_\mathit{HH}\circ \mb{\kappa}^\infty_\mathit{HH}=0$. 
Note that $\mb{f}\circ \mb{\xi}\in \Hom\big(S(H)\otimes S^k H, \sC\big){\kkbar}$
and $\mb{K}^\infty_\mathit{H\sC}\big(\mb{f}\circ \mb{\xi}\big) 
= \mb{f}\circ\big(\mb{\kappa}^\infty_\mathit{HH} \mb{\omega}\big)$ 
since
$\mb{K}\circ \mb{f} =\mb{f}\circ \mb{\kappa}$.

We define a $\Bbbk$-linear operator $\nabla^\infty_{(-\kbar)^{-1}}$ on $\Hom\big(S(H)\otimes S^j H, \sC\big){\kkbar}$
as follows: given $\mb{\Xi}\in \Hom\big(S(H)\otimes S^j H, \sC\big){\kkbar}$,
\eqn{hhds}{
(-\kbar)\nabla^\infty_{(-\kbar)^{-1}}\mb{\Xi} \coloneqq{}
\mb{\Xi} 
-\mb{f}\circ h\circ \Xi 
- \mb{K}^\infty_\mathit{H\sC} (s\circ \Xi) 
- s\circ K\circ \Xi,
}
where 
$\Xi\in \Hom\big(S(H)\otimes S^j H, \sC\big)$ is the classical limit of $\mb{\Xi}$.
The classical limit of the right hand side of Eq.~\eqref{hhds} is 
${\Xi} -{f}\circ h\circ \Xi - K\circ s\circ \Xi - s\circ K\circ \Xi=0$
so it is divisible by $\kbar$. Therefore, we have
$\nabla^\infty_{\kbar^{-1}}\mb{\Xi}\in  \Hom\big(S(H)\otimes S^j H, \sC\big){\kkbar}$.
Assume that the classical limit of $\mb{\Xi}$ vanishes.  Then 
we have $\nabla_{(-\kbar)^{-1}}\mb{\Xi}=\Fr{1}{(-\kbar)}\mb{\Xi}$.
The following is direct.
\begin{lemma}\label{hodgea}  
For a triple  $(\mb{\Xi}, \mb{M}, \mb{\xi})$ satisfying, for $k\geq 1$;
\[
\mb{K}^\infty_{\!H\!\sC} \mb{\Xi}= (-\kbar)^{k} \mb{M} - \mb{f}\circ \mb{\xi},
\text{ where }
\begin{cases}
\mb{\Xi},\mb{M}\in \Hom\big(S(H)\otimes S^j H, \sC\big){\kkbar},\\ 
\mb{\xi} \in \Hom\big(S(H)\otimes S^j H, H\big){\kkbar},
\end{cases}
\]
we have
\eqn{xsolvafw}{
\mb{K}^\infty_{\!H\!\sC}\mb{\Xi}^{[k]}=\mb{M} - \mb{f}\circ \mb{\xi}^{[k]}
\quad\text{ and }\quad 
\begin{cases}
\mb{\Xi}^{[k]}\in \Hom\big(S(H)\otimes S^j H, \sC\big){\kkbar}
,\\
\mb{\xi}^{[k]} \in \Hom\big(S(H)\otimes S^j H, H\big){\kkbar}
,
\end{cases}
}
where
\eqn{ysolvaww}{
\begin{aligned}
(-\kbar)^{k}\mb{\Xi}^{[k]}
\coloneqq{}&
\mb{\Xi}
-\mb{f}\circ\left(\sum_{\mathclap{i=0}}^{k-1}(-\kbar)^i h\circ {\Xi}^{[i]}\right)
-\mb{K}^\infty_{H\sC}\left(\sum_{\mathclap{i=0}}^{k-1}(-\kbar)^i {s} \circ {\Xi}^{[i]} \right)
\text{ and}\\
(-\kbar)^k \mb{\xi}^{[k]}\coloneqq{}&\mb{\xi} +\sum_{\mathclap{i=0}}^{k-1} (-\kbar)^i\mb{\kappa}^\infty_{\!H\!H}\left( h\circ \Omega^{[i]}\right)
.
\end{aligned}
}
\end{lemma}

\subsection{The master equation for level one quantum correlators} 

%Recall that $\grave{\mb{\pi}}_1$ is the identity map on $H$ and  $\grave{\mb{\pi}}_n$, $n\geq 2$, has the following expansion:
%\[
%\grave{\mb{\pi}}_n =\grave{\pi}_n + (-\kbar)\grave{\pi}^{(1)}_n+\ldots+(-\kbar)^{n-2} \grave{\pi}^{n-2}_n, 
%\]
%where $\grave{\pi}^{(k)}_n \in \Hom\big(S^n H, H\big)^0$ for all $k=0,1,\ldots, n-2$. 

Let $\big( \grave{\mb{\pi}}^0, \mb{\eta}^{-1}, \grave{\bell}, \mb{\phi}^0\big)$ be the canonical solution 
to the level zero quantum master equation, which can be written in the following simpler form:
\eqn{mstzero}{
\mb{\Pi}^0= \mb{f}\circ  \grave{\mb{\pi}}^0 +\mb{K}^{\infty}_{\mathit{H\sC}}\mb{\eta}^{-1}
,\qquad 
\mb{\kappa}^{\infty}_{\mathit{HH}}\grave{\mb{\pi}}^0=0,
}
where $\mb{\Pi}^0 \in \Hom\big(\Xbar{S}(H), \sC\big)^{0}{\kkbar}$ is defined by the formula:
\eqn{zerocor}{
\mb{\Pi}^0(v_1\odot\ldots\odot v_n) \equiv
\mb{\Pi}^0_n(v_1,\ldots, v_n)=\sum_{\mathclap{\mp \in P(n)}} (-\kbar)^{n-|\mp|}\e(\mp)
\mb{\phi}^0\big(v_{B_1}\big)\cdot\ldots\cdot \mb{\phi}^0\big(v_{B_{|\mp|}}\big).
}
Note also that 
$\mb{K}^{\infty}_{\mathit{H\sC}}\mb{\Pi}^0=0$ 
since
$\mb{K}^{\infty}_{\mathit{H\sC}}\mb{\Pi}^0=\mb{K}^{\infty}_{\mathit{H\sC}}\big( \mb{f}\circ  \grave{\mb{\pi}}^0 \big)
= \mb{f}\circ \mb{\kappa}^{\infty}_{\mathit{HH}} \grave{\mb{\pi}}^0 =0$.

\begin{definition}\label{masterone}
The level one quantum master equation is 
a system of equations
for a tuple $\left\{\grave{\mb{\pi}}^{-1}, \mb{\eta}^{-2},\grave{\mb{m}}^0,\mb{\phi}^{-1}\right\}$,
where
\[
\begin{aligned}
\grave{\mb{\pi}}^{-1}&  \in \Hom\left({S}(H)\otimes S^2H, H\right)^{-1}{\kkbar}
,\\
\grave{\mb{m}}^0&\in \Hom\left({S}(H)\otimes S^2H, H\right)^{0}{\kkbar}
,
\end{aligned}
\qquad
\begin{aligned}
\mb{\eta}^{-2}  &\in \Hom\left({S}(H)\otimes S^2H,\sC\right)^{-2}{\kkbar}
,\\
\mb{\phi}^{-1} &\in \Hom\left({S}(H)\otimes S^2H,\sC\right)^{-1}{\kkbar}
,
\end{aligned}
\]
are defined recursively for $n\geq 2$ by the equations
\begin{align*}
\mb{f}\Big(\grave{\mb{\pi}}^{-1}_n(v_1,\ldots,v_n)\Big) 
=
&
 \mb{\eta}^{-1}_n\big(v_1,\ldots,v_n\big) -\mb{K}^\infty_{\mathit{H\sC}}\mb{\eta}^{-2}_n\big(v_1,\ldots,v_n\big)
\\
&
-\sum_{\clapsubstack{\mp \in P(n)\\  n-1\sim_\mp n }}
(-\kbar)^{n-|\mp|-1}
\e(\mp)
\mb{\phi}^0\big(J{v}_{B_1}\big)\cdots\mb{\phi}^0\big(J{v}_{B_{|\mp|-1}}\big)\cdot\mb{\phi}^{-1}({v}_{B_{|\mp|}})
\\
&
-\sum_{\clapsubstack{\mp \in P(n)\\|B_{|\mp|}|=n-|\mp|+1\\ n-1\sim_\mp n\\ |\mp|\neq 1}} 
(-\kbar)^{n-|\mp|-1}
\e(\mp)\; 
\mb{\eta}^{-1}_{|\mp|}\left( {v}_{B_1}, \cdots, {v}_{B_{\mp-1}},
\grave{\mb{m}}^0\big({v}_{B_{|\mp|}}\big)\right)
%\\
%&
%+\sum_{\clapsubstack{\mp \in P(n)\\|B_{|\mp|}|=n-|\mp|+1\\ n-1\nsim_\mp n}}
%(-\kbar)^{n-|\mp|}
%\e(\mp)\; 
%\mb{\eta}^{-2}_{|\mp|}
%\left( J{v}_{B_1},\dotsc, J{v}_{B_{i-1}},\grave{\bell}\!\!
%\left({v}_{B_i}\right), {v}_{B_{i+1}},\dotsc, {v}_{B_{|\mp|}}\right)
,\\
\grave{\mb{\pi}}^0_{n} ({v}_1,\dotsc,{v}_n)
=
&
\sum_{\clapsubstack{\mp \in P(n)\\|B_{|\mp|}|=n-|\mp|+1\\ n-1\sim_\mp n  }}
(-\kbar)^{n-|\mp|-1}
\e(\mp)\; 
\grave{\mb{\pi}}^0_{|\mp|}\left({v}_{B_1}, \cdots, {v}_{B_{\mp-1}},
\grave{\mb{m}}^0\big({v}_{B_{|\mp|}}\big)\right)
\\
&
-\mb{\kappa}^\infty_{\mathit{HH}}\grave{\mb{\pi}}^{-1}_{n}({v}_1,\dotsc,{v}_n)
%\\
%&
%-\sum_{\clapsubstack{\mp \in P(n)\\|B_{|\mp|}|=n-|\mp|+1\\ n-1\nsim_\mp n}}
%(-\kbar)^{n-|\mp|}
%\e(\mp)\; 
%\grave{\mb{\pi}}^{-1}_{|\mp|}\left( J{v}_{B_1},\dotsc, J{v}_{B_{i-1}},\grave{\bell}\!\!
%\left({v}_{B_i}\right), {v}_{B_{i+1}},\dotsc, {v}_{B_{|\mp|}}\right)
,
\end{align*}
with the initial conditions
\[
\grave{\mb{\pi}}^{-1}_{2}=0
,\qquad
\mb{\eta}^{-2}_2=0
,\qquad
\grave{\mb{m}}^0_{2}=\grave{\mb{\pi}}^0_2
,\qquad
\mb{\phi}^{-1}_2=\mb{\eta}^{-1}_2
.
\]
\end{definition}

\begin{remark}
The leading, $n=2$, level one quantum master equation is
\begin{align*}
\mb{f}\Big(\grave{\mb{\pi}}^{-1}_2(v_1,v_2)\Big) =
& \mb{\eta}^{-1}_2(v_1,v_2) - \mb{\phi}^{-1}_2(v_1,v_2) - \mb{K}^\infty_{\mathit{H\sC}} \mb{\eta}^{-1}_2(v_1,v_2),
\\
\grave{\mb{\pi}}^0_2(v_1,v_2)
=
& 
\grave{\mb{m}}^0_2(v_1,v_2) - \mb{\kappa}^\infty_{\mathit{HH}}\mb{\eta}^{-2}_2(v_1,v_2),
\end{align*}
which is solved by the initial conditions.
\naturalqed

\end{remark}

\begin{remark}
Define
%It is convenient to express the level one quantum master equation in the following form:  for all $n\geq 2$, let
\begin{equation}
\begin{aligned}
\mb{\Omega}^{-1}_n (v_1,\ldots, &v_n)
\coloneqq{}
\mb{\eta}^{-1}_n(v_1,\ldots,v_n)
\\
&
-\sum_{\clapsubstack{\mp \in P(n)\\|B_{|\mp|}|=n-|\mp|+1\\ n-1\sim_\mp n \\ |\mp|\neq 1}} 
(-\kbar)^{n-|\mp|-1}
\e(\mp)\; 
\mb{\eta}^{-1}_{|\mp|}\left( {v}_{B_1}, \cdots, {v}_{B_{\mp-1}},
\grave{\mb{m}}^0\big({v}_{B_{|\mp|}}\big)\right)
\\
&
-\sum_{\clapsubstack{\mp \in P(n)\\  n-1\sim_\mp n \\\color{red} |\mp|\neq 1}}
(-\kbar)^{n-|\mp|-1}
\e(\mp)
\mb{\phi}^0\big(J{v}_{B_1}\big)\cdots\mb{\phi}^0\big(J{v}_{B_{|\mp|-1}}\big)\cdot\mb{\phi}^{-1}({v}_{B_{|\mp|}})
,
\end{aligned}
\end{equation}
and
\begin{equation}
\begin{aligned}
\label{mstoneb}
\grave{\mb{\varpi}}^0_{n}({v}_1,\dotsc, &{v}_n)
\coloneqq{}
\grave{\mb{\pi}}^0_{n}({v}_1,\dotsc,{v}_n)
\\
&-
\sum_{\clapsubstack{\mp \in P(n)\\|B_{|\mp|}|=n-|\mp|+1\\ n-1\sim_\mp n\\ \color{red}|\mp|\neq 1}}
(-\kbar)^{n-|\mp|-1}
\e(\mp)\; 
\grave{\mb{\pi}}^0_{|\mp|}\left({v}_{B_1}, \cdots, {v}_{B_{\mp-1}},
\grave{\mb{m}}^0\big({v}_{B_{|\mp|}}\big)\right)
.
\end{aligned}
\end{equation}
Note that  
$\mb{\O}^{-1}_{n}$ is in $\Hom\big(S^{n-2} H\otimes S^2H, \sC\big)^{-1}{\kkbar}$,
that $\grave{\mb{\varpi}}^0_n$ is in $\Hom\big(S^{n-2}H\otimes S^2H,H\big)^{0}{\kkbar}$, and both depend only on the families
\[
\left\{\grave{\mb{\pi}}^{-1}_2,\ldots,\grave{\mb{\pi}}^{-1}_{n-1}\right\}
,\quad
\left\{\mb{\eta}^{-2}_2,\ldots,\mb{\eta}^{-2}_{n-1}\right\}
,\quad
\left\{\mb{\phi}^{-1}_2,\ldots,\mb{\phi}^{-1}_{n-1}\right\}
,\quad
\left\{\grave{\mb{m}}^0_2,\ldots,\grave{\mb{m}}^0_{n-1}\right\},
\]
and the canonical solution to the level zero quantum master equation. Then the level one quantum master equation can be redefined as
\eqnalign{mstone}{
(-\kbar)^{n-2} \mb{\phi}^{-1}_n
=
& 
\mb{\Omega}^{-1}_n
-\mb{f}\circ\grave{\mb{\pi}}^{-1}_{n} 
-\mb{K}^\infty_{\mathit{H\sC}}\mb{\eta}^{-2}_{n}
,\\
(-\kbar)^{n-2}\grave{\mb{m}}^0_n
=
&
\grave{\mb{\varpi}}^0_n 
+\mb{\kappa}^\infty_{\mathit{H\sC}}\grave{\mb{\pi}}^{-1}_n
,
}
Thus the key to solving these equations is to show that
there is a pair $\big(\grave{\mb{\pi}}^{-1}_{n},\mb{\eta}^{-2}_n\big)$ so that the expressions on the right hand side of Eq.~\eqref{mstone} 
are divisible by $\kbar^{n-2}$.
% \[
% \Fr{1}{(-\kbar)^{n-2}} \left(\mb{\Omega}^{-1}_n
% - \mb{f}\circ\grave{\mb{\pi}}^{-1}_{n} 
% -\mb{K}^\infty_{\mathit{H\sC}}\mb{\eta}^{-2}_{n}
% \right)\text{ and } 
% \Fr{1}{(-\kbar)^{n-2}} \left(\grave{\mb{\varpi}}^0_n 
% +\mb{\kappa}^\infty_{\mathit{H\sC}}\grave{\mb{\pi}}^{-1}_n\right)
% \]
% are in $\Hom\big(S^{n-2} H\otimes S^2H, \sC\big)^{-1}{\kkbar}$ and 
% $\Hom\big(S^{n-2} H\otimes S^2H, H\big)^{0}{\kkbar}$, respectively. 
\end{remark}

\begin{proposition}[Integrability of Definition \ref{masterone}]
\label{descentone}
Let $\left\{\grave{\mb{\pi}}^{-1}, \mb{\eta}^{-2},\grave{\mb{m}}^0,\mb{\phi}^{-1}\right\}$ 
be a solution to the level one quantum master
equation. Then for all $n\geq 2$ and ${v}_1,\ldots, {v}_n \in H$,
\begin{align*}
(-\kbar)\mb{\phi}^0_n({v}_1,\ldots, {v}_n)
=
&
\mathbin{\hphantom{-}}\sum_{\clapsubstack{\mp \in P(n)\\ \big|B_{|\mp|}\big|=n -|\mp|+1\\ n-1\sim_\mp n \\ |\mp|\neq n}}
\e(\mp)
\mb{\phi}^0_{|\mp|}\Big({v}_{B_1},\ldots, {v}_{B_{|\mp|-1}}, \grave{\mb{m}}^0({v}_{B_{|\mp|}})\Big)
\\
&
-\sum_{\clapsubstack{\mp \in P(n)\\|\mp|=2\\ n-1\nsim_\mp n}}
\ep(\mp)
\mb{\phi}^0\big({v}_{B_1}\big)\cdot\mb{\phi}^0\big({v}_{B_{2}}\big)
\\
&
+\sum_{\clapsubstack{\mp \in P(n)\\ |B_i| = n-|\mp|+1\\ n-1 \sim_\mp n}}
\e(\mp)
\bell_{|\mp|}\left(\mb{\phi}^0(J{v}_{B_{1}}), \ldots,\mb{\phi}^0(J{v}_{B_{|\mp|-1}}),  
\mb{\phi}^{-1}({v}_{B_{|\mp|}})\right)
\\
&
+\sum_{\clapsubstack{\mp \in P(n)\\ |B_i| = n-|\mp|+1\\ n-1\nsim_\mp n}}
\e(\mp)
\mb{\phi}^{-1}_{|\mp|} \left( J{v}_{B_1},\dotsc, J{v}_{B_{i-1}},
\grave{\bell}({v}_{B_i}), {v}_{B_{i+1}},\dotsc, {v}_{B_{|\mp|}}\right)
.
\end{align*}
\end{proposition}

Here is our strategy to prove the above proposition.
Consider the level one quantum master equation in the form of~	\eq{mstone}.
Applying $\mb{K}^\infty_{\mathit{H\sC}}$ to the first equation in \eq{mstone} and
using the second equation in \eq{mstone}, we obtain that
\[
\mb{K}^\infty_{\mathit{H\sC}} \mb{\Omega}^{-1}_n
+\mb{f}\circ\grave{\mb{\varpi}}^0_n  =(-\kbar)^{n-2}\left( \mb{f}\circ\grave{\mb{m}}^0_n +\mb{K}^\infty_{\mathit{H\sC}} \mb{\phi}^{-1}_n
\right).
\]
Therefore, there must be some 
$\mb{M}^0_n \in \Hom\big(S^{n-2} H\otimes S^2H, \sC\big)^{0}{\kkbar}$, for $n\geq 2$, 
such that
$\mb{K}^\infty_{\mathit{H\sC}} \mb{\Omega}^{-1}_n+\mb{f}\circ\grave{\mb{\varpi}}^0_n=(-\kbar)^{n-2}\mb{M}^0_n$,
which implies that
\eqn{dstone}{
\mb{M}^0_n=\mb{f}\circ\grave{\mb{m}}^0_n +\mb{K}^\infty_{\mathit{H\sC}} \mb{\phi}^{-1}_n.
}
We claim that, for all $n\geq 2$,
\eqnalign{zxzaxc}{
\mb{M}^0_{n}({v}_1,\ldots, {v}_n)
=
&
(-\kbar)\mb{\phi}^0_n({v}_1,\ldots, {v}_n)
+\sum_{\clapsubstack{\mp \in P(n)\\|\mp|=2\\ n-1\nsim_\mp n}}
\ep(\mp)
\mb{\phi}^0\big({v}_{B_1}\big)\cdot\mb{\phi}^0\big({v}_{B_{2}}\big)
\\
&
-\sum_{\clapsubstack{\mp \in P(n)\\ \big|B_{|\mp|}\big|=n -|\mp|+1\\ n-1\sim_\mp n\\ \color{red} |\mp|\neq 1}}
\e(\mp)
\mb{\phi}^0_{|\mp|}\Big({v}_{B_1},\ldots, {v}_{B_{|\mp|-1}}, \grave{\mb{m}}^0({v}_{B_{|\mp|}})\Big)
\\
&
+\sum_{\clapsubstack{\mp \in P(n)\\ n-1 \sim_\mp n\\ \color{red} |\mp|\neq 1}}
\e(\mp)
\bell_{|\mp|}\left(\mb{\phi}^0\big(J{v}_{B_{1}}\big), \ldots,
\mb{\phi}^0\big(J{v}_{B_{|\mp|-1}}\big),  \mb{\phi}^{-1}\big({v}_{B_{|\mp|}}\big)\right)
,
}
so that the relation \eq{dstone} is equivalent to the above proposition.
Therefore, it suffices to show that
$\mb{K}^\infty_{\mathit{H\sC}} \mb{\Omega}^{-1}_n =(-\kbar)^{n-2}\mb{M}^0_n -\mb{f}\circ\grave{\mb{\varpi}}^0_n$
for all $n\geq 2$.

\begin{lemma}\label{keyone}
Let $\left\{\grave{\mb{\pi}}^{-1}_2,\ldots,\grave{\mb{\pi}}^{-1}_{n-1}
\big|\mb{\eta}^{-2}_2,\ldots,\mb{\eta}^{-2}_{n-1}
\big|\mb{\phi}^{-1}_2,\ldots,\mb{\phi}^{-1}_{n-1}
\big|\grave{\mb{m}}^0_2,\ldots,\grave{\mb{m}}^0_{n-1}\right\}$ be a solution to the level one quantum master equation for
$k=2,\ldots, n-1$. Then, we have
\[
\mb{K}^\infty_{\mathit{H\sC}} \mb{\Omega}^{-1}_n =(-\kbar)^{n-2}\mb{M}^0_n -\mb{f}\circ\grave{\mb{\varpi}}^0_n.
\]
\end{lemma}

We remark that
the triple $\big(\mb{\Omega}^{-1}_n, \mb{M}^0_n,\grave{\mb{\varpi}}^0_n\big)$ depends only on
\[
\left\{\grave{\mb{\pi}}^{-1}_2,\ldots,\grave{\mb{\pi}}^{-1}_{n-1}
\big|\mb{\eta}^{-2}_2,\ldots,\mb{\eta}^{-2}_{n-1}
\big|\mb{\phi}^{-1}_2,\ldots,\mb{\phi}^{-1}_{n-1}
\big|\grave{\mb{m}}^0_2,\ldots,\grave{\mb{m}}^0_{n-1}\right\}.
\]
along with the canonical solution to the level zero quantum master equation.
The above lemma will be the key to solving the level one quantum master equation, but its proof is purely technical.

\begin{proof}

For $n=2$, we have $\mb{M}^0_2=\mb{\Pi}^0_2$, $\mb{\O}^{-1}_n =\mb{\eta}^{-1}_2$, 
and $\grave{\mb{m}}_2=\grave{\mb{\pi}}_2$.
Therefore the level zero quantum master equation 
$\mb{f}\circ \grave{\mb{\pi}}^0_2 =\mb{\Pi}^0_2 -\mb{K}^\infty_{\mathit{H\sC}}\mb{\eta}^{-1}_2$ for $n=2$  
implies that $\mb{K}^\infty_{\mathit{H\sC}} \mb{\O}^{-1}_2 =\mb{M}^0_2 -\mb{f}\circ\grave{\mb{\varpi}}^0_2.$

For $n\geq 3$ assume that 
$\mb{K}^\infty_{\mathit{H\sC}} \mb{\Omega}^{-1}_k =(-\kbar)^{k-2}\mb{M}^0_k -\mb{f}\circ\grave{\mb{\varpi}}^0_k$
for $k=2,\ldots, n-1$, which implies that
\eqn{zxzb}{
\mb{K}^\infty_{\mathit{H\sC}} \mb{\phi}^{-1}_k =\mb{M}^0_k- \mb{f}\circ\grave{\mb{m}}^0_k, \qquad 2\leq k \leq n-1.
}
It remains only to  check that 
$\mb{K}^\infty_{\mathit{H\sC}} \mb{\Omega}^{-1}_n =(-\kbar)^{n-2}\mb{M}^0_n -\mb{f}\circ\grave{\mb{\varpi}}^0_n$.
Rewrite $\mb{\O}^{-1}_n$ as follows:
\eqnalign{mstonec}{
\mb{\Omega}^{-1}_n (v_1,\ldots, v_n)
=
&
\mb{\eta}^{-1}_n(v_1,\ldots,v_n)
\\
&
-\sum_{\clapsubstack{\mp \in P(n)\\|B_{|\mp|}|=n-|\mp|+1\\ n-1\sim_\mp n \\ |\mp|\neq 1}}
(-\kbar)^{n-|\mp|-1}
\e(\mp)\; 
\mb{\eta}^{-1}_{|\mp|}\left( {v}_{B_1}, \cdots, {v}_{B_{\mp-1}},
\grave{\mb{m}}^0\big({v}_{B_{|\mp|}}\big)\right)
\\
&
-\sum_{\clapsubstack{\mp \in P(n)\\|\mp|=2\\ n-1\sim_\mp n}}(-\kbar)^{|B_2|-2}\e(\mp)
\mb{\Pi}^0\big(Jv_{B_1})\cdot\mb{\phi}^{-1}\big(v_{B_2}\big)
,\\
}
where we have used the identity
\begin{align*}
\sum_{\clapsubstack{\mp \in P(n)\\  n-1\sim_\mp n \\ |\mp|\neq 1}}
(-\kbar)^{n-|\mp|-1}
\e(\mp)&
\mb{\phi}^0\big(J{v}_{B_1}\big)\cdots\mb{\phi}^0\big(J{v}_{B_{|\mp|-1}}\big)\cdot\mb{\phi}^{-1}({v}_{B_{|\mp|}})
\\
=&
\sum_{\clapsubstack{\mp \in P(n)\\|\mp|=2\\ n-1\sim_\mp n}}
(-\kbar)^{|B_2|-2}\e(\mp)\mb{\Pi}^0\big(Jv_{B_1})\cdot\mb{\phi}^{-1}\big(v_{B_2}\big).
\end{align*}
It will also be convenient to decompose  $\mb{M}^0_k = \mb{X}_k + \mb{Y}_k + \mb{Z}_k$, for $2\leq k \leq n-1$,
where
\begin{align*}
\mb{X}_{k}({v}_1,\ldots, {v}_k)
\coloneqq{}
&
(-\kbar)\mb{\phi}_k({v}_1,\ldots, {v}_k)
+\sum_{\clapsubstack{\mp \in P(k)\\|\mp|=2\\ k-1\nsim_\mp k}}
\ep(\mp)
\mb{\phi}^0\big({v}_{B_1}\big)\cdot\mb{\phi}^0\big({v}_{B_{2}}\big)
,\\
\mb{Y}_{k}({v}_1,\ldots, {v}_k)
\coloneqq{}
&
-\sum_{\clapsubstack{\mp \in P(k)\\ \big|B_{|\mp|}\big|=k -|\mp|+1\\ k-1\sim_\mp k\\  |\mp|\neq 1}}
\e(\mp)
\mb{\phi}^0_{|\mp|}\Big({v}_{B_1},\ldots, {v}_{B_{|\mp|-1}}, \grave{\mb{m}}^0({v}_{B_{|\mp|}})\Big)
,\\
\mb{Z}_{k}({v}_1,\ldots, {v}_k)
\coloneqq{}
&
\sum_{\clapsubstack{\mp \in P(k)\\ k-1 \sim_\mp k\\  |\mp|\neq 1}}
\e(\mp)
\bell_{|\mp|}\left(\mb{\phi}^0\big(J{v}_{B_{1}}\big), \ldots,
\mb{\phi}^0\big(J{v}_{B_{|\mp|-1}}\big),  \mb{\phi}^{-1}\big({v}_{B_{|\mp|}}\big)\right)
,
\end{align*}
Applying $\mb{K}^\infty_{\mathit{H\sC}}$ to \eq{mstonec} and using $\mb{K}^\infty_{\mathit{H\sC}}\mb{\Pi}^0=0$, we have
\begin{align*}
\mb{K}^\infty_{\mathit{H\sC}} \mb{\Omega}^{-1}_n (v_1,\ldots, v_n)
=
&
\mb{K}^\infty_{\mathit{H\sC}} \mb{\eta}^{-1}_n(v_1,\ldots,v_n)
\\
&
-\sum_{\clapsubstack{\mp \in P(n)\\|B_{|\mp|}|=n-|\mp|+1\\ n-1\sim_\mp n \\ |\mp|\neq 1}}
(-\kbar)^{n-|\mp|-1}
\e(\mp)\; 
\mb{K}^\infty_{\mathit{H\sC}} \mb{\eta}^{-1}_{|\mp|}\left( {v}_{B_1}, \cdots, {v}_{B_{\mp-1}},
\grave{\mb{m}}^0\big({v}_{B_{|\mp|}}\big)\right)
\\
&
-\sum_{\clapsubstack{\mp \in P(n)\\|\mp|=2\\ n-1\sim_\mp n}}(-\kbar)^{|B_2|-2}\e(\mp)\mb{\Pi}^0\big(v_{B_1}\big)
\cdot\mb{K}^\infty_{\mathit{H\sC}}\mb{\phi}^{-1}\big(v_{B_2}\big)
\\
&
-\sum_{\clapsubstack{\mp \in P(n)\\|\mp|=2\\ n-1\sim_\mp n}}(-\kbar)^{|B_2|-3}\e(\mp)
\bell_2\Big(\mb{\Pi}^0\big(Jv_{B_1}\big),\mb{\phi}^{-1}\big(v_{B_2}\big)\Big)
.\\
\end{align*}
Using \eq{mstzero} and \eq{zxzb}, we obtain the decomposition 
$\mb{K}^\infty_{\mathit{H\sC}} \mb{\Omega}^{-1}_n(v_1,\ldots, v_n)  = \bm{R}_1+ \bm{R}_2 +  \bm{R}_3+\bm{R}_4,
$ where
\begingroup
\allowdisplaybreaks
\begin{align*}
\bm{R}_1
=
&
-\mb{f}\Big( \grave{\mb{\pi}}^0_n (v_1,\ldots, v_n) 
-\sum_{\clapsubstack{\mp \in P(n)\\|B_{|\mp|}|=n-|\mp|+1\\ n-1\sim_\mp n \\ |\mp|\neq 1}}
(-\kbar)^{n-|\mp|-1}
\e(\mp)\; 
\grave{\mb{\pi}}^0_{|\mp|}\left( {v}_{B_1}, \cdots, {v}_{B_{\mp-1}},
\grave{\mb{m}}^0\big({v}_{B_{|\mp|}}\big)\right)
\Big)
,\\
\bm{R}_2=
&
+\mb{\Pi}^0_n(v_1,\ldots,v_n)
-\sum_{\clapsubstack{\mp \in P(n)\\|\mp|=2\\ n-1\sim_\mp n}}(-\kbar)^{|B_2|-2}\e(\mp)\mb{\Pi}\big(v_{B_1}\big)
\cdot \mb{X}\big({v}_{B_{|\mp|}}\big)
,
\\
\bm{R}_3=
&
-\sum_{\clapsubstack{\mp \in P(n)\\|B_{|\mp|}|=n-|\mp|+1\\ n-1\sim_\mp n \\ |\mp|\neq 1}}
(-\kbar)^{n-|\mp|-1}
\e(\mp)\; 
\mb{\Pi}^0_{|\mp|}\left( {v}_{B_1}, \cdots, {v}_{B_{\mp-1}},
\grave{\mb{m}}^0\big({v}_{B_{|\mp|}}\big)\right)
\\
&
+\sum_{\clapsubstack{\mp \in P(n)\\|\mp|=2\\ n-1\sim_\mp n}}(-\kbar)^{|B_2|-2}\e(\mp)
\mb{\Pi}^0\big(v_{B_1}\big)\cdot \mb{\phi}^0_1\big(\grave{\mb{m}}^0({v}_{B_{2}})\big)
\\
&
-\sum_{\clapsubstack{\mp \in P(n)\\|\mp|=2\\ n-1\sim_\mp n}}(-\kbar)^{|B_2|-2}
\e(\mp)
\mb{\Pi}^0\big(v_{B_1}\big)\cdot\mb{Y}({v}_{B_{2}})
,\\
 \bm{R}_4=
&
-\sum_{\clapsubstack{\mp \in P(n)\\|\mp|=2\\ n-1\sim_\mp n}}(-\kbar)^{|B_2|-3}\e(\mp)
\bell_2\Big(\mb{\Pi}^0\big(Jv_{B_1}\big),\mb{\phi}^{-1}\big(v_{B_2}\big)\Big)
\\
&
-\sum_{\clapsubstack{\mp \in P(n)\\|\mp|=2\\ n-1\sim_\mp n}}(-\kbar)^{|B_2|-2}\e(\mp)
\mb{\Pi}^0\big(v_{B_1}\big)
\cdot\mb{Z}({v}_{B_{2}})
%\\
%&
%-\sum_{\clapsubstack{\mp \in P(n)\\  n-1\sim_\mp n \\\color{red}|\mp|\neq 1}}
%(-\kbar)^{n-|\mp|}
%\e(\mp)\bell_2\Big(
%\mb{\phi}\big(J{v}_{B_1}\big)\cdots\mb{\phi}\big(J{v}_{B_{|\mp|-1}}\big), \mb{\phi}^{-1}({v}_{B_{|\mp|}})
%\Big)
.
\end{align*}
\endgroup

From the definition of $\grave{\mb{\varpi}}^{0}_n$ in \eq{mstoneb}, we have
\[
\bm{R}_1 = -\mb{f}\circ \grave{\mb{\varpi}}^0_n (v_1,\ldots, v_n).
\]
From the definition of $\mb{\Pi}^0$ in \eq{zerocor}, we obtain the following  identity:
\[
\mb{\Pi}^0_{n}(v_1,\ldots, v_n)
= 
(-\kbar)^{n-2}\mb{X}_{n}(v_1,\ldots, v_n)
+\sum_{\clapsubstack{\mp \in P(n)\\|\mp|=2\\ n-1\sim_\mp n}}
(-\kbar)^{|B_2|-2}\e(\mp)\mb{\Pi}^0\big(v_{B_1})\cdot\mb{X}\big(v_{B_2}\big)
,
\]
which implies that  
\[
\bm{R}_2 =(-\kbar)^{n-2}\mb{X}_n (v_1,\ldots, v_n).
\]
We can check from the definition of $\mb{\Pi}^0_k$ and $\mb{Y}_k$, $2\leq k\leq n-1$, that
\[
\bm{R}_3
%=-(-\kbar)^{n-2}\sum_{\clapsubstack{\mp \in P(n)\\ \big|B_{|\mp|}\big|=n -|\mp|+1\\ n-1\sim_\mp n\\ |\mp|\neq 1}}
%\e(\mp)
%\mb{\phi}^0_{|\mp|}\Big({v}_{B_1},\ldots, {v}_{B_{|\mp|-1}}, \grave{\mb{m}}^0({v}_{B_{|\mp|}})\Big)
=(-\kbar)^{n-2}\mb{Y}_n (v_1,\ldots, v_n)
.
\]
Finally, the formula in Remark \ref{remBKftalg} and the definition of $\mb{Z}_k$ for $2\leq k\leq n-1$, imply that
\[
\bm{R}_4
%=+(-\kbar)^{n-2}\sum_{\clapsubstack{\mp \in P(n)\\ n-1 \sim_\mp n\\ |\mp|\neq 1}}
%\e(\mp)
%\bell_{|\mp|}\left(\mb{\phi}^0\big(J{v}_{B_{1}}\big), \ldots,
%\mb{\phi}^0\big(J{v}_{B_{|\mp|-1}}\big),  \mb{\phi}^{-1}\big({v}_{B_{|\mp|}}\big)\right)
=(-\kbar)^{n-2}\mb{Z}_n (v_1,\ldots, v_n)
.
\]
Therefore we have 
$\mb{K}^\infty_{\mathit{H\sC}} \mb{\Omega}^{-1}_n =(-\kbar)^{n-2}\mb{M}^0_n -\mb{f}\circ\grave{\mb{\varpi}}^0_n$.
\naturalqed
\end{proof}

\begin{theorem}\label{solvemstone}
There is a canonical solution
$\left\{\grave{\mb{\pi}}^{-1}, \mb{\eta}^{-2},\grave{\mb{m}},\mb{\phi}^{-1}\right\}$
to  the level one quantum master equation in Definition \ref{masterone}
with
\begin{equation}
\label{eq: redundant level one}
\grave{\mb{\pi}}^{-1}_{2}=0
,\quad
\mb{\eta}^{-2}_2=0
,\quad
\grave{\mb{m}}_{2}=\grave{\mb{\pi}}_2
,\quad
\mb{\phi}^{-1}_2=\mb{\eta}^{-1}_2
\end{equation}
and, for all $n\geq 3$, 
\begin{align*}
\grave{\mb{\pi}}^{-1}_n& \coloneqq{}\sum_{\mathclap{i=0}}^{n-3}(-\kbar)^i h\circ \big(\Omega^{-1}_n\big)^{[i]}
,&
\mb{\grave{m}}_n 
&=
\Fr{1}{(-\kbar)^{n-2} }\left(\mb{\grave{\varpi}}_n 
+\mb{\kappa}^\infty_{\mathit{HH}}\grave{\mb{\pi}}^{-1}_n\right)
,
\phantom{\Big)^{[i]}}
\\
\mb{\eta}^{-2}_n& \coloneqq{}\sum_{\mathclap{i=0}}^{n-3}(-\kbar)^i s\circ \big(\Omega^{-1}_n\big)^{[i]}
,
&
\mb{\phi}^{-1}_n 
&= 
\left(\mb{\Omega}^{-1}_n\right)^{[n-2]}
,
\phantom{\Big)^{[i]}}
\end{align*}
where
$\left(\mb{\Omega}^{-1}_n\right)^{[i]}\coloneqq{}\left(\nabla^\infty_{(-\kbar)^{-1}}\right)^i \mb{\Omega}^{-1}_n$
and $\left({\Omega}^{-1}_n\right)^{[i]}$ is the classical limit of $\left(\mb{\Omega}^{-1}_n\right)^{[i]}$ for
$0\leq i\leq n-3$.  
\end{theorem}
\begin{remark}
Again, the conditions of eq.~\eqref{eq: redundant level one} are mostly redundant; all except the third of them are explicitly required by Definition~\ref{masterone}.
\end{remark}

\begin{proof}
Fix $n\geq 3$ and let
$\left\{\grave{\mb{\pi}}^{-1}_2,\ldots,\grave{\mb{\pi}}^{-1}_{n-1}
\big|\mb{\eta}^{-2}_2,\ldots,\mb{\eta}^{-2}_{n-1}
\big|\mb{\phi}^{-1}_2,\ldots,\mb{\phi}^{-1}_{n-1}
\big|\grave{\mb{m}}^0_2,\ldots,\grave{\mb{m}}^0_{n-1}\right\}$ be a solution 
to the level one quantum master equation for
$k=2,\ldots, n-1$ with all the  properties in Theorem~\ref{solvemstone}. 
Then, we have  (via Lemma \ref{keyone})
\[
\mb{K}^\infty_{\mathit{H\sC}} \mb{\Omega}^{-1}_n =(-\kbar)^{n-2}\mb{M}^0_n -\mb{f}\circ\grave{\mb{\varpi}}^0_n.
\]

Therefore, we can apply Lemma \ref{hodgea} to obtain that
\begin{align*}
(-\kbar)^{n-2}\big(\mb{\Omega}_n^{-1}\big)^{[n-2]}
=&
\mb{\Omega}^{-1}_n
-\mb{f}\circ\left(\sum_{\mathclap{i=0}}^{n-3}(-\kbar)^i h\circ \big(\Omega^{-1}_n\big)^{[i]}\right)\\
&\quad{}
-\mb{K}^\infty_{H\sC}\left(\sum_{\mathclap{i=0}}^{n-3}(-\kbar)^i s\circ \big(\Omega^{-1}_n\big)^{[i]} \right)
,\\
(-\kbar)^{n-2} \big(\grave{\mb{\varpi}}^0_n\big)^{[n-2]}
=&\grave{\mb{\varpi}}^0_n 
+\sum_{\mathclap{i=0}}^{n-3} (-\kbar)^i\mb{\kappa}_{\!H\!H}\left( h\circ \big(\Omega^{-1}_n\big)^{[i]}\right)
,\\
\mb{K}^\infty_{\!H\!\sC}\big(\mb{\Omega}^{-1}_n\big)^{[n-2]}
=&\mb{M}^0_n - \mb{f}\circ \big(\grave{\mb{\varpi}}^0_n\big)^{[n-2]}
.
\end{align*}
Recall that $\big(\mb{\O}^{-1}_n\big)^{[j]}=\mb{\nabla}_{(-\kbar)^{-1}}^{j}\mb{\O}^{-1}_n$ is in $\Hom\big(S^{n-2}H\otimes S^2H,\sC\big)^{-1}{\kkbar}$ for $j$ in the range from $0$ to $n-2$, and 
$\big(\grave{\mb{\varpi}}^0_n\big)^{[n-2]}   \in \Hom\Big(S^{n-2}H\otimes S^2H,\sC\Big)^{0}{\kkbar}$. Setting 

\begin{align*}
\grave{\mb{\pi}}^{-1}_n& \coloneqq{}\sum_{\mathclap{i=0}}^{n-3}(-\kbar)^i h\circ \big(\Omega^{-1}_n\big)^{[i]}
%\sum_{\mathclap{i=0}}^{n-2}(-\kbar)^i h\pi^{(i)}_n
&\text{ in }& \Hom\big(S^{n-2}H\otimes S^2H, \sC\big)^{-1}
,\\
\mb{\eta}^{-2}_n& \coloneqq{}\sum_{\mathclap{i=0}}^{n-3}(-\kbar)^i s\circ \big(\Omega^{-1}_n\big)^{[i]} &\text{ in }& \Hom\Big(S^{n-2}H\otimes S^2H, \sC\Big)^{-2}
,\\
\mb{\phi}^{-1}_n&\coloneqq{}\big(\mb{\O}^{-1}_n\big)^{[n-2]}&\text{ in }& \Hom\Big(S^{n-2}H\otimes S^2H,\sC\Big)^{-1}{\kkbar}
,\\
\grave{\mb{m}}^0_n &\coloneqq{}\big(\grave{\mb{\varpi}}^0_n\big)^{[n-2]} &\text{ in }& \Hom\Big(S^{n-2}H\otimes S^2H,\sC\Big)^{0}{\kkbar}
,
\end{align*}
we obtain that
\begin{align*}
(-\kbar)^{n-2} \mb{\phi}^{-1}_n=
& \mb{\Omega}^{-1}_n
- \mb{f}\circ\grave{\mb{\pi}}^{-1}_{n} 
-\mb{K}^\infty_{\mathit{H\sC}}\mb{\eta}^{-2}_{n}
,\\
(-\kbar)^{n-2}\grave{\mb{m}}^0_n  =&\grave{\mb{\varpi}}^0_n +\mb{\kappa}^\infty_{\mathit{H\sC}}\grave{\mb{\pi}}^{-1}_n
,
\end{align*}
so we are done by induction.
\naturalqed
\end{proof}

Now we specialize to the anomaly-free case that $\mb{\kappa}=0$.
Then the following lemma shows that the family $\underline{\grave{\mb{m}}}
=\grave{\mb{m}}^0_2, \grave{\mb{m}}^0_3,\ldots$ does not depend on $\kbar$.

\begin{lemma}\label{dalem}
Let  $\mb{\kappa}=0$. Then, we have 
$\grave{\mb{m}}^0_n =\grave{m}^0_n=\grave{\pi}^{0(n-2)}_n \in\Hom\big(S^n H,H\big)^0$, for all $n\geq 2$,
so that  the family $\underline{\grave{{m}}}^0
=\grave{{m}}^0_2, \grave{{m}}^0_3,\ldots$ 
determines $\underline{\grave{\mb{\pi}}}^0$ as follows:
for all $n\geq 2$ and homogeneous elements $v_1,\ldots,v_n \in H$,
\[
\grave{\mb{\pi}}^0_{n} ({v}_1,\dotsc,{v}_n)
=
\sum_{\clapsubstack{\mp \in P(n)\\|B_{|\mp|}|=n-|\mp|+1\\ n-1\sim_\mp n  }}
(-\kbar)^{n-|\mp|-1}
\e(\mp)\; 
\grave{\mb{\pi}}^0_{|\mp|}\left({v}_{B_1}, \cdots, {v}_{B_{\mp-1}},
\grave{m}^0\big({v}_{B_{|\mp|}}\big)\right)
.
\]
\end{lemma}
\begin{proof}
From the condition $\mb{\kappa}=0$, we have $\mb{\kappa}^\infty_{\mathit{HH}}=0$ and 
\eqn{fses}{
\grave{\mb{\pi}}^0_{n} ({v}_1,\dotsc,{v}_n)
=
\sum_{\clapsubstack{\mp \in P(n)\\|B_{|\mp|}|=n-|\mp|+1\\ n-1\sim_\mp n  }}
(-\kbar)^{n-|\mp|-1}
\e(\mp)\; 
\grave{\mb{\pi}}^0_{|\mp|}\left({v}_{B_1}, \cdots, {v}_{B_{|\mp|-1}},
\grave{\mb{m}}^0\big({v}_{B_{|\mp|}}\big)\right)
.
}
Recall that $\grave{\mb{\pi}}^0_1 =\I_H$ and 
$\grave{\mb{\pi}}^0_n = \sum_{{j=0}}^{n-2}(-\kbar)^j \grave{\pi}^{0(j)}_n$, where
$\grave{\pi}^{0(j)}_n \in \Hom\big(S^n H, H\big)^0$, for all $n\geq 2$.
For $n=2$, we have $\grave{\mb{\pi}}^0_{2} = \grave{\mb{m}}^0_2$.
 It follows that $\grave{\mb{m}}^0_2= \grave{{m}}^0_2=\grave{\pi}^{0(0)}_2 \in \Hom\big(S^2 H, H\big)^0$ 
 since $\grave{\mb{\pi}}^0_{2} = \grave{{\pi}}^0_{2} \in \Hom\big(S^2 H, H\big)^0$. 
Fix $n\geq 3$ and assume that 
$\grave{\mb{m}}^0_k =\grave{{\pi}}^{0(k-2)}_k= \grave{{m}}^0_k \in \Hom\big(S^k H, H\big)^0$ for $2\leq k < n$.
Then  relation \eq{fses} can be rewritten as follows:
\begin{align*}
(-\kbar)^{n-2} \grave{\mb{m}}^0_n(& v_1,\ldots, v_n)
= (-\kbar)^{n-2}\grave{{\pi}}^{0(n-2)}_n({v}_1,\dotsc,{v}_n)
\\
&
+\sum_{\mathclap{j=0}}^{n-3}(-\kbar)^{n-3}\grave{{\pi}}^{0(j)}_n ({v}_1,\dotsc,{v}_n)
\\
&
-\sum_{\clapsubstack{\mp \in P(n)\\|B_{|\mp|}|=n-|\mp|+1\\ n-1\sim_\mp n }}\ \ 
\sum_{\mathclap{j=0}}^{|\mp|-2}
(-\kbar)^{n-|\mp| + j -1}
\e(\mp)\; 
\grave{{\pi}}^{0(j)}_{|\mp|}\left({v}_{B_1}, \cdots, {v}_{B_{|\mp|-1}},
\grave{{m}}^0\big({v}_{B_{|\mp|}}\big)\right)
.
\end{align*}
The term on the right hand side of the above equality with the highest power in $\kbar$ is 
$(-\kbar)^{n-2}\grave{{\pi}}^{0(n-2)}_n$. Note also that the right hand side  must be divisible by $(-\kbar)^{n-2}$ since  
$\grave{\mb{m}}^0_n \in \Hom\big(S^{n-2} H\otimes S^2 H, H\big)^{0}{\kkbar}$.
It follows that $\grave{\mb{m}}^0_n=\grave{{\pi}}^{0(n-2)}_n=  \grave{{m}}^0_n\in \Hom\big(S^n H, H\big)^0$.  
Therefore, by induction, we have 
$\grave{\mb{m}}^0_n=  \grave{{m}}^0_n\in \Hom\big(S^n H, H\big)^0$, for all $n\geq 2$.
\naturalqed
\end{proof}

\section{Applications to anomaly-free theory}
\label{section: anomaly-free theory}

In this subsection, we specialized to 
an anomaly-free binary QFT algebra, leading to Theorems \ref{intra}, \ref{intrb} and \ref{intrc}
together with some physical interpretations and two algorithms 
to compute the universal algebraic structure governing quantum correlations.
We further specialize to the case that $H$ is finite dimensional
to discuss relationships 
with the notion of special coordinates and the WDDV equation
in topological string theory.

\subsection{The algebra governing quantum correlation functions}

In this subsection, we
consider an anomaly-free binary QFT algebra $\sC{\kkbar}_\BQ=
\big(\sC{\kkbar}, 1_\sC, \,\cdot\,,\mb{K}\big)$ 
with quantum descendant $\big(\sC{\kkbar}, 1_\sC, \underline{\bell}\big)$.
Fix a classical off-to-on-shell retraction $(f,h,s)$ and let  $(\mb{f},\mb{h},\mb{s})$ be a quantization of it. Then $(\mb{f},\mb{h},\mb{s})$ is the data of a homotopy equivalence between $\big(H{\kkbar}, 1_H, 0\big)$ and $\big(\sC{\kkbar}, 1_\sC, \mb{K}\big)$.
Let $v_1,\ldots, v_n$ be homogeneous elements in $H$.

The results in Section~\ref{subsec: Master equation for the level zero quantum correlators} specialized to the case $\mb{\kappa}=0$ 
can be presented as follows:
\begin{theorem}\label{lakia}
There is a distinguished unital $sL_\infty$-quasi-isomorphism 
\[
\xymatrix{\underline{\mb{\phi}}^0:
\big(H{\kkbar}, 1_H,  \underline{0}\big)\ar@{..>}[r] & \big(\sC{\kkbar}, 1_\sC,  \underline{{\bell}}\big)}
\]
such that $\mb{\phi}^0_1=\mb{f}$ and, for all $n\geq 1$,
\eqn{spmstzero}{
\mb{\Pi}^0_n= \mb{f}\circ \grave{\mb{\pi}}^0_n+ \mb{K} \circ\mb{\eta}^{-1}_n,
}
where $\mb{\Pi}^0_n \in \Hom\big(S^n H, \sC\big)^{0}{\kkbar}$ is defined by
\eqn{zerocorr}{
\mb{\Pi}^0_n(v_1,\ldots, v_n)\coloneqq{}\sum_{\mathclap{\mp\in P(n)}}
(-\kbar)^{n-|\mp|}\e(\mp)\,
\mb{\phi}^0\big({v}_{B_1}\big)\cdot\dotsc\cdot\mb{\phi}^0\big({v}_{B_{|\mp|}}\big)
,
}
and the families $\underline{\grave{\mb{\pi}}}^0$ and $\underline{\mb{\eta}}^{-1}$ have
the following properties:
\begin{itemize}
\item $\grave{\mb{\pi}}_1 =\I_{H}$
and, for all $n\geq 2$,
\[
\grave{\mb{\pi}}^0_n =\sum_{\mathclap{j=0}}^{n-2}(-\kbar)^j \grave{\pi}^{0(j)}_n,
\quad
\text{ where }
\grave{\pi}^{0(j)}_n\in\Hom\big(S^n H,H)^0,
\]
and $\grave{\mb{\pi}}^0_{n+1}(v_1,\ldots, v_n, 1_H)= \grave{\mb{\pi}}^0_{n+1}(v_1,\ldots, v_n)$, for all $n\geq 1$;

\item  $\mb{\eta}^{-1}_1 =0$ and, for all $n\geq 2$,
\[
\mb{\eta}^{-1}_n  =\sum_{\mathclap{j=0}}^{n-2}(-\kbar)^j \eta^{-1(j)}_n,
\quad\text{ where }
\eta^{-1(j)}_n \in \Hom(S^n H,\sC)^{-1}
\]
and $\mb{\eta}^{-1}_{n+1}(v_1,\ldots, v_{n}, 1_H)=\mb{\eta}^{-1}_{n+1}(v_1,\ldots, v_{n})$, for all $n\geq 1$.
\end{itemize}

\end{theorem}

\begin{remark}\label{kramz}
The first part of the above theorem establishes that the classical cohomology $H$ of an anomaly-free binary QFT algebra, viewed 
as a topologically-free unital $sL_\infty$-algebra $\big(H{\kkbar}, 1_H, \underline{0}\big)$ with zero $sL_\infty$-structure
$\underline{0}$ is quasi-isomorphic to the quantum descendant unital unital $sL_\infty$-algebra 
$\big(\sC{\kkbar}, 1_\sC, \underline{\bell}\big)$. Moreover, we have constructed a particular 
quasi-isomorphism $\underline{\mb{\w}}$ and the rest of the theorem is about its distinguished properties:
let $\underline{\mb{\w}}: \big(H{\kkbar}, 1_H, \underline{0}\big) \longrightarrow 
\big(\sC{\kkbar}, 1_\sC, \underline{\bell}\big)$ be an arbitrary  unital $sL_\infty$-quasi-isomorphism
and $\underline{\mb{\Pi}}^{\mb{\w}}$ be the associated  family of quantum correlators (which satisfy
$\mb{K}\circ {\mb{\Pi}}^{\mb{\w}}_n=0$, for all $n\geq 1$). It follows that
${\mb{\Pi}}^{\mb{\w}}_n = \mb{f}\circ {\mb{\pi}}^{\mb{\w}} + \mb{K}\circ \mb{\eta}^{\mb{\w}}_n$, for all $n\geq 1$,
where ${\mb{\pi}}^{\mb{\w}}_n\coloneqq{}\mb{h} \circ\mb{\Pi}^{\mb{\w}}_n \in \Hom\big(\Xbar S(H) , H\big)^{0}{\kkbar}$
and $\mb{\eta}^{\mb{\w}}_n\coloneqq{}\mb{s}\circ \mb{\Pi}^{\mb{\w}}_n \in \Hom\big(\Xbar S(H) , H\big)^{-1}{\kkbar}$.
However, the families $\underline{{\mb{\pi}}}^{\mb{\w}}$ and $\underline{\mb{\eta}}^{\mb{\w}}$ 
do not,  in general, have the properties of 
the families $\underline{\grave{\mb{\pi}}}^0$ and $\underline{\mb{\eta}}^{-1}$.
\naturalqed
\end{remark}

\begin{remark}\label{kramy}
We reproduce the algorithm to determine the families  $\underline{\mb{\phi}}^0$, $\underline{\grave{\mb{\pi}}}^0$ and
$\underline{\mb{\eta}}^{-1}$, which becomes much simpler due to the condition $\mb{\kappa}=0$.
Note that the operator 
$\nabla_{-1/\kbar}: \Hom\big(\Xbar{T}(H), \sC\big){\kkbar} \rightarrow  \Hom\big(\Xbar{T}(H), \sC\big){\kkbar}$
defined by \eq{hdbh} reduces in the following manner:  
for any $\mb{\O} \in \Hom\big(\Xbar{T}(H), \sC\big){\kkbar}$ whose classical limit $\O$ satisfies $K\circ \O=0$,
we have
$(-\kbar)\nabla_{-1/\kbar}\mb{\O}  = \mb{\O} - \mb{f}\circ h\circ \O -\mb{K}\circ s\circ \O$.

To begin with, we set  
\eqn{whyx}{
\mb{\phi}^0_1=\mb{f}
,\quad
\grave{\mb{\pi}}^0_1=\I_H
,\quad
\mb{\eta}^{-1}_1=0.
}
Note that $\mb{\Pi}^0_1= \mb{\phi}^0_1$. Therefore, we have 
\eqn{whya}{
\mb{\Pi}^0_1= \mb{f}\circ \grave{\mb{\pi}}^0_1+ \mb{K} \circ\mb{\eta}^{-1}_1
\quad\Longrightarrow\quad
\mb{K}\circ \mb{\phi}^0_1=0.
}

Assume that we have 
$\left\{{\mb{\phi}}^0_1,\ldots, {\mb{\phi}}^0_{n-1}\big|\grave{\mb{\pi}}^0_1, \ldots, \grave{\mb{\pi}}^0_{n-1}
\big| \mb{\eta}^{-1}_1, \ldots, \mb{\eta}^{-1}_{n-1}\right\}$ for $n\geq 2$, satisfying the initial conditions \eq{whyx} and, 
for all $k=1,\ldots, n-1$,
\eqn{whye}{
\mb{\Pi}^0_k = \mb{f}\circ \grave{\mb{\pi}}^0_k+\mb{K} \mb{\eta}^{-1}_k
\quad\Longrightarrow\quad
\sum_{\clapsubstack{\mp \in P(k)}}
\e(\mp)\,
\bell_{|\mp|}\!\left(\mb{\phi}^0\left( {v}_{B_1}\right), \cdots, \mb{\phi}\big(  {v}_{B_{|\mp|}}\big)\right)=0
.
}
Define $\mb{\O}^0_n \in \Hom\big(S^nH,\sC\big)^{0}{\kkbar}$ and
$\mb{L}_n \in \Hom\big(S^nH,\sC\big)^{1}{\kkbar}$ as follows:
\begin{align*}
\mb{\O}^0_n(v_1,\ldots,v_n) 
\coloneqq{}
& \sum_{\clapsubstack{\mp\in P(n)\\ \color{red}|\mp|\neq 1}}
(-\kbar)^{n-|\mp|}\e(\mp)\,
\mb{\phi}^0\big({v}_{B_1}\big)\cdot\dotsc\cdot\mb{\phi}^0\big({v}_{B_{|\mp|}}\big)
,\\
\mb{L}_n(v_1,\ldots,v_n) \coloneqq{}
&
\sum_{\clapsubstack{\mp \in P(n)\\\color{red}|\mp|\neq 1}}
\e(\mp)\,
\bell_{|\mp|}\!\left(\mb{\phi}^0\left( {v}_{B_1}\right), \cdots, \mb{\phi}\big(  {v}_{B_{|\mp|}}\big)\right)
.
\end{align*}
From \eq{whye} and the definition of $\bell_1,\ldots, \bell_{n}$,
we obtain that 
\[
\mb{K}\circ\mb{\O}^0_n=(-\kbar)^{n-1}\mb{L}_n.
\]
Define $\left(\mb{\O}^0_n\right)^{[j]} \in  \Hom\big(S^nH,\sC\big)^{0}{\kkbar}$ for $j=0,1,\ldots, n-1$,
by starting with $\left(\mb{\O}^0_n\right)^{[0]} =\mb{\O}^0_n$ and then recursively definining
$\left(\mb{\O}^0_n\right)^{[i+1]}= \nabla_{-1/\kbar}\left(\mb{\O}^0_n\right)^{[i]}$, $i=0,\ldots, n-2$,
so that we have
\[
(-\kbar)\left(\mb{\O}^0_n\right)^{[i+1]} 
= 
\left(\mb{\O}^0_n\right)^{[i]}
- \mb{f}\circ \grave{\pi}^{0(i)}_n 
-\mb{K}\circ \eta^{-1(i)}_n
,
\]
where
\eqnalign{whyb}{
\grave{\pi}^{0(i)}_n&\coloneqq{}h\circ  \left({\O}^0_n\right)^{[i]}  \in \Hom\big(S^nH,H\big)^0
,\\
 \eta^{-1(i)}_n &\coloneqq{}s\circ  \left({\O}^0_n\right)^{[i]} \in \Hom\big(S^nH,\sC\big)^{-1}
,
}
and $\left({\O}^0_n\right)^{[i]} \in  \Hom\big(S^nH,\sC\big)^0$ denotes the classical
limit of $\left(\mb{\O}^0_n\right)^{[i]}$.  
We remark that $\mb{K}\circ \left(\mb{\O}^0_n\right)^{[j]} =(-\kbar)^{n-1-j}\mb{L}_n$
for $j=0,\ldots, n-1$, and $K\circ\left({\O}^0_n\right)^{[i]} =0$ for $i=0,\ldots, n-2$.
It follows that 
\[
(-\kbar)^{n-1}\left(\mb{\O}^0_n\right)^{[n-1]}=\mb{\O}^0_n  
- \sum_{\mathclap{i=0}}^{n-2}(-\kbar)^i \mb{f}\circ\grave{\pi}^{0(i)}_n 
-\sum_{\mathclap{i=0}}^{n-2}(-\kbar)^i  \mb{K}\circ {\eta}^{-1(i)}_n.
\]
Finally, we set
\eqn{whyf}{
\mb{\phi}^0_n =- \left(\mb{\O}^0_n\right)^{[n-1]}
,\qquad
\grave{\mb{\pi}}^0_n=\sum_{\mathclap{i=0}}^{n-2}(-\kbar)^i\grave{\pi}^{0(i)}_n
,\qquad
\mb{\eta}^{-1}_n =\sum_{\mathclap{i=0}}^{n-2}(-\kbar)^i\eta^{-1(i)}_n
,
}
and obtain
\[
-(-\kbar)^{n-1}\mb{\phi}^0_n
=\mb{\O}^0_n  - \mb{f}\circ \grave{\mb{\pi}}^0_n -\mb{K} \circ\mb{\eta}^{-1}_n
\quad
\Longrightarrow\quad  \mb{K}\circ \mb{\phi}^0_n+ \mb{L}_n=0.
\]
From $\mb{\Pi}^0_n = \mb{\O}^0_n+(-\kbar)^{n-1}\mb{\phi}^0_n$, it follows that
\eqn{whyg}{
\mb{\Pi}^0_n = \mb{f}\circ \grave{\mb{\pi}}^0_n+\mb{K} \mb{\eta}^{-1}_n
\quad
\Longrightarrow
\quad
\sum_{\clapsubstack{\mp \in P(n)}}
\e(\mp)\,
\bell_{|\mp|}\!\left(\mb{\phi}^0\left( {v}_{B_1}\right), \cdots, \mb{\phi}\big(  {v}_{B_{|\mp|}}\big)\right)=0.
}
This concludes the construction of our well-defined algorithm to determine the families  $\underline{\mb{\phi}}^0$, $\underline{\grave{\mb{\pi}}}^0$ and
$\underline{\mb{\eta}}^{-1}$.
\naturalqed
\end{remark}

Now we turn to some physical interpretations of Theorem \ref{lakia}.

Due to the condition $\mb{\kappa}=0$,
we now call a homogeneous element $v \in H$ simply an {\em observable} and 
$\left\langle\mb{f}(v)\right>\rangle_{\bm{\mc}}\coloneqq{}\bm{\mc}\big(\mb{f}(v)\big)\in \Bbbk{\kkbar}$
the {\em quantum expectation value} of  the observable $v$ with respect to the quantum expectation $\bm{\mc}$ 
--- a pointed cochain map from $\big(\sC{\kkbar}, 1_\sC, \mb{K}\big)$ to $\big(\Bbbk{\kkbar}, 1_\sC, 0\big)$. 
From $\mb{K}\circ \mb{f}=0$,  it follows that $\bm{\mc}\big(\mb{f}(v)\big)$ depends only on the homotopy type of $\bm{\mc}$
for all $v\in H$. We also call $\bm{\mc}\circ \mb{f} \in \Hom\big(H, \Bbbk\big)^{0}{\kkbar}$ the {\em on-shell quantum expectation}
with respect to $\bm{\mc}$. 

We call $\underline{\mb{\Pi}}^0=\mb{\Pi}^0_1,\mb{\Pi}^0_2,\ldots$ the family of {\em level $0$ quantum correlators}.
For example,
\begin{align*}
\mb{\Pi}^0_1 =&\mb{\phi}^0_1,
\\
\mb{\Pi}^0_2(v_1,v_2)= &  \mb{\phi}^0_1(v_1)\cdot \mb{\phi}^0_1(v_2) +(-\kbar)\mb{\phi}^0_2(v_1,v_2)
,\\
\mb{\Pi}^0_3(v_1,v_2,v_3) 
=
& \mb{\phi}^0_1(v_1)\cdot \mb{\phi}^0_1(v_2)\cdot \mb{\phi}^0_1(v_3)
+(-\kbar) \mb{\phi}^0_1(v_1)\cdot \mb{\phi}^0_2(v_2,v_3)
\\
&
+(-\kbar) (-1)^{|v_1||v_2|}\mb{\phi}^0_1(v_2)\cdot \mb{\phi}^0_2(v_1,v_3)
+(-\kbar) \mb{\phi}^0_2(v_1,v_2)\cdot \mb{\phi}^0_2(v_3)
\\
&
+(-\kbar)^2 \mb{\phi}^0_3(v_1,v_2,v_3)
,
\end{align*}
We call $\bm{\mc}\circ \underline{\mb{\Pi}}^0$ the family of {\em level $0$ quantum correlation functions}
with respect to  $\bm{\mc}$.  
From \eq{spmstzero}, we have $\mb{K}\circ  \mb{\Pi}^0_n =0$
so that $\bm{\mc}\circ \mb{\Pi}^0_n \in \Hom\big(S^n H, \Bbbk\big)^{0}{\kkbar}$ 
depends only on the homotopy type of $\bm{\mc}$.
This also implies that, for all $n\geq 1$, 
\eqn{ywha}{
\bm{\mc}\circ \mb{\Pi}^0_n =\bm{\mc}\circ \mb{f}\circ \grave{\mb{\pi}}^0_n
.
}
Therefore the whole family $\bm{\mc}\circ \underline{\mb{\Pi}}^0$  of level $0$ quantum correlation functions is
determined by  the on-shell quantum expectation $\bm{\mc}\circ \mb{f}$ and the family 
$\underline{\grave{\mb{\pi}}}^0$.  Note that $\bm{\mc}\circ \mb{f}$ is a formal power series in $\kbar$
so that there may be divergence issues if we try to evaluate $\kbar$ away from zero. On the other hand 
$\grave{\mb{\pi}}_n$ is at most an order $n-2$ polynomial in $\kbar$ for $n\geq 2$, and
it is the family $\underline{\grave{\mb{\pi}}}^0$ that governs quantum correlations.

For any set of observables $\{v_1,\ldots, v_k\}$, we can define the associated family of joint quantum moments (at level zero)
as follows:
\eqn{ywhb}{
\Big\{\left\langle {\mb{\Pi}}^{0}_n(v_{j_1}, \ldots, v_{j_n})\right\rangle_{\bm{\mc}}
=\bm{\mc}\circ {\mb{\Pi}}^{0}_n(v_{j_1}, \ldots, v_{j_n})\Big|n\geq 1 \,;\, 1\leq j_1,\ldots, j_n\leq k\Big\}.
}
Then, the joint quantum distribution of the set  $\{v_1,\ldots, v_k\}$ of observables 
can be defined as a $\Bbbk$-linear  map
\[
\hat{\mb{\m}}: \Bbbk[t_1,\dotsc, t_k]\rightarrow \Bbbk{\kkbar}
\]
as follows.  For any monomial $t_{j_1}\cdots t_{j_n}$ in $\Bbbk[t_1,\dotsc, t_k]$ satisfying $1\leq j_1,\ldots, j_n\leq k$,
we have 
\[
\hat{\mb{\m}}\big(t_{j_1}\cdots t_{j_n}\big) = 
\left\langle {\mb{\Pi}}^{0}_n(v_{j_1}, \ldots, v_{j_n})\right\rangle_{\bm{\mc}}.
\]
It is clear that $\hat{\mb{\m}}$ is determined completely by the family of joint quantum moments
which is conveniently described by its generating function 
\[
\bm{Z}(t_1,\dotsc, t_k)=
\left\langle e^{-\Fr{1}{\kbar}\mb{\Theta}(\g)}\right\rangle_{\bm{\mc}}
=
1+\sum_{\mathclap{n=1}}^\infty\Fr{1}{n!(-\kbar)^n}\left\langle {\mb{\Pi}}^{0}_n(\g, \ldots, \g)\right\rangle_{\bm{\mc}}
\in 
\Bbbk[\![t_1,\dotsc,t_k]\!](\!(\kbar)\!),
\]
where $\g =\sum_{i=1}^k t_i v_i$ and
\[
\mb{\Theta}(\g)=
\sum_{\mathclap{n=1}}^\infty \Fr{1}{n!} \mb{\phi}_n(\g,\ldots, \g) 
=\sum_{\mathclap{n=1}}^\infty \sum_{{1\leq j_1,\ldots, j_n\leq k}}\Fr{1}{n!}t_{j_1}\cdots t_{j_n} \mb{\phi}_n(v_{j_1},\ldots, v_{j_n}). 
\]
Then, we have
\[
\hat{\mb{\m}}\big(t_{j_1}\cdots t_{j_n}\big) 
=(-\kbar)^n \left. \Fr{\rd^n \bm{Z}(t_1,\dotsc, t_k)}{\rd t_{j_1}\cdots\rd t_{j_n}}\right|_{t_1,\ldots, t_k=0}.
\]

An important part of the results in Section~\ref{subsec: homotopy h-divisibility 2}, specialized to the case $\mb{\kappa}=0$, can now be presented as
follows.

\begin{theorem}\label{lakib}
Let $\grave{m}^0_n \coloneqq{} \grave{\pi}^{0(n-2)}_n \in \Hom\big(S^n H, H\big)^0$ for $n\geq 2$. 
Then the family  $\underline{\grave{m}}^0=\grave{m}^0_2,\grave{m}^0_3,\ldots$ 
determines the family  $\underline{\grave{\mb{\pi}}}^0$ recursively 
with initial condition $\grave{\mb{\pi}}^0_1=\I_H$ along with the following equation  for $n\geq 2$:
\eqn{bsna}{
\grave{\mb{\pi}}^0_{n}({v}_1,\dotsc,{v}_n)
=
\sum_{\clapsubstack{\mp \in P(n)\\|B_{|\mp|}|=n-|\mp|+1\\ n-1\sim_\mp n}}
(-\kbar)^{n-|\mp|-1}
\e(\mp)\; 
\grave{\mb{\pi}}^0_{|\mp|}\left({v}_{B_1}, \cdots, {v}_{B_{\mp-1}},
\grave{m}^0\big({v}_{B_{|\mp|}}\big)\right).
}
Moreover, there is a distinguished family 
$\underline{\mb{\phi}}^{-1} =\mb{\phi}^{-1}_2, \mb{\phi}^{-1}_3,\ldots$ 
of $\mb{\phi}^{-1}_n$ in $\Hom\big(S^{n-2}\otimes S^2 H, \sC\big)^{-1}{\kkbar}$ so that
$\mb{\phi}^{-1}_2=\mb{\eta}^{-1}_2$ and
, for all $n\geq 2$,
\eqn{bsnb}{
\mb{M}^0_n=\mb{f}\circ \grave{m}^0_n + \mb{K}\circ \mb{\phi}^{-1}_n
,
}
where $\mb{M}^0_n \in \Hom\big(S^{n-2}H\otimes S^2H,\sC\big)^{0}{\kkbar}$ is defined as follows:
\eqnalign{bsnc}{
\mb{M}^0_n({v}_1,\ldots, {v}_n)
=
&
(-\kbar)\mb{\phi}^0_n({v}_1,\ldots, {v}_n)
+\sum_{\clapsubstack{\mp \in P(n)\\|\mp|=2\\ n-1\nsim_\mp n}}
\ep(\mp)
\mb{\phi}^0\big({v}_{B_1}\big)\cdot\mb{\phi}^0\big({v}_{B_{2}}\big)
\\
&
-\sum_{\clapsubstack{\mp \in P(n)\\ \big|B_{|\mp|}\big|=n -|\mp|+1\\ n-1\sim_\mp n \\ \color{red}|\mp|\neq 1}}
\e(\mp)
\mb{\phi}^0_{|\mp|}\Big({v}_{B_1},\ldots, {v}_{B_{|\mp|-1}}, \grave{{m}}^0({v}_{B_{|\mp|}})\Big)
\\
&
-\sum_{\clapsubstack{\mp \in P(n)\\ |B_i| = n-|\mp|+1\\ n-1 \sim_\mp n\\ \color{red}|\mp|\neq 1}}
\e(\mp)
\bell_{|\mp|}\left(\mb{\phi}^0(J{v}_{B_{1}}), \ldots,\mb{\phi}^0(J{v}_{B_{|\mp|-1}}),  
\mb{\phi}^{-1}({v}_{B_{|\mp|}})\right)
.
}
\end{theorem}

\begin{remark}\label{kramx}
One significance to this theorem is that it suggests a different method than the one in Remark \ref{kramy} to compute the family
$\underline{\grave{\mb{\pi}}}^0$.

For $n\geq 1$ let $\ell_n$ be the classical limit of $\bell_n$ --- note that
$\ell_1=K$. Then $\big(\sC, 1_\sC, \underline{\ell}\big)$ is a unital $sL_\infty$-algebra over $\Bbbk$.
Let $\phi^0_n$ be the classical limit of $\mb{\phi}^0_n$ for $n\geq 1$ --- note that
$\phi^0_1=f$.  
Then 
$\xymatrix{\underline{\phi}^0: \big(H, 1_H,\underline{0}\big)\ar@{..>}[r]& \big(\sC, 1_\sC, \underline{\ell}\big)}$
is a distinguished unital $sL_\infty$-quasi-isomorphism.  
Let $\phi^{-1}_n$ be the classical limit of $\mb{\phi}^{-1}_n$ for $n\geq 2$.
Then, from the classical limit of \eq{bsnb}, we obtain 
\eqn{lakiz}{
M^0_n = f\circ \grave{m}^0_n + K\circ \phi^{-1}_n,
} 
for $n\geq 2$, where
\begin{align*}
M^0_n(v_1,&\ldots, v_n)
\\
=
&
\sum_{\clapsubstack{\mp \in P(n)\\|\mp|=2\\ n-1\nsim_\mp n}}
\ep(\mp)
{\phi}\big({v}_{B_1}\big)\cdot {\phi}\big({v}_{B_{2}}\big)
-\sum_{\clapsubstack{\mp \in P(n)\\ \big|B_{|\mp|}\big|=n -|\mp|+1\\ n-1\sim_\mp n\\ \color{red} |\mp|\neq 1}}
\e(\mp)
{\phi}_{|\mp|}\Big({v}_{B_1},\ldots, {v}_{B_{|\mp|-1}}, \grave{{m}}^0({v}_{B_{|\mp|}})\Big)
\\
&
+\sum_{\clapsubstack{\mp \in P(n)\\ n-1 \sim_\mp n\\ \color{red} |\mp|\neq 1}}
\e(\mp)
\ell_{|\mp|}\left({\phi}\big(J{v}_{B_{1}}\big), \ldots,{\phi}\big(J{v}_{B_{|\mp|-1}}\big), \phi^{-1}\big({v}_{B_{|\mp|}}\big)\right)
.
\end{align*}
The relation \eq{lakiz} implies that  $K\circ M_n =0$ and $\grave{m}^0_n =h\circ M_n$ for all $n\geq 2$.
Therefore we can determine the family $\underline{\grave{m}}^0$, (and hence the
family  $\underline{\grave{\mb{\pi}}}^0$) via \eq{bsna},
by taking the classical $K$ cohomology class of the
family $\underline{M}^0$.  Note that $M_n \in \Hom\big(S^{n-2}H\otimes S^2H, \sC\big)^0$, while
$\grave{m}^0_n=h\circ M_n \in \Hom\big(S^{n}H, H\big)^0$.
\naturalqed
\end{remark}

\begin{remark}
Note that the recursive definition of $M_n$ depends only on the three families
$\big\{\phi^0_1,\ldots, \phi_{n-1}\big\}$, $\{\phi^{-1}_2,\ldots, \phi^{-1}_{n-1}\big\}$, and $\{\grave{m}^0_2,\ldots, \grave{m}^0_{n-1}\}$.
For example, we have
\begin{align*}
M^0_2(v_1,v_2)
=
&
\phi^0_1(v_1)\cdot \phi^0_1(v_2)
,\\
M^0_3(v_1,v_2,v_3)
=
&
\phi^0_2(v_1,v_2)\cdot \phi^0_1(v_3)
+(-1)^{|v_1||v_2|}\phi^0_1(v_2)\cdot \phi^0_2(v_1, v_3)
-\phi^0_2\big(v_1, \grave{m}^0_2(v_2,v_3)\big)
\\
&
+\ell_2\big(Jv_1, \phi^{-1}_2(v_2,v_3)\big)
,
\end{align*}
etc. 
Recall that the combined classical limit 
$\big(\sC, 1_\sC, \,\cdot\, ,\underline{\ell}\big)$ of a binary QFT algebra and its quantum descendant unital $sL_\infty$-algebra
is called a binary CFT algebra. (See Lemma \ref{cldesalg}). The unital CDGA part 
$\big(\sC, 1_\sC, \,\cdot\, K\big)$ induces the structure of a unital $\Z$-graded commutative and associative
algebra on $H$, whose product is exactly $\grave{m}^0_2$
---
from $\phi_1=f$, we have $M^0_2(v_1,v_2)= f(v_1)\cdot f(v_2)$ and $\grave{m}^0_2=h\circ M^0_2$
defines the $\Z$-graded commutative and associative algebra $(H, 1_H, \grave{m}^0_2)$.

Note  that
$f=\phi^0_1: (H, 1_H, \grave{m}^0_2)\rightarrow (\sC, 1_H, \,\cdot\,, K)$
 is both a pointed cochain quasi-isomorphism and an algebra homomorphism up to the homotopy $\phi^{-1}_2$ by \eq{lakiz}:
\eqn{whaa}{
\phi^0_1\big(\grave{m}^0_2(v_1,v_2)\big) -  \phi^0_1(v_1)\cdot\phi^0_1(v_2) = - K \phi^{-1}_2(v_1,v_2)
}
Note also that $(H, 1_H, \underline{0})$ is a unital $sL_\infty$-algebra with zero  $sL_\infty$-structure $\underline{0}$,
and
\eqn{whab}{
\ell_2\big(\phi^0_1(v_1), \phi^0_1(v_2) \big) = -K \phi^0_2(v_1,v_2).
}
Finally, note that $\ell_2$ is a derivation of the product $\cdot$. Therefore, we have the following
identity:
\begin{align*}
\ell_2\big(\phi^0_1(v_1), \phi^0_1(v_2)\cdot\phi^0_1(v_3) \big)
=&\ell_2\big(\phi^0_1(v_1), \phi^0_1(v_2)\big)\cdot\phi^0_1(v_3) 
\\
&
+(-1)^{|v_1||v_2|}\phi^0_1(v_2)\cdot\ell_2\big(\phi^0_1(Jv_1), \phi^0_1(v_3) \big).
\end{align*}
Substituting  \eq{whaa} and \eq{whab} into this equation, we obtain that $K\circ M_3=0$.
Therefore, we may call $\grave{m}^0_3=h\circ M_3$ the compatibility class between
the product $\cdot$ and the bracket $\ell_2$.
\naturalqed
\end{remark}

%
%
%
%Note that  Theorem \ref{intra} is a combination of 
%Theorems \ref{lakia} and \ref{lakib} and Theorem \ref{intrb}
%is a corollary of them.

\begin{theorem}\label{xdalem}
The  family $\underline{\grave{m}}^0=\grave{m}^0_2, \grave{m}^0_3,\ldots$ 
has the following properties:
\begin{itemize}
\item (symmetry): $\grave{m}^0_n \in \Hom\big(S^n H, H\big)^0$, for all $n\geq 2$;

\item (unity): $\grave{m}^0_2(1_H, v_1)=v_1$, while $\grave{m}^0_{n+1}(1_H, v_1,\ldots,v_n)=0$ for all $n\geq 2$;

\item (generalized associativity): for all $n\geq 0$
and homogeneous  $v_1,\ldots, v_n, w_1,w_2,w_3 \in H$, we have
\begin{align*}
\sum_{\mathclap{\vs\subset [n]}} &
\e(\vs\sqcup\vs^c)
\grave{m}^0\big(v_{\vs}\odot \grave{m}^0(v_{\vs^c}\odot w_1\odot w_2)\odot w_{3} \big)
\\
&=\sum_{\mathclap{\vs\subset [n]}}
\e(\vs\sqcup\vs^c)
(-1)^{|v_{\vs^c}||u_1|}
\grave{m}^0\big(v_{\vs}\odot w_{1}\odot \grave{m}^0(v_{\vs^c}\odot w_2\odot w_3) \big)
,
\end{align*}
where the sums are over all  subsets 
$\vs\subset [n]$ of  the  set $[n]=\{1,2,\ldots, n\}$,
as ordered sets, we use
$\vs^c$ to denote the complement of $\vs$, we use $|\vs|$ for the number of elements in $\vs$,
 we write $v_\vs = v_{j_1}\otimes \ldots \otimes  v_{j_{|\vs|}}$
if $\vs=\{j_1,\ldots, j_{|\vs|}\}$, $j_1 <\ldots< j_{|\vs|}$, 
and $\e(\vs\sqcup \vs^c)$ is notation for 
the Koszul sign for the permutation 
$\check{\s}\left(v_1\otimes \ldots \otimes v_n \right)=v_{\vs}\otimes v_{\vs^c}$.
\end{itemize}
\end{theorem}

\begin{proof}
The symmetry is obvious since $\grave{m}^0_n = \grave{\pi}^{0(n-2)}_n \in \Hom(S^nH, H)^0$.
We can also deduce that  $\grave{m}^0_2(v, 1_H)=v$
and $\grave{m}^0_{n+1}(v_1,\ldots, v_n, 1_H) =0$ if $n\geq 2$
from  the property that $\mb{\pi}^0_{n+1}(v_1,\ldots, v_n, 1_H) = \mb{\pi}^0_{n}(v_1,\ldots, v_n)$,
for all $n\geq 1$.
It remains to show generalized associativity, which follows from
 $\grave{\mb{\pi}}^0 \in \Hom\big(\Xbar{S}(H), H\big)^{0}{\kkbar}$.
Note that the set of relations in \eq{bsna} 
is equivalent to $\grave{\mb{\pi}}^0_2 =\grave{m}^0_2$ and,  for all $n\geq 0$,
\begin{align*}
\grave{\mb{\pi}}^0_{n+3} ({v}_1,&\dotsc,{v}_n, w_1,w_2,w_3)
\\
=
&
\sum_{\clapsubstack{\vs \subset [n] }}
\e(\vs\sqcup \vs^c)
(-\kbar)^{|\vs^c|+1}
\grave{\mb{\pi}}^0\left({v}_{\vs}\odot
\grave{m}^0\big({v}_{\vs^c}\odot w_1\odot w_2\odot w_3\big)\right)
\\
&
+\sum_{\clapsubstack{\vs \subset [n] }}
\e(\vs\sqcup \vs^c)
(-1)^{|w_1||v_{\vs^c}|}
(-\kbar)^{|\vs^c|}
\grave{\mb{\pi}}^0\left({v}_{\vs}\odot w_1\odot
\grave{m}^0\big({v}_{\vs^c}\odot w_2\odot w_3\big)\right)
.
\end{align*}
Then,
from $\grave{\mb{\pi}}^0_{n+3} ({v}_1,\dotsc,{v}_n, w_1,w_2,w_3)=(-1)^{|w_1||w_2|}
\grave{\mb{\pi}}^0_{n+3} ({v}_1,\dotsc,{v}_n, w_2,w_1,w_3)$
and the symmetry of $\underline{\grave{m}}^0$,
we obtain that
\begin{align*}
&
\sum_{\clapsubstack{\vs \subset [n] }}
\e(\vs\sqcup \vs^c)
(-1)^{|w_1||v_{\vs^c}|}
(-\kbar)^{|\vs^c|}
\grave{\mb{\pi}}^0\left({v}_{\vs}\odot w_1\odot
\grave{m}^0\big({v}_{\vs^c}\odot w_2\odot w_3\big)\right)
\\
&
=
(-1)^{|w_1||w_2|}
\sum_{\clapsubstack{\vs \subset [n] }}
\e(\vs\sqcup \vs^c)
(-1)^{|w_2||v_{\vs^c}|}
(-\kbar)^{|\vs^c|}
\grave{\mb{\pi}}^0\left({v}_{\vs}\odot w_2\odot
\grave{m}^0\big({v}_{\vs^c}\odot w_1\odot w_3\big)\right)
.
\end{align*}
Then the relation $\grave{\pi}^{0(n-2)}_n =\grave{m}^0_n$ implies that
\begin{align*}
\sum_{\clapsubstack{\vs \subset [n] }}&
\e(\vs\sqcup \vs^c)
(-1)^{|w_1||v_{\vs^c}|}
\grave{m}^0\left({v}_{\vs}\odot w_1\odot
\grave{m}^0\big({v}_{\vs^c}\odot w_2\odot w_3\big)\right)
\\
=&
(-1)^{|w_1||w_2|}
\sum_{\clapsubstack{\vs \subset [n] }}
\e(\vs\sqcup \vs^c)
(-1)^{|w_2||v_{\vs^c}|}
\grave{m}^0\left({v}_{\vs}\odot w_2\odot
\grave{m}^0\big({v}_{\vs^c}\odot w_1\odot w_3\big)\right)
.
\end{align*}
The right hand side of the above equality is equivalent to
\begin{align*}
&\quad{}\,(-1)^{|w_1|(|w_2|+|w_3|)}
\sum_{\clapsubstack{\vs \subset [n] }}
\e(\vs\sqcup \vs^c)
(-1)^{|w_2||v_{\vs^c}|}
\grave{m}^0\left({v}_{\vs}\odot w_2\odot
\grave{m}^0\big({v}_{\vs^c}\odot w_3\odot w_1\big)\right)
\\
&=
(-1)^{|w_1|(|w_2|+|w_3|)+|w_2||w_3|}
%\\
%\qquad\qquad\qquad\times
\sum_{\clapsubstack{\vs \subset [n]} }
\e(\vs\sqcup \vs^c)
(-1)^{|w_3||v_{\vs^c}|}
\grave{m}^0\left({v}_{\vs}\odot w_3\odot
\grave{m}^0\big({v}_{\vs^c}\odot w_2\odot w_1\big)\right)
\\
&=
(-1)^{(|w_1|+|w_2|)|w_3|}
\sum_{\clapsubstack{\vs \subset [n] }}
\e(\vs\sqcup \vs^c)
(-1)^{|w_3||v_{\vs^c}|}
\grave{m}^0\left({v}_{\vs}\odot w_3\odot
\grave{m}^0\big({v}_{\vs^c}\odot w_1\odot w_2\big)\right)
\\
&=
\sum_{\clapsubstack{\vs \subset [n] }}
\e(\vs\sqcup \vs^c)
\grave{m}^0\left({v}_{\vs}\odot
\grave{m}^0\big({v}_{\vs^c}\odot w_1\odot w_2\big)\odot w_3\right)
,
\end{align*}
so that we have generalized associativity.
\naturalqed
\end{proof}

\begin{definition}
We call the triple $\big(H, 1_H, \underline{\grave{m}}^0\big)$ the {\em on-shell quantum correlation algebra}
of the anomaly-free binary QFT algebra $\sC{\kkbar}_\BQ$
and call $\underline{\grave{\mb{\pi}}}^0$  the \emph{family of iterated quantum correlation products}
generated by $\underline{\grave{m}}^0$.
\end{definition}

A large portion of Section~\ref{subsec: homotopy h-divisibility 2} can be viewed as an algorithm
%,
%dubbed as the level one quantum master equation,
to determine
the distinguished family $\underline{\mb{\phi}}^{-1}$ in Theorem \ref{lakib} 
if we specialize to the case $\mb{\kappa}=0$.

\begin{theorem}
\label{lakic}
There are families $\underline{\grave{\mb{\pi}}}^{-1}=\grave{\mb{\pi}}^{-1}_2, \grave{\mb{\pi}}^{-1}_3,\ldots $
and $\underline{\mb{\eta}}^{-2} ={\mb{\eta}}^{-2}_2, {\mb{\eta}}^{-2}_3,\ldots$ 
such that,  for all $n\geq 2$,
\eqn{bsnz}{
\mb{\Pi}^{-1}_n= \mb{f}\circ \grave{\mb{\pi}}^{-1}_n+ \mb{K} \circ\mb{\eta}^{-2}_n,
}
where
\eqnalign{bsny}{
\mb{\Pi}^{-1}_n(v_1,&\ldots,v_n) 
\coloneqq{}
 \mb{\eta}^{-1}_n\big(v_1,\ldots,v_n\big)
\\
&
-\sum_{\clapsubstack{\mp \in P(n)\\  n-1\sim_\mp n }}
(-\kbar)^{n-|\mp|-1}
\e(\mp)
\mb{\phi}^0\big(J{v}_{B_1}\big)\cdots\mb{\phi}^0\big(J{v}_{B_{|\mp|-1}}\big)\cdot\mb{\phi}^{-1}({v}_{B_{|\mp|}})
\\
&
-\sum_{\clapsubstack{\mp \in P(n)\\|B_{|\mp|}|=n-|\mp|+1\\ n-1\sim_\mp n}}
(-\kbar)^{n-|\mp|-1}
\e(\mp)\; 
\mb{\eta}^{-1}_{|\mp|}\left( {v}_{B_1}, \cdots, {v}_{B_{\mp-1}},
\grave{{m}}^0\big({v}_{B_{|\mp|}}\big)\right)
,
}
and
\begin{itemize}

\item we have $\grave{\mb{\pi}}^{-1}_2 =0$
and, for all $n\geq 3$,
\[
\grave{\mb{\pi}}^{-1}_n =\sum_{\mathclap{j=0}}^{n-3}(-\kbar)^j \grave{\pi}^{-1(j)}_n,
\text{ where }
\grave{\pi}^{-1(j)}_n \in\Hom\big(S^{n-2} H\otimes S^2H,H)^0
\]
and $\grave{\mb{\pi}}^{-1}_{n}(v_1,\ldots, v_n)=0$ whenever $v_i=1_H$;

\item we have $\mb{\eta}^{-2}_2 =0$ and, for all $n\geq 3$,
\[
\mb{\eta}^{-2}_n  =\sum_{\mathclap{j=0}}^{n-3}(-\kbar)^j \eta^{-2(j)}_n,
\text{ where }
\eta^{-2(j)}_n \in \Hom(S^n H\otimes S^2H,\sC)^{-2}
\]
and $\mb{\eta}^{-2}_{n}(v_1,\ldots, v_{n},)=0$ whenever $v_i=1_H$. 
\end{itemize}

\end{theorem}

Note that $\mb{\Pi}^{-1}_n \in \Hom\big(S^{n-2}H\otimes S^2H, \sC\big)^{-1}{\kkbar}$, for all $n\geq 2$.
The relation in \eq{bsnz} implies that $\mb{K}\circ \mb{\Pi}^{-1}_n=0$, for all $n\geq 3$. It follows that
$\bm{\mc}\circ \mb{\Pi}^{-1}_n$ depends only on the homotopy type of the quantum expectation $\bm{\mc}$. 
We call the family $\underline{\mb{\Pi}}^{-1}= \mb{\Pi}^{-1}_2, \mb{\Pi}^{-1}_3,\ldots$ the {\em level one quantum correlators}
and $\bm{\mc}\circ \underline{\mb{\Pi}}^{-1}$ the family of {\em level one quantum correlation functions} with respect to $\bm{\mc}$.
The relation of \eq{bsnz} also implies that, for all $n\geq 2$,
\eqn{bsnx}{
\bm{\mc}\circ {\mb{\Pi}}^{-1}_n =\bm{\mc}\circ \mb{f}\circ  \grave{\mb{\pi}}^{-1}_n,
}
so that the $n$-fold level one quantum correlation function $\bm{\mc}\circ {\mb{\Pi}}^{-1}_n$ is 
determined by the on-shell quantum expectation $\bm{\mc}\circ \mb{f}$ and $\grave{\mb{\pi}}^{-1}_n$,
which is at most a degree $n-3$ polynomial in $\kbar$.
The physical interpretation of $\bm{\mc}\circ \underline{\mb{\Pi}}^{-1}$ remains elusive but we have presented
an algorithm to determine the family $\underline{\grave{\mb{\pi}}}^{-1}$.
With some effort, one can obtain Theorem \ref{lakib} as the appropriate integrability condition of this theorem.

\begin{remark}
We reproduce the algorithm to determine the families  $\underline{\mb{\phi}}^{-1}$, $\underline{\grave{\mb{\pi}}}^{-1}$ and
$\underline{\mb{\eta}}^{-2}$, which become much simpler due to the condition $\mb{\kappa}=0$.

Set  
\eqn{awhyx}{
\mb{\phi}^{-1}_2=\mb{\eta}^{-1}_2 
,\qquad
\grave{\mb{\pi}}^{-1}_2=0
,\qquad
\mb{\eta}^{-2}_2=0.
}
Note that $\mb{\Pi}^{-1}_2=\mb{\eta}^{-1}_2 - \mb{\phi}^{-1}_2=0$. Therefore, we have 
\eqn{awhya}{
\begin{cases}
\mb{\Pi}^{-1}_2= \mb{f}\circ \grave{\mb{\pi}}^{-1}_2+ \mb{K} \circ\mb{\eta}^{-2}_2
,\\
\mb{M}^0_2 =\mb{f}\circ \grave{m}^0_2 + \mb{K}\circ \mb{\phi}^{-1}_2
,\\
\grave{\mb{\pi}}^0_2 =\grave{m}^0_2
,
\end{cases}
}
where we used $\mb{\pi}^0_2=\grave{m}_2$ and
$\mb{M}^0_2(v_1,v_2) \coloneqq{}(-\kbar)\mb{\phi}^0_2(v_1,v_2) +\mb{\phi}^0_1(v_1)\cdot \mb{\phi}^0_1(v_2)
=\mb{\Pi}^0_2(v_1,v_2)$.

Assume that we have 
$\left\{{\mb{\phi}}^{-1}_2,\ldots, {\mb{\phi}}^{-1}_{n-1}\big|\grave{\mb{\pi}}^{-1}_2, \ldots, \grave{\mb{\pi}}^{-1}_{n-1}
\big| \mb{\eta}^{-2}_2, \ldots, \mb{\eta}^{-1}_{n-1}\right\}$ for $n\geq3$ satisfying the initial conditions of \eq{awhyx} and, 
for all $k=2,\ldots, n-1$,
\eqn{whyee}{
\begin{aligned}
\mb{\Pi}^{-1}_k &= \mb{f}\circ \grave{\mb{\pi}}^{-1}_k+\mb{K} \mb{\eta}^{-2}_k
,\\
\mb{M}^0_k&=\mb{f}\circ \grave{m}^0_k + \mb{K}\circ \mb{\phi}^{-1}_k
,\\
\grave{\mb{\pi}}_{k}({v}_1,\dotsc,{v}_k)
&=
\sum_{\clapsubstack{\mp \in P(k)\\|B_{|\mp|}|=k-|\mp|+1\\ k-1\sim_\mp k}}
(-\kbar)^{k-|\mp|-1}
\e(\mp)\; 
\grave{\mb{\pi}}_{|\mp|}\left({v}_{B_1}, \cdots, {v}_{B_{\mp-1}},
\grave{m}\big({v}_{B_{|\mp|}}\big)\right)
.
\end{aligned}
}
Define $\mb{\O}^{-1}_n \in \Hom\big(S^{n-2}H\otimes S^2H,\sC\big)^{-1}{\kkbar}$ and
$\grave{\mb{\varpi}}^0_n \in \Hom\big(S^{n-2}H\otimes S^2H, H\big)^{0}{\kkbar}$ as follows:
\begin{align*}
\mb{\O}^{-1}_n(v_1,\ldots,&v_n) 
\coloneqq{}
 \mb{\eta}^{-1}_n\big(v_1,\ldots,v_n\big)
\\
&
-\sum_{\clapsubstack{\mp \in P(n)\\  n-1\sim_\mp n \\ \color{red}|\mp|\neq 1}}
(-\kbar)^{n-|\mp|-1}
\e(\mp)
\mb{\phi}^0\big(J{v}_{B_1}\big)\cdots\mb{\phi}^0\big(J{v}_{B_{|\mp|-1}}\big)\cdot\mb{\phi}^{-1}({v}_{B_{|\mp|}})
\\
&
-\sum_{\clapsubstack{\mp \in P(n)\\|B_{|\mp|}|=n-|\mp|+1\\ n-1\sim_\mp n}}
(-\kbar)^{n-|\mp|-1}
\e(\mp)\; 
\mb{\eta}^{-1}_{|\mp|}\left( {v}_{B_1}, \cdots, {v}_{B_{\mp-1}},
\grave{{m}}^0\big({v}_{B_{|\mp|}}\big)\right)
,\\
\grave{\mb{\varpi}}^0_{n}({v}_1,\dotsc,&{v}_n)
\coloneqq{}
\grave{\mb{\pi}}^0_{n}({v}_1,\dotsc,{v}_n)
\\
&-
\sum_{\clapsubstack{\mp \in P(n)\\|B_{|\mp|}|=n-|\mp|+1\\ n-1\sim_\mp n\\ \color{red}|\mp|\neq 1}}
(-\kbar)^{n-|\mp|-1}
\e(\mp)\; 
\grave{\mb{\pi}}^0_{|\mp|}\left({v}_{B_1}, \cdots, {v}_{B_{\mp-1}},
\grave{{m}}^0\big({v}_{B_{|\mp|}}\big)\right)
.
\end{align*}
Note that we have
\[
\grave{\mb{\varpi}}^0_{n}=\sum_{\mathclap{i=0}}^{n-2}(-\kbar)^i \grave{{\varpi}}^{0(i)}_{n},
\quad\mathit{where}\quad
\grave{{\varpi}}^{0(i)}_{n} \in  \Hom\big(S^{n-2}H\otimes S^2H, H\big)^0
,
\]
and $\grave{{\varpi}}^{0(n-2)}_{n}=\grave{{\pi}}^{0(n-2)}_{n} =\grave{m}^0_n$.
%We shall see that $\grave{{\varpi}}^{0(j)}_{n}=0$ for $j=0,\ldots, n-3$.

From \eq{whyee} and the definition of $\bell_1,\ldots, \bell_{n}$,
we obtain that 
\eqn{awhyd}{
\mb{K}\circ\mb{\O}^{-1}_n=(-\kbar)^{n-2}\mb{M}^0_n -\mb{f}\circ \grave{\mb{\varpi}}^0_n
.
}
Note that the classical limit of this equation is
$K\circ {\O}^{-1}_n=-{f}\circ \grave{{\varpi}}^{0(0)}_n$, which implies
that $K\circ {\O}^{-1}_n=0$ and $\grave{{\varpi}}^{0(0)}_n=0$.
Therefore, we have
\eqn{awhye}{
\mb{K}\circ\mb{\O}^{-1}_n=(-\kbar)^{n-2}\mb{M}^0_n -\sum_{\mathclap{i=1}}^{n-2}(-\kbar)^{i}\mb{f}\circ 
 \grave{{\varpi}}^{0(i)}_{n}
.
}
Set $\left(\mb{\O}^{-1}_n\right)^{[0]} \coloneqq{}\mb{\O}^{-1}_n$ and 
$\left(\mb{\O}^{-1}_n\right)^{[1]}
\coloneqq{}  \nabla_{-1/\kbar}\left(\mb{\O}^{-1}_n\right)^{[0]} \in  \Hom\big(S^{n-2}H\otimes S^2H,\sC\big)^{-1}{\kkbar}$, 
so that we have
\[
(-\kbar)\left(\mb{\O}^{-1}_n\right)^{[1]} = \left(\mb{\O}^{-1}_n\right)^{[0]}
- \mb{f}\circ \grave{\pi}^{-1(0)}_n -\mb{K}\circ \eta^{-2(0)}_n
,\quad\mathit{where}\quad 
\begin{cases}
\grave{\pi}^{-1(0)}_n&\coloneqq{}h\circ  \left({\O}^{-1}_n\right)^{[0]}
,\\
 \eta^{-2(0)}_n &\coloneqq{}s\circ  \left({\O}^{-1}_n\right)^{[0]} 
.
\end{cases}
\]
From \eq{awhye}, we have
\eqn{awhyf}{
\mb{K}\circ\left(\mb{\O}^{-1}_n\right)^{[1]}=(-\kbar)^{n-3}\mb{M}^0_n -\sum_{\mathclap{i=1}}^{n-2}(-\kbar)^{i-1}\mb{f}\circ 
 \grave{{\varpi}}^{0(i)}_{n},
}
and the classical limit for $n >3$ is then 
$K\circ\left({\O}^{-1}_n\right)^{[1]}= -\mb{f}\circ  \grave{{\varpi}}^{0(1)}_{n}$,
which implies that $K\circ\left({\O}^{-1}_n\right)^{[1]}=0$ and $\grave{{\varpi}}^{0(1)}_{n}=0$.
Working inductively after setting
$
\left(\mb{\O}^{-1}_n\right)^{[i+1]}\coloneqq{} \nabla_{-1/\kbar}\left(\mb{\O}^{-1}_n\right)^{[i]},
$
we can check that
$K\circ\left({\O}^{-1}_n\right)^{[i]}=0$
and $\grave{{\varpi}}^{0(i)}_{n}=0$
for $i=0,\ldots, n-3$.
It follows that
\eqn{awhyg}{
\begin{aligned}
(-\kbar)^{n-2}\left(\mb{\O}^{-1}_n\right)^{[n-2]}
&=\mb{\O}^{-1}_n  
- \sum_{\mathclap{i=0}}^{n-3}(-\kbar)^i \mb{f}\circ\grave{\pi}^{-1(i)}_n 
-\sum_{\mathclap{i=0}}^{n-3}(-\kbar)^i  \mb{K}\circ {\eta}^{-2(i)}_n
,\\
\mb{K}\circ\left(\mb{\O}^{-1}_n\right)^{[n-2]}
&=
\mb{M}^0_n -\mb{f}\circ \grave{{m}}^{0}_{n}
,\\
\grave{{\varpi}}^{0(n-2)}_{n}&=\grave{m}^0_n
,
\end{aligned}
}
where
\[
\begin{aligned}
\grave{\pi}^{-1(i)}_n&\coloneqq{}h\circ  \left({\O}^{-1}_n\right)^{[i]}& \text{ in }&\Hom\big(S^{n-2}H\otimes S^2H,H\big)^{-1}
,\\
 \eta^{-2(i)}_n &\coloneqq{}s\circ  \left({\O}^{-1}_n\right)^{[i]}& \text{ in }& \Hom\big(S^{n-2}H\otimes S^2H,\sC\big)^{-2}
,
\end{aligned}
\]
and $\left({\O}^0_n\right)^{[i]} \in  \Hom\big(S^{n-2}H\otimes S^2H,\sC\big)^0$ denotes the classical
limit of $\left(\mb{\O}^0_n\right)^{[i]}$. 

Finally, we set
\eqn{xawhyf}{
\mb{\phi}^{-1}_n =\left(\mb{\O}^{-1}_n\right)^{[n-2]}
,\quad
\grave{\mb{\pi}}^{-1}_n=\sum_{\mathclap{i=0}}^{n-3}(-\kbar)^i\grave{\pi}^{-1(i)}_n
,\quad
\mb{\eta}^{-2}_n =\sum_{\mathclap{i=0}}^{n-3}(-\kbar)^i\eta^{-2(i)}_n
,
}
and note that  $\mb{\Pi}^{-1}_n = \mb{\O}^{-1}_n+(-\kbar)^{n-2}\mb{\phi}^{-1}_n$.
Then, from \eq{awhyg}, 
we conclude that
\begin{align*}
\mb{\Pi}^{-1}_n&= \mb{f}\circ \grave{\mb{\pi}}^{-1}_n+\mb{K} \mb{\eta}^{-2}_n
,\\
\mb{M}^0_n&=\mb{f}\circ \grave{m}^0_n + \mb{K}\circ \mb{\phi}^{-1}_n
,\\
\grave{\mb{\pi}}_{n}({v}_1,\dotsc,{v}_n)
&=
\sum_{\clapsubstack{\mp \in P(n)\\|B_{|\mp|}|=n-|\mp|+1\\ n-1\sim_\mp n}}
(-\kbar)^{n-|\mp|-1}
\e(\mp)\; 
\grave{\mb{\pi}}_{|\mp|}\left({v}_{B_1}, \cdots, {v}_{B_{\mp-1}},
\grave{m}\big({v}_{B_{|\mp|}}\big)\right)
.
\end{align*}
Therefore, we have a  well-defined algorithm to determine the families  
$\underline{\mb{\phi}}^{-1}$, $\underline{\grave{\mb{\pi}}}^{-1}$ and
$\underline{\mb{\eta}}^{-2}$.
\naturalqed
\end{remark}

Before leaving this subsection, we prove Theorem \ref{intrd}.

Consider an another binary QFT algebra $\sC^\pr{\kkbar}_\BQ=\big(\sC^\pr{\kkbar}, 1_{\sC^\pr}, \,\cdot^\pr\, , \mb{K}^\pr\big)$
which is homotopy equivalent to our binary QFT algebra $\sC{\kkbar}_\BQ$.
Realize this equivalence by a quasi-isomorphism
 $\xymatrix{\bm{\mf}: \sC{\kkbar}_\BQ \ar[r] & \sC^\pr{\kkbar}_\BQ}$ of binary QFT algebras.
% which is a homotopy equivalence.
It is obvious that $\sC^\pr{\kkbar}_\BQ$ is also anomaly-free, and we will use the induced  isomorphism
$H(\mf^{(0)}): H \rightarrow H^\pr$  to identify $H^\pr$ with $H$.
Write $\big(\sC^\pr{\kkbar}, 1_{\sC^\pr}, \underline{\bell}^\pr\big)$ for the quantum descendant, and
let $\xymatrix{\underline{\mb{\psi}}:\big(\sC{\kkbar}, 1_{\sC}, \underline{\bell}\big)\ar@{..>}[r]&
\big(\sC^\pr{\kkbar}, 1_{\sC^\pr}, \underline{\bell}^\pr\big)}$ be the quantum descendant
of $\bm{\mf}$ which is a quasi-isomorphism of the unital $sL_\infty$-algebras.
Then, by definition, we have 
\eqn{asdsit}{
\bm{\mf}\circ \mb{\pi} =\mb{\pi}^\pr\circ \mb{\Psi}_{\mb{\psi}}
,
}
where $\mb{\pi}(\bm{x}_1\bm{\odot}\ldots\bm{\odot}\bm{x}_n)\coloneqq{} \bm{x}_1\cdot\ldots\cdot\bm{x}_n$
and  $\mb{\pi}^\pr(\bm{x}^\pr_1\bm{\odot}\ldots\bm{\odot}\bm{x}^\pr_n)\coloneqq{} \bm{x}^\pr_1\cdot^\pr\ldots\cdot^\pr\bm{x}^\pr_n$
for all $n\geq 1$, $\bm{x}_1,\ldots, \bm{x}_n \in \sC{\kkbar}$ and  $\bm{x}^\pr_1,\ldots, \bm{x}^\pr_n \in \sC{\kkbar}$.
Consider the distinguished $sL_\infty$-quasi-isomorphism
$\xymatrix{\underline{\mb{\phi}}^0:
\big(H{\kkbar}, 1_H,  \underline{0}\big)\ar@{..>}[r] & \big(\sC{\kkbar}, 1_\sC,  \underline{{\bell}}\big)}$
of {Theorem} \ref{lakia}.   Recall that the definition of
the associated family $\underline{\mb{\Pi}}^0$ of level zero quantum correlators
is equivalent to  $\mb{\Pi}^0 \coloneqq{} \mb{\pi}\circ \mb{\Psi}_{\mb{\phi}^{0}} \in \Hom\left(\Xbar{S}(H) \sC\right)^0{\kkbar}$
and, from \eq{spmstzero}, we have
\eqn{asdsita}{
\mb{\pi}\circ \mb{\Psi}_{\mb{\phi}^{0}} =\mb{f}\circ \grave{\mb{\pi}}^0 + \mb{K}\circ \mb{\eta}^{-1}.
}
Now we consider the following 
unital $sL_\infty$-quasi-isomorphism 
\[
\xymatrix{\underline{\mb{\phi}}^{\pr 0} \coloneqq{}\underline{\mb{\psi}}\bullet \underline{\mb{\phi}}^{0}:
\big(H{\kkbar}, 1_{H}, \underline{0}\big)\ar@{..>}[r] & \big(\sC^\pr{\kkbar}, 1_{\sC^\pr}, \underline{\bell}^\pr\big)}
\]
and the associated family 
$\underline{\mb{\Pi}}^{\pr 0}$ of level zero quantum correlators,
whose defintion is equivalent to $\mb{\Pi}^{\pr 0} \coloneqq{} \mb{\pi}^\pr\circ \mb{\Psi}_{\mb{\phi}^{\pr 0}} 
\in \Hom\left(\Xbar{S}(H) \sC^\pr\right)^0{\kkbar}$.
From the combination of the equation $\mb{\Psi}_{\mb{\phi}^{\pr 0}} = \mb{\Psi}_{\mb{\psi}\bullet \mb{\phi}^0}
=  \mb{\Psi}_{\mb{\psi}}\circ \mb{\Psi}_{\mb{\phi}^0}$,
\eq{asdsit}, \eq{asdsita}, and $\bm{\mf}\circ \mb{K} =\mb{K}^\pr\circ \bm{\mf}$,
we obtain the following:
\begin{align*}
\mb{\Pi}^{\pr 0} 
=
&
 \mb{\pi}^\pr\circ \mb{\Psi}_{\mb{\psi}}\circ \mb{\Psi}_{\mb{\phi}^{0}}
= \bm{\mf}\circ \mb{\pi}\circ \mb{\Psi}_{\mb{\phi}^{0}}
= \bm{\mf}\circ \big( \mb{f}\circ \grave{\mb{\pi}} + \mb{K}\circ \mb{\eta}^{-1}\big)
=\bm{\mf}\circ  \mb{f}\circ \grave{\mb{\pi}}  +  \mb{K}^\pr \circ \bm{\mf}\circ\mb{\eta}^{-1}
\\
= 
&\bm{\mf}^\pr\circ \grave{\mb{\pi}}  +  \mb{K}^\pr \circ \mb{\eta}^{\pr-1},
\end{align*}
where $\bm{\mf}^\pr \coloneqq{} \bm{\mf}\circ  \mb{f}$ and  $\mb{\eta}^{\pr-1}\coloneqq{}\bm{\mf}\circ\mb{\eta}^{-1}$.
Therefore,
under the isomorphism $H{\kkbar}\cong H^\pr{\kkbar}$,
we can identify $\underline{\mb{\phi}}^{\pr 0} \coloneqq{}\underline{\mb{\psi}}\bullet \underline{\mb{\phi}}^{0}$
and $\mb{\Pi}^{\pr 0}$
with a distinguished unital $sL_\infty$-quasi-isomorphism  from 
$\big(H^\pr{\kkbar}, 1_{H^\pr}, \underline{0}\big)$ to $\big(\sC^\pr{\kkbar}, 1_{\sC^\pr}, \underline{\bell}^\pr\big)$
and the associated family of level zero quantum correlators. It follows that
a homotopy equivalence of anomaly-free binary QFT algebras induces an isomorphism of on-shell quantum correlation algebras
and, in particular, sends a quantum structure to  a quantum structure.

\subsection{Relations with the WDDV equation}

In the remaining part of this subsection, we further specialize  to
the case that $H$ is finite dimensional as a $\Z$-graded vector space. We shall see some interesting aspects
of our solutions. 

Assume that $H$ is finite dimensional. Choose homogeneous coordinates
$t_H =\{t^\a\}$ so that $\{\rd_\a=\rd /\rd t^\a\}$ form a basis of $H$ and $\rd_0= 1_H$. 
Extend $\{\rd_\a\}$ as a derivation of $\Bbbk[\![t_H]\!]\cong \widehat{S}(H^*)$, which is the completed symmetric
algebra generated by the dual $\Z$-graded vector space $H^*$.

Using the unital $sL_\infty$-quasi-morphism 
$\xymatrix{\underline{\mb{\phi}}^0:
\big(H{\kkbar}, 1_H,  \underline{0}\big)\ar@{..>}[r] & \big(\sC{\kkbar}, 1_\sC,  \underline{{\bell}}\big)}$
of Theorem \ref{lakia}, we
define 
\[
\mb{\Theta} 
\coloneqq{} \sum_{\mathclap{n=1}}^\infty\Fr{1}{n!}t^{\a_n}\cdots t^{\a_1} \mb{\phi}^0_n(\rd_{\a_1},\ldots, \rd_{\a_n})
\in \big(\Bbbk[\![t_H]\!]\widehat{\otimes} \sC\big)^{0}{\kkbar}
.
\]
Then, by the property that $\underline{\mb{\phi}}^0$ is a unital $sL_\infty$-morphism, we obtain that
\eqnalign{dashs}{
\mb{K}\mb{\Theta} 
+\sum_{\mathclap{n\geq 2}}\Fr{1}{n!}\bell_n\big(\mb{\Theta} ,\ldots,\mb{\Theta} \big)=0
&\Longleftrightarrow 
\mb{K} e^{-\fr{1}{\kbar}\mb{\Theta}} =0
,\\
\rd_0\mb{\Theta}=1_\sC
&\Longleftrightarrow 
(-\kbar)\rd_0 e^{-\fr{1}{\kbar}\mb{\Theta}} =e^{-\fr{1}{\kbar}\mb{\Theta}}
.
}
From the property that $\mb{\phi}_1:\big(H{\kkbar},1_H, 0\big)\rightarrow \big(\sC{\kkbar},1_\sC, \mb{K}\big)$
is a cochain quasi-isomorphism, we see that $\mb{\Theta}$ is a universal solution to the Maurer--Cartan
equation of the unital $sL_\infty$-algebra $\big(\sC{\kkbar}, 1_\sC,\underline{\bell}\big)$.
We also have the following identity:
\eqn{spfina}{
e^{-\fr{1}{\kbar}\mb{\Theta}} = 1_\sC 
+\sum_{\mathclap{n=1}}^\infty\Fr{1}{(-\kbar)^n}\Fr{1}{n!} t^{\r_n}\cdots t^{\r_1}\mb{\Pi}^0_n\big(\rd_{\r_1},\ldots,\rd_{\r_n}\big).
}
Therefore $e^{-\fr{1}{\kbar}\mb{\Theta}}$ is a generating function 
of the family $\underline{\mb{\Pi}}^0$ of level $0$ quantum correlators.
From the families $\underline{\grave{\mb{\pi}}}^0$ and $\underline{\mb{\eta}}^{-1}$
in Theorem \ref{lakia},
define
\begin{align*}
\grave{\bm{T}}^\g \coloneqq{}& t^\g + \sum_{\mathclap{n=2}}^\infty \Fr{1}{n! (-\kbar)^{n-1}}t^{\r_n}\cdots t^{\r_1} 
\grave{\mb{\pi}}^0_{\r_1\cdots\r_n}{}^\g 
\in \Bbbk[\![t_H]\!][\![\kbar^{-1}]\!]
,\\
\bm{\S} 
\coloneqq{}
&\sum_{\mathclap{n=2}}^\infty \Fr{1}{(-\kbar)^{n-1}}
\Fr{1}{n!}t^{\Xbar\r_n}\cdots t^{\Xbar\r_1} \mb{\eta}^{-1}_n(\rd_{\r_1},\ldots, \rd_{\r_n})
\in \big(\Bbbk[\![t_H]\!]\hat{\otimes} \sC\big)[\![\kbar^{-1}]\!]^{-1}
,
\end{align*}
where $\big\{ \grave{\mb{\pi}}^0_{\a_1\cdots\a_n}{}^\g\big\} \in \Bbbk{\kkbar}$ are structure constants, i.e., 
$\grave{\mb{\pi}}^0_n\big(\rd_{\a_1},\ldots,\rd_{\a_n}\big) = \grave{\mb{\pi}}^0_{\a_1\cdots\a_n}{}^\g \rd_\g$
and $t^{\Xbar\r} = (-1)^{\gh(\rd_\r)}t^\r$.
It is easy to check that the relation of \eq{spmstzero} is equivalent to the following identity:
\eqn{spfinb}{
e^{-\fr{1}{\kbar}\mb{\Theta}} 
= 1_\sC +\Fr{1}{(-\kbar)}\grave{\bm{T}}^\g \mb{f}(e_\g) 
+\Fr{1}{(-\kbar)}\mb{K}\bm{\S}
.
}
From the property that 
$\grave{\mb{\pi}}^0_{n+1}(v_1,\ldots, v_n, 1_H)= \grave{\mb{\pi}}^0_{n+1}(v_1,\ldots, v_n)$ for all $n\geq 1$, 
we also have
\eqn{spfinc}{
\rd_0 \grave{\bm{T}}^\g = \d_{0}{}^\g -\Fr{1}{\kbar} \grave{\bm{T}}^\g.
}
Note that \eq{spfinb} and \eq{spfinc} imply the relations in \eq{dashs}.
%\[
%\eqalign{
%1_\C +\Fr{1}{(-\kbar)}\Pi^0_1 +\Fr{1}{2(-\kbar)^2}\Pi^0_2+
%\ldots = 
%&
%1_\sC +\Fr{1}{(-\kbar)}\bm{t}^\g \mb{f}(e_\g) +\Fr{1}{2!(-\kbar)^2}t^\b t^\a \pi_{\a\b}{}^\g \mb{f}(e_\g) 
%+\ldots
%\\
%&
%+\Fr{1}{2!(-\kbar)^2}t^\b t^\a \eta^{-1}_{\a\b}
%++\Fr{1}{3!(-\kbar)^3}t^\g t^\b t^\a \eta^{-1}_{\a\b\g}
%}
%\]

We emphasis that both $\grave{\bm{T}}^\g$ and $\bm{\S}$ are formal power series in $\kbar^{-1}=1/\kbar$,
since both $\grave{\mb{\pi}}^0_n$ and  ${\mb{\eta}}^{-1}_n$ are polynomials in $\kbar$ with degree
at most $n-2$ for $n\geq 2$. These both follow from $\underline{\mb{\phi}}^0$ being a distinguished $sL_\infty$-quasi-morphism.

\begin{remark}
Let $\underline{\mb{\w}}:\big(H{\kkbar}, 1_H,  \underline{0}\big)\dasharrow \big(\sC{\kkbar}, 1_\sC,  \underline{{\bell}}\big)$,
be an arbitrarily chosen  $sL_\infty$-quasi-isomorphism. 
We may regard $\underline{\mb{\w}}$ as a universal homotopical family of quantum observables
--- let $\underline{\mb{\Pi}}^{\mb{\w}}$ be the associated family  of quantum correlators. 
Then we  also have a universal solution  to the Maurer--Cartan equation 
of the unital $sL_\infty$-algebra $\big(\sC{\kkbar}, 1_\sC,\underline{\bell}\big)$ given by 
$\mb{\Theta}^{\mb{\w}}\coloneqq{} \sum_{{n=1}}^\infty\Fr{1}{n!}t^{\a_n}\cdots t^{\a_1} \mb{\w}_n(\rd_{\a_1},\ldots, \rd_{\a_n})$, 
and $e^{-\fr{1}{\kbar}\mb{\Theta}}$ is a generating function of $\underline{\mb{\Pi}}^{\mb{\w}}$:
\[
e^{-\fr{1}{\kbar}\mb{\Theta}} = 1_\sC 
+\sum_{\mathclap{n=1}}^\infty\Fr{1}{(-\kbar)^n}\Fr{1}{n!} t^{\r_n}\cdots t^{\r_1}\mb{\Pi}^{\mb{\w}}_n\big(\rd_{\r_1},\ldots,\rd_{\r_n}\big).
\]
From $\mb{K}\circ \mb{\Pi}^{\mb{\w}}_n=0$ for all $n\geq 1$, we have 
$\mb{\Pi}^{\mb{\w}}_n\big(\rd_{\r_1},\ldots,\rd_{\r_n}\big) =
{\mb{\pi}}^{\mb{\w}}_{\r_1\cdots \r_n}{}^\g\mb{\w}_1(\rd_\g) + \mb{K}\mb{\eta}^{\mb{\w}}_{\r_1\cdots \r_n}$
for some ${\mb{\pi}}^{\mb{\w}}_{\r_1\cdots \r_n}{}^\g \in \Bbbk{\kkbar}$
and $\mb{\eta}^{\mb{\w}}_{\r_1\cdots \r_n}\in \sC^{-1}{\kkbar}$. 
Let 
\begin{align*}
\bm{T}_{\mb{\w}}^\g \coloneqq{}& t^\g + \sum_{\mathclap{n=2}}^\infty \Fr{1}{n! (-\kbar)^{n-1}}t^{\r_n}\cdots t^{\r_1} 
{\mb{\pi}}^{\mb{\w}}_{\r_1\cdots\r_n}{}^\g
,\\
\bm{\S}_{\mb{\w}}
\coloneqq{}&\sum_{\mathclap{n=2}}^\infty \Fr{1}{(-\kbar)^{n-1}}
\Fr{1}{n!}t^{\Xbar\r_n}\cdots t^{\Xbar\r_1} \mb{\eta}^{\mb{\w}}_n(\rd_{\r_1},\ldots, \rd_{\r_n})
.
\end{align*}
Then we have the identity $e^{-\fr{1}{\kbar}\mb{\Theta}^{\mb{\w}}} 
= 1_\sC +\Fr{1}{(-\kbar)}\bm{T}_{\mb{\w}}^\g \mb{\w}_1(e_\g) 
+\Fr{1}{(-\kbar)}\mb{K}\bm{\S}_{\mb{\w}}$.
On the other hand, $\bm{T}_{\mb{\w}}^\g$ is a formal Laurent series in $\kbar$ in general.

Let $\underline{\tilde{\mb{\w}}}\sim \underline{\mb{\w}}$ be another unital  $sL_\infty$-quasi-morphism homotopic
to $ \underline{\mb{\w}}$. Then $\mb{\Theta}^{\tilde{\mb{\w}}}$ is another universal solution  to the Maurer--Cartan equation
but is gauge equivalent to $\mb{\Theta}^{\mb{\w}}$
and we have $\bm{T}_{\mb{\w}}^\g = \bm{T}_{\tilde{\mb{\w}}}^\g$.
 A gauge equivalence class of such universal solutions can be viewed
as a choice of affine coordinates on the based formal moduli space $\sM_o$ defined by the Maurer--Cartan functor
of the unital $sL_\infty$-algebra.  Therefore, our distinguished  unital  $sL_\infty$-quasi-morphism $\underline{\mb{\phi}}^0$
defines a distinguished choice of affine coordinates on $\sM_o$, which we call the {\em quantum coordinates}.
\naturalqed
\end{remark}

From the family $\underline{\grave{m}}^0$ in Theorem \ref{lakib}, we define
\[
\grave{A}_{\a\b}{}^\g =  \grave{m}^0_{\a\b}{}^\g  + \sum_{\mathclap{n=1}}^\infty \Fr{1}{n!}t^{\r_n}\cdots t^{\r_1}
\grave{m}^0_{\a\b\r_1\cdots\r_n}{}^\g
\in \Bbbk[\![t_H]\!],
\]
where
$\big\{ \grave{m}^0_{\a_1\cdots\a_n}{}^\g\big\} \in \Bbbk$ is defined by
$\grave{m}^0_n\big(\rd_{\a_1},\ldots,\rd_{\a_n}\big) = \grave{m}^0_{\a_1\cdots\a_n}{}^\g \rd_\g$.
Then, we can check that the formula of \eq{bsna}, relating $\underline{\grave{m}}^0$ and $\underline{\grave{\mb{\pi}}}^0$, 
implies the following identity:
\eqn{spfind}{
\kbar\rd_\a\rd_\b \grave{\bm{T}}^\g - \grave{A}_{\a\b}{}^\r\rd_\r \grave{\bm{T}}^\g=0
.
}

From the family $\underline{\mb{\phi}}^{-1}$ in Theorem \ref{lakib}, we define
\[
\mb{\La}_{\b\g} 
\coloneqq{}
\sum_{\mathclap{n=0}}^\infty\Fr{1}{n!}t^{\Xbar\r_n}\cdots t^{\Xbar\r_1} \mb{\phi}^{-1}_{n+2}(\rd_\a,\rd_\b,\rd_{\r_1},\ldots, \rd_{\r_n})
\in \big(\Bbbk[\![t_H]\!]\hat{\otimes} \sC\big)^{\gh(\rd_\a)+\gh(\rd_\b)-1}{\kkbar}.
\]
Then, it can be checked that the relation of \eq{bsnb} is equivalent to the following identity:
\eqn{spfine}{
(-\kbar)^2\rd_\a \rd_\b  e^{-\fr{1}{\kbar}\mb{\Theta}}  
=(-\kbar) \grave{A}_{\a\b}{}^\g \rd_\g e^{-\fr{1}{\kbar}\mb{\Theta}} 
+ \mb{K}\left(e^{-\fr{1}{\kbar}\mb{\Theta}}\cdot\mb{\La}_{\a\b} \right)
.
}

From the property of the family $\underline{\grave{m}}^0$  in Theorem \ref{xdalem}, 
we obtain that   $\left\{\grave{A}_{\a\b}{}^\g\right\}$ has the following properties:
\begin{itemize}
\item unity: $\grave{A}_{0\b}{}^\g =\d_\b{}^\g$,

\item symmetry: 
$\grave{A}_{\a\b}{}^\g =(-1)^{|t^\a||t^\b|}\grave{A}_{\b\a}{}^\g$ and
$\rd_\a  \grave{A}_{\b\g}{}^\s =(-1)^{|t^\a||t^\b|}\rd_\b  \grave{A}_{\a\g}{}^\s$,

\item generalized associativity:
$\grave{A}_{\a\b}{}^\r \grave{A}_{\r \g}{}^\s=\grave{A}_{\b\g}{}^\r \grave{A}_{\a\r}{}^\s$,
\end{itemize}
so that
$\big(\Bbbk[\![t_H]\!], \rd_0, \diamond \big)$, where $\rd_\a\diamond \rd_\b  \coloneqq{}\grave{A}_{\a\b}{}^\g \rd_\g$,
is a unital super-commutative associative algebra over $\Bbbk[\![t_H]\!]$.
This kind of structure is related to that of a \emph{Frobenius manifold}.
The notion of a formal Frobenius super-manifold \cite{Dub,KoMa} originated in Saito's flat structure \cite{Saito} in the context of singularity theory 
and the WDVV equation \cite{DVV,WittenA} associated with topological string theories. 
The homology $H$ doesn't quite have the structure of a formal Frobenius super-manifold, but rather a formal $F$-manifold, which is the same thing as a formal Frobenius super-manifold without the invariant metric or inner product \cite{HerMa}. 
 % Hence, $H$
% is formal super $F$-manifold. 

\begin{example}
Let $\sL =\C[x_0,\ldots, x_N]$ and $S_{\cl} \in \sL$. Introduce $\eta_0,\ldots, \eta_n$ with $\gh=-1$ such
that $\eta_i\cdot  x_j =x_j\cdot \eta_i$ and $\eta_i\cdot\eta_j =-\eta_j\cdot \eta_i$.
Let $\sC =\Bbbk[x_0,\ldots, x_n, \eta_1,\ldots, \eta_n]$ which is a unital $\Z$-graded commutative associative algebra 
$\big(\sC, 1_\sC, \,\cdot\,\big)$ with $1_\sC =1$. Note that $\sC^0 =\sL$
and the ghost numbers of $\sC$ are concentrated in non-positive integers. 
Define the following $\Bbbk$-linear operators of ghost number $1$:
\[
 \Delta \coloneqq{}\sum_{\mathclap{i=0}}^N \Fr{\rd^2}{\rd \eta_i \rd x_i}
,\qquad
K\coloneqq{}\sum_{\mathclap{i=0}}^N \Fr{\rd S_{\cl}}{\rd x_i} \Fr{\rd}{\rd \eta_i}
.
\]
Then it is trivial that 
$\Delta S_{\cl} =0$ and $K_{\cl}\circ K_{\cl}= K_{\cl}\circ \Delta +\Delta\circ K_{\cl} =\Delta\circ \Delta=0$. 
Let $\mb{K}= -\kbar \Delta  + K_{\cl}$. Then,
$\sC_{\BQ}=\big(\sC{\kkbar}, 1_\sC, \,\cdot\,, \mb{K}\big)$ is a BV-QFT algebra over $\C$
with quantum descendant 
sDGLA $\big((\sC{\kkbar}, 1_\sC, \mb{K}, (-,-)_{\mathit{BV}}\big)$,
where $(\a_1, \a_2)_{\mathit{BV}}= \Delta(\a_1\cdot \a_2) - \Delta\a\cdot \a_2 -(-1)^{|\a_1|}\a_1\cdot\Delta \a_2$,
$\a_1,\a_2 \in \sC$.

Assume that $S_{\cl} \in \sL$ has isolated singularities, which implies that 
\[
H=H^0\cong \Bbbk[x_0,\ldots, x_N]\Big/\left(\Fr{\rd S_{\cl}}{\rd x_i}\right)
\] 

and $H^0$ is finite dimensional.
Note that $1_\sC=1$ is non-trivial in cohomology. We denote its cohomology class by $1_H$.
Let $\{e_\a\}$ be a basis of  $H$ such that $e_0=1_H$. Choose a representative $f(e_\a) \in \sC^0\subset \sC$ such that
$f(e_0)=1$ and extend linearly over $H=H^0$. Then $f:(H, 1_H, 0)\rightarrow (\sC, 1_\sC, K)$ is a quasi-isomorphism.
Note that $\Im f \subset \sC^0$, where $\Delta$ vanishes. Therefore we have $K\circ f =\Delta\circ f=0$. Hence
$\mb{K}\circ f =0$ and $\mb{f}=f:(H{\kkbar}, 1_H, 0)\rightarrow (\sC{\kkbar}, 1_\sC, \mb{K})$ is a quasi-isomorphism.
Therefore the BV-QFT algebra $\sC_\BQ$ is anomaly-free with finite dimensional classical cohomology.
Then our construction reproduces the $F$-manifold structure on $H=H^0$, the space of the universal unfolding of the isolated
singularities, equivalent to that of K.\ Saito after forgetting the flat metric.
His theory also contains
integrals over the vanishing cycles, which correspond to quantum expectations.
\naturalqed
\end{example}

\begin{example}
Now we consider the construction  in \cite{BaKo}.
Let  $(X, \omega^{n,0})$ be a complex $n$-dimensional Calabi--Yau manifold.
Let $T_X$ be the holomorphic tangent bundle to $X$,  $\Xbar{T}^{*}_X$ be the anti-holomorphic
cotangent bundle to $X$, and
\[
\sC =\bigoplus_{k=-n}^n \sC^k
, \qquad 
\sC^k=\bigoplus_{\clapsubstack{q-p=k\\ q,p =0,\ldots,n}} \G\left(\wedge^p T_X \otimes \wedge^q \Xbar{T}^{*}_X\right)
\]
Note that $\big(\sC, 1, \wedge,\Xbar\rd\big)$ is a unital CDGA over $\C$.
From the differential $\rd$ and  $\omega^{n,0}$ define $\Delta: \sC^\bullet \rightarrow \sC^{\bullet +1}$ by
the formula $(\Delta \g )\vdash \omega^{n,0} =\rd (\g\vdash \omega^{n,0})$. 
Then, we have $\Delta\circ \Delta=\Delta\circ \Xbar\rd +\Xbar\rd \circ \Delta =0$ so that 
$\sC{\kkbar}_\BQ=\big(\sC{\kkbar}, 1, \wedge, \mb{K} =-\kbar \Delta +\Xbar\rd\big)$ is a BV QFT algebra with 
quantum descendant unital sDGLA $\big(\sC{\kkbar}, 1_\sC, \mb{K}, (\ ,\,)_{SN}\big)$. 
Here $(\g_1, \g_2)_{\mathit{SN}}= \Delta(\g_1\wedge \g_2) - \Delta\g\wedge \g_2 -(-1)^{|\g_1|}\g_1\wedge\Delta \g_2$
and is equivalent to the holomorphic Schoutens--Nijenhuis bracket.\footnote{
We remark that our grading conventions differ from those in \cite{BaKo}; the differences are not important.}

Note that the classical cohomology $H$ is the $\Xbar\rd$-cohomology.
 Then  the $\rd\Xbar\rd$-lemma for K\"ahler manifolds \cite{DGMS} implies
that every $\Xbar\rd$-cohomology class has a unique representative belonging to the kernel of $\Delta$.
Choose a basis $\{e_\a\}$ of $H$ such that $e_0=1_H$, and choose representatives $f(e_\a)$ satisfying $\Delta f(e_\a)=0$.
Then we have $K\circ f =\Delta\circ f=0$. Thus
$\mb{K}\circ f =0$ and $\mb{f}=f:(H{\kkbar}, 1_H, 0)\rightarrow (\sC{\kkbar}, 1_\sC, \mb{K})$ is a quasi-isomorphism.
Therefore the BV-QFT algebra $\sC{\kkbar}_\BQ$ is anomaly-free with finite dimensional classical cohomology
$H$.
Then our construction reproduces the super $F$-manifold structure on $H$, isomorphic to the Dolbeault cohomology of $X$, equivalent to the formal Frobenius supermanifold
of Barannikov--Kontsevich after forgetting the invariant flat metric.
Their theory also contains period integrals over the middle dimensional homology cycles on $X$, 
which correspond to quantum expectations.
\naturalqed
\end{example}

%
%\begin{example}
%
%Assume that $S_\cl = x_0\cdot G(x_1,\ldots, x_n)$, where
%\[
%\sum_{\mathclap{i=1}}^n q_i x_i\Fr{\rd}{\rd x_i}G(x_1,\ldots, x_n)= d\cdot G(x_1,\ldots, x_n), \qquad q_i \in \Z.
%\]
%Assume that $G(x_1,\ldots, x_n)$ is a generic polynomial such that $\Fr{\rd G}{\rd x_1}=\ldots=\Fr{\rd G}{\rd x_n}=0$ is isolated
%at the origin. Also assume that $q_1+\ldots + q_n = d$. Then the complex zeros of $G$ modulo the $\mathbb{C}^*$-action
%$\l.(x_1,\ldots, x_n)= \big(\l^{q_1}x_1,\ldots, \l^{q_n}x_n\big)$, $\l \in \mathbb{C}^*$ is a Calabi--Yau hypersurface
%$X_G \subset \mathbb{CP}_{q_1,\ldots, q_n}$ in the weight projective space.
%Then $H = H^0 \oplus H^{-1}$, where $H^0$ is isomorphic to the primitive  middle dimensional cohomology of $X_G$
%and $H^{-1}$ is one dimensional spanned by the cohomology class of 
%\[
%\l=-d\cdot x_0\eta_0 + q_1\cdot x_1\eta_1 +\ldots + q_n\cdot x_n\cdot \eta_n. 
%\]
%Note that $K_\cl \l = -d\cdot x_0\cdot G + q_1\cdot x_1\cdot x_0\Fr{\rd G}{\rd x_1}+ q_n\cdot x_n\cdot x_0\Fr{\rd G}{\rd x_n}=0$
%and $\Delta \l = -d + q_1+\ldots +q_n=0$.
%Therefore the QFT algebra $\sC_\BQ$ is anomaly-free with finite dimensional classical cohomology.
%\end{example}

\begin{remark}
One can check that \eq{spfind}, viewed as a formal differential equation for $\{\grave{\bm{T}}^\g\}\in \Bbbk[\![t_H]\!][\![\kbar^{-1}]\!]$,
has a unique solution with the initial conditions 
$\grave{\bm{T}}^\g\big|_{t_H=0}=0$ and $\rd_\b \grave{\bm{T}}^\g\big|_{t_H=0}=\d_\b{}^\g$.
The equation \eq{spfinc} has the following
integrability condition:
\[
\left(\kbar\left(\rd_\a  \grave{A}_{\b\g}{}^\s -(-1)^{|t^\a||t^\b|}\rd_\b  \grave{A}_{\a\g}{}^\s\right)  
+\grave{A}_{\b\g}{}^\r \grave{A}_{\a\r}{}^\s -\grave{A}_{\a\g}{}^\r \grave{A}_{\b\r}{}^\s\right)\rd_\s \grave{\bm{T}}^\g=0.
\]
Note that the condition $\rd_\b \grave{\bm{T}}^\g\big|_{t_H=0}=\d_\b{}^\g$ implies that the matrix $\grave{\bm{\sG}}$
with entries $\grave{\bm{\sG}}_\b{}^\g=\rd_\b \grave{\bm{T}}^\g$  is invertible, so that 
we have
\[
\kbar\left(\rd_\a  \grave{A}_{\b\g}{}^\s -(-1)^{|t^\a||t^\b|}\rd_\b \grave{A}_{\a\g}{}^\s\right)  
+\grave{A}_{\b\g}{}^\r \grave{A}_{\a\r}{}^\s - (-1)^{|t^\a||t^\b|}\grave{A}_{\a\g}{}^\r\grave{A}_{\b\r}{}^\s=0.
\]
It follows that
$\rd_\a \grave{A}_{\b\g}{}^\s =(-1)^{|t^\a||t^\b|}\rd_\b  \grave{A}_{\a\g}{}^\s$ and
\[
\grave{A}_{\b\g}{}^\r \grave{A}_{\a\r}{}^\s - (-1)^{|t^\a||t^\b|}\grave{A}_{\a\g}{}^\r \grave{A}_{\b\r}{}^\s=0 \Longleftrightarrow 
\grave{A}_{\a\b}{}^\r \grave{A}_{\r \g}{}^\s-\grave{A}_{\b\g}{}^\r \grave{A}_{\a\r}{}^\s=0,
\]
since $\grave{A}_{\a\b}{}^\g$ does not depend on $\kbar$.
\naturalqed
\end{remark}

From the families $\underline{\grave{\mb{\pi}}}^{-1}$ and $\underline{\mb{\eta}}^{-2}$
in Theorem \ref{lakic}, we
define
\[
\grave{\bm{U}}_{\a\b}{}^\g \coloneqq{} \sum_{\mathclap{n=1}}^\infty \Fr{1}{n! (-\kbar)^{n}}t^{\Xbar\r_n}\cdots t^{\Xbar\r_1} 
\grave{\mb{\pi}}^{-1}_{\a\b\r_1\cdots\r_n}{}^\g 
\]
in $\Bbbk[\![t_H]\!][\![\kbar^{-1}]\!]^{-1}$ and 
\[
\bm{\Xi}_{\a\b}
\coloneqq{}
\sum_{\mathclap{n=1}}^\infty \Fr{1}{(-\kbar)^{n}}\Fr{1}{n!}t^{\r_n}\cdots t^{\r_1} 
\mb{\eta}^{-2}_{n+2}(\rd_\a,\rd_\b, \rd_{\r_1},\ldots, \rd_{\r_n})
\] in
$\big(\Bbbk[\![t_H]\!]\hat\otimes \sC\big)[\![\kbar^{-1}]\!]^{\gh(\rd_\a)+\gh(\rd_\b)-2}$, 
where $\big\{ \grave{\mb{\pi}}^{-1}_{\a_1\cdots\a_n}{}^\g\big\} \in \Bbbk{\kkbar}$ is defined for $n\geq 3$ by
$\grave{\mb{\pi}}^{-1}_n\big(\rd_{\a_1},\ldots,\rd_{\a_n}\big) = \grave{\mb{\pi}}^{-1}_{\a_1\cdots\a_n}{}^\g \rd_\g$.
We emphasize that both $\grave{\bm{U}}_{\a\b}{}^\g$ and $\bm{\Xi}_{\a\b}$ 
are formal power series in $\kbar^{-1}=1/\kbar$.
Then, the relation of \eq{bsnz} is equivalent to the following identity:
\eqn{spfinf}{
(-\kbar)\rd_{\Xbar\a} \rd_{\Xbar \b} \bm{\S}
-e^{-\fr{1}{\kbar}\mb{\Theta}}\cdot\mb{\La}_{\a\b}
- \grave{A}_{\a\b}{}^\r\rd_{\Xbar \r} \bm{\S}
= \grave{\bm{U}}_{\a\b}{}^\g \mb{f}(e_\g) + \mb{K}\bm{\Xi}_{\a\b}.
}

\begin{remark}
Apply $\mb{K}$ to \eq{spfinf} to obtain that
\[
\left((-\kbar)\rd_{\a} \rd_{\b}
- \grave{A}_{\a\b}{}^\r\rd_{\r}\right) \mb{K}\bm{\S}
=\mb{K}\left(e^{-\fr{1}{\kbar}\mb{\Theta}}\cdot\mb{\La}_{\a\b}\right).
\]
From  \eq{spfinb} we know that $\mb{K}\bm{\S} =(-\kbar)e^{-\fr{1}{\kbar}\mb{\Theta}} 
-(-\kbar)1_\sC -\grave{\bm{T}}^\g \mb{f}(e_\g)$, so we have
\begin{align*}
(-\kbar)^2\rd_\a\rd_\b e^{-\fr{1}{\kbar}\mb{\Theta}}  
-(-\kbar)\grave{A}_{\a\b}{}^\r\rd_\r e^{-\fr{1}{\kbar}\mb{\Theta}} 
-\mb{K}\left(e^{-\fr{1}{\kbar}\mb{\Theta}}\cdot\mb{\La}_{\a\b}\right)
&
\\
=
 \left((-\kbar)\rd_\a\rd_\b\grave{\bm{T}}^\g -\grave{A}_{\a\b}{}^\r\rd_\r \grave{\bm{T}}^\g\right)\mb{f}(e_\g)
,&
\end{align*}
which implies that
the relations \eq{spfind} and  \eq{spfine} arise as integrability conditions of the relation \eq{spfinf}.
\naturalqed
\end{remark}

\begin{remark}
For any quantum expectation $\bm{\mc}$, define the following generating series of the level zero quantum correlation
functions:
\[
\bm{\CZ}_{\bm{\mc}} 
\coloneqq{}\bm{\mc}\left(e^{-\fr{1}{\kbar}\mb{\Theta}} \right) \equiv \left\langle e^{-\fr{1}{\kbar}\mb{\Theta}} \right\rangle_{\bm{\mc}} 
= 1+\sum_{\mathclap{n=1}}^\infty\Fr{1}{(-\kbar)^n}\Fr{1}{n!} t^{\a_n}\cdots t^{\a_1}
\left\langle \mb{\Pi}^0_n\big(\rd_{\a_1},\ldots,\rd_{\a_n}\big)\right\rangle_{\bm{\mc}}.
\]
From \eq{spfinb},  \eq{spfinc} and \eq{spfind}, we obtain that
\[
\bm{\CZ}_{\bm{\mc}}  = 1 +\Fr{1}{(-\kbar)}\bm{T}^\g \left\langle \mb{f}(e_\g)\right\rangle_{\bm{\mc}}
,\quad
-\kbar\rd_0  \bm{\CZ}_{\bm{\mc}}= \bm{\CZ}_{\bm{\mc}}
,\quad
-\kbar\rd_\a\rd_\b \bm{\CZ}_{\bm{\mc}} = \grave{A}_{\a\b}{}^\r\rd_\r \bm{\CZ}_{\bm{\mc}}
.
\]
We remind the reader that $\bm{\CZ}_{\bm{\mc}} \in \Bbbk[\![t_H]\!](\!(\kbar)\!)^0$.
Assume that the above quantum expectation  $\bm{\mc}$ is not just a pointed cochain map
from $\big(\sC{\kkbar}, 1_\sC, \mb{K}\big)$ to $\big(\Bbbk{\kkbar},1, 0\big)$ but also a morphism
of binary QFT algebras.  Let $\underline{\mb{\chi}} =\bm{\mK}(\bm{\mc})$ be the
quantum descendant of $\bm{\mc}$. Define $\bm{\CF}_{\bm{\mc}} \in \Bbbk[\![t_H]\!]^{0}{\kkbar}$
be defined by
\[
\bm{\CF}_{\bm{\mc}}\coloneqq{}\sum_{\mathclap{n=1}}^\infty\Fr{1}{n!}t^{\a_n}\cdots t^{\a_1} 
\big(\underline{\mb{\chi}}\bullet \underline{\mb{\phi}}^0\big)_n \big(\rd_{\a_1},\ldots,\rd_{\a_n}\big)
.
\]
Then, we have the identity $e^{-\Fr{1}{\kbar}\bm{\CF}_{\bm{\mc}}} = \bm{\CZ}_{\bm{\mc}}$.
\naturalqed
\end{remark}

Finally, we briefly return to a general binary QFT algebra without assuming either the anomaly-free condition
or the finite dimensionality of $H$.  

Consider the binary QFT algebra $\big(\sC{\kkbar},1_\sC,\,\cdot\,,\mb{K})$
along with its quantum descendant $\big(\sC{\kkbar},1_\sC,\underline{\bell}\big)$ as
in Sect.~\ref{section: mastering qc}.
Let 
$\left\{\grave{\mb{\pi}}^{0}, \mb{\eta}^{-1},\grave{\bell},\mb{\phi}^{0}
\big|\grave{\mb{\pi}}^{-1}, \mb{\eta}^{-2},\grave{\mb{m}},\mb{\phi}^{-1}\right\}$ 
be the canonical solutions of the levels zero and one quantum master equations given in Theorems
\ref{solvemstzero} and \ref{solvemstone}.

\begin{definition}
A  {\em superselection sector}   of the binary QFT algebra
is a finite-dimensional pointed subspace $W$ of $H$
such that
\begin{enumerate}[label = \roman*., ref=\roman*]
\item (unitality) $1_W\coloneqq{}1_H \in W$, 
\label{ss property 1}
\item ($\mb{\kappa}$-triviality) $\mb{\kappa}w =0$ for all $w \in W$, 
\label{ss property kappa}
\item (level zero higher triviality) $\grave{\mb{\pi}}^0_n(w_1,\ldots, w_n) \in W$
for all $n\geq 2$ and $w_1,\ldots, w_n \in W$, and 
\label{ss property pi0}
\item (level one higher triviality) $\grave{\mb{\pi}}^{-1}_n(w_1,\ldots, w_n) \in W$
for all $n\geq 3$ and $w_1,\ldots, w_n \in W$. 
\label{ss property pi-1}
\end{enumerate}
\end{definition}
From properties~\eqref{ss property kappa} and~\eqref{ss property pi0} of $W$, we have
$\grave{\bell}_n(w_1,\ldots, w_n) =0$ for all $n\geq 1$ and $w_1,\ldots, w_n \in W$
so that $\underline{\mb{\phi}}^{0}:\big(W{\kkbar}, 1_W, \underline{0}\big) \dasharrow
\big(\sC{\kkbar}, 1_\sC, \underline{\bell}\big)$ is a unital $sL_\infty$-morphism.
From properties~\eqref{ss property kappa} and~\eqref{ss property pi-1} of $W$, we obtain that, for all $n\geq 2$ and homogeneous $w_1,\ldots, w_n \in W$,
\[
\grave{\mb{\pi}}^0_{n} ({w}_1,\dotsc,{w}_n)
=
\sum_{\clapsubstack{\mp \in P(n)\\|B_{|\mp|}|=n-|\mp|+1\\ n-1\sim_\mp n  }}
(-\kbar)^{n-|\mp|-1}
\e(\mp)\; 
\grave{\mb{\pi}}^0_{|\mp|}\left({w}_{B_1}, \cdots, {w}_{B_{|\mp|-1}},
\grave{\mb{m}}^0\big({v}_{B_{|\mp|}}\big)\right)
.
\]
It follows that $\grave{\mb{m}}^0_n(w_1,\ldots, w_n) = \grave{\pi}^{0(n-2)}_n(w_1,\ldots, w_n)$
for all $n\geq 2$ and homogeneous $w_1,\ldots, w_n \in W$. Define 
the family $\underline{\mathring{m}}^{0}=\mathring{m}^{0}_2,\mathring{m}^{0}_3,\ldots$
of $\mathring{m}^{0}_n \in \Hom\big(S^n W, W\big)^0$ for all $n\geq 2$
to be $\mathring{m}^{0}_n(w_1,\ldots, w_n) \coloneqq{}  \grave{\pi}^{0(n-2)}_n(w_1,\ldots, w_n)$. 
It is clear that $\big(W, 1_W, \underline{\mathring{m}}^{0}\big)$ is unital, symmetric, 
and satisfies generalized associativity.

Introduce homogeneous coordinates  $t_W=\{t^a\}$ on $W$ so that
$\{\rd_a =\rd/\rd\! t^a\}$ form a homogeneous basis of $W$ with distinguished element $\rd_0=1_W$.
Then extend $\rd_a$ as a derivation on $\Bbbk[\![t_W]\!]$. Consider the structure constants 
$\left\{ \mathring{m}_{a_1\cdots a_n}{}^\g\right\}$ and  $\left\{\mathring{\mb{\pi}}_{\a_1\cdots \a_n}{}^\g\right\}$, where
$\mathring{m}^0_n\big(\rd_{a_1}, \ldots, \rd_{a_n}\big) =\mathring{m}_{a_1\cdots a_n}{}^c \rd_c$ and
$\grave{\mb{\pi}}^0_n\big(\rd_{a_1}, \ldots, \rd_{a_n}\big) = \mathring{\mb{\pi}}_{a_1\cdots a_n}{}^c \rd_c$,
and define
\begin{align*}
{\mb{\Theta}}_W =&\sum_{\mathclap{n=1}}^\infty \Fr{1}{n!}t^{a_n}\cdots t^{a_1}
\mb{\phi}_n\big(\rd_{a_1},\ldots, \rd_{a_n}\big)
\in \big(\sC[\![t_W]\!]\big)^{0}{\kkbar}
,\\
\mathring{\bm{T}}^c = &t^c + \sum_{\mathclap{n=2}}^\infty \Fr{1}{n! (-\kbar)^{n-1}}t^{\a_n}\cdots t^{\a_1} 
\mathring{\mb{\pi}}_{a_1\cdots a_n}{}^c 
\in \Bbbk[\![t_W]\!][\![\kbar^{-1}]\!]
,\\
\mathring{A}_{\a\b}{}^c = & \mathring{m}^0_{a b}{}^c  + \sum_{\mathclap{n=1}}^\infty \Fr{1}{n!}t^{a_n}\cdots t^{a_1}
\mathring{m}^0_{a_1\cdots a_n ab}{}^c
\in \Bbbk[\![t_W]\!]
.
\end{align*}
Then,  we have $\mb{K}e^{-\Fr{1}{\kbar}\mb{\Theta}_W} =0$
and the generating function $\bm{\CZ}^W_{\bm{\mc}}$ of quantum correlation functions 
on the superselection sector $W$ with respect to a quantum expectation $\bm{\mc}$ 
can be defined as follows:
\[
\bm{\CZ}^W_{\bm{\mc}}
\coloneqq{} \Big< e^{-\Fr{1}{\kbar}\mb{\Theta}_W}\Big>_{\bm{\mc}}
= 1+\sum_{\mathclap{n=1}}^\infty \Fr{1}{n! (-\kbar)^n}t^{a_n}\cdots t^{a_1} 
\Big<\mb{\Pi}_n^W(\rd_{a_1},\ldots, \rd_{a_n})\Big>_{\bm{\mc}}.
\]
It follows that $\bm{\CZ}^W_{\bm{\mc}} = 1 -\Fr{1}{\kbar} \grave{\bm{T}}^a \big< \mb{f}(\rd_a)\big>_{\bm{\mc}}$
so that the quantum expectation values (namely $\left\{ \big< \mb{f}(\rd_a)\big>_{\bm{\mc}}\right\}$ and 
$\left\{\mathring{\bm{T}}^a\right\}$) determine every quantum correlation function in the superselection sector $W$.
Moreover, 
$\left\{\mathring{\bm{T}}^a\right\}$ is the unique solution in formal power series in $\kbar^{-1}$ 
to the following system of formal differential equations:
\[
\kbar\rd_b\rd_c \mathring{\bm{T}}^a + \mathring{A}_{bc}{}^e\rd_e \mathring{\bm{T}}^a=0
,\qquad
\rd_0 \mathring{\bm{T}}^a= \d_{0}{}^a -\Fr{1}{\kbar} \mathring{\bm{T}}^a
,
\]
with the boundary condition $\rd_b \mathring{\bm{T}}^a\big|_{t_W=0}=\d_b{}^a$.
Finally, $\left\{\mathring{A}_{ab}{}^c\right\}$ has the requisite compatibilities to make $\big(W\otimes\Bbbk[\![t_W]\!], \rd_0, \diamond\big)$, where $\rd_a\diamond \rd_b \coloneqq{} \mathring{A}_{ab}{}^c \rd_c$,
into a unital super-commutative associative algebra over $\Bbbk[\![t_W]\!]$ --- hence $W$ is a formal super $F$-manifold.
Therefore, we conclude the following.
\begin{lemma}\label{intrac}
Every finite dimensional super-selection sector of a binary QFT algebra is a formal super $F$-manifold
and comes with a well-defined quantum distribution.
\end{lemma}

\begin{appendix}
\section{On homotopy Lie algebras}
\label{appendix: supersection}

This appendix is intended as a self-contained introduction to the homotopy category of $sL_\infty$-algebras
and a homotopy functorial description of the bar construction of  $sL_\infty$-algebras.

An $sL_\infty$-algebra is a version of a homotopy Lie algebra,
also known as an $L_\infty$-algebra, with its degree shifted by one. The notion of an $L_\infty$-algebra first arose in the rational homotopy theory of Sullivan
\cite{Sullivan} in disguise before its  explicit form was used  in the deformation theoretic approach \cite{Stasheff,SS} to
rational homotopy theory.  For our purposes the shifted version, $sL_\infty$-algebras, are more natural.
There is no  conceptual originality to the contents of this appendix, and we refer to \cite{Konst}
for the formal super-geometric aspects of $L_\infty$-algebras.

%This appendix is a summary of algebraic package from (commutative) homotopy probability theory, whose
%variant shall be used in the main body of this paper.
%This package includes a self-contained introduction to the homotopy category of $sL_\infty$-algebras, 
%a homotopy functorial description of
%the bar construction of  $sL_\infty$-algebras and the definitions of homotopy category of homotopy probability
%algebras with the descendant functor to the homotopy category of  unital $sL_\infty$-algebras.

Throughout this appendix  $R$ is a fixed commutative ground ring of characteristic zero with unit $1_R$. 
(In the body of the paper $R$ is either a field $\Bbbk$ of characteristic zero or $\Bbbk{\kkbar}$.)

Let  $V=\bigoplus_{k \in \Z} V^k$ be a $\Z$-graded $R$-module.
An element  of $V$  is  homogeneous if it lies in fixed degree. 
For a homogeneous element $v\in V$, we use the notation $|v|$ for its degree and write $Jv$ for $(-1)^{|v|}v$.
We denote by  $\Hom_R(V,V^\pr)^j$  the space of $R$-module homomorphisms from $V$ to $V^\pr$
of degree $j$, and by $\Hom_R(V,V^\pr)$ the sum $\bigoplus_{j\in \Z}\Hom_R(V,V^\pr)^j$. 
The tensor product over $R$ is denoted by $\otimes$.% or by $\otimes$.

%The {\em free tensor module} generated by a $\Z$-graded $R$-module $V$ is $T(V)= \bigoplus_{n=0}^\infty T^n V$
%where $T^n V=V^{\otimes n}$ and $T^0 V=R$.  
The reduced free tensor module generated by a $\Z$-graded $R$-module  $V$ is 
$\Xbar{T}(V) = \bigoplus_{n=1}^\infty T^n V$, where $T^n V=V^{\otimes n}$,
which has the induced structure of a $\Z$-graded $R$-module. 
Let $\Perm_n$ be the group of permutations of the set $[n]=\{ 1,\ldots, n\}$. For each $\s \in \Perm_n$, 
we define the map $\check\s :T^n V\rightarrow T^n V$ by specifying, for homogeneous elements $v_1,\dotsc, v_n \in V$,
\[
\check\s\left(v_1\otimes v_2\otimes\cdots \otimes v_n\right) 
=
\e(\s)v_{\s(1)}\otimes v_{\s(2)}\otimes\cdots\otimes v_{\s(n)}
.
\]
In this equation $\ep(\s)=\pm 1$ is the Koszul sign determined by decomposing $\s$ as composition of transpositions 
$\check{\tau}:v_1\otimes v_2 \mapsto (-1)^{|v_1||v_2|} v_2\otimes v_1$. 
%For example 
%$\check\s (v_1\otimes v_2) =(-1)^{|v_1||v_2|}v_2\otimes v_1$.
We denote by $S^nV$ is the submodule of $T^n V$ that is fixed by $\check\s$.
Elements in $S^nV$ are generated by elements of the form $v_1\odot\ldots \odot v_n$,
which is a weighted sum over the orbit of $v_1\otimes\cdots\otimes v_n$.
% where for all $ \s \in \Perm_n$,
% \[
% v_1\odot\ldots \odot v_n=
% \ep(\s)v_{\s(1)}\odot\ldots\odot v_{\s(n)}.
% \]
The reduced free symmetric module generated by $V$ is
$\Xbar{S}(V) = \bigoplus_{n=1}^\infty S^n V$. We denote by
$\eb_{S^n V}: S^n V\rightarrow \Xbar{S}(V)$
and $\proj_{S^n V}: \Xbar{S}(V)\rightarrow S^n V$ for $n\geq 1$
the canonical embedding and projection.

We say $L_n\in \Hom_R\big(T^n V, V^\pr\big)^{j}$ descends to a $R$-linear map from $S^n V$ to $V^\pr$,
and denote this map by  $L_n\in \Hom_R\Big(S^n V, V^\pr\Big)^{j}$,
if
$L_n\big(v_1\otimes\cdots\otimes v_n\big)
= \ep(\s)L_n\big(v_{\s(1)}\otimes\cdots\otimes v_{\s(n)}\big)$.
A family $\underline{L}=L_1,L_2,\dots $ of $L_n\in \Hom\Big(S^n V, V^\pr\Big)^{j}$ for all $n\geq 1$
determines $L \in  \Hom\Big(\Xbar{S}(V), V^\pr\Big)^{j}$ by the convention that 
$L\big(v_1\odot\cdots\odot v_n\big)=L_n\big(v_1\odot\cdots\odot v_n\big)$, for all $n\geq 1$,
and vice versa. That is, $L_n = L\circ \eb_{S^n V}$.
We use the notation $\underline{L}$ and $L$ interchangeably. We also use the notation
$L\big(v_1\odot\cdots\odot v_n\big)=L_n\big(v_1\odot\cdots\odot v_n\big)=L_n\big(v_1,\ldots, v_n)$.

\subsection{The homotopy category of \texorpdfstring{$sL_\infty$}{sL-infinity}-algebras} 
\label{appendix: l infinity algebras}

We begin with generator-relation definitions of  $sL_\infty$-algebras, morphisms and homotopy types
of morphisms.

\begin{definition}\label{slinftyalg}
An $sL_\infty$-algebra over $R$ is a tuple $\big(V, \underline{{l}}\big)$, where
$V$ is a $\Z$-graded $R$-module and $\underline{{l}}={l}_1,{l}_2,\ldots$ is a family
of operations ${l}_k  \in \Hom_R\big(S^k V, V\big)^1$ for $k\geq 1$,  such that, for all $n\geq 1$ and homogeneous
$v_1,\ldots, v_n \in V$,
\[
\sum_{\clapsubstack{|\mp| \in P(n)\\ |B_i|=n-|\mp|+1}}\e(\mp){l}_{|\mp|}\left(
Jv_{B_1},\ldots, Jv_{B_{i-1}}, {l}(v_{B_{i}}), v_{B_{i+1}}, \ldots, v_{B_{|\mp|}}\right)=0.
\]
\end{definition}
The sum in the above formula is over all classical partitions of $[n]$ satisfying the condition that every block  in the partition $\mp$
has a single element except possibly one block which has $n-|\mp|+1$ elements.
For example, we have
\begin{align*}
d^2\big(v_1\big)&=0
,\\
d{l}_2\big(v_1,v_2\big) + {l}_2\big(dv_1,v_2\big) +{l}_2\big(Jv_1, dv_2\big)&=0
,\\
d{l}_3\big(v_1,v_2,v_3\big)+{l}_3\big(dv_1,v_2,v_3\big)+{l}_3\big(Jv_1,dv_2,v_3\big)+{l}_3\big(Jv_1,Jv_2,dv_3\big)
&
\\
+{l}_2\big({l}_2(v_1,v_2),v_3) +{l}_2\big(Jv_1, {l}_2(v_2,v_3)\big)
+(-1)^{|v_1||v_2|}{l}_2\big(Jv_2, {l}_2(v_1,v_3)\big)
&=0
,\\
\end{align*}
etc., 
where $d={l}_1$. Note that $(V,d)$ is a cochain complex over $R$ whose cohomology $H$ is called
the cohomology of the $sL_\infty$-algebra $\big(V, \underline{{l}}\big)$.
An $sL_\infty$-algebra $\big(V, \underline{{l}}\big)$ is called minimal if $l_1=0$.

\begin{definition} \label{slinftymor}
A morphism of $sL_\infty$-algebras  from $\big(V, \underline{{l}}\big)$ to $\big(V^\pr, \underline{{l}}^\pr\big)$
is a family $\underline{\w}=\w_1,\w_2,\ldots$ of  $\w_k  \in \Hom_R\big(S^k V, V^\pr\big)^0$, $k\geq 1$,  
such that, for all $n\geq 1$ and homogeneous
$v_1,\ldots, v_n \in V$,
\begin{align*}
\sum_{\clapsubstack{|\mp| \in P(n)}}&
\e(\mp)
{l}^\pr_{|\mp|}\Big(
\w\left(v_{B_1}\right),\ldots,\w\big( v_{B_{|\mp|}}\big)
\Big) 
\\
&
=
\sum_{\clapsubstack{|\mp| \in P(n)\\ |B_i|=n-|\mp|+1}}
\e(\mp)
\w_{|\mp|}\Big(Jv_{B_1},\ldots, Jv_{B_{i-1}}, {l}(x_{B_{i}}), v_{B_{i+1}},\ldots, v_{B_{|\mp|}}\Big)
.
\end{align*}
\end{definition}
For example, we have
\begin{align*}
d^\pr\w_1(v_1)&=\w_1(d v_1)
,\\
\w_1\big({l}_2\big(v_1,v_2\big)\big) -{l}^\pr_2\big(\w_1(v_1), \phi_1(v_2)\big) 
&= 
d\w_2(v_1,v_2)
- \w_2\big(dv_1,v_2\big) -\w_2\big(Jv_1, dv_2\big)
,
\end{align*}
etc. Note that $\phi_1$ is a cochain map from $\big(V, d\big)$ to $\big(V^\pr, d^\pr\big)$.  
Recall that a cochain map is a cochain quasi-isomorphism if it induces an isomorphism on cohomology.
An $sL_\infty$ morphism $\underline{\w}$ a {\em quasi-isomorphism}
if $\w_1$ is a cochain quasi-isomorphism between the underlying cochain complexes.

\begin{definition}[Lemma]\label{lmorcomp}
Let
$
\xymatrix{
\big(V, \underline{{l}}\big)\ar@{..>}[r]^{\underline{\w}}&\big(V^\pr, \underline{{l}}^\pr\big)\ar@{..>}[r]^{\underline{\w}^\pr}&
\big(V^\ppr, \underline{{l}}^\ppr\big)
}
$ be  consecutive $sL_\infty$-morphisms. Then,
the composition $\underline{\w}^\pr\bullet \underline{\w}$ defined
by the following equation for all $n\geq 1$ and $v_1,\ldots,v_n \in V$:
\[
\left(\underline{\w}^\pr\bullet \underline{\w}\right)_n(v_1,\ldots,v_n)\coloneqq{}
\sum_{\clapsubstack{|\mp| \in P(n)}}
\e(\mp)
\w^\pr_{|\mp|}\big(
\w(v_{B_1}),\ldots,\w( v_{B_{|\mp|}})
\big),
\]
is  an $sL_\infty$-morphism from $\big(V, \underline{{l}}\big)$
to $\big(V^\ppr, \underline{{l}}^\ppr\big)$.
The operation $\bullet$ is associative.
\end{definition}
The {\em  category of  $sL_\infty$-algebras} over $R$ is the category $\category{sL}_\infty(R)$
whose objects are $sL_\infty$-algebras over $R$ and whose morphisms are $sL_\infty$-morphisms with composition operation $\bullet$.

Now we turn to the homotopy category of $sL_\infty$-algebras, which will require a bit of preparation.

\begin{definition}\label{slinftyhflow}
A homotopy pair 
of $sL_\infty$-algebras  from $\big(V, {\underline{l}}\big)$ to 
$\big(V^\pr, {\underline{l}}^\pr\big)$
is a pair 
\[
\big(\w(\t), \l(\t)\big)\in \Hom_R\big(\Xbar{S}(V), V^\pr\big)^0[\t]\oplus \Hom_R\big(\Xbar{S}(V), V^\pr\big)^{-1}[\t]
\]
satisfying the following system of equations:
for all $n\geq 1$ and homogeneous
$v_1,\ldots, v_n \in V$,
\begin{align*}
\Fr{d}{d\t}\w(\t)&(v_1\odot\ldots\odot v_n)
\\
= 
&
\sum_{\clapsubstack{\mp \in P(n)\\ |B_i|=n-|\mp| +1}}
\e(\mp)
\l(\t)\Big(J\!v_{B_1}\odot\ldots\odot J\!v_{B_{i-1}}\odot{l}(x_{B_{i}})\odot v_{B_{i+1}}\odot\ldots\odot v_{B_{|\mp|}}\Big)
\\
&
+\sum_{\clapsubstack{\mp \in P(n) }}\ \ \sum_{\mathclap{i=1}}^{|\mp|}
\e(\mp)
{l}^\pr\Big(\w(\t)(J\!v_{B_1})\odot\ldots\odot\w(\t)(J\!v_{B_{i-1}})\odot\l(\t)(v_{B_i})
\\
&\qquad\qquad\qquad\qquad\qquad\qquad\qquad
\odot\w(\t)(v_{B_{i+1}})\odot
\ldots\odot\w(\t)(v_{B_{|\mp|}})\Big)
.
\end{align*}
\end{definition}

The first two equations of this system are
\begin{align*}
\Fr{d}{d\t}\w_1(\t)(v_1)=& \l_1(\t)(dv_1) + d^\pr \l_1(\t)(v_1)
,\\
\Fr{d}{d\t}\w_2(\t)(v_1,v_2)=& \l_2(\t)(dv_1,v_2) +\l_2(\t)(Jv_1, dv_2) +\l_1(\t)\big({l}_2(v_1,v_2)\big)
\\
&
+d^\pr \l_2(\t)(v_1,v_2) + {l}^\pr_2\big(\l_1(\t)(v_1), \w_1(v_2)\big) + {l}_2\big(\w_1(Jv_1), \l_1(\t)(v_2)\big)
.
\end{align*}
Working recursively from $n=1$, it is obvious that the system of equations in Definition \ref{slinftyhflow}
has a unique solution  $\w(\t)$, modulo an initial condition ${\w}(0)$, with respect to ${\l}(\t)$.
\begin{lemma}
Let $\big({\w}(\t), {\l}(\t)\big)$ be a homotopy pair of  $sL_\infty$-algebras
such that $\underline{\w}(0)$ is an $sL_\infty$-morphism.  
Then $\w(\t)$ is a (uniquely defined) family of $sL_\infty$-morphisms.
\end{lemma}

Now we are ready to define homotopy types of $sL_\infty$-morphisms.

\begin{definition}
Two $sL_\infty$-morphisms $\underline{\w}$ and $\underline{\tilde\w}$ are \emph{homotopic}, 
which we denote $\underline{\w}\sim\underline{\tilde\w}$, 
or have the same \emph{homotopy type},  denoted $\left[\underline{\w}\right]=\left[\underline{\tilde\w}\right]$,
if there is
a $sL_\infty$-homotopy pair $\big({\w}(\t), {\l}(\t)\big)$
such that $\underline{\w}= \underline{\w}(0)$ and $\underline{\tilde\w}= \underline{\w}(1)$.
\end{definition}
It is clear that $\sim$ is an equivalence relation.  
The homotopy category  of $sL_\infty$-algebras over $R$ is the category $\mathit{ho}\category{sL}_\infty(R)$, whose objects
are $sL_\infty$-algebras over $R$ and morphisms are homotopy types of $sL_\infty$-morphisms.
It remains is to check that $\mathit{ho}\category{sL}_\infty(R)$ is indeed a category.

\begin{lemma}\label{unihflowcom}
Given ``composable'' $sL_\infty$-homotopy pairs 
\[
\xymatrix{
\big(V, {{l}}\big)\ar@{:>}[rr]^-{\big({\w}(\t), {\l}(\t)\big)}&&\big(V^\pr, {{l}}^\pr\big)
\ar@{:>}[rr]^-{\big({\w}(\t), {\l}(\t)\big)}&&\big(V^\ppr, {{l}}^\ppr\big)
}
\]
the composition 
$\big({\w}^\pr(\t), {\l}^\pr(\t)\big)\bullet \big({\w}(\t), {\l}(\t)\big)=\big({\w}^\ppr(\t), {\l}^\ppr(\t)\big)$
defined for all $n\geq 1$ and homogeneous $v_1,\ldots, v_n \in V$ by the equations
\begin{align*}
\w^\ppr(\t)\big(v_1\odot\ldots\odot v_n\big) \coloneqq{} 
&
 \sum_{\clapsubstack{|\mp| \in P(n)}}
\e(\mp)
\w^\pr(\t)\Big(
\w(\t)(v_{B_1})\odot\ldots\odot\w(\t)( v_{B_{|\mp|}})
\Big)
,\\
\l^\ppr(\t)\big(v_1\odot\ldots\odot v_n\big) \coloneqq{} 
& 
\sum_{\clapsubstack{|\mp| \in P(n)}}
\e(\mp)
\l^\pr(\t)\Big(
\w(\t)(v_{B_1})\odot\ldots\odot\w(\t)( v_{B_{|\mp|}})
\Big)
\\
+\sum_{\clapsubstack{|\mp| \in P(n)}}
\e(\mp)\sum_{\mathclap{i=1}}^{|\mp|}
&
\w^\pr(\t)\Big(
\w(\t)(Jv_{B_1})\odot\ldots\odot\l(\t)(v_{B_{i}})\odot\ldots\odot \w(\t)( v_{B_{|\mp|}})
\Big)
\end{align*}
is an $sL_\infty$-homotopy pair from $\big(V, \underline{{l}}\big)$
to $\big(V^\ppr, \underline{{l}}^\ppr\big)$.
The operation $\bullet$ is associative.
\end{lemma}

Consider  $sL_\infty$-morphisms
$
\xymatrix{
\big(V,\underline{{l}}\big)
\ar@{..>}@/^/[r]^{\underline{\w}}
\ar@{..>}@/_/[r]_{\underline{\tilde\w}}
&\big(V^\pr,\underline{{l}}^\pr\big)
\ar@{..>}@/^/[r]^{\underline{\w}^\pr}
\ar@{..>}@/_/[r]_{\underline{\tilde\w}^\pr}
&\big(V^\ppr,\underline{{l}}^\ppr\big)
}
$
and assume that $\underline{\w}\sim \underline{\tilde\w}$ and  $\underline{\w}^\pr\sim \underline{\tilde\w}^\pr$.
Then,  there are corresponding $sL_\infty$-homotopy pairs as follows:
\begin{itemize}
\item $\big(\w(\t), \l(\t)\big)$ such that $\underline{\w}(0)=\underline{\w}$ and $\underline{\w}(1)=\underline{\tilde\w}$;
\item $\big(\w^\pr(\t), \l^\pr(\t)\big)$ 
such that $\underline{\w}^\pr(0)=\underline{\w}^\pr$ and $\underline{\w}(1)^\pr=\underline{\tilde\w}^\pr$.
\end{itemize}
By Lemma \ref{unihflowcom}, the composition 
$\big(\w^\ppr(\t), \l^\ppr(\t)\big)=\big({\w}^\pr(\t), {\l}^\pr(\t)\big)\bullet \big({\w}(\t), {\l}(\t)\big)$ is an
$sL_\infty$-homotopy pair  from $\big(V,\underline{{l}}\big)$ to $\big(V^\ppr,\underline{{l}}^\ppr\big)$
such that 
\[
\underline{\w}^\ppr(0) = \underline{\w}^\pr\bullet \underline{\w},\qquad
\underline{\w}^\ppr(1) = \underline{\tilde\w}^\pr\bullet \underline{\tilde\w}
.
\]
It follows that
$\underline{\w}^\pr\bullet \underline{\w} \sim \underline{\tilde\w}^\pr\bullet \underline{\tilde\w}$
or, equivalently, $\left[\underline{\w}^\pr\bullet \underline{\w}\right]= \left[\underline{\tilde\w}^\pr\bullet \underline{\tilde\w}\right]$
whenever $\underline{\tilde\w} \sim \underline{\w}$ and $\underline{\tilde\w}^\pr\sim \underline{\w}^\pr$
so that
the  homotopy type $\left[\underline{\w}^\pr\bullet \underline{\w}\right]$ of $\underline{\w}^\pr\bullet \underline{\w}$ 
depends only on the homotopy types $\left[\underline{\w}\right]$ and $\left[\underline{\w}^\pr\right]$
of $\underline{\w}$ and $\underline{\w}^\pr$, respectively. 
Therefore the composition which takes $\left[\underline{\w}\right]$ and $\left[\underline{\w}^\pr\right]$ to
$\left[\underline{\w}^\pr\right]\bullet_h \left[\underline{\w}\right]\coloneqq{} \left[\underline{\w}^\pr\bullet \underline{\w}\right]$ is well-defined.
It is obvious that $\bullet_h$ is associative.

\begin{definition}
The homotopy category of $sL_\infty$-algebras over $R$
is the category
$\mathit{ho}\category{sL}_\infty(R)$,
whose objects are $sL_\infty$-algebras over $R$ and whose morphisms are homotopy types 
of $sL_\infty$-morphisms with composition $\bullet_h$.

\end{definition}

In this paper, we shall work primarily with the category and homotopy category of {\em unital} $sL_\infty$-algebras.
\begin{definition}
\begin{itemize}
\item
A  unital $sL_\infty$-algebra over $R$ is a tuple 
$\big(V, 1_V,\underline{{l}}\big)$, where
$\big(V, \underline{{l}}\big)$ is an $sL_\infty$-algebra over $R$ 
and  $1_V$ is an element of $V^0$ such that, 
for all $n\geq 1$ and homogeneous
$v_1,\ldots, v_{n-1} \in V$,
\[
{l}_{n}\big(v_1,\ldots, v_{n-1}, 1_V)=0.
\]
\item
A unital $sL_\infty$-morphism from
$\big(V, 1_V,\underline{{l}}\big)$ to $\big(V^\pr, 1_{V^\pr},\underline{{l}}^\pr\big)$ is an
$sL_\infty$-morphism $\underline{\w}$ from $\big(V,\underline{{l}}\big)$ to $\big(V^\pr, \underline{{l}}^\pr\big)$
such that, for  all $n\geq 1$ and homogeneous
$v_1,\ldots, v_{n-1} \in V$,
\[
\w_{n}\big(v_1,\ldots, v_{n-1}, 1_V)=\d_{n,1}\times 1_{V^\pr},
\]
where $\d_{n,1}$ is  the Kronecker delta: $1$ for $n=1$ and $0$ otherwise.
\end{itemize}
\end{definition}
The composition of two unital $sL_\infty$-morphisms, as $sL_\infty$-morphisms,
is a unital $sL_\infty$-morphism.
\begin{notation}
We denote by $\category{UsL}_\infty(R)$ the {\em  category of  unital $sL_\infty$-algebras}
whose objects are unital $sL_\infty$-algebras over $R$ and whose morphisms are unital $sL_\infty$-morphisms.
\end{notation}
\begin{definition}
A unital $sL_\infty$-homotopy pair is  $sL_\infty$-homotopy pair 
$\big({\w}(\t), {\l}(\t)\big)$
such that for all $n\geq 1$ and homogeneous
$v_1,\ldots, v_{n-1} \in V$, 
\[
\l(\t)\big(v_1\odot\ldots\odot v_{n-1}\odot 1_V)=0.
\]
\end{definition}
Then, $\underline{\w}(\t)$ is a smooth $1$-parameter family of unital $sL_\infty$-morphisms
if $\underline{\w}(0)$ is a unital $sL_\infty$-morphism, so that we can define homotopy of unital $sL_\infty$-morphisms accordingly.
\begin{notation}
We denote by $ho\category{UsL}_\infty(R)$ is the {\em  homotopy category of  unital $sL_\infty$-algebras}
whose objects are unital $sL_\infty$-algebras over $R$ and whose morphisms are homotopy types of unital $sL_\infty$-morphisms.
\end{notation}
The following important lemma is well-known:
\begin{lemma}\label{htrans}
On the cohomology $H$ of an $sL_\infty$-algebra $\big(W,\underline{\ell}\big)$ over a field $\Bbbk$ of characteristic zero, there is 
the structure of a minimal $sL_\infty$-algebra $\big(H, \underline{\grave{\ell}}\big)$ and a quasi-isomorphism
$\underline{\w}: \big(H, \underline{\grave{\ell}}\big)\rightarrow \big(W,\underline{\ell}\big)$.
\end{lemma}

\begin{proof}
Choose the data of a strong deformation retract $(f,h,s)$ between the cochain complex $\big(W, d\coloneqq{}\ell_1\big)$ over $\Bbbk$
and its homology $(H,0)$.
So $s\in\Hom_\Bbbk\big(W, W\big)^{-1}$ and $f\in \Hom_\Bbbk(H, W\big)^0$ and $h\in \Hom_\Bbbk(W, H\big)^0$
satisfy $f\circ h = \I_{W} - d\circ s -s\circ d$ and $h\circ f= \I_H$.
Define $\underline{\w}=\w_1,\w_2,\ldots $ and $\underline{\grave{\ell}}=\grave{\ell}_1, \grave{\ell}_2,\ldots$
recursively as follows: $\w_1=f$ and  $\grave{\ell}_1=0$, while
$\grave{\ell}_n \coloneqq{}h\circ L_n$ and $\w_n\coloneqq{}-s\circ L_n$ for all $n\geq 2$, 
where $L_n \in \Hom\big(S^nH, W\big)^1$  is given for homogeneous $v_1,\ldots,v_n \in H$ by
\begin{align*}
L_n(v_1,\ldots, v_n)\coloneqq{}
&
\sum_{\clapsubstack{|\mp| \in P(n)\\ |\mp|\neq 1}}
\e(\mp)
\ell_{|\mp|}\big(
\w(v_{B_1}),\w( v_{B_{|\mp|}})
\big) 
\\
&
-
\sum_{\clapsubstack{|\mp| \in P(n)\\ |B_i|=n-|\mp|+1\\ |\mp|\neq n,1}}
\e(\mp)
\w_{|\mp|}\Big(Jv_{B_1},\ldots, Jv_{B_{i-1}}, \grave{\ell}(x_{B_{i}}), v_{B_{i+1}},\ldots, v_{B_{|\mp|}}\Big)
.
\end{align*}
Note that $L_n$ depends only on $\w_1,\ldots, \w_{n-1}$ and $\grave{\ell}_2, \ldots, \grave{\ell}_{n-1}$.
For $n\geq 2$, let
\[
F_n(v_1,\ldots, v_n)\coloneqq{}\sum_{\clapsubstack{|\mp| \in P(n)\\ |B_i|=n-|\mp|+1\\ |\mp|\neq n, 1}}\e(\mp){l}_{|\mp|}\left(
Jv_{B_1},\ldots, Jv_{B_{i-1}}, {l}(v_{B_{i}}), v_{B_{i+1}}, \ldots, v_{B_{|\mp|}}\right),
\]
which depends only on $\grave{\ell}_2, \ldots, \grave{\ell}_{n-1}$.
From $L_2(v_1,v_2) = \big(\w_1(v_1),\w_1(v_2)\big)$, we have $d\circ L_2=0$ since $d\circ \w_1=0$ and $d$ is a derivation
of the bracket $(\ ,\,)$. From $\grave{\ell}_2 =h\circ L_2$
and $\w_2 =-s\circ L_2$, we obtain that $L_2=\w_1\circ \grave{\ell}_2 - d\circ \w_2$. Note that $F_2=0$.
Fix $n\geq 3$ and assume that $L_k=\w_1\circ \grave{\ell}_k - d\circ \w_k$ and $F_k =0$ for all $k=2,\ldots, n-1$.
Then it is straightforward to check that $d\circ L_n=f\circ F_n$, which implies that
$h\circ d\circ L_n=h\circ f\circ F_n=F_n=0$ and $d\circ L_n=0$. From $\grave{\ell}_n \coloneqq{}h\circ L_n$ and $\w_n\coloneqq{}-s\circ L_n$,
we obtain that $L_n=\w_1\circ \grave{\ell}_n - d\circ \w_n$. Therefore we have proved that $d\circ \w_1=\ell_1=0$
and $L_n+ d\circ \w_n-\w_1\circ \grave{\ell}_n =F_n=0$ for all $n\geq 2$, which are exactly the conditions
for $\big(H, \underline{\grave{\ell}}\big)$ to be a minimal $sL_\infty$-algebra over $\Bbbk$ and for
$\underline{\w}: \big(H, \underline{\grave{\ell}}\big)\rightarrow \big(W,\underline{\ell}\big)$ to be an $sL_\infty$-morphism,
which is a quasi-isomorphism since $\w_1=f: (H,0)\rightarrow (W,d)$ is a cochain quasi-isomorphism.
\naturalqed
\end{proof}

\subsection{The homotopy category of dg coalgebras}
\label{appendix: coalgebras}

A {\em dg coalgebra} over $R$ is a tuple $\Big(C, \triangle_C, d_C\Big)$, where 

- $\Big(C, d_C\Big)$ is a cochain complex over $R$, i.e., $d_C \in \Hom\big(C, C\big)^1$ and $d_C\circ d_C=0$,
and

- $\Big(C, \triangle_C\Big)$ is a $\Z$-graded coassociative coalgebra over $R$, i.e.,
\[
\triangle_C \in \Hom\big(C, C\otimes C\big)^0,\qquad
\big(\triangle_C\otimes \I_C\big)\circ  \triangle_C = \big(\I_C\otimes \triangle_C\big)\circ  \triangle_C,
\]
such that $d_C$ is a coderivation of $\triangle_C$, i.e.,
$\triangle_C \circ d_C =\big(d_C\otimes \I_C +\I_C\otimes d_C)\circ \triangle_C$.

A {\em morphism of  dg coalgebras}  from $\Big(C, \triangle_C, d_C\Big)$ to $\Big(C^\pr, \triangle_{C^\pr}, d_{C^\pr}\Big)$
is both a cochain map and coalgebra map, i.e., 
\[
F \in \Hom(C, C^\pr)^0
,\qquad 
d_{C^\pr}\circ F = F\circ d_C
,\qquad 
\triangle_{C^\pr}\circ F= \big(F\otimes F)\circ \triangle_C
.
\]
It is straightforward to check that the composition $F^\pr\circ F$ of dg coalgebra morphisms 
as $R$-linear maps is a dg coalgebra morphism. Therefore, we have the category $\category{dgC}(R)$ of 
dg coalgebras over $R$.  

\begin{definition}
A homotopy pair 
$\xymatrix{\Big(\mF(\t), {\La}(\t)\Big):\Big(C, \triangle_C, d_C\Big)\ar@{:>}[r]&\Big(C^\pr, \triangle_{C^\pr}, d_{C^\pr}\Big)}$
of dg coalgebras is a pair 
$\mF(\t)\oplus \La(\t):[0,1]\rightarrow\Hom\big(C, C^\pr\big)^0[\t]\oplus \Hom\big(C, C^\pr\big)^{-1}][\t]$
such that the following relations are satisfied:
\begin{align*}
 \Fr{d}{d\t}\mF(\t) &= d_{C^\pr}\circ \La(\t) +\La(\t)\circ d_C,
\\
\triangle_{C^\pr}\circ \mF(\t) &=\big(\mF(\t)\otimes \mF(\t)\big)\circ \triangle_{C}
 ,\\
\triangle_{C^\pr}\circ \La(\t) &=\big(\mF(\t)\otimes \La(\t) + \La(\t)\otimes \mF(\t)\big)\circ \triangle_{C}.
\end{align*}
\end{definition}
Then, it is straightforward to show that $\mF(\t)$ is determined uniquely with respect to $\La(\t)$
for a given initial condition $\mF(0)$ and is a smooth family of dg coalgebra morphisms if $\mF(0)$ is a dg coalgebra
morphism.
\begin{definition}
Two  dg coalgebra morphisms $F$ and $\tilde F$ are \emph{homotopic}, denoted $F\sim \tilde F$,
or have the same \emph{homotopy type}, denoted by $\big[F\big]=\big[\tilde F\big]$,
if there is a dg coalgebra homotopy pair  $\big(\mF(\t), \La(\t)\big)$  such that $\mF(0)=F$ and $\mF(1)=\tilde F$.
\end{definition}

It is clear that $\sim$ is an equivalence relation.  
The homotopy category  of dg coalgebras over $R$ shall be a category $ho\category{dgC}(R)$, whose objects
are dg coalgebras over $R$  and whose morphisms are homotopy types of dg coalgebra morphisms.
We check that $ho\category{dgC}(R)$ is indeed a category in the following Lemma:

\begin{lemma}\label{xunihflowcom}
Given ``composable'' homotopy pairs of dg coalgebras
\[
\xymatrix{
\Big(C, \triangle_C, d_C\Big)
\ar@{:>}[rr]^-{\big(\mF(\t), {\La}(\t)\big)}&&
\Big(C^\pr, \triangle_{C^\pr}, d_{C^\pr}\Big)
\ar@{:>}[rr]^-{\big({\mF}^\pr(\t), {\La}^\pr(\t)\big)}&&
\Big(C^\ppr, \triangle_{C^\ppr}, d_{C^\ppr}\Big)
}
\]
the composition 
$\Big({\mF}^\ppr(\t), {\La}^\ppr(\t)\Big)=\Big({\mF}^\pr(\t), {\La}^\pr(\t)\Big)\circ \Big({\mF}(\t), {\La}(\t)\Big)$,
defined by the formulas
\begin{align*}
{\mF}^\ppr(\t)\coloneqq{}& {\mF}^\pr(\t)\circ {\mF}(\t)
,\\
{\La}^\ppr(\t)\coloneqq{}& {\mF}^\pr(\t)\circ {\La}(\t) +{\La}^\pr(\t)\circ {\mF}(\t)
,
\end{align*}
is a homotopy pair of dg coalgebras from $\Big(C, \triangle_C, d_C\Big)$ to $\Big(C^\ppr, \triangle_{C^\ppr}, d_{C^\ppr}\Big)$ and $\circ$ is associative.
\end{lemma}

Consider the following diagram in $\category{dgC}(R)$
\[
\xymatrix{
\Big(C, \triangle_C, d_C\Big)\ar@/^/[r]^{F} \ar@/_/[r]_{\tilde F} 
& \Big(C^\pr, \triangle_{C^\pr}, d_{C^\pr}\Big)\ar@/^/[r]^{F^\pr} \ar@/_/[r]_{\tilde F^\pr} 
& \Big(C^\ppr, \triangle_{C^\ppr}, d_{C^\ppr}\Big)
}
\]
and assume that $F \sim \tilde F$ and $F^\pr \sim \tilde F^\pr$.
Then there are homotopy pairs 

- $\Big(\mF(\t), {\La}(\t)\Big)$ such that $\mF(0)=F$ and $\mF(1)=\tilde F$;

- $\Big(\mF^\pr(\t), {\La}^\pr(\t)\Big)$ such that $\mF^\pr(0)=F^\pr$ and $\mF^\pr(1)=\tilde F^\pr$.

By Lemma~\ref{xunihflowcom}, their composition $\Big({\mF}^\ppr(\t), {\La}^\ppr(\t)\Big)$ is a homotopy pair
such that ${\mF}^\ppr(0)=F^\pr\circ F$ and ${\mF}^\ppr(1)=\tilde F^\pr\circ \tilde F$. It follows that
$F^\pr\circ F \sim \tilde F^\pr\circ \tilde F$ or, equivalently, $\big[F^\pr\circ F \big]=\big[\tilde F^\pr\circ \tilde F\big]$
whenever $F \sim \tilde F$ and $F^\pr \sim \tilde F^\pr$
and the homotopy type $\big[F^\pr\circ F\big]$ of $F^\pr\circ F$ depends only on the homotopy types
$\big[F^\pr\big]$ and $\big[ F\big]$ of $F^\pr$ and $F$, respectively. 
Therefore we can define the composition of $\big[F^\pr\big]$ and $\big[ F\big]$ by
$\big[F^\pr\big]\circ_h \big[ F\big]\coloneqq{} \big[F^\pr\circ F\big]$. Again it is obvious that $\circ_h$ is associative.

\begin{definition}
The homotopy category  of dg coalgebras over $R$ is the category $ho\category{dgC}(R)$ whose objects
are dg coalgebras over $R$  and whose morphisms are homotopy types of dg coalgebra morphisms with composition $\circ_h$.
\end{definition}

A dg coalgebra $\Big(C, \triangle_C, d_C\Big)$ is cocommutative if $\triangle_C =\check{\tauup}\circ\triangle_C$.
We use the notation $\category{cocdgC}(R)$ and $\mathit{ho}\category{cocdgC}(R)$ for the category and homotopy category of
cocommutative dg coalgebras over $R$, which are full subcategories of $\category{dgC}(R)$ and $\mathit{ho}\category{dgC}(R)$,
respectively.

\subsection{The bar functor.}
\label{appendix: bar}

The bar construction of $sL_\infty$-algebras is a 
homotopy functor $\mathfrak{B}$ %$\category{sL}_\infty(R)\rightsquigarrow \category{cocdgC}(R)$
from the category $\category{sL}_\infty(R)$ of $sL_\infty$-algebras 
to the category $\category{cocdgC}(R)$ of cocommutative dg-coalgebras.

The reduced symmetric module $\Xbar{S}(V)$ generated by a $\Z$-graded $R$-module $V$ has the structure of a
$\Z$-graded cocommutative and coassociative coalgebra 
$\Xbar{S}^{\mathit{co}}(V)=\big(\Xbar{S}(V), \Xbar{\blacktriangle}\big)$ over $R$ 
called the reduced symmetric coalgebra cogenerated by $V$,
where
the coproduct
$\Xbar{\blacktriangle}: \Xbar{S}(V)\rightarrow \Xbar{S}(V){\otimes} \Xbar{S}(V)$ is defined for all $n\geq 1$ and homogeneous
$v_1,\ldots, v_n \in V$ to be
\[
\Xbar{\blacktriangle}(v_1\odot\ldots \odot v_n)=\sum_{\mathclap{r=1}}^{n-1} \sum_{{\s \in Sh(r, n-r)}}\e (\s)
v_{\s(1)}\odot\ldots \odot v_{\s(r)}{\otimes} v_{\s(r+1)}\odot\ldots \odot v_{\s(n)}.
\]
Here the second sum is over all $(r, n-r)$-shuffles---these are those permutations of $n$ such that
$\s(1)< \ldots < \s(r),$ and $\s(r+1)< \ldots < \s(n)$.

The reduced symmetric coalgebras have the following properties:
\begin{lemma}\label{fcoder}
For any  ${l} \in \Hom_R\left(\Xbar{S}(V), V\right)^1$, 
define $\mathfrak{D}({l})\in \Hom_R\left(\Xbar{S}(V), \Xbar{S}(V)\right)^1$ 
for all $n\geq 1$ and homogeneous
element $v_1,\ldots,v_n \in V$ by the equation
\[
\mathfrak{D}({l})(v_1\odot\ldots\odot v_n) =
\sum_{\clapsubstack{|\mp| \in P(n)\\ |B_i|=n-|\mp|+1}}
\e(\mp)
Jv_{B_1}\odot \ldots\odot Jv_{B_{i-1}}\odot {l}(v_{B_{i}}) \odot v_{B_{i+1}}\odot \ldots \odot v_{B_{|\mp|}}.
\]
Then, $\mathfrak{D}({l})$ is the unique coderivation of
$\Xbar{S}^{\mathit{co}}(V)$ with the property $\proj_V\circ \mathfrak{D}={l}$,
where $\proj_V:\Xbar{S}(V)\rightarrow V$ is the natural projection.
Conversely, any degree $1$ coderivation $\mathfrak{D}$ of $\Xbar{S}^{\mathit{co}}(V)$,
$\Xbar{\blacktriangle}\circ \mathfrak{D} = \big(\mathfrak{D}\otimes \I +\I\otimes \mathfrak{D}\big)\circ \Xbar{\blacktriangle}$,
is in the form
$\mathfrak{D}({l})$, where ${l} =\proj_V\circ \mathfrak{D} \in \Hom\big(\Xbar{S}(V), V\big)^1$.
\end{lemma}

%Here are some explicit formulas for $\sD=\sD({l})$:
%\[
%\eqalign{
%\sD(v_1)=&{l}_1(v_1)
%,\\
%\sD(v_1\odot v_2 )=& {l}_2(v_1,v_2) +{l}_1(v_1)\odot v_2 +Jv_1\odot {l}_1(v_2)
%,\\
%\sD(v_1\odot v_2 \odot v_3)=
%&
%{l}_3(v_1,v_2,v_3)
%\\
%& + {l}_2(v_1,v_2)\odot v_3 +Jv_1\odot {l}_2(v_2,v_3)
%+(-1)^{|v_1||v_2|}Jv_2\odot {l}_2(v_1,v_3)
%\\
%&
%+{l}_1(v_1)\odot v_2\odot v_3 +Jv_1\odot {l}_1(v_2)\odot v_3 
%+Jv_1\odot v_2\odot {l}_1( v_3)
%.\\
%}
%\]
%
%

\begin{lemma}\label{fcoalg}
For any pair 
$\big(\w, \l\big)\in \Hom_R\left(\Xbar{S}(V), V^\pr\right)^0\oplus
\Hom_R\left(\Xbar{S}(V), V^\pr\right)^{-1}$,
define a pair $\big( \mF(\w), \La(\w, \l)\big)
\in 
\Hom_R\left(\Xbar{S}(V), \Xbar{S}(V)\right)^0\oplus\Hom_R\left(\Xbar{S}(V), \Xbar{S}(V)\right)^{-1}$
for all $n\geq 1$ and homogeneous
element $v_1,\ldots,v_n \in V$ via the equations
\begin{align*}
\mF\big(\w\big)(v_1\odot \ldots\odot v_n)\coloneqq{}
&
\sum_{\clapsubstack{|\mp| \in P(n)}}
\e(\mp)
\w\big(v_{B_1}\big)\odot\ldots\odot\w\big( v_{B_{|\mp|}}\big)
,\\
\La\big(\w,\l\big)(v_1\odot \ldots\odot v_n)\coloneqq{}
&
\sum_{\clapsubstack{|\mp| \in P(n)}}
\e(\mp)
\sum_{\mathclap{i=1}}^{|\mp|}
\w\big(Jv_{B_1}\big)\odot\ldots\odot\w\big(Jv_{B_{i-1}}\big)
\\
&
\qquad\qquad\qquad\quad
\odot \l\big(v_{B_i}\big)\odot\w\big(Jv_{B_{i+1}}\big)\odot\ldots\odot
\w\big( v_{B_{|\mp|}}\big)
.
\end{align*}
Then,  we have
\[
\begin{aligned}
\Xbar{\blacktriangle}^\pr\circ \mF(\w)
&=\big( \mF(\w)\otimes  \mF(\w)\big)\circ \Xbar{\blacktriangle}
,\\
\Xbar{\blacktriangle}^\pr\circ \La(\w,\l)
&=\big( \La(\w,\l)\otimes  \mF(\w)+ \mF\otimes  \La(\w,\l)\big)\circ \Xbar{\blacktriangle}
,
\end{aligned}
\qquad
\begin{aligned}
\proj_V\circ {\mF}(\w)&=\w
,\\
\proj_V\circ {\La}(\w,\l)&=\l
.
\end{aligned}
\]
Conversely, any pair 
$\big(\mF, \La\big)$ of $\mF\in  \Hom_R\left(\Xbar{S}(V), \Xbar{S}(V^\pr)\right)^0$ and 
$\La \in \Hom\left(\Xbar{S}(V), \Xbar{S}(V^\pr)\right)^{-1}$
satisfying  $\Xbar{\blacktriangle}^\pr\circ \mF=\big( \mF\otimes  \mF\big)\circ \Xbar{\blacktriangle}$
and $\Xbar{\blacktriangle}^\pr\circ \La=\big( \La\otimes  \mF+ \mF\otimes  \La\big)\circ \Xbar{\blacktriangle}$
is in the form $\big(\mF(\w), \La(\w,\l)\big)$, where
$\w = \proj_V\circ {\mF} \in \Hom_R\left(\Xbar{S}(V), V^\pr\right)^0$
and $\l = \proj_V\circ {\La} \in \Hom_R\left(\Xbar{S}(V), V^\pr\right)^{-1}$.
\end{lemma}

\begin{definition}
The bar construction of an $sL_\infty$-algebra $\big(V,\underline{{l}}\big)$ is the cocommutative dg-coalgebra
defined as follows:
\[
\mathfrak{B}\big(V,\underline{{l}}\big)=\big(\Xbar{S}^{\mathit{co}}(V),\mD({l})\big).
\]
The bar construction of an $sL_\infty$-morphism  $\xymatrix{\big(V,\underline{{l}}\big)\ar@{..>}[r]^{\underline{\w}} &\big(V^\pr,\underline{{l}}^\pr\big)}$ is  $\mathfrak{B}\big(\underline{\w}\big) =\mF(\w)$.
\end{definition}
\begin{lemma}\label{barfunctor}
The bar construction $\mathfrak{B}$ is a functor from $\category{sL}_\infty(R)$ to $\category{cocdgC}(R)$.
\end{lemma}

\begin{proof}
Lemma \ref{fcoder} implies that, for all $n\geq 1$ and homogeneous $v_1,\ldots,v_n \in V$,
\begin{align*}
&\Big(\proj_V\circ \mathfrak{D}({l})\circ \mathfrak{D}({l})\Big)(v_1\odot\ldots\odot v_n) 
\\
&\qquad\qquad
=
\sum_{\clapsubstack{|\mp| \in P(n)\\ |B_i|=n-|\mp|+1}}
\e(\mp ){l}\big(
Jv_{B_1}\odot \ldots\odot Jv_{B_{i-1}}\odot{l}(x_{B_{i}})\odot v_{B_{i+1}}\odot\ldots\odot v_{B_{|\mp|}}\big)
.
\end{align*}
From  Definition \ref{slinftyalg} of an $sL_\infty$-algebra, 
the coderivation $\mD({l})$ satisfies the condition 
$\proj_V\circ\mD({l})\circ \mD({l})=0$,
which  can be checked, by a straightforward induction, 
to be equivalent to the condition that $\mD({l})\circ \mD({l})=0$.
Therefore 
$\mathfrak{B}\big(V,\underline{{l}}\big)=\big(\Xbar{S}^{\mathit{co}}(V),\mD({l})\big)$
is a cocommutative dg-coalgebra.

Lemma \ref{fcoalg} implies that, for all $n\geq 1$ and homogeneous $v_1,\ldots,v_n \in V$,
\begin{align*}
&\Big(\proj_V\circ\big(\mD({l}^\pr)\circ \mF(\w)- \mF(\w)\circ\mathfrak{D}({l})\big)\Big)(v_1\odot \ldots\odot v_n)
\\
&\qquad\qquad
=
\sum_{\clapsubstack{|\mp| \in P(n)}}
\e(\mp){l}^\pr\big(
\w\big(v_{B_1}\big)\odot\ldots\odot \w\big( v_{B_{|\mp|}}\big)\big)
\\
&\qquad\qquad\quad
-\sum_{\clapsubstack{|\mp| \in P(n)\\ |B_i|=n-|\mp|+1}}
\e(\mp)\w\big(
Jv_{B_1}\odot\ldots\odot Jv_{B_{i-1}}\odot {l}(x_{B_{i}})\odot v_{B_{i+1}}\odot \ldots\odot v_{B_{|\mp|}}\big)
.
\end{align*}
From Definition \ref{slinftymor} of $sL_\infty$-morphism, 
we have $\proj_{V^\pr}\circ\big(\mD({l}^\pr)\circ \mF(\w)- \mF(\w)\circ\mathfrak{D}({l})\big)=0$,
which  can be checked to be equivalent to the condition that $\mD({l}^\pr)\circ \mF(\w)= \mF(\w)\circ\mathfrak{D}({l})$.
Therefore $\mathfrak{B}\big(\underline{\w}\big) =\mF(\w)$ is a morphism of cocommutative dg-coalgebras from
$\mathfrak{B}\big(V,\underline{{l}}\big)=\big(\Xbar{S}^c(V),\mD({l})\big)$ to
$\mathfrak{B}\big(V^\pr,\underline{{l}}^\pr\big)=\big(\Xbar{S}^c(V^\pr),\mD({l}^\pr)\big)$.

Consider the sequence of $sL_\infty$-morphisms
$
\xymatrix{
\big(V,\underline{{l}}\big)\ar@{..>}[r]^{\underline{\w}} 
%\ar@/_1.5pc/@{..>}[rr]_{\underline{\w}^\ppr \coloneqq{} \underline{\w}^\pr\bullet \underline{\w}} 
&\big(V^\pr,\underline{{l}}^\pr\big)
\ar@{..>}[r]^{\underline{\w}^\pr} &\big(V^\ppr,\underline{{l}}^\ppr\big)}$. 
Then $ \underline{\w}^\ppr\coloneqq{} \underline{\w}^\pr\bullet \underline{\w}$ is an $sL_\infty$-morphism from
$\big(V,\underline{{l}}\big)$ to $\big(V^\ppr,\underline{{l}}^\ppr\big)$. Now
Lemma \ref{fcoalg} and Definition \ref{lmorcomp}
imply that
$\proj_{V^\ppr}\circ\Big(\mF({\w}^\ppr) - \mF(\w^\pr)\circ \mF(\w)\Big)=0$,
which can be checked to be equivalent to the  condition:
$\mF({\w}^\ppr)=\mF(\w^\pr)\circ \mF(\w)$.
Therefore, we have 
$\mB\big(\underline{\w}^\pr\bullet \underline{\w}\big) = \mB\big(\underline{\w}^\pr\big)\circ  \mB\big( \underline{\w}\big)$
so that
$\mathfrak{B}: \category{sL}_\infty(R)\rightsquigarrow \category{cocdgC}(R)$ is a functor.
\naturalqed
\end{proof}

The next lemma implies that $\mathfrak{B}: \category{sL}_\infty(R)\rightsquigarrow \category{cocdgC}(R)$ is
a homotopy functor,  so that it induces a functor 
$\mathit{ho}\mathfrak{B}: ho\category{sL}_\infty(R)\rightsquigarrow 
\mathit{ho}\category{cocdgC}(R)$.

\begin{lemma}\label{flowtoflow}
For each $sL_\infty$-homotopy pair $\big({\w}(\t), {\l}(\t)\big)$  from
$\big(V,\underline{{l}}\big)$ to $\big(V^\pr,\underline{{l}}^\pr\big)$, 
define
\[
\sB\big(\big({\w}(\t), {\l}(\t)\big)\big) \coloneqq{} \big(\mF\big(\w(\t), \La\big(\w(\t),\l(\t)\big)\big).
\]
Then $\sB\big({\w}(\t), {\l}(\t)\big)$ is a homotopy pair of cocdg-coalgebras
from $\mathfrak{B}\big(V,\underline{{l}}\big)$ to $\mathfrak{B}\big(V^\pr,\underline{{l}}^\pr\big)$.
Furthermore, for composable $sL_\infty$-homotopy pairs 
\[
\xymatrix{
\big(V, {{l}}\big)\ar@{:>}[rr]^-{({\w}(\t), {\l}(\t))}&&\big(V^\pr, {{l}}^\pr\big)
\ar@{:>}[rr]^-{({\w}^\pr(\t), {\l}^\pr(\t))}&&\big(V^\ppr, {{l}}^\ppr\big)
}
\]
we have
$\sB\big(\big({\w}^\pr(\t), {\l}^\pr(\t)\big)\bullet \big({\w}(\t), {\l}(\t)\big)\big)= 
\sB\big(\big({\w}^\pr(\t), {\l}^\pr(\t)\big)\big)\circ\sB\big( \big({\w}(\t), {\l}(\t)\big)\big)$.
\end{lemma}

\begin{proof}
Set $\big({\w}^\ppr(\t), {\l}^\ppr(\t)\big)=\big({\w}^\pr(\t), {\l}^\pr(\t)\big)\bullet \big({\w}(\t), {\l}(\t)\big)$
and
\[
\begin{aligned}
{\mF}(\t) &= \mF\big(\w(\t)\big),
&
{\mF}^\pr(\t) &= \mF\big(\w^\pr(\t)\big)
,
&
{\mF}^\ppr(\t) &= \mF\big(\w^\ppr(\t)\big)
,
\\
\La(\t) &= \La\big(\w(\t),\l(\t)\big)
,
&
\La^\pr(\t) &= \La\big(\w^\pr(\t),\l^\pr(\t)\big)
,
&
\La^\ppr(\t) &= \La\big(\w^\ppr(\t),\l(\t)\big)
.
\end{aligned}
\]
From Lemma \ref{fcoalg}, we have
\begin{align*}
\Xbar{\blacktriangle}^\pr\circ \mF(\t)
-\big( \mF(\t)\otimes  \mF(\t)\big)\circ \Xbar{\blacktriangle}
=0
,\\
\Xbar{\blacktriangle}^\pr\circ \La(\t)
-\big( \La(\t)\otimes  \mF(\t)+ \mF(\t)\otimes  \La(\t)\big)\circ \Xbar{\blacktriangle}
=0
.
\end{align*}
Note that the system of equations given in Definition \ref{slinftyhflow}
for the $sL_\infty$-homotopy pair $\big({\w}(\t), {\l}(\t)\big)$
is equivalent to  $\proj_{V^\pr}\circ\left(\Fr{d}{d\t}\mF(\t)-\sD(l^\pr)\circ \La(\t)- \La(\t)\circ \sD(l)\right)=0$, 
which implies  that 
\[
\Fr{d}{d\t}\mF(\t)=\sD(l^\pr)\circ \La(\t)+\La(\t)\circ \sD(l).
\]
Therefore
$\sB\Big(\big({\w}(\t), {\l}(\t)\big)\Big)=\big({\mF}(\t), \La(\t)\big)$ is a homotopy pair of cdg-coalgebras
from $\mathfrak{B}\big(V,\underline{{l}}\big)$ to $\mathfrak{B}\big(V^\pr,\underline{{l}}^\pr\big)$.
From Lemma \ref{unihflowcom} and Lemma \ref{fcoalg},  we have
\begin{align*}
\proj_{V^\ppr}\circ\Big(\mF^\ppr(\t) - \mF^\pr(\t)\circ \mF(\t)\Big)= 0
,\\
\proj_{V^\ppr}\circ\Big(\La^\ppr(\t)  - \mF^\pr(\t)\circ \La(\t) - \La^\pr(\t)\circ \mF(\t)\Big)=0
,
\end{align*}
Working inductively, it is straightforward to show that the above conditions imply that
\[
\begin{cases}
\mF^\ppr(\t) =\mF^\pr(\t)\circ \mF(\t)
\\
\La^\ppr(\t) =\mF^\pr(\t)\circ \La(\t)
+ \La^\pr(\t)\circ \mF(\t)
\end{cases}
\] so that
\[
\Big(\mF^\ppr(\t), \La^\ppr(\t)\Big)=\Big(\mF^\pr(\t), \La^\pr(\t)\Big)\circ \Big(\mF(\t), \La(\t)\Big).
\]
Restoring the original notation, we have the desired composition relation. \naturalqed
% \[
% \sB\Big(\big({\w}^\pr(\t), {\l}^\pr(\t)\big)\bullet \big({\w}(\t), {\l}(\t)\big)\Big)= 
% \sB\Big(\big({\w}^\pr(\t), {\l}^\pr(\t)\big)\Big)\circ\sB\Big( \big({\w}(\t), {\l}(\t)\big)\Big)
% .\naturalqedmath
% \]
\end{proof}

Consider $sL_\infty$-morphisms as follows
\[
\xymatrix{
\Big(V,\underline{{l}}\Big)
\ar@{..>}@/^/[r]^{\underline{\w}}
\ar@{..>}@/_/[r]_{\underline{\tilde\w}}
&\Big(V^\pr,\underline{{l}}^\pr\Big)
\ar@{..>}@/^/[r]^{\underline{\w}^\pr}
\ar@{..>}@/_/[r]_{\underline{\tilde\w}^\pr}
&\Big(V^\ppr,\underline{{l}}^\ppr\Big),
}
\] 
and assume that $\underline{\w} \sim\underline{\tilde\w}$ and
$\underline{\w}^\pr \sim \underline{\tilde\w}^\pr$. Then the first part of Lemma~\ref{flowtoflow} implies
that $\mB(\underline{\w})$ and $\mB(\underline{\tilde\w})$ are homotopic morphisms of cocdg-algebras
from $\mathfrak{B}\big(V,\underline{{l}}\big)$ to $\mathfrak{B}\big(V^\pr,\underline{{l}}^\pr\big)$,
so that the homotopy type $\left[\mB(\underline{\w})\right]$  of  $\mB(\underline{\w})$ 
depends only the the homotopy type $\left[\underline{\w}\right]$ of $\underline{\w}$.
Define $\mathit{ho}\mB\Big(\left[\underline{\w}\right]\Big)\coloneqq{} \left[\mB(\underline{\w})\right]$.
Now the second part of Lemma~\ref{flowtoflow} implies that
 $\mB(\underline{\w}^\pr\bullet\underline{\w})= \mB(\underline{\w}^\pr)\circ \mB(\underline{\w})$ 
is homotopic to 
$\mB(\underline{\tilde\w}^\pr\bullet \underline{\tilde\w})
= \mB(\underline{\tilde\w}^\pr)\circ \mB(\underline{\tilde\w})$ 
as morphisms of cocdg-algebras from $\mathfrak{B}\big(V,\underline{{l}}\big)$ 
to $\mathfrak{B}\big(V^\ppr,\underline{{l}}^\ppr\big)$,
so that the homotopy type $\left[\mB(\underline{\w}^\pr\bullet\underline{\w})\right]$ 
of $\mB(\underline{\w}^\pr\bullet\underline{\w})$ 
depends only on the homotopy types $\left[\underline{\w}^\pr\right]$ and $\left[\underline{\w}\right]$.
Combining everything, we have 
$\mathit{ho}\mB\Big(\left[\underline{\w}\right]\bullet_h\left[\underline{\w}\right]\Big)
=\mathit{ho}\mB\Big(\left[\underline{\w}\right]\Big)\circ_h \mathit{ho}\mB\Big(\left[\underline{\w}\right]\Big)$.
Therefore, we conclude the following.

\begin{lemma}
$\mathfrak{B}: \category{sL}_\infty(R)\rightsquigarrow \category{cocdgC}(R)$ is a homotopy functor: it induces a functor 
$\mathit{ho}\mathfrak{B}: \mathit{ho}\category{sL}_\infty(R)\rightsquigarrow \mathit{ho}\category{cocdgC}(R)$.
\end{lemma}

\subsection{On complete towers of classical symmetries and the classical BV master action}
\label{appendix: classical symmetries tower}

We explain  the notion of a complete tower of infinitesimal classical symmetries as well as a recipe to construct a classical BV master action and subsequent gauge fixing.
%incorporating both such the classical symmetry, the proper notion of classical observables and subsequent gauge fixing, 
%using both  $sL_\infty$-algebra and quasi-isomorphism  as well as curved $L_\infty$-algebra. 
The key point is that every notion in {\em off-shell} classical physics should be defined modulo the classical equation of motion and coherence issues are naturally resolved using the language of $sL_\infty$-algebras.  
%We assume no originality
%and omit references for this account as it is a recipe extracts from the collective wisdom and practice of theoretical physics -
%and an apology for many original thinkers behind this.
%For a brevity, we will consider bosonic theory only.

\begin{definition}
A BV-CFT algebra is a tuple $\big(\sC, 1_\sC, \,\cdot\,, K, (\ ,\,)\big)$, where
$\big(\sC, 1_\sC, \,\cdot\,, K\big)$ is a unital CDGA and $\big(\sC, 1_\sC,  K, (\ ,\,)\big)$ is a unital sDGLA,
such that the degree $1$ bracket is a derivation of the product.
\end{definition}

A typical example of BV-CFT algebra with geometric origin can be built from a smooth manifold with a distinguished function on it.
We present a dictionary for classical field theory.

Let $\sL_{\cl}$ be the space of smooth functions on a smooth manifold  $\mL_{\cl}$ with
a distinguished element  $S_{\cl}\in \sL_{\cl}$. Let $\sC_\cl = \cdots\oplus \sC_\cl^{-2}\oplus \sC_\cl^{-1}\oplus \sC^0_\cl$
be the $\Z$-graded space of smooth poly-vector fields on $\mL_{\cl}$, where $\sC_\cl^{-k} =\G\big(\mL_{\cl},\La^k T_{\mL_{\cl}}\big)$. We regard a $k$-poly vector field as an element of ghost number $-k$ in $\sC_\cl$. Note that $\sC_\cl^0=\sL_\cl$.
Then $\sC_\cl$ has the structure $\big(\sC_\cl, 1_{\sC_\cl}, \,\cdot\,, K_\cl, (\ ,\,)_\cl\big)$ of a BV-CFT algebra,
where the bracket $(\ ,\,)_\cl$ is the Schoutens--Nijenhuis bracket, the differential is defined by 
$K_\cl =\big(S_\cl, \,\big)_\cl$, the product $\cdot$ is the exterior product and the unit $1_{\sC_\cl}$ is the
constant function on $\mL_\cl$ with the value $1$. 
Recall that the Schoutens--Nijenhuis bracket
is the unique extension of the Lie bracket to the vector fields on $\mL_\cl$ as a derivation of the exterior product
of poly-vector fields. It follows that $K_\cl\circ K_\cl=0$, since $\big(S_\cl, S_\cl\big)_\cl=0$, and $K_\cl$ is a derivation
of both the product and the bracket.
Equivalently, we may regard $\sC_\cl$ as the space of smooth
functions on the total space $T^*[-1]\mL_{\cl}$ of the cotangent bundle to $\mL_{\cl}$  after twisting the fiber by ghost number $-1$.
We remark that $T^*[-1]\mL_{\cl}$ has a  canonical odd symplectic structure of ghost number $-1$, whose associated odd Poisson bracket
of ghost number $1$ is the bracket $(\ ,\,)_\cl$.

We regard $\mL_{\cl}$ as the space of {\em classical fields} and $S_\cl$ as the {\em classical action}  of a classical field theory.
Choose Darboux coordinates $\{z^I| z^\bullet_I\}$ of $T^*[-1]\mL_{\cl}$ with
$\gh(z^I)=0$ and $\gh(z^\bullet_I)=-1$, and call $\{z^I\}$ the classical fields and
$\{z^\bullet_I\}$ the anti-fields for the classical fields.  Then, the differential $K_{\cl}$ can be expressed
as  
\eqn{appcla}{
K_{\cl} = \left( \Fr{\d S_{\cl}}{\d z^I}\right)\Fr{\d}{\d z^\bullet_I}
,
}
where we use deWitt--Einstein notation.
Regarded as an odd vector field on $T^*[-1]\mL_{\cl}$, 
the vanishing loci
of $K_{\cl}$ is the solution space  $\mL_{\mathit{onshell}}\subset \mL_{\cl}$
of the {\em classical equation of motion}: 
\eqn{appclb}{
\Fr{\d S_{\cl}}{\d z^I}=0,\quad \forall I
.
}
In general, any element in $\Im K_\cl\cap \sC_\cl$
vanishes by the classical equation of motion.

Consider the cohomology $H_\cl$ of the cochain complex $\big(\sC_{\cl}, K_{\cl}\big)$. 
Then, by Lemma \ref{htrans}, there is
a minimal $sL_\infty$-structure $\big(H_{\cl}, 0, \grave{\ell}_2, \grave{\ell}_3,\ldots\big)$ on $H_{\cl}$
and an $sL_\infty$-quasi-isomorphism 
$\xymatrix{\underline{\w}:\big(H_{\cl}, 0, \grave{\ell}_2, \grave{\ell}_3,\ldots\big)
\ar@{..>}[r]& \big(\sC_{\cl},  K_{\cl},(\ ,\,)_\cl\big)}$. 
This is closely related with the notion of a complete tower of infinitesimal classical symmetries.

We begin with a physical interpretation of the cohomology $H_\cl=\cdots\oplus H^{-2}_\cl\oplus H^{-1}_\cl\oplus H^{0}_\cl$.
\begin{itemize}

\item
We say two elements
$O_\cl$ and $\tilde O_\cl$ of $\sL_\cl$ are equivalent if 
$\tilde  O_\cl - O_\cl = \l(z)^I \Fr{\d S_{\cl}}{\d z^I}$ for some $\{\l(z)^I\} \in  \sL_\cl$
--- they induce the same function on $\mL_{\mathit{onshell}}$. 
Then,  $H^0_{\cl}$ is exactly the set of such equivalence classes:   
every element in $\sC^0_\cl=\sL_\cl$ belongs to $\ker K_\cl$,
and any $\La \in \sC^{-1}_\cl$ is in the form $\La =\l(z)^I z^\bullet_I$ and $K_\cl \La= \l(z)^I \Fr{\d S_{\cl}}{\d z^I}$.
Note that $\big(\sC_\cl, 1_{\sC_\cl}, \,\cdot\,, K_\cl\big)$ is a unital CDGA, which induces
the structure $\big(H^0_\cl, 1_{H_\cl}, \grave{m}_2\big)$ of a unital commutative and associative algebra on $H^0_\cl$,
which is viewed as the algebra $\sL_{\mathit{onshell}}$ of functions on
the solution space
 $\mL_{\mathit{onshell}}$ of the classical equation of motion.

\medskip
\item 
An element $R=R(z)^I z^\bullet_I \in \sC^{-1}_\cl$ is called an infinitesimal symmetry
vector field (for the classical action $S_\cl$) if 
$K_\cl R\equiv R(z)^I \Fr{\d S_{\cl}}{\d z^I}=0$, i.e., $R \in \Ker K_\cl\cap \sC^{-1}_\cl$.
Two infinitesimal symmetry vector fields $R$ and $\tilde R$ are equivalent if 
$\tilde  R - R= \l(z)^{IJ} \Fr{\d S_{\cl}}{\d z^J}$, $\l(z)^{IJ}=-\l(z)^{JI}$ 
--- they induce the same vector fields on $\mL_{\mathit{onshell}}$.
Then  the cohomology group $H^{-1}_{\cl}$ is exactly the set of such equivalence classes: 
any $\La \in \sC^{-2}_\cl$ is in the form $\La =\Fr{1}{2}\l(z)^{IJ} z^\bullet_I z^\bullet_J$ and 
$K_\cl \La= \l(z)^{IJ} \Fr{\d S_{\cl}}{\d z^J}z^\bullet_I$.
Note that $\big(H^{-1}_\cl, \grave{\ell}_2\big)$ is a Lie algebra, which should be the Lie algebra of gauge symmetries of $S_{\cl}$.
It is straightforward to check that  $H^0_\cl$ is a module of the Lie algebra $H^{-1}_\cl$ with the action
$\grave{\vr}:H^{-1}_\cl\times H^0_\cl \rightarrow H^0_\cl$ given by 
$(\xi, \upsilon)\mapsto \grave{\vr}(\xi)(\upsilon)=\xi.\upsilon\coloneqq{}\grave{\ell}_2(\xi, \upsilon)$.
In fact, the entire space $H_\cl$ is a $\Z$-graded module over the Lie algebra $H^{-1}_\cl$ with a similarly defined action.
\end{itemize}

In general, we call 
$H^0_\cl$ the space of on-shell classical observables,
$H^{-1}_\cl$ the space of gauge symmetries, $H^{-2}_\cl$ the space of symmetries of the gauge symmetry, etc.

\begin{definition}
A \emph{tower of infinitesimal classical symmetries} of a BV-CFT algebra is a minimal $sL_\infty$-algebra
$\big(\mg, \ell_2^\mg,\ell_3^\mg,\ldots\big)$ together with an $sL_\infty$-morphism
$\underline{\r}=\r_1,\r_2,\ldots$ to  $\big({\sC}_{\cl},  K_{\cl},(\ ,\,)_{\cl}\big)$. 
Such a tower is \emph{complete} if  $\r_1$ induces an isomorphism from 
$\mg$ to $\Xbar{H}_{\cl}\coloneqq{} \cdots\oplus H^{-2}_{\cl}\oplus H^{-1}_{\cl}$.
\end{definition}

Assume, for demonstrative purposes  that $\mg$ is concentrated in degree $-1$ 
so that $\mg$ is a Lie algebra with bracket  $\ell_2^\mg$. 

\begin{itemize}
\item
The first condition for 
$\underline{\r}:\big(\mg, \ell_2^\mg\big)
\rightarrow \big({\sC}_{\cl},  K_{\cl},(\ ,\,)_{\cl}\big)$ 
to be an $sL_\infty$-morphism is 
\eqn{appclc}{K_\cl\circ \r_1=0,
}
and the image of $\r_1:\mg \rightarrow \sC_\cl$ lies on $\sC^{-1}_{\cl}$.
The condition  $K_\cl\circ \r_1=0$
is equivalent to the condition that $\vr(g)\big(S_\cl)=\big(\r_1(g), S_\cl\big)=0$.
Therefore, 
$\r_1(g) \in \Ker K_\cl\cap\sC^{-1}_\cl$ is an infinitesimal symmetry vector field of the classical
action $S_\cl$ for all $g \in \mg$.  Recall that
$\sC^{-1}_{\cl}$ is the space of vector fields on $\mL_\cl$, which is equivalent to the space $\Der\big(\sL_\cl\big)
\subset\End\big(\sL_\cl\big)$ of derivations of $\sL_{\cl}$.  Therefore $\r_1$ induces a
linear map $\vr : \mg \rightarrow \End\big(\sL_\cl\big)$ defined for all $g \in \mg$ and $O_\cl \in \sL_\cl$ by the equation
\eqn{appcld}{
\vr(g)\big(O_\cl)\coloneqq{} \big(\r_1(g), O_\cl\big)_\cl
.
}

\medskip

\item
The second condition for $\underline{\r}=\r_1,\r_2,\ldots$ to be an $sL_\infty$-morphism is that,  for all $g_1,g_2 \in \mg$,
\eqn{appcle}{
\r_1\big(\ell_2^\mg(g_1,g_2)\big) - \big(\r_1(g_1),\r_1(g_2)\big)_{\cl}=K_\cl \r_2(g_1,g_2).
}
The image of $\r_2: S^2\mg \rightarrow \sC_\cl$ lies on $\sC^{-1}_{\cl}$ and $K_\cl \r_2(g_1,g_2)$
vanishes by the classical equation of motion.  Therefore,  the linear map $\vr: \mg \rightarrow \End\big(\sL_\cl\big)$ 
is almost a  representation of the Lie algebra $\mg$ whose failure vanishes by  the classical equation of motion. It follows that
$\vr$ induces a representation 
$\grave{\vr}: \mg \rightarrow \End\big(\sL_{\mathit{onshell}}\big)$
of the Lie algebra $\mg$, and {\em this is exactly what is relevant for classical physics}. 

\medskip

\item
Note that the  relations in \eq{appcle} come with a coherence issue.
From the Jacobi-identities of   $\ell_2^\mg$ and $(\ ,\,)_{\cl}$, it can be checked
that $K_\cl\circ P_3=0$, where
$P_3$ is the element of $\Hom\big(S^3 \mg, \sC_\cl\big)^1$ defined to be the sum over cyclic permutations of the indices of the arguments of the expression
\[
\r_2\big(\ell^\mg_2(g_1,g_2), g_3\big)_\cl -  \big(\r_2(g_1,g_2), \r_1(g_3)\big)_\cl.%+\mathit{cyclic\; sum}.
\]
Choose the data of a strong deformation retract 
%a splitting $s: \sC_\cl\rightarrow \sC_\cl$ and let 
$\xymatrix{
\big(H_\cl,0\big)\ar@/^/[r]^f &\ar@/^/[l]^{h}\big(\sC_\cl, K_\cl\big)\ar@(ul,ur)^{s}
}$. %be the corresponding homotopy equivalence. 
Define $\zeta_3$ to be $h\circ P_3 \in \Hom\big(S^3 \mg, H_\cl\big)^1$
and $\r_3$ to be $s\circ P_3\in \Hom\big(S^3 \mg, H_\cl\big)^0$.  Then we have
$f\circ \zeta_3 = P_3  - K_\cl\circ\r_3$. Provided that $\zeta_3=0$, we have $ P_3 = K_\cl\circ\r_3$,
which is precisely the thirdc ondition for $\underline{\r}=\r_1,\r_2,\r_3,\ldots$ to be an $sL_\infty$-morphism.
In general, we have $\zeta_3\neq 0$ so that the coherence issue can not be resolved within the Lie algebra.

\end{itemize}

To summarize, the correct notion of an infinitesimal classical symmetry is supposed to be an action of a Lie algebra $\mg$ on the
solution space $\mL_{\mathit{onshell}}$ of the classical equation of motion or, equivalently, a representation
$\grave{\vr}:\mg \rightarrow \End\big(\sL_{\mathit{onshell}}\big)$ of the Lie algebra $\mg$. 
Any lifting of such a representation to an off-shell $\mL_{\cl}$ should allow a weaker notion of representation 
${\vr}:\mg \rightarrow \End\big(\sL_{\mathit{cl}}\big)$ modulo the classical equation of motion, which introduces a potentially infinite sequence of coherence issues.

Resolution of all those coherence issues can be achieved,  leading to the notion of a complete tower of infinitesimal symmetries 
as a representative of the maximal symmetry modulo equivalence. To describe this solution,
we consider the reduced cochain complex $\big(\Frozenbar{\sC}_{\cl},  K_{\cl}\big)$,
where $\Frozenbar{\sC}_{\cl}=\cdots\oplus \sC^{-2}_{\cl}\oplus \sC^{-1}_{\cl}\cap \Ker K_{\cl}$. 
It follows that the cohomology of the reduced cochain complex is  isomorphic 
to $\Xbar{H}_{\cl}= \cdots\oplus H^{-2}_{\cl}\oplus H^{-1}_{\cl}$.
From the properties that the bracket $(\ ,\,)_\cl$ has ghost number $1$ and that $K_{\cl}$ is a derivation of the bracket,
it follows that $\big(\Frozenbar{\sC}_{\cl}, K_{\cl}, (\ ,\,)_{\cl}\big)$ is also an sDGLA.
Therefore  $\Xbar{H}_{\cl}$ also admits the structure $\big(\Xbar{H}_{\cl}, \grave{\ell}_2, \grave{\ell}_3, \ldots\big)$
of  a minimal $sL_\infty$-algebra together with
an $sL_\infty$-quasi-morphism 
$\underline{\w}: \big(\Xbar{H}_{\cl}, \grave{\ell}_2, \grave{\ell}_3, \ldots\big)
\rightarrow \big(\Frozenbar{\sC}_{\cl},  K_{\cl},(\ ,\,)_\cl\big)$. Then this data constitutes a complete tower of infinitesimal classical
symmetries, which is supposed to be encoded by a classical BV master action $S=S_\cl +\cdots$.

Choose a homogeneous basis $\{e_a\}_{a \in  \mathscr{I}}$ of $\Xbar{H}_{\cl}$ and work out 
the set of structure constants $\big\{C_{a_1a_2}{}^b,C_{a_1a_2a_3}{}^b,\ldots\big\}$
so that 
$\grave{\ell}_n(e_{a_1},\ldots, e_{a_n}) = C_{a_1\ldots a_n}{}^b e_b$,
where $n\geq 2$ and $a_1,\ldots, a_n,b\in  \mathscr{I}$. 
Then, we consider the following super-manifold:
\[
{\mL} 
=\mL_{\cl}\times \Xbar{H}_{\cl}\times \Xbar{H}^*_{\cl}\times \Xbar{H}^*_{\cl}[1],
\]
with homogeneous affine coordinate system $\big\{\mq^A\big\}= \big\{z^I, \eta^a, \Xbar\eta_a, \l_a\}$, where 
$\gh(z^I)=0$, $\gh(\eta^a)=- \gh(e_a)$, 
$\gh(\Xbar\eta_a) = -\gh(\eta^a)$ and $\gh(\l_a) = \gh(\Xbar\eta_a) +1$.
Then the algebra $\sL$ of functions on $\mL$ is isomorphic to $\sL_{\cl}\big[\!\big[ \eta^a,\Xbar\eta_a, \l_a\big]\!\big]$.
In physics terminology,  we call
$\big\{\eta^a\big\}$  the  ghost fields,
$\big\{\Xbar\eta_a\big\}$ anti-ghost fields, and $\big\{\l_a\big\}$ auxiliary or Lagrangian multiplier fields.
In particular $\big\{\eta^a\big|\gh(\eta^a)=1\big\}$ 
is the Faddeev--Popov ghosts,  $\big\{\eta^a\big|\gh(\eta^a)=2\big\}$ the ghosts of the the Faddeev--Popov ghosts, etc.
Collectively, one calls $\big\{\mq^A\big\}$ the {\tt fields}.  

Now we consider $T^*[-1]\mL$, which has the canonical symplectic structure
$\Omega_{\mathit{BV}}$ with ghost number $-1$, whose expression in terms of Darboux coordinates
\[
\left\{q^A|q^\bullet_A\right\}=\left\{z^I, \eta^a, \Xbar\eta_a, \l_a| z^\bullet_I, \eta^\bullet_a, \Xbar\eta^a_\bullet, \l^a_\bullet\right\},
\qquad \gh(q^\bullet_A) = -\gh(q^A) -1,
\]
is $\Omega=d z^\bullet_I\wedge d z^I+ d \eta^\bullet_a \wedge d\eta^a+d \Xbar\eta^a_\bullet \wedge d\Xbar\eta_a
+d \l^a_\bullet \wedge d\l_a$. We call $q^\bullet_A$ the {\tt antifield} for the {\tt field} $q_A$.
Then,  the space 
$\sC\cong \sL_{\cl}\big[\!\big[ \eta^a, \Xbar\eta_a, \l_a, \eta^\bullet_a, \Xbar\eta^a_\bullet, \l^a_\bullet\big]\!\big]_{a \in \mathscr{J}}$
of functions on $T^*[-1]\mL$ has the structure of BV-CFT algebra $\big(\sC, 1_\sC, \,\cdot\,  (\ ,\,)_{\mathit{BV}}\big)$ with zero differential, where $(\ ,\,)_{\mathit{BV}}$ is the odd Poisson bracket associated with $\Omega_{\mathit{BV}}$
and $1_\sC = 1_{\sC_\cl}$.

Now we define a classical BV master action $S \in \sC^0$ as follows:
\eqn{appclf}{
S = S_{\cl} 
+  \sum_{\mathclap{n\geq 2}}\Fr{1}{n!}\eta^{a_1}\cdots \eta^{a_n} C_{a_1\ldots a_n}{}^b \eta^{\bullet}_b
+ \sum_{\mathclap{n\geq 1}}\Fr{1}{n!}\eta^{a_1}\cdots \eta^{a_n}\w_n(e_{a_1},\ldots, e_{a_n})
+\l_a \Xbar\eta^a_\bullet
.
}
Then, it is straightforward to check that $S$ satisfies  the so called {\em classical BV master equation}:
\eqnalign{appclg}{
(S, S)_{\mathit{BV}}&=0
,\\
S\big|_{\mL}&=S_{\cl}
.
}
Define $K \coloneqq{}(S, -)_{\mathit{BV}}$. It follows that $\big(\sC, 1_\sC, \,\cdot\, , K, (\ ,\,)_{\mathit{BV}}\big)$
is a BV-CFT algebra.

We remark that $\w_n(e_{a_1},\ldots,e_{a_n}) \in \sC_{\cl}$
and  $\gh\big(\w_n(e_{a_1},\ldots,e_{a_n})\big)=\gh(e_{a_1})+\ldots+\gh(e_{a_n})\leq -1$
for all $n\geq 1$ and $e_{a_1},\ldots,e_{a_n} \in \Xbar H_{\cl}$. 
Therefore, we have
\[
\w_n(e_{a_1},\ldots,e_{a_n})=R(z)_{a_1\ldots a_n}{}^{I_1\ldots I_k}z^\bullet_{I_1}\cdots z^\bullet_{I_k},
\]
where $k =- \big(\gh(e_{a_1})+\ldots+\gh(e_{a_n})\big)\geq 1$, and we obtain the following  more explicit form of $S$ as defined 
in \eq{appclf}:
\[
S = S_{\cl} 
+  \sum_{\mathclap{n\geq 2}}\Fr{1}{n!}\eta^{\a_1}\cdots \eta^{\a_n} C_{a_1\ldots a_n}{}^b \eta^{\bullet}_b
+ \sum_{\mathclap{n\geq 1}}\Fr{1}{n!}\eta^{a_1}\cdots \eta^{a_n}R(z)_{a_1\ldots a_n}{}^{I_1\ldots I_k}z^\bullet_{I_1}\cdots z^\bullet_{I_k}
+\l_a \Xbar\eta^a_\bullet
.
\]

Encoding a complete tower of infinitesimal classical symmetries by the classical BV master action $S$, we turn to a general gauge fixing procedure.

\begin{definition}
A \emph{gauge fermion} is an element $\psi \in \sL^{-1}$ and the \emph{gauge fixed classical action} 
$S^\psi_\cl\in \sL^{0}$ with respect to the gauge fermion $\p$ is
\begin{align*}
S^\psi_\cl 
&\coloneqq{} 
S_\cl + \sum_{\mathclap{n\geq 1}}\Fr{1}{n!}\n_n(\psi,\ldots, \psi)
\\
&=
S_\cl + \l_a \Fr{\d \psi}{\d \Xbar\eta_a}+  \sum_{\mathclap{n\geq 2}}\Fr{1}{n!}\eta^{a_1}\cdots \eta^{a_n} 
C_{a_1\ldots a_n}{}^b \Fr{\d \psi}{\d \eta^{b}}
\\
&\quad{}
+ \sum_{\mathclap{n\geq 1}}\Fr{1}{n!}\eta^{a_1}\cdots \eta^{a_n}R(z)_{a_1\ldots a_n}{}^{I_1\ldots I_k}
\left(\Fr{\d \psi}{\d z^{I_1}}\right)\cdots \left(\Fr{\d \psi}{\d z^{I_k}}\right)
.
\end{align*}
\end{definition}

\begin{remark}
Define the family $\underline{\n}=\n_0,\n_1,\n_2,\ldots$ by declaring that $\n_0 =S_{\cl}$ and,
for all $n\geq 1$ and homogeneous $\g_1,\ldots, \g_n \in \sL$,
\[
\n_n(\g_1,\ldots, \g_n) \coloneqq{}\left(\left(\cdots\left( (S, \g_1)_{\mathit{BV}},\g_2\right)_{\mathit{BV}},\cdots \right)_{\mathit{BV}}
,\g_n\right)_{\mathit{BV}}\Big|_{\mL}.
\]
Then, we have 
\[
S^\psi_\cl= \n_0 + \n_1(\psi) +\Fr{1}{2!}\n_2(\psi,\psi) +\Fr{1}{3!}\n_3(\psi,\psi,\psi) +\ldots.
\]
In fact, it is easy to check that $\big(\sL, \underline{\n}\big)$ 
is a weakly homotopy Lie algebra, also known as a curved $L_\infty$-algebra 
--- for example the differential $\n_1: \sL \rightarrow \sL$, which is often called a BRST operator,
does not satisfy $\n_1\circ \n_1 =0$ strictly but only modulo the classical equation of motion. 
\naturalqed
\end{remark}

Alternatively, we consider an element $\Psi \in \sC^{-1}$ and  the canonical transformation generated by $\Psi$.
Then, we have
\[
S\rightarrow S^\Psi = S + (S, \Psi)_{\mathit{BV}} +\Fr{1}{2!}( (S, \Psi)_{\mathit{BV}},\Psi)_{\mathit{BV}} +\ldots \in \sC^0
\]
Let $\psi= \Psi\big|_{\mL}$.
Then, from \eq{appclg} and by definitions, we obtain that
\begin{align*}
(S^\Psi, S^\Psi)_{\mathit{BV}}&=0
,\\
S^\Psi\big|_{\mL}&=S^\psi_{\cl}
.
\end{align*}
Equivalently, we can interpret the canonical transformation generated by $\Psi=\psi$ as a deformation of the Lagrangian subspace
$\mL$ into $\mL^\psi$ so that $S^\Psi\big|_{\mL}=S\big|_{\mL^\psi}$, i.e., from the zeros of the section $\mq^\bullet_A$ 
to the zeros of the section $\mq^\bullet_A - \Fr{\d \psi}{\d \mq^A}$.

\begin{example}
Now we consider a simple case, where the standard BRST-FP method \cite{FaPo,BRST} can be used.
Assume that $\Xbar{H}_{\cl} = H^{-1}_\cl$. Then, the reduced 
sDGLA $\big(\Frozenbar{\sC}, K,(\ ,\,)_{\mathit{BV}}\big)$ is formal for degree reasons. 
Then the minimal $sL_\infty$-structure on $\Xbar{H}_{\cl}$ is   $\big(H^{-1}_\cl, \grave{\ell}_2\big)$,
which is  just a Lie algebra $\mg$ --- which we assume to be the Lie algebra of  a simply connected Lie group $\mG$.
Also assume that we have an $sL_\infty$-quasi-isomorphism $\underline{\w}$ to the sDGLA such that $\w_n=0$, for all $n\geq 2$.
Note that $\w_1(e_{a})=R(z)_a{}^I z^\bullet_I \in \sC^{-1}$.
Then the classical BV master action $S$ is reduced to the following simple form
\[
S = S_{\cl} 
+\l_a \Xbar\eta^a_\bullet 
+\Fr{1}{2}\eta^{a_1} \eta^{a_2} f_{a_1a_2}{}^b \eta^{\bullet}_b
+ \eta^{a}R(z)_{a}{}^{I}z^\bullet_{I}
.
\]
where $\gh(\eta^{a})=1$, $\gh(\Xbar\eta_{a})=-1$ and $\gh(\l_a)=0$. Then we have
$\big(\sL, \underline{\n}\big)$ where $\n_1= S_{\cl}$, $\n_1=\d_{\mathit{BRST}}$ and $\n_k=0$ for all $k\geq 3$,
where $\d_{\mathit{BRST}}$ is the so-called BRST operator given by
\[
\d_{\mathit{BRST}} = 
\l_a \Fr{\d}{\d \Xbar\eta_a} + \d_{\mathit{Lie}},\qquad
\d_{\mathit{Lie}}=\Fr{1}{2}\eta^{a_1} \eta^{a_2} f_{a_1a_2}{}^b \Fr{\d}{\d  \eta^b}
+ \eta^{a}R(z)_{a}{}^{I}\Fr{\d}{\d z^I}
\]
Then we have $\d_{\mathit{BRST}}\circ \d_{\mathit{BRST}}=\d_{\mathit{Lie}}\circ \d_{\mathit{Lie}}=0$
and $\d_{\mathit{BRST}} S_{\cl}\equiv \d_{\mathit{Lie}} S_{\cl}=0$.
Note that $\d_{\mathit{Lie}}$ is the differential for the standard Lie algebra cohomology.
In coordinates, any gauge fermion can be written $\psi = \Xbar\eta_a G^a(z)$. It follows
that 
\[
S^\psi_{\cl} = S_{\cl}  +\d_{\mathit{BRST}}\psi 
= 
S_{\cl}  
+\l_a G(z)^a
+ \eta^{a}R(z)_{a}{}^{I}\Fr{\d G(z)^b}{\d z^I}\Xbar\eta_b,
\]
which is precisely the gauge fixed action according to the Fadeev--Popov procedure:
the integral over $\{\l_a\}$ imposes the constraint $\{G(z)^a=0\}$ --- a gauge fixing ---
and the integral over $\{\eta^a, \Xbar\eta_a\}$ induces the so called Fadeev--Popov determinant,
which are the two ingredients for  constructing a quotient measure on $\mL_\cl\big/\mG$.
\naturalqed
\end{example}

Finally, we refer to \cite{Schwarz} for the geometry of quantum BV master action
and \cite{Cos} for  many subtleties in dealing with infinite dimensional spaces of fields from both the classical and quantum perspectives.   For a passage from the idea of infinitesimal symmetry of quantum expectation to binary QFT algebra, we refer
to  Section $5$ in \cite{Park15}, where $\kbar=1$.

\end{appendix}

\newpage

\renewcommand\refname{Bibliography}

\end{document}